\documentclass[11pt]{amsart}

\usepackage{amssymb,amsfonts,latexsym}
\usepackage[centertags]{amsmath}
\usepackage{amsfonts}
\usepackage{amssymb}
\usepackage{amsthm}
\usepackage[all]{xy}
\usepackage{cite}
\usepackage{graphicx,color}
\usepackage{todo}

\newtheorem{cont}{cont}[section]
\newtheorem{theorem}[cont]{Theorem}
\newtheorem{proposition}[cont]{Proposition}
\newtheorem{lemma}[cont]{Lemma}
\newtheorem{corollary}[cont]{Corollary}
\newtheorem{definition}[cont]{Definition}
\newtheorem{pblm}[cont]{Problem}
\newtheorem{question}[cont]{Question}

\newtheorem*{Enunciato*}{Enunciato}
\newtheorem{conj}[cont]{Conjecture}

\numberwithin{equation}{section}
\newtheorem{remark}[cont]{Remark}
\newtheorem{example}[cont]{Example}
\newtheorem*{not*}{Notation}

\newcommand{\cO}{{\mathcal O}}
\newcommand{\shF}{\mathcal{F}}

\newcommand{\shG}{\mathcal{G}}

\newcommand{\cH}{\mathcal{H}}
\newcommand{\cA}{\mathcal{A}}
\newcommand{\cB}{\mathcal{B}}
\newcommand{\cC}{\mathcal{C}}
\newcommand{\cD}{\mathcal{D}}
\newcommand{\shI}{\mathcal{I}}

\newcommand{\PP}{\mathbb{P}}
\newcommand{\HH}{\mathbb{H}}

\newcommand{\ZZ}{\mathbb{Z}}

\newcommand{\QQ}{\mathbb{Q}}

\newcommand{\odi}[1]{\mathcal{O}_{#1}}

\DeclareMathOperator{\GradAlg}{GradAlg}
\DeclareMathOperator{\Hl}{H} 
\DeclareMathOperator{\Hilb}{Hilb} 
 \DeclareMathOperator{\Hom}{Hom}
\DeclareMathOperator{\Der}{Der}

\DeclareMathOperator{\Spec}{Spec}
\DeclareMathOperator{\depth}{depth}
\DeclareMathOperator{\codepth}{codepth}
\DeclareMathOperator{\Proj}{Proj} 
\DeclareMathOperator{\codim}{codim}
\DeclareMathOperator{\coker}{coker}
\DeclareMathOperator{\Ext}{Ext} \DeclareMathOperator{\pd}{pd}
\DeclareMathOperator{\Tor}{Tor} \DeclareMathOperator{\ext}{ext}
 \DeclareMathOperator{\im}{im}

\begin{document}
\title{Deformation and unobstructedness of determinantal schemes}

\author[Jan O.\ Kleppe, Rosa M.\ Mir\'o-Roig]{Jan O.\ kleppe, Rosa M.\
Mir\'o-Roig$^{*}$}
\address{Oslo Metropolitan University,
  Faculty of Technology, Art and Design, PB 4 St. Olavs plass, N-0130
  Oslo, Norway}
\email{jank@oslomet.no}
\urladdr{https://www.cs.hioa.no/\textasciitilde jank/}
\address{Facultat de Matem\`atiques i Inform\`{a}tica,
Departament de Matem\`{a}tiques i Inform\`{a}tica, Gran Via de les Corts Catalanes
585, 08007 Barcelona, SPAIN }
\urladdr{https://webgrec.ub.edu/webpages/000006/ang/miro.ub.edu.html}
 \email{miro@ub.edu}

\date{\today}
\thanks{$^*$ Partially supported by MTM2016-78623-P and PID2019-104844GB-100
\\ {\it 2010 Mathematic Subject Classification.} 14M12,14C05,14H10,14J10.}

\begin{abstract}
  A closed subscheme $X\subset \PP^n$ 
  is said to be {\em determinantal} if its homogeneous saturated ideal can be
  generated by the $s\times s$ minors of a homogeneous $p\times q$ matrix
  satisfying $(p-s+1)(q-s+1)=n - \dim X$ and it is said to be {\em standard
    determinantal} if, in addition, $s=\min(p,q)$. Given integers
  $a_1\le a_2\le \cdots \le a_{t+c-1}$ and $b_1\le b_2 \le \cdots \le b_t$ we
  consider $t\times (t+c-1)$ matrices $\cA=(f_{ij})$ with entries
  homogeneous forms of degree $a_j-b_i$ and we denote by
  $\overline{W(\underline{b};\underline{a};r)}$ the closure of the locus
  $W(\underline{b};\underline{a};r)\subset \Hilb ^{p(t)}(\PP^{n})$ of
  determinantal schemes defined by the vanishing of the
  $(t-r+1)\times (t- r+1)$ minors of such $\cA$ for $\max\{1,2-c\} \le r < t$.
  $W(\underline{b};\underline{a};r)$  is an irreducible algebraic set.

  First of all, we compute an upper $r$-independent bound for the dimension of
  $W(\underline{b};\underline{a};r)$ in terms of $a_j$ and $b_i$ which is
  sharp for $r=1$. In the linear case ($a_j = 1, b_i=0$) and cases
  sufficiently close, we conjecture and to a certain degree prove that this
  bound is achieved for all $r$. Then, we study to what extent the family
  $W(\underline{b};\underline{a};r)$ fills in a generically smooth open
  subset of the corresponding component of the Hilbert scheme
  $\Hilb ^{p(t)}(\PP^{n})$ of closed subschemes of $\PP^n$ with Hilbert
  polynomial $p(t)\in \QQ[t]$. Under some weak numerical assumptions on the
  integers $a_j$ and $b_i$ (or under some depth conditions) we conjecture and
  often prove that $\overline{W(\underline{b};\underline{a};r)}$ is a
  generically smooth component. Moreover, we also study the depth of the
  normal module of the homogeneous coordinate ring of
  $(X)\in W(\underline{b};\underline{a};r)$ and of a closely related module.
  We conjecture, and in some cases prove, that their codepth is often 1 (resp.
  $r$). These results extend previous results on {\em standard determinantal}
  schemes to {\em determinantal} schemes; i.e. previous results of the authors
  on $W(\underline{b};\underline{a};1)$ to $W(\underline{b};\underline{a};r)$
  with $1\le r < t$ and $c\ge 2-r$. Finally, deformations of exterior powers
  of the cokernel of the map determined by $\cA$ are studied and proven to be
  given as deformations of $X \subset \PP^n$ if $\dim X \ge 3$.

  The work contains many examples which illustrate the results obtained and a
  considerable number of open problems; some of them are collected as
  conjectures in the final section.
 \end{abstract}

\maketitle

\tableofcontents

\section{Introduction}
In this monograph, we generalize previous results on deformation of standard
determinantal schemes to cover deformations of determinantal schemes. Recall
that a codimension $c$ subscheme $X\subset \PP^n$ is called {\em
  determinantal} if its homogeneous ideal $I(X)\subset R:=k[x_0.,\cdots ,x_n]$
can be generated by the $s\times s$ minors of a homogeneous $p\times q$ matrix
$\cA$ and $c=(p-s+1)(q-s+1)$. $X$ is said to be {\em standard determinantal}
if, in addition,  $s=\min(p,q)$. We would like to parameterize all determinantal schemes $X$
by looking at $X$ as a point $(X)$ of a component of the Hilbert scheme
$\Hilb ^{p_X(t)}(\PP^n)$, to study to what extent the family
$W(\underline{b};\underline{a};r)$ of determinantal schemes fills in this
component and whether $\Hilb ^{p_X(t)}(\PP^n)$ is generically smooth along
$W(\underline{b};\underline{a};r)$. The locus of $\Hilb ^{p_X(t)}(\PP^n)$ along
standard determinantal schemes is quite well understood and in this work we
study carefully 
the structure and dimension of $\Hilb ^{p_X(t)}(\PP^n)$ along determinantal
schemes.

A large and important class of classical varieties are cut by minors of a
homogeneous matrix: Veronese varieties, Segre varieties, rational normal
curves, Bordiga surface, Palatini scrolls, certain varieties of quasi-minimal
degree, and even more any variety is isomorphic to a determinantal variety
given by a matrix with linear entries (see \cite[Pg. 112]{Ha}). Determinantal
schemes have been a central topic in both algebraic geometry and commutative
algebra and their study has received considerable attention in the last
decades. For instance, in \cite{EKS}, Eisenbud, Koh and Stillman proved that
the homogeneous ideal of any curve $C\subset \PP^r$ of degree $d$ and genus
$g$ with $d\ge 4g+2$ is generated by the $2\times 2$ minors of a matrix with
linear entries. A quite recent result which shows the importance of
determinantal schemes is due to Sidman and Smith. In \cite[Theorem 1.1]{SS},
they prove that a sufficiently ample line bundle on a connected scheme $X$ is
determinantally presented. Here a property for a sufficiently
ample line bundle on $X$ holds if there is a line bundle $E$ such that the
property holds for all $L\in Pic(X)$ for which $L\otimes E^{-1}$ is ample and,
moreover, given a scheme $X$ embedded in $\PP ^n$ by a complete linear system
$L$ we say that $L$ is {\em determinantally presented} if $I(X)$ can be
generated by the $2\times 2$ minors of a 1-generic matrix.

As we have just said the determinantal ideals $I_s(\cA )$ generated by the
$s\times s$ minors of a $p\times q$ homogeneous matrix $\cA $ have been
extensively studied by many people. One of the first important results is due
to Eagon and Hochster who proved that $I_s(\cA)$ is a Cohen-Macaulay ideal
\cite{e-h}; hence $R/I_s(\cA )$ has a minimal free $R$-resolution of length
equal to $n+1-\dim R/I_s(\cA )$. To find explicitly such a minimal free
resolution is a problem with a long history behind it. For $s=1$,
$s=\min(p,q)$ or $s=\min(p,q)-1$ such minimal free $R$-resolution is given by
the Koszul complex, the Eagon-Northcott complex \cite{e-n} and the
Akin-Buchsbaum-Weyman complex \cite{ABW0}, respectively; while the first in
giving a minimal free $R$-resolution of $I_s(\cA)$ for any
$1\le s\le \min(p,q)$ was Lascoux in \cite{L}. All these results are crucial in
our work.

Given integers $a_1,a_2,...,a_{t+c-1}$ and $b_1,...,b_t$ we denote
by $W(\underline{b};\underline{a};r)\subset \Hilb ^{p_X(t)}(\PP^{n})$
 the locus of  determinantal schemes $X\subset \PP^{n}$  defined by the $(t-r+1)\times (t-r+1)$ minors of a $t\times
(t+c-1)$ matrix $\cA=(f_{ij})^{i=1,...,t}_{j=1,...,t+c-1}$ where
$f_{ij}\in k[x_0,x_1,...,x_{n}]$ is a homogeneous polynomial of
degree $a_j-b_i$. $W(\underline{b};\underline{a};r)$
 is an irreducible algebraic set by Lemma \ref{lemaaux2}.

In this paper, we address the following four fundamental problems:

\begin{pblm} \label{MainPblms}

\begin{itemize}
\item[(1)] To determine  the dimension of
$W(\underline{b};\underline{a};r)$  in terms of $a_j$ and $b_i$ for all $r$.
\item[(2)] To determine whether the closure of
  $W(\underline{b};\underline{a};r)$ is an irreducible component of
  $\Hilb ^{p_X(t)}(\PP^{n})$.
\item[(3)] To determine  when $\Hilb
^{p_X(t)}(\PP^{n})$ is generically smooth along
$W(\underline{b};\underline{a};r)$.
\item[(4)] To determine whether any deformation of $X$ with
  $(X) \in W(\underline{b};\underline{a};r)$ comes from deforming its
  associated homogeneous matrix $\cA$.
\end{itemize}
\end{pblm}

Due to  Lemma~\ref{unobst}, if Problem \ref{MainPblms}{\rm (4)}
  holds,  Problems \ref{MainPblms}{\rm (2)} and {\rm (3)} also hold.

  The first important contribution to these problems was made in 1975 by
  Ellingsrud \cite{E}. He proved that any arithmetically Cohen–Macaulay,
  closed subscheme $X\subset \PP^n$, $n\ge 3$, of codimension 2 is
  unobstructed and computed the dimension of the Hilbert scheme at $(X)$ (see
  \cite{F} for the case $n=2$). The purpose of this work is to extend
  Ellingsrud's Theorem, viewed as a statement on standard determinantal
  schemes of codimension 2, to arbitrary determinantal schemes and to show that, in general, the component of the Hilbert scheme which parameterizes determinantal schemes behaves well, contrary to what may happen more generally as predicted by Vakil's Murphy's law for singularities of the Hilbert scheme (see \cite{Va}). The case of
  codimension $c$ standard determinantal schemes (i.e. the case $r=1$) was
  mainly solved in \cite{KMMNP} for $c=3$, in \cite{KM2005} for $c=4$ and
  partially for $c = 5$, and some cases when $c \ge 6$ and in \cite{K2014} for
  arbitrary $c$ (See also \cite{FF}, \cite{K2011}, \cite{K2018}, and
  \cite{KM2011} for more details). As our results and conjectures in this
  paper show, we think it is possible to solve the above four problems in
  full generality also for $r>1$, provided $\dim X > 2$. In \cite{KM2009} we
  focused our attention on the first unsolved case and we dealt with
  codimension-4 determinantal schemes $X\subset \PP^n$ defined by the
  submaximal minors of a homogeneous square matrix. In this monograph, we will
  address the general case following the ideas developed by the authors in
  \cite{KM2009}, see also \cite{KMMNP, KM2005, KM2011}. Indeed, we will prove
  our results by considering the smoothness of the Hilbert flag scheme of
  chains of closed subschemes obtained by deleting suitable columns and its
  natural projections into the usual Hilbert schemes. Then we recursively
  prove that the closure of $W(\underline{b};\underline{a};r)$ is a
  generically smooth component of $\Hilb ^{p_X(t)}(\PP^{n})$ under some mild
  numerical assumptions by starting the induction from the case where
  $W(\underline{b};\underline{a};1)$ parameterizes standard determinantal
  schemes (i.e. schemes defined by maximal minors). Our proof is based on the
  following key points: (i) A deep analysis of whether any deformation of a
  determinantal scheme $X$ of $W(\underline{b};\underline{a};r)$ comes from
  deforming its associated matrix, and (ii) the fact that any determinantal
  scheme $X=\Proj(A)$ 
  is defined by a regular section of a "nice" sheaf $\tilde
  {N}$ on a determinantal scheme
  $Y=\Proj(B)$ of lower codimension. More precisely 
let $\varphi:\oplus _{i=1}^tR(b_i)\longrightarrow \oplus _{j=1}^{t+c-1}R(a_j)$
be the morphism of free graded $R$-modules induced by the transpose of $\cA$
and
$\varphi _{t+c-2}:\oplus _{i=1}^tR(b_i)\longrightarrow \oplus
_{j=1}^{t+c-2}R(a_j)$
the morphism obtained deleting the last column of $\cA$. Set
$A=R/I_{t+1-r}(\varphi ^*)$, $B=R/I_{t+1-r}(\varphi ^*_{t+c-2})$,
$MI=\coker (\varphi^*)$ and $N=\coker (\varphi _{t+c-2}^*)$. Then there is a
regular section $\sigma ^*:B(-a_{t+c-1})\longrightarrow N\otimes B$ fitting
into an exact sequence
\begin{equation*}
0\longrightarrow B(-a_{t+c-1})\stackrel{\sigma ^* }{\longrightarrow} N\otimes
B \longrightarrow MI\otimes B\longrightarrow 0\, ,
\end{equation*}
such that $A=B/\im
(\sigma )$. Here $\sigma ^* $ is locally given by all minors of size
$t-r+1$
involving the deleted row. Such an exact sequence with properties as described
in Theorem~\ref{mainthm} has turned out to be very useful in the case of
maximal minors, not only in several of our papers where we often consider
deformation problems, but also for other purely algebraic classification
problems (\cite{KMNP}). In \cite{KM2009} this was generalized to submaximal
minors ($r=2$)
when $c=1$,
but now in full generality for minors of any size, only requiring $\dim
Y \ge
r$ and the codimension of the singular locus being at least 2. Our result
extends to cover even the artinian case of $A$,
as well as the case where the assumption $k$
a field, is replaced by $k$
a local artinian ring over which (the lifting of) $B$
is flat and allows a regular section of the corresponding $N\otimes
B$ (see Theorem~\ref{teo3} and the end of its proof for details). This is
sufficient for showing the just mentioned main result in (i) above, namely
that under some assumptions, any deformation of a determinantal scheme
$X$
comes from deforming its associated matrix (Theorem~\ref{teo3}). The key
results (i) and (ii) together with important results of Bruns in \cite{B} on
the maximal Cohen-Macaulayness of $N\otimes
B$ and its $B$-dual
$M$
which we prove lead to the vanishing of several $\Ext^i_B$-groups
involving $N\otimes
B$, $M$, $B$ and $I_{A/B} := \ker(B \to
A)$ (Propositions~\ref{vanish} and \ref{vanish2}), provide the basis upon
which we are able to partially solve Problem \ref{MainPblms}. For instance the
vanishing of these $\Ext^i_B$-groups
rather immediately solves Problems \ref{MainPblms}(1)-(3), without using (4),
for the case of submaximal minors satisfying $
_0\!\Ext^1_B(I_B/I^2_B, I_{A/B})=0$ (Proposition~\ref {propoA}).

\vskip 4mm

Next we outline the structure of the paper. Section 2 of the paper provides
background and preliminary results needed later on. Section 3 is inserted both
for sake of completeness and for including slightly new generalizations. It
contains a summary of the main results on deformation and unobstructedness of
{\em standard determinantal} schemes. In the next sections we will see that in
many cases the same behaviour can be established for {\em all} determinantal
schemes. The main result of section 4 states that any determinantal scheme
$X=\Proj(A)$ can be defined as the degeneracy locus of a regular section of a
sheaf $\tilde{N}$ on a determinantal scheme $Y=\Proj(B)$ of lower codimension
(Theorem \ref{mainthm}). In this section, we also show that if
$\depth_{J_A}A \ge 4$, $J_A:=I_{t-r}(\varphi ^*)$ then
$\Hom_B(I_{A/B},A) \cong MI\otimes A(a_{t+c-1})$ and
$ \Ext^1_A(I_{A/B}/I_{A/B}^2,A)=0$ (Proposition~\ref{vanish}). Since we
recursively want to transfer properties (dimension and smoothness) of
$ \Hilb ^{p_Y(t)}(\PP^{n})$ at $(Y)$ to $ \Hilb ^{p_X(t)}(\PP^{n})$ at $(X)$
the above result is important because it implies that the $2^{nd}$ projection,
$$p_2: \Hilb ^{p_X(t),p_Y(t)}(\PP^{n})\longrightarrow \Hilb ^{p_Y(t)}(\PP^{n})
\ { \ \rm given \ by \ } \ (X' \subset Y') \mapsto (Y') \, , $$
defined over the Hilbert-flag scheme, is smooth at $(X \subset Y)$ and with
tangential fiber dimension $ \dim (MI\otimes A)_{a_{t+c-1}}$.

We start section 5 characterizing when any deformation of a determinantal ring
$A$ comes from deforming its associated homogeneous matrix $\cA $ in terms of
the surjectivity of its tangent map (Lemma \ref{lemma2}). As an application we
get that it holds for so-called {\em generic determinantal} rings
$A_{(s)}:=R/I_s(\cA)$, $s=t+1-r$ where we have $R=k[x_{ij}]$, $1\le i \le t$,
$1\le j \le t+c-1$, and $\cA=(x_{ij})$ the $t\times (t+c-1)$ matrix of
indeterminates of $R$. Indeed every deformation of $A_{(s)}$ comes from
deforming $\cA=(x_{ij})$ provided $(s,c)\ne (t,1)$ (Proposition
\ref{deforminggenericcase}). Then we prove the main result of this section,
namely, Theorem~\ref{teo3} which
more precisely states that if the following property: ``any deformation of a
determinantal ring comes from deforming its associated matrix'', holds for
$B$, it also holds for $A$ provided $ _0\!\Ext^1_B(I_B/I^2_B, I_{A/B})=0$.
Thus it will be important to show the vanishing of
$ _0\!\Ext^1_B(I_B/I^2_B, I_{A/B})$, also because
one knows that its vanishing implies that the $1^{st}$
projection
$$p_1: \Hilb ^{p_X(t),p_Y(t)}(\PP^{n})\longrightarrow \Hilb ^{p_X(t)}(\PP^{n})
\ {\rm \ \ given \ by \ } \ \ (X' \subset Y') \mapsto (X') \, $$
defined over the Hilbert-flag scheme, is smooth at $(X \subset Y)$ and with
tangential fiber dimension \
$ \dim\,_0\!\Hom_B(I_B/I^2_B, I_{A/B})$.

In section 6, we address Problem \ref{MainPblms}(1). We fix integers
$\underline{b}=(b_1,...,b_t)$,
$\underline{a}=(a_1,a_2,...,a_{t+c-1})$,
$\underline{a'}=(a_1,\cdots ,a_{t+c-2})$ and
$1\le r < t$ and we give an upper $r$-independent bound for
$\dim W(\underline{b};\underline{a};r)$ in terms of $a_{i}$ and $b_{j}$
(Theorem \ref{ineqdimW}). The bound is achieved for $r=1$ provided
$a_{i-1}>b_{i}$ for $2\le i\le t$. We carefully analyze under which numerical
hypothesis the bound is also achieved for $r>1$. To do so,
we need to compute $ \dim (MI\otimes A)_{a_{t+c-1}}$ and
$ \dim\,_0\!\Hom_B(I_B/I^2_B, I_{A/B})$. We compute the former under the
assumption
$$(*): a_{t+c-1} -b_1 < \sum _{i=1}^{t-r+1}(a_i-b_{r+i-1})\ ,$$
which holds in the linear case and in cases sufficiently close. For the latter
we succeed in showing the expected formula
$\dim\, _0\!\Hom_B(I_B/I^2_B, I_{A/B})=\sum _{j=1}^{t+c-2}
{a_j-a_{t+c-1}+n\choose n}$
under some assumptions (e.g. $b_t = b_1 < a_1$ and $a_{t-r+1} < a_{t+c-1}$)
leading to the main result of this section (Theorem~\ref{dimW}) which states
that if also every deformation of $\Proj(B)$ of
$W(\underline{b};\underline{a'};r)$ comes from deforming its matrix,
$\dim B > r+2$, $ \dim W(\underline{b};\underline{a'};r)=\lambda _{c-1}$ and
$(*)$ holds, then $\dim W(\underline{b};\underline{a};r) = \lambda_c\, $ where
$$\lambda_c=\sum_{i,j}
    \binom{a_i-b_j+n}{n} + \sum_{i,j} \binom{b_j-a_i+n}{n} - \sum _{i,j}
    \binom{a_i-a_j+n}{n}- \sum _{i,j} \binom{b_i-b_j+n}{n} + 1\ .
    $$
    The expected formula for $\dim\, _0\!\Hom_B(I_B/I^2_B, I_{A/B})$, which
    leads to $\dim W(\underline{b};\underline{a};r) = \lambda_c$, remains
    conjectural if $\dim B > r+2$, while for $\dim B = r+2$ and $r=2$, i.e.
    $\dim A= 2$ there are counterexamples for which
    $\dim W(\underline{b};\underline{a};r) = \lambda_c-1$, see
    Remark~\ref{remconjdimW} and Example~\ref{qmdeg}. Moreover we include a
    lot of examples to illustrate our results and based on examples and
    results we conjecture that
    $\dim W(\underline{b};\underline{a};r) = \lambda_c $ for $\dim A > 2$
    (resp. $\dim A > 3$ if $c=1$) if
    $(*)$ and $a_1 > b_t$ hold.

    Section 7 is entirely devoted to solve Problems \ref{MainPblms}(2) and (3)
    by mainly using (4) to see when the closure of
    $W(\underline{b};\underline{a};r)$ fills in a generically smooth component
    of $\Hilb ^{p_X(t)}(\PP^n)$. The main result of this section (Theorem
    \ref{Wsmooth}) weakens the assumption in (4); every deformation of $A$
    comes from deforming its associated matrix $\cA $, to the corresponding
    assumption on $B$, by in addition assuming the natural map
    $\gamma: \ _0\!\Hom_R(I_A ,A)\to\ _0\!\Ext^1_B(I_B/I_B^2,I_{A/B})$ to be
    zero. Since we have not be able to show
    $\ _0\!\Ext^1_B(I_B/I_B^2,I_{A/B})=0$ under weak enough assumptions, we
    sometimes use Macaulay2 in which case one may check that $\gamma = 0$ by
    showing the equality of dimensions of ``normal modules'' in formula
    \eqref{thm61cond} (if directly showing
    $\ _0\!\Ext^1_B(I_B/I_B^2,I_{A/B})=0$ is too time-consuming or impossible
    or even false). Although these conditions are somewhat technical, it can
    be shown to be satisfied in a wide number of cases which we make explicit
    in a series of Remarks (cf. Remarks \ref{remthm61} and
    \ref{remconjWsmooth}), Corollaries and Examples (cf. Examples
    \ref{Wsmooth2}, \ref{examples712} and \ref{exbcn}). In particular, Problems
    \ref{MainPblms}(2) and (3) hold for generic determinantal rings (see
    Proposition \ref{deforminggenericcase} and Corollary \ref{corgendetW}).
    Moreover by considering the degrees of the relations in the
    Lascoux-resolution of $B$ and using the inclusion
    $\ _0\!\Ext^1_B(I_B/I_B^2,I_{A/B}) \subset \ _0\!\Ext^1_R(I_B,I_{A/B})$,
    we always get $\ _0\!\Ext^i_B(I_B/I_B^2,I_{A/B})=0$ for $i=0, 1$ if the
    increase in the sequence of numbers $a_{t-r+1} < a_{t-r+2} < a_{t-r+3}...$
    is large enough (cf. Corollary \ref{corWsmooth3}). Based on our results
    and examples we conjecture that Problems \ref{MainPblms}(2) and (3) hold
    provided $\dim A \ge 4$ (resp. $\dim A \ge 3$) for $c=1$ (resp. $c \ne 1$)
    and $a_1 > b_t$. Finally, we have also considered examples and results for
    the vanishing of $\ \Ext^i_B(I_B/I_B^2,I_{A/B})$ and
    $\ \Ext^i_B(I_B/I_B^2,B)$ for several $i \ge 1$ and they seem to coincide
    with expecting $\Hom_B(I_B/I_B^2,I_{A/B})$ and $\Hom_B(I_B/I_B^2,B)$ to
    have small codepth (Conjecture \ref{conjdepthNb}). Conjecture
    \ref{conjdepthNb} is true for generic determinantal schemes for $c \ge 0$
    by Proposition~\ref{genericdetring}. Moreover, Conjecture
    \ref{conjdepthNb} for $\Hom_B(I_B/I_B^2,I_{A/B})$ implies exactly that
    $\ _0\!\Ext^1_B(I_B/I_B^2,I_{A/B}) = 0$ for $\dim A \ge 4$ (resp.
    $\dim A \ge 3$) for $c=1$ (resp. $c \ne 1$) by Proposition~\ref{conj3ext}
    provided $c \ge 4-r$, whence that Problems \ref{MainPblms}(2) and (3)
    hold. In the remaining case $c=3-r$, we have
    $\ _0\!\Ext^1_B(I_B/I_B^2,I_{A/B}) \ne 0$ even when $X$ is a generic
    determinantal scheme (see Example~\ref{rem74ex}). Fortunately the
    condition $\gamma=0$ of Theorem \ref{Wsmooth} seems to hold in the generic
    case.

    In section 8, we weaken the assumptions in Theorems~\ref{dimW} and
    \ref{Wsmooth} by explicitly considering the recursively transfer of the
    property "any
    deformation 
    comes from deforming its associated homogeneous matrix" for the rings in a
    flag and thereby skipping this assumption on $B$ in Theorems~\ref{dimW}
    and \ref{Wsmooth} even though $\ _0\!\Ext^1_B(I_B/I_B^2,I_{A/B}) \ne 0$
    for $c=3-r$ provide problems. Also weakening $(*)$ in Theorem~\ref{dimW}
    and $\gamma = 0$ in Theorem~\ref{Wsmooth} are considered. Indeed, by
    successively deleting columns from the right hand side we construct a flag
    of determinantal rings
    $$A_{2-r}\twoheadrightarrow \cdots \twoheadrightarrow A_0\twoheadrightarrow
    A_1\twoheadrightarrow A_2\twoheadrightarrow \cdots \twoheadrightarrow
    A_{c-1}\twoheadrightarrow A_c\,,$$
    and letting $B:=A_{c-1}\twoheadrightarrow A:=A_c$, we show in
    Corollary~\ref{fiberpr2} how $ _0\!\Ext^i_B(I_B/I^2_B, I_{A/B})$ is
    related to $ _0\!\Ext^i_{A_{j}}(I_{A_{j}}/I^2_{A_{j}}, I_{A/B})$ for
    $i=0,1$ and $j \in \{2-r,3-r\}$. This has implications to the smoothness
    and the fiber dimension (tangential) of
    $p_1: \Hilb ^{p_X(t),p_Y(t)}(\PP^{n})\longrightarrow \Hilb
    ^{p_X(t)}(\PP^{n})$.
    We also consider the tangential fiber dimension of the projection
    $p_2: \Hilb ^{p_X(t),p_Y(t)}(\PP^{n})\longrightarrow \Hilb
    ^{p_Y(t)}(\PP^{n})$.
    This leads to Theorem~\ref{corWsmoothc} for the case $c=3-r$ and to
    Theorems~\ref{corWsmooth}, \ref{corWsmoothcnew} and \ref{corWsmoothnew2}
    when $c \ge 4-r$.
    We should have liked to recursively compute the dimension of
    $W(\underline{b};\underline{a};r)$ in terms of $a_{i}$ and $b_{j}$,
    starting the induction from the standard determinantal case ($j=2-r$), but
    we somehow have to start it from $j=3-r$ leaving the transfer from $j=3-r$
    to $j=2-r$ as an assumption on $\gamma$ in Theorem \ref{corWsmoothnew2}.
    Similarly in Theorem~\ref{corWsmooth} where the vanishing of two
    $\gamma's$ are assumptions. In Theorem~\ref{corWsmoothcnew} we manage to
    replace this vanishing of $\gamma's$ by a generalization of
    Theorems~\ref{dimW} and \ref{Wsmooth} where we weaken $\gamma = 0$ in the
    latter and transfer the assumptions on
    $\dim\,_0\!\Hom_B(I_B/I^2_B, I_{A/B})$ to
    $ \dim\,_0\!\Hom_{A_{3-r}}(I_{A_{3-r}}/I^2_{A_{3-r}}, I_{A/B})$ in the
    former, leading to formulas (\ref{a3r}) or (\ref{a3rr}) which imply a
    solution to Problem \ref{MainPblms}, see Corollary~\ref{corWsmoothnew3}.
    In Corollary~\ref{corWsmooth4}, we see how our results influence on the
    increase in the sequence of numbers
    $a_{t-r+1} < a_{t-r+2} < a_{t-r+3},\cdots $. After $ a_{t-r+3}$, no
    increase in $a_{t-r+3} \le a_{t-r+4} \le a_{t-r+5}$ is required, see
    Example~\ref{Wsmooth2f}.

    In section 9, we show that under some depth assumptions there is an
    isomorphism between the deformation functor of the exterior power
    $\wedge ^r MI$ of $MI$ and the deformation functor of the surjection
    $R\twoheadrightarrow R/I_{t-r+1}(\varphi ):=A_r$ (see Theorem
    \ref{defpower}). Moreover assuming 
    (*) above, e.g. $\cA$ linear, we are in Theorem~\ref{defWed} and
    Remark~\ref{rema97} able to highlight important consequences. Indeed using
    proven results, or say we assume that Conjectures~
    \ref{conjdimW} and \ref{conjWsmooth} hold for all $X_i:=\Proj(A_i)$ and that $\dim A_r \ge 4$
    and $c \ge 2$, then the local Hilbert functor $\Hilb_{X_i}(-)$ are for
    {\it all} $i$, $1 \le i \le r$ isomorphic to the local Hilbert functor
    $\Hilb_{X_1}(-)$ of deforming the scheme $X_1 \subset \PP^{n}$ defined by
    maximal minors. In particular, for all $i$,
    $ \overline{ W(\underline{b};\underline{a};i)} \subset \Hilb ^{p_{X_i}(t)}
    (\PP^{n})$
    is a generically smooth irreducible component of the same
    dimension $\lambda _c$ even though their Hilbert polynomials are very
    different.

    Finally, in the last section, we collect some of the questions which
    naturally come up from our work, apart from the ones stemming from the
    conjectures posed along this monograph. In particular it concerns the
    problems which are not completely solved:
\begin{itemize}
\item[(1)]
 \ Show:
$\dim\, _0\!\Hom_B(I_B/I^2_B, I_{A/B})=\sum _{j=1}^{t+c-2}
{a_j-a_{t+c-1}+n\choose n}$ if  $\dim A \ge 3$\,.

\item[(2)] \ Show: \ $_0\!\Ext^1_B(I_B/I^2_B, I_{A/B})= 0$
  if $ \dim A \ge 3$ (resp. $\dim A \ge 4$)  for $c \ne 1$ (resp. $c=1$)\,.

\item[(3)] \ Determine $ \dim (MI\otimes A)_{a_{t+c-1}}$ in the case:
  $a_{t+c-1} -b_1 \ge \sum _{i=1}^{t-r+1}(a_i-b_{r+i-1})$\,.
\end{itemize}

The work contains many examples. In some of them  we use  Macaulay2 \cite{Mac2} and others follow from our results, but most of them can be computed with Macaulay2. In the Appendix we include the code that we have used
 to compute the examples not covered by our results, to control our results or to give new evidences to the conjectures stated along the monograph; and the reader can check them.

\vskip 4mm
\noindent \underline{Acknowledgment} Part of this work was done while the
second author was a guest of the Oslo and Akershus University College of
Applied Sciences and she thanks the Oslo and Akershus University College for
its hospitality. The first author thanks the Universitat de Barcelona for its
kind hospitality during several visits to Barcelona.

\vskip 4mm
\noindent \underline{Notation.}
  Throughout this paper $k$ will be an algebraically closed field of characteristic zero, $R=k[x_0, x_1, \cdots ,x_n]$,  $\mathfrak{m}=(x_0, \ldots,x_n)$ and $\PP^n=\Proj(R)$.
For any graded Cohen-Macaulay quotient $A$ of $R$ of codimension $c$,
we let $I_A=\ker(R\twoheadrightarrow A)$ and  $K_A=\Ext^c_R (A,R)(-n-1)$ be its canonical module.  If $M$ is a
finitely generated graded $A$-module, let $\depth_{J}{M}$ denote
the length of a maximal $M$-sequence in a homogeneous ideal $J$
and let $\depth {M} = \depth_{\mathfrak m}{ M}$.
If $\shF$ and $\shG$ are two coherent $\cO
_X$-modules, we denote the group of morphisms from $\shF$ to $\shG$
by $\Hom_{\cO_X}(\shF,\shG)$ while $\cH om_{\cO_X}(\shF,\shG)$ denotes
the sheaf of local morphisms of $\shF$ into $\shG$. We often omit
$\cO_X$ in $\Hom_{\cO_X}(\shF,\shG)$ (resp. $\cH
om_{\cO_X}(\shF,\shG)$) if the underlying scheme $X$ is evident and we set $\hom(\shF,\shG)=\dim_k\Hom(\shF,\shG)$
 where $\dim_k$ denotes the dimension as
$k$-vector space. For any closed subschemes $X\subset Y\subset \PP^n$ we denote by ${\mathcal N}_Y$ the normal sheaf of $Y$ and we write ${\mathcal N}_{Y|X}:={\mathcal H}om_{{\mathcal O}_{\PP^n}} ({\mathcal I}_Y,{\mathcal O}_X)$.

Finally, we denote by $\Hilb ^{p_X(t)}(\PP^{n})$ the Hilbert scheme which
parameterizes closed subschemes $(X')$ of $\PP^{n}$ with Hilbert polynomial
$p_{X'}=p_X$ and we let $\Hilb ^{p_X(t),p_Y(t)}(\PP^{n})$ be the Hilbert flag
scheme parameterizing pairs of closed subschemes $(X' \subset Y')$ of
$\PP^{n}$ with Hilbert polynomials $p_{X'}=p_X$ and $p_{Y'}=p_Y$,
respectively. By definition a closed subscheme $X\subset \PP^n$ is {\em
  unobstructed} if $\Hilb ^{p_X(t)}(\PP^{n})$ is smooth at $(X)$. When we
define $A$ by writing $X=\Proj(A)$ for a given closed subscheme $X$ of $\PP^n$, we always take
$A=R/H^0_*(\shI _X)$, so $I_A=H^0_*(\shI _X)$ and we also write
$I(X)=H^0_*(\shI _X)$.


\section{Preliminaries}
\label{preli}

In this section we fix the definitions, notation and some basic results
that we are going to use in the sequel.

\begin{definition}
  \rm If $\cA$ is a homogeneous $p\times q$ matrix, we denote by $I(\cA)$ the
  ideal of $R$ generated by the maximal minors of $\cA$ and by $I_s(\cA)$ the
  ideal of $R$ generated by the $s \times s$ minors of $\cA$. Assume $p\le q$.
  $I(\cA)$ is called a \emph{standard determinantal} ideal if
  $\depth I(\cA)=q-p+1$, and $I_s(\cA)$ and $R/I_s(\cA)$ are said to be
  \emph{determinantal} if $\depth I_s(\cA)=(p-s+1)(q-s+1)$.
 \end{definition}

Let
$\varphi
: F=\oplus _{i=1}^tR(b_i)\longrightarrow G=\oplus _{j=1}^{t+c-1} R(a_j)$  be a morphism of free graded $R$-modules of
rank $t$ and $t+c-1$, respectively,
 $\varphi _{t+c-2}:F\longrightarrow \oplus _{j=1}^{t+c-2}R(a_j)=:G_{t+c-2}$ the morphism obtained
 deleting the last row,
 $\cA$  the homogeneous matrix of $\varphi ^*$ and ${\mathcal B}$ the homogeneous matrix of $\varphi ^*_{t+c-2}$ (i.e. $\cB$ is obtained deleting the last column of $\cA$).  For any integer $r$, $1\le r\le t$, let $B=R/I_{t+1-r}(\varphi _{t+c-2}^*)$, $A=R/I_{t+1-r}(\varphi ^*)$, $N=\coker (\varphi _{t+c-2}^*)$ and $MI=\coker (\varphi ^*).$ A codimension $c$ subscheme $X\subset \PP^n$ is said to be {\em standard determinantal} if its homogeneous saturated ideal $I(X)=I_t(\cA)$ for some $t\times (t+c-1)$ homogeneous matrix $\cA$ and it is said to be {\em determinantal} if $I(X)=I_s(\cA)$ for some $p\times q$ homogeneous matrix $\cA$ and $c=(p-s+1)\times (q-s+1)$.

\begin{example} \rm (i)  Complete intersections schemes $X\subset \PP^n$ of codimension $c\ge 1$ are examples of standard determinantal schemes.

\vskip 2mm
(ii)  Let $X\subset \PP^n$ be a
rational normal curve  defined as the
image of the map
$$v_n:\PP^1 \longrightarrow \PP^n$$
$$[a:b]\mapsto [a^n:a^{n-1}b:\cdots : ab^{n-1}:b^n].$$  The homogeneous ideal $I(X)$ of $X$ is
generated by the maximal minors of the homogeneous matrix $$ \left (
  \begin{array}{ccccccc} x_0 & x_1 & \cdots & x_{n-1} \\ x_1 & x_2 & \cdots &
    x_{n}
\end{array} \right).$$ Therefore, $X$ is a standard determinantal
subscheme of $\PP^n$.

\vskip 2mm
(iii)  Let $S\subset \PP^8$ be the 4-dimensional Segre variety  defined as the
image of the map
$$ \hskip -4.5cm s_{2,2}:\PP^2 \times \PP^2 \longrightarrow \PP^8$$
$$([a:b:c],[x:y:z])\mapsto [ax:bx:cx:ay:by:cy:az:bz:cz].$$  Fix coordinates $x_0,x_1,\cdots ,x_8$ in $\PP^8$. The homogeneous ideal $I(S)$ of $S$ is
generated by the $2\times 2$ minors of the homogeneous matrix $$ \left (
\begin{array}{ccccccc} x_0 & x_1 & x_{2} \\ x_3 & x_4  & x_{5} \\ x_6 & x_7 & x_8
\end{array} \right).$$ Therefore, $S$ is a  determinantal
subscheme of $\PP^8$.
\end{example}

Along this work, we assume $t\ge 2$ and $r<t$. The cases $t=1$ or $t=r$ for
determinantal ideals corresponds to the well known case of complete
intersections. When $c=1$ we also assume $r\ge 2$ (unless explicitly considering
$(c,r)=(1,1)$) since the case $(c,r)=(1,1)$ corresponds to hypersurfaces (see,
however, section \ref{conjec} for some interesting problems about
determinantal hypersurfaces).

\begin{definition}\label{defor} \rm We say that
  every deformation of $X=\Proj(A)$ (or $A$) comes from deforming
  $\cA=(f_{ij})$ if for every local artinian ring $T$ with residue field
  $k=T/{\mathfrak m}_T$ and every graded deformation $A_T$ of $A$ to $T$ there
  exists a homogeneous matrix $\cA _T$ with entries
  $f_{ij}^T\in R\otimes _kT:=R_T $ that map to $f_{ij}$ via
  $T\twoheadrightarrow k$. Such matrix $\cA _T$ is called a lifting of $\cA$
  to $T$. (By
  Remark~\ref{remMthm61}, the definition in \cite{K2011, K2014} which requires $T$ above to
  be a local ring, essentially of finite type over $k$, with
  $T/{\mathfrak m}_T=k$, is equivalent to Definition~\ref{defor}.)
\end{definition}

For $c=2$ and $r=1$, Ellingsrud proved in \cite{E} that any deformation of
$R/I_t(\varphi ^*)$ comes from deforming the $t\times (t+1)$ matrix $\cA $
associated to $\varphi ^*$. This was generalized in \cite[Theorem 5.8]{K2014}
to cover any $c \ge 2$ when $r=1$ and $n-c \ge 2$ but it is not always true
when $r=1$ and $n-c<2$ (see, for instance, \cite[Example 4.1]{K2011}). One of
the main goals of this monograph will be to generalize these results to
$1\le r < t$ and $c\ge 2-r$; and to study whether any deformation of
$R/I_{t-r+1}(\varphi ^*)$ comes from deforming the homogeneous matrix $\cA $.
We think that quite often it will be true that any deformation of a
determinantal scheme comes from deforming its associated homogeneous matrix
(cf. Conjecture \ref{conjWsmooth}). As in the case of standard determinantal
schemes, we prove that it is indeed true in many cases (see Lemma
\ref{lemma2}, Theorem \ref{teo3}, Corollary \ref{evdef}, Theorem \ref{Wsmooth}
and Corollary \ref{corWsmooth4}) but we have a few exceptions (see, for
instance, Example \ref{dimW2}).

\vskip 2mm Let
$W(\underline{b};\underline{a};r)\subset \Hilb ^{p_X(t)}(\PP^{n})$ be the
locus of determinantal schemes $X\subset \PP^{n}$ of codimension $r(r+c-1)$
defined by the $(t-r+1)\times (t-r+1)$ minors of a $t\times (t+c-1)$
homogeneous matrix $\cA=(f_{ij})^{i=1,...,t}_{j=1,...,t+c-1}$ where
$f_{ij}\in k[x_0,x_1,...,x_{n}]$ is a homogeneous polynomial of degree
$a_j-b_i$. Correspondingly, set $\underline{a'}=(a_1,a_2,...,a_{t+c-2})$ and
let
$W_{(\underline{b};\underline{a};r)}^{(\underline{b};\underline{a'};r)}
\subset \Hilb ^{p_X(t),p_Y(t)}(\PP^{n})$
be the locus of determinantal flags $X\subset Y\subset \PP^n$ of the
Hilbert-flag scheme with $X\in W(\underline{b};\underline{a};r)$ and
$Y\in W(\underline{b};\underline{a'};r)$ defined by the
$(t-r+1)\times (t-r+1)$ minors of the $t\times (t+c-2)$ matrix ${\mathcal B}$
that we obtain deleting the last column of $\cA$. Note that the restriction of
the natural projections
$$p_1: \Hilb ^{p_X(t),p_Y(t)}(\PP^{n})\longrightarrow \Hilb
^{p_X(t)}(\PP^{n})$$
and
$$p_2: \Hilb ^{p_X(t),p_Y(t)}(\PP^{n})\longrightarrow \Hilb
^{p_Y(t)}(\PP^{n})$$
to $W_{(\underline{b};\underline{a};r)}^{(\underline{b};\underline{a'};r)}$
maps into $W(\underline{b};\underline{a};r)$ and
$ W(\underline{b};\underline{a'};r)$, respectively. Since we will only
consider the case $W(\underline{b};\underline{a};1)\ne \emptyset$, from now on
we will assume (see \cite[(2.2)]{KM2011}):
\begin{equation} \label{paperassump1} a_1\le a_2\le \cdots \le a_{t+c-1},
  \quad b_1\le \cdots \le b_t, \quad b_{i}\le a_{i} \text{ for all }i \text{
    and } b_{i_0}<a_{i_0} \text{ for some } i_0\,.
\end{equation}
\begin{definition} \label{PYgeneric} \rm Let $X\subset Y \subset \PP^n$ be
  closed subschemes. We say that $X$ is {\em $p_Y$-generic} if there is an
  open subset of $\Hilb ^{p_X(t)}(\PP^n)$ containing $(X)$ whose members
  $(X')$ are subschemes of some closed scheme $Y'$ with Hilbert polynomial
  $p_Y$.
\end{definition}

In this paper we address the following  fundamental problems:

\begin{pblm} \label{mainPBLM}\rm
\begin{itemize} \item[(1)] To determine the dimension of
$W(\underline{b};\underline{a};r)$ in terms of $a_j$ and
$b_i$;
\item[(2)] To determine whether the closure of
$W(\underline{b};\underline{a};r)$ is an irreducible
component of $\Hilb ^{p_X(t)}(\PP^{n})$;
 \item[(3)] To determine
when $\Hilb ^{p_X(t)}(\PP^{n})$ is generically smooth along
$W(\underline{b};\underline{a};r)$, and;
\item[(4)] To determine whether any deformation of $X$,
  $(X) \in W(\underline{b};\underline{a};r)$ comes from deforming its
  associated homogeneous matrix $\cA$.
\end{itemize}
\end{pblm}

Due to  Lemma~\ref{unobst}, if Problem \ref{mainPBLM} {\rm (4)}
  holds,  Problems \ref{mainPBLM} {\rm (2)} and {\rm (3)} also hold.

The case $r=1$ was treated in \cite{E} and \cite{F} for $c=2$, \cite{KMMNP}
for $c=3$ and \cite{K2014} and \cite{KM2005}  for arbitrary $c$ (See also
\cite{FF}, \cite{K2011}, \cite{K2018}, \cite{KM2011} and
\cite{rm} for more details); while the case $r=2$ and $c=1$ was considered in
\cite{KM2009}. In the next section of this work, for sake of completeness, we
summarize and slightly generalize the case $r=1$ and we devote the remaining
part of the paper trying to generalize the mentioned results to $r>1$ and to
solve the above problems under weak numerical assumptions on the
integers $a_i$ and $b_j$ and under some depth or other cohomological conditions.

In this paper we often switch between a subscheme $X\subset \PP^n$, or a pair
$X\subset Y \subset \PP^n$ of subschemes, and the corresponding homogeneous
coordinate rings and their surjections $R \to A$ and $R \to B \to A$. Indeed
there is a scheme $\GradAlg(H_A)$ parameterizing graded surjections $R \to A$
of $\depth_{\mathfrak m}A \geq 1$ which is the stratum of Grothendieck's
Hilbert scheme $ \Hilb^{ p_X} (\PP^n)$ consisting of points
$(X=\Proj(A) \subset \PP^n)$ with Hilbert function $ H_X$ where
$H_X(v)=p_X(v)$ for $v \gg 0$, cf. \cite{K2007}. Note that we define the Hilbert
function of $X$, or $A$, by $H_X(v)=H_A(v):= \dim A_v$. $\GradAlg(H_A)$ has a
natural scheme structure whose tangent (resp.\ ``obstruction'') space at
$(X \subset \PP^n)$ is
$\ _0\!\Hom_A(I_{A}/I_A^2,A) \simeq \ _0\!\Hom_R(I_{A},A) \ $ (resp.\
$\ _0\!\Ext^1_A(I_{A}/I_A^2,A)$ provided $R \to A$ is generically a complete
intersection). $A$ is called {\em unobstructed} (as a graded
$R$-algebra) if $\GradAlg(H_A)$ is smooth at $(R \rightarrow A)$. And we may
use \cite{E} to see that the open subscheme of $\GradAlg(H_A)$ consisting of
graded Cohen-Macaulay quotients satisfying $\dim A \ge 2$ and the open
subscheme of $\!\ \Hilb^{p_X}(\PP^n)$ consisting of arithmetically
Cohen-Macaulay (ACM) schemes satisfying $\dim X \ge 1$ where $X=\Proj(A)$ are
isomorphic as {\it schemes}. Indeed if $\depth_{\mathfrak m}A \geq 2$, we have
\begin{equation} \label{Grad} \GradAlg(H_A) \simeq \Hilb^{ p_X}(\PP^n) \ \ \
  {\rm at} \ \ \ (X = \Proj(A) \subset \PP^n)\,.
\end{equation}
In particular if $\dim X \ge 1$ we can replace the phrase ``every graded
deformation $A_T$ of $A$ to $T$'' in Definition~\ref{defor} by ``every
deformation $X_T \subset \Proj(R \otimes_k T)$ of $X \subset \PP^n $ to $T$''.

Similarly there is a scheme $\GradAlg(H_B,H_A)$ representing the functor
deforming flags (surjections) $B \rightarrow A$ of graded quotients of $R$ of
positive depth at ${\mathfrak m}$ and with Hilbert functions $H_B$ and $H_A$
respectively. Also in this case, if $\depth_{\mathfrak m}A \geq 2$ and
$\depth_{\mathfrak m}B \geq 2$ we have
\begin{equation} \label{GradPair}
\GradAlg(H_B,H_A) \simeq
\Hilb
  ^{p_X(t),p_Y(t)}(\PP^{n}) \ \ \ {\rm at} \ \ \ (X \subset Y \subset \PP^n)\,.
\end{equation}
In particular the open subscheme of $\GradAlg(H_B,H_A)$ consisting of
Cohen-Macaulay pairs $ (B \twoheadrightarrow A)$ satisfying $\dim A \ge 2$ are
isomorphic (as {\it schemes}) to the corresponding open set of
$\Hilb ^{p_X(t),p_Y(t)}(\PP^{n})$. Thus {\it restricting} to
these open subschemes, the projection
$$p : \GradAlg(H_B,H_A) \rightarrow \GradAlg(H_A)$$ induced by sending
$(B' \rightarrow A')$ onto $(A')$ is isomorphic to the $1^{st}$ projection
$$p_1: \Hilb ^{p_X(t),p_Y(t)}(\PP^{n})\longrightarrow \Hilb ^{p_X(t)}(\PP^{n})$$
and $q : \GradAlg(H_B,H_A) \rightarrow \GradAlg(H_B)$ to the $2^{nd}$
projection $p_2$ which maps into $\Hilb ^{p_Y(t)}(\PP^{n})$. Since in this
paper we almost always work under assumptions that imply
$\depth_{\mathfrak m}A \geq 2$ (and $\depth_{\mathfrak m}B \geq 2$ if $B$ is
present), there is really no restriction to work with $\GradAlg$ instead of
$\Hilb$.

 \vskip 2mm

 We finish this section with the following useful comparison of cohomology
 groups. If $Z\subset X\subset \PP^{n}$ is a closed subset such that
 $U=X\setminus Z$ is a local complete intersection, $L$ and $N$ are finitely
 generated $R/I(X)$-modules, $\widetilde{N}$ is locally free on $U$ and
 $\depth_{I(Z)}L\ge r+1$, then the natural map
\begin{equation} \label{NM}
\Ext^{i}_{R/I(X)}(N,L)\longrightarrow
\Hl_{*}^{i}(U,{\mathcal H}om_{\odi{X}}(\widetilde{N},\widetilde{L}))\cong \oplus _{\nu \in \ZZ}\Ext^{i}_{\odi{X}}(\widetilde{N},\widetilde{L}(\nu ))
\end{equation}
is an isomorphism, (resp. an injection) for $i<r$ (resp. $i=r$)
cf. \cite[expos\'{e} VI]{SGA2}. Recall that we interpret $I(Z)$ as $\mathfrak{m}$ if $Z=\emptyset $.


\section{Families of standard determinantal schemes}
\label{satandard}

Until now all results we have proved about families of determinantal ideals
deal with standard determinantal schemes i.e. schemes of codimension $c$ in
$\PP^n$ defined by the maximal minors of a $t\times (t+c-1)$ homogeneous
matrix (according to our notation they correspond to the case $r=1$), except
those in \cite{KM2009} where the authors considered ideals generated by the
submaximal minors of a square homogeneous matrix (i.e. $c=1$ and $r=2$). This
last paper can be seen as a first attempt to address the general case (i.e.
$r\ge 1$ and $c\ge 2-r$). In this section, we state without proof what is
known about families of standard determinantal schemes and we refer to
\cite{FF}, \cite{K2011}, \cite{K2014}, \cite{K2018}, \cite{KM2005},
\cite{KM2011}, \cite{KMMNP} and \cite{rm} for a complete proof of these
results. A few generalizations with proofs are, however, also included.

Given sequences of integers $\underline{a}=(a_1,a_2,...,a_{t+c-1})$ and
$\underline{b}=(b_1,b_2,...,b_{t})$ satisfying \eqref{paperassump1} we denote
by
 $$W(\underline{b};\underline{a}):=W(\underline{b};\underline{a};1)\subset \Hilb ^{p_X(t)}(\PP^{n})$$
 the locus of standard determinantal schemes. As always we assume $t\ge 2$.
 Using the convention $ \binom{m}{n}=0$ for $m < n$ we define for any $r$,
 $\max\{1,2-c\} \le r < t$, the following invariants
\begin{equation}\label{lamda}
\begin{array}{rcl}
\ell_i & :=  & \sum_{j=1}^{t+i-1}a_j-\sum_{k=1}^tb_k, \ i \ge 1,\\
h_{i-3} & := &
  2a_{t+i-1}-\ell_i +n, \ \text{ for } 3\le i\le c,\\
\lambda_c & := & \sum_{i,j}
    \binom{a_i-b_j+n}{n} + \sum_{i,j} \binom{b_j-a_i+n}{n} - \sum _{i,j}
    \binom{a_i-a_j+n}{n}- \sum _{i,j} \binom{b_i-b_j+n}{n} + 1, \\
    K_3 & := & \binom{h_0}{n}, \\
    K_4 &:= & \sum_{j=1}^{t+2} \binom{h_1+a_j}{n}-
\sum_{i=1}^{t} \binom{h_1+b_i}{n}, \text{ and, } \\
%
K_{i} &: = &\sum _{p+q=i-3
  \atop p , q \ge 0} \sum _{1\le i_1< ...< i_{p}\le t+i-2 \atop 1\le
  j_1\le...\le j_q \le t } (-1)^{i-1-p} \binom{h_{i-3}+a_{i_1}+\cdots
  +a_{i_p}+b_{j_1}+\cdots +b_{j_q} }{n} \text{ for } 3\le i \le c.  \end{array}
\end{equation}

\vskip 2mm
Using the generalized Koszul complexes associated to a codimension $c$
standard determinantal ideal $I_t(\cA)$, one knows that the {\em
  Eagon-Northcott complex} (see \cite[Theorem 2.20]{b-v} or \cite[Corollaries
A2.12 and A2.13]{eise}) yields the following minimal free $R$-resolution of
$A:=R/I_t(\cA)$
 \begin{equation}\label{EN}0 \longrightarrow \wedge^{t+c-1}G^* \otimes S_{c-1}(F)\otimes
  \wedge^tF\longrightarrow \wedge^{t+c-2} G ^*\otimes S _{c-2}(F)\otimes \wedge ^tF\longrightarrow
  \ldots  \end{equation}
$$ \longrightarrow
\wedge^{t}G^* \otimes S_{0}(F)\otimes \wedge^tF\longrightarrow R \longrightarrow R/I_t(\cA) \longrightarrow 0
$$
 and that the  {\em Buchsbaum-Rim complex} yields a minimal free $R$-resolution of $\coker (\varphi ^* ) $
\begin{equation}\label{BR} 0 \longrightarrow \wedge^{t+c-1} G^* \otimes S_{c-2}F \otimes
  \wedge^t F \longrightarrow \wedge^{t+c-2} G^* \otimes S _{c-3}F \otimes
  \wedge ^t F \longrightarrow \ldots \longrightarrow
 \end{equation}
 $$\wedge^{t+1}G^* \otimes
  S_{0}F \otimes \wedge^t F \longrightarrow G^* \stackrel {\varphi ^* }{
    \longrightarrow}  F^* \longrightarrow \coker (\varphi ^* ) \longrightarrow
  0. $$
Note that (\ref{EN}) shows that any standard determinantal scheme is arithmetically Cohen-Macaulay (ACM). Analogous result holds for determinantal schemes. Indeed, using the minimal free $R$-resolution of a determinantal ideal given by Lascoux in \cite[Theorem 3.3]{L} we get that any determinantal scheme is ACM.

Using essentially induction on the codimension by successively deleting
columns from the right hand side (assuming, for instance, that the $c-2$
columns that we delete do not contain units, to be sure that the rings we get
by deleting columns are standard determinantal, cf. \cite{Br}), the
Eagon-Northcott complex associated to a standard determinantal ideal
$I_t(\cA)$, the Buchsbaum-Rim complex, the theory of Hilbert flags schemes and
the depth of certain mixed determinantal ideals, we can solve (for $r=1$)
Problem \ref{mainPBLM}(1) under some weak numerical assumptions and prove:

\begin{theorem}\label{std_det_case} Fix integers $t\ge 2$, $c\ge 2$,  $\underline{a}=(a_1,a_2,...,a_{t+c-1})$ and $\underline{b}=(b_1,b_2,...,b_{t})$ as above. Assume $a_{i-1}\ge b_i$ for $2 \le i \le t$. It holds:
\begin{itemize}
\item[(i)] $\dim W(\underline{b};\underline{a})\le \lambda _c+K_3+K_4+\cdots +K_c$.

\item[(ii)] For $c=2$ and $n-c\ge 0$, $\dim W(\underline{b};\underline{a})= \lambda _c$.

\item[(iii)] Assume $c\ge 3$. If $n-c>0$, or $n-c=0$, $a_1>b_t$ and $a_{t+c-1}>a_{t-1}$, then
 $$\dim W(\underline{b};\underline{a})= \lambda _c+K_3+K_4+\cdots +K_c.$$
\end{itemize}
\end{theorem}

\begin{proof} (i) See \cite[Theorem 3.5]{KM2005} and Remark \ref{aug2023}(3) below.

  (ii) See \cite{E} for the case $n-c\ge 1$, \cite[Remark
  4.6]{KM2005} (cf.\cite{F} and \cite[Appendix 3]{Got}) for the case $n-c=0$
  and \cite[Theorem 5.16]{K2014} for a generalization.

  (iii) See \cite[Theorems 3.5, 4.5 and Corollaries 4.7, 4.10, 4.14 and
  4.15]{KM2005} and \cite[Corollary 5.6]{K2014} for the case $n-c>0$. For the
  case $n-c=0$ the reader could look at \cite[Theorem 3.2]{KM2011} for the
  case $3\le c\le 5$ and \cite[Theorem 6.1]{K2018} for the case $c\ge 3$. (See
  also \cite{KMMNP} for the cases $c=3$).
\end{proof}

\begin{remark} \label{aug2023} \rm (1) For $n-c>0$ Theorem~\ref{std_det_case} as well as
  \cite[Corollary 5.6]{K2014} or \cite{FF} prove Conjecture 4.1 stated in
  \cite{KM2011}.

  (2) For $c\ge 3$ and $n-c=0$ the hypothesis $a_{t+c-1}>a_{t-1}$ in Theorem
  \ref{std_det_case} cannot be deleted. In fact,
  we consider the case $t=2$, $a_1=\cdots =a_{c+1}=1$ and $b_1=b_2=0$, i.e.,
  $\cA$ is a $2\times (c+1)$ matrix with entries general linear forms. The
  vanishing of the $2\times 2$ minors of $\cA$ defines a reduced scheme
  $X\subset \PP^c$ which consists in $c+1$ different points in $\PP^c$. So,
  $\dim _{(X)} \Hilb ^{p_X(t)}(\PP^c)=c(c+1)$ while $K_i=0$ for $3\le i \le c$
  and $\lambda _c=c^2+2c-2$. Hence,
  $\dim _{(X)} \Hilb ^{p_X(t)}(\PP^c)< \lambda _c$ for $c\ge3$, cf.
  \cite[Example 3.3]{K2011}.

  (3) Note that even though \cite{KM2005} assumes $\cA = (f_{ij})$ minimal,
  i.e. $f_{ij}=0$ when $ \deg f_{ij} = a_j - b_i=0$, \cite[Theorem
  3.5]{KM2005} holds without this assumption owing to \eqref{paperassump1}!
  Indeed Proposition 3.12 and (3.3) of \cite{KM2005} yield Theorem 3.5. But
  \cite[Proposition 3.12]{KM2005} is generalized in \cite[Proposition
  2.4]{K2018} where we only assume the last $c-2$ columns to be without units.
  Now as the degrees of the entries of any $\cA$ obey \eqref{paperassump1}, we
  either have $a_i = b_i$ or $a_i > b_i$. Take the smallest $i$ (say $i=i_0$)
  with $a_i = b_i$. First if $f_{ii} = 0$ and $f_{ij} \neq 0$ for some
  $j > i$ (resp. $f_{ij} = 0$ for all $j > i$) with $a_j = a_i$, then
  interchange the columns $\cA^i$ and $\cA^j$ of $\cA$ (resp. leave $\cA$
  unchanged). Next if $f_{ii} $ is (or has become by the operation above) a
  unit, then replace all columns $\cA^j$
  with $j > i$ and $a_j = a_i$ by
$$\cA^j - \frac{f_{ij}}{f_{ii}} \cA^i $$
  Then repeat the above taking the next smallest $i$ (say $i=i_1$) with
  $a_i = b_i$ and so on, and note that the column operations we perform when
  $i = i_1$ don't change the $0$'s already created in the $i_0$-th row in the
  case $a_{i_1} = a_{i_0}$ and certainly not when $a_{i_1} > a_{i_0}$. Finally
  if $a_i > b_i$ for some $i := i_2$, then all $f_{i_2j}$ for $j > i_2$ are
  non-units (in the $i_2$-th row). Note that the above elementary column
  operations don't change $I_t(\cA)$, cf. Fitting's lemma. In addition the
  $\cA$ we create in this way is without units in the last $c-1$ columns. Thus
  as \cite[(3.3)]{KM2005} holds (this doesn't dependent on a given $\cA$),
  \cite[Theorem 3.5 and Proposition 3.12]{KM2005} hold under the assumptions
  of this paper.

  (4) The following interesting observation also makes a correction to
  \cite[Proposition 5.1]{K2018}; note that the locus
  $W_s(\underline{b};\underline{a})$ inside $\GradAlg(H)$ of 
  graded, possibly artinian, $A:=R/I_t(\cA)$ of
  \cite{K2018} 
  equals the $W(\underline{b};\underline{a})$ of this paper via
  $(A) \mapsto (\Proj(A))$ whenever $\dim A > 0$ (i.e. when
  $H(v):= \dim A_v \neq 0$ for all $v \gg 0)$: Now suppose there exists
  $i,j'$
  with $a_{j'}
  = b_i$. Then using \eqref{paperassump1} as above there's a $j \le
  i$ with $a_j = b_i$, namely $j := j'$ if $j'< i$ and $j := i$ if $j'\ge
  i$. Fix $i$ and $j \le i$ with $a_j =
  b_i$. Repeatedly using the two column operations in (3) (also when $j' <
  j$), 
  we get a matrix $\cA$
  with $f_{ij}$
  either $0$
  or a unit and with $f_{ij'} = 0$ for all $j' \neq j$ with $a_{j'} = a_j$.

  In the case $f_{ij}$ is a unit we can further use the displayed formula in
  (3) to create a matrix with $f_{ij'}=0$ for all $j' \neq j$. Then if $\cA'$
  is the matrix obtained by deleting the $i$-th row and $j$-th column,
  $I_t(\cA) = I_{t-1}(\cA')$. Thus if
  $\underline{a'} = a_1,a_2,...,a_{j-1},a_{j+1},...,a_{t+c-1}$ and
  $\underline{b'}= b_1,b_2,...,b_{i-1},b_{i+1},...,b_t$ and if all $\cA$ with
  $(R/I_t(\cA)) \in W_s(\underline{b};\underline{a})$ belongs to this case
  where $f_{ij}$ has become a unit, we get
  $W_s(\underline{b'};\underline{a'})= W_s(\underline{b};\underline{a})$ as
  conversely, given $\cA'$ we can create an $\cA$ by inserting a new $i$-th
  row and $j$-th column consisting only of $0$'s except at the $i$-th and
  $j$-th coordinate where we put the unit $1$. 

  In the case $f_{ij}$ of $\cA$ is zero and
  $(R/I_t(\cA)) \in W_s(\underline{b};\underline{a})$, we consider the matrix
  $\cA_g$ obtained by letting $f_{ij}=u$ with $u \in k$ a unit and letting all
  the other $(l,m)$ entries be those of $\cA$ plus a general enough
  $g_{lm} \in R_{a_m-b_l}$. More precisely taking $\{u_{\nu}\}_{\nu=1}^N$ with
  $N = \dim {\HH}om(F,G)$ as indeterminates with $u_1=u$ and ${\it all}$ the
  other $g_{lm}$ as a sum of different $u_{\nu}$ times the monomials of
  $ R_{a_m-b_l}$, $2 \le \nu \le N$, then $R[u_1,...,u_N]/I_t(\cA_g)$
  restricts by e.g. \cite[Lemma 2.6]{K2018} to a flat family over some open
  dense subset of $\mathbb A^N_k$ containing $(0,...,0)$, with dense image in
  $ W_s(\underline{b};\underline{a})$ (cf. the text after \eqref{flag}) in which
  the zero $N$-tuple represents $R/I_t(\cA)$ and a general enough $N$-tuple
  represents a $k$-point, say,
  $R/I_t(\cA_1)$ 
  contained in the subset $W_s(\underline{b'};\underline{a'})$
  of $W_s(\underline{b};\underline{a})$
  (much as in \cite[Proposition 5.1, see third paragraph of its proof]{K2018};
  the argument above is more accurate concerning "general enough". As
  $(R/I_t(\cA_1))
  \in W_s(\underline{b'};\underline{a'})$, the generization
  $R/I_t(\cA_1)$
  removes all ghost terms in the Eagon-Northcott resolution of
  $R/I_t(\cA)$
  "arising from $a_j
  = b_i$", e.g. see \cite[Examples 5.3 and 5.4]{K2018}).

  Moreover $W_s(\underline{b'};\underline{a'})$
  is open in $
  W_s(\underline{b};\underline{a})$ as they both are constructible
  by 
  Chevalley's theorem (cf. \eqref{flag}); moreover if we take any generization
  $R_T/I_t(\cA_T)$
  {\bf in} ${W_s(\underline{b};\underline{a})}$,
  $T$
  some local $k$-algebra
  with residue field $k$
  and $\cA_T:=(f_{lm,T})$
  a matrix with $f_{lm,T}
  \in R_T:=R \otimes_k T$, $\deg f_{lm,T} = a_m-b_l$ satisfying $f_{lm,T}
  \otimes_T k =
  f_{lm}$ as in \cite[(2.14)]{K2018}, of an arbitrary
  $R/I_t(\cA_1)$
  of $W_s(\underline{b'};\underline{a'})$,
  then $f_{ij,T}$
  must be a unit as $f_{ij}$
  of $\cA_1$
  is, i.e. $R_T/I_t(\cA_T)$
  is a $T$-point of $ W_s(\underline{b'};\underline{a'})$. 
  So the latter is open in $
  {W_s(\underline{b};\underline{a}})$; in fact with
  $T$
  as the local ring at $x:=(R/I_t(\cA_1))$
  of an open affine $\Spec(V)
  \subset W_s(\underline{b};\underline{a})_{red}
  $ (reduced scheme structure), then there's an open dense subset
  $U$
  of $\Spec(V)$,
  $x
  \in U$, in which all entries
  $f_{lm,T}$,
  thus the pullback of $R_T/I_t(\cA_T)$
  to $U$,
  are defined; whence $U
  \subset W_s(\underline{b'};\underline{a'})$ as
  $f_{ij,T}$
  is a unit (so correct the $2^{nd}$
  sentence of \cite[Proposition 5.1]{K2018} "we have an inclusion $
  {W_s(\underline{b'};\underline{a'})} \subset
  {W_s(\underline{b};\underline{a}})$ of open irreducible subsets of
  $\GradAlg(H)$"
  by deleting "open" in "open irreducible.." and replacing "an inclusion" by
  "an open inclusion"). Hence we always have $\overline
  {W_s(\underline{b'};\underline{a'})} = \overline
  {W_s(\underline{b};\underline{a}})$ when $ a_j=b_i$ for some $i,
  j$. Thus we may prior to computing the invariants of
  $\eqref{lamda}$,
  successively delete from $\underline{b},
  \underline{a}$ all pairs $(a_j,b_i)$ with $a_j =
  b_i$, i.e. changing
  $\underline{b},\underline{a}$
  in this way will not change the invariants of $\eqref{lamda}$.
  \end{remark}

Concerning Problems \ref{mainPBLM} (2) and (3) we also gave in \cite{KM2005}
and \cite{KM2011}  an affirmative answer to both questions in the range $2\le
c\le 4$ and $n-c\ge 2$ and in the case $c\ge 5$ and $n-c\ge 1$ provided
certain numerical assumptions are verified. More precisely, we have got:

\begin{theorem}\label{std_det_case1} Fix integers $t\ge 2$, $c\ge 2$,
  $\underline{a}=(a_1,a_2,...,a_{t+c-1})$ and
  $\underline{b}=(b_1,b_2,...,b_{t})$ as above. Assume  $a_{i+1-min \{3,t\}
  }\ge b_{i}$ for $´min \{3,t \} \le i \le t$ and $n-c\ge 1$. It holds:

\begin{itemize}
\item[(i)] If $c=2$ then $\overline{W(\underline{b};\underline{a})}$ is a
  generically smooth irreducible component of $\Hilb ^{p_X(t)}(\PP^n)$ of
  dimension $\lambda _2$. Moreover $\Hilb ^{p_X(t)}(\PP^n)$ is smooth at every
  point $(X) \in W(\underline{b};\underline{a})$.

\item[(ii)]  Assume $c\ge 3$. If $n-c>1$, or $n-c=1$, $a_1>b_t$ and
  $a_{t+c-1}>a_t+a_{t+1}-b_1$, then
  $\overline{W(\underline{b};\underline{a})}$ is a generically smooth
  irreducible component of $\Hilb ^{p_X(t)}(\PP^n)$ of dimension $\lambda
  _c+K_3+\cdots +K_c$.
    \end{itemize}
\end{theorem}

\begin{proof} (i) See  \cite[Th\'{e}or\`{e}me 2]{E} and \cite[Theorem 5.16]{K2014} for a generalization.

  (ii) Let $n-c\ge 1$. For $c=3$ and $4$ the reader can see \cite[Corollary
  5.10]{KM2005}. The case $c\ge 5$ was proved in \cite[Corollary 3.8]{KM2011}
  under the hypothesis $a_{t+4}>a_t+a_{t+1}-b_1$ improving quite a lot the
  previous results of the authors in \cite[Corollary 5.9]{KM2005}. If $n-c>1$,
  these results are generalized in \cite[Corollary 5.9]{K2014} (which is a
  corollary of Theorem~\ref{Amodulethm3}(ii) below), only assuming $n-c> 1$.

  We will now prove (ii) in the case $n-c = 1$ and
  $a_{t+c-1}>a_t+a_{t+1}-b_1$, generalizing \cite[Corollary 3.8]{KM2011} a
  bit. We get $ \dim \overline {W(\underline{b};\underline{a})}$ from
  Theorem~\ref{std_det_case}. Moreover by successively deleting columns from
  the right hand side of $\cA$, and taking maximal minors, one gets a flag of
  standard determinantal quotients
 $$ A_2\twoheadrightarrow A_3\twoheadrightarrow \cdots \twoheadrightarrow A_c
  =A$$
  and we let $I_{A_i}:=\ker(R\twoheadrightarrow A_i)$ and
  $I_i:=I_{A_{i+1}}/I_{A_i}$. Then since $\dim A_{c-1} = 3$, we get that every
  deformation of $\Proj(A_{c-1})$ comes from deforming its matrix $\cA_{c-1}$
  by Theorem~\ref{Amodulethm3}(ii) below. Due to \cite[Theorem 4.6 and
  Lemma 4.4]{K2011} it suffices to show that
  $\ _0\! \Ext^1_{A_{c-1}}(I_{A_{c-1}}/I_{A_{c-1}}^2,I_{c-1})=0$. Since
  $$\ _0\! \Ext^1_{A_{c-1}}(I_{A_{c-1}}/I_{A_{c-1}}^2,I_{c-1}) \subset \ _0\!
  \Ext^1_{A_{2}}(I_{A_{2}}/I^2_{A_{2}},I_{c-1})$$
  by \cite[(3.3) and (2.17)]{KM2011} and
  $\ _0\! \Ext^1_{A_{2}}(I_{A_{2}}/I^2_{A_{2}},I_{c-1})$ is a subgroup of
  $\ _0\! \Ext^1_{R}(I_{A_{2}},I_{c-1})$, it suffices to show
  $\ _0\! \Ext^1_{R}(I_{A_{2}},I_{c-1})=0$. By the
  Eagon-Northcott resolution \eqref{EN} we see that the largest possible
  degree of a relation for $I_{A_{2}}$ is $\ell_2-b_1$ and the smallest
  possible degree of a generator of $I_{c-1}$ is
  $\ell _c-\sum _{j=t}^{t+c-2} a_{j}$. Since
  $\ell _c= \ell _2 + \sum_{j=t+2}^{t+c-1}a_j$, we get
  $\ _0\! \Ext^1_{R}(I_{A_{2}},I_{c-1})=0$ from
  $$\ell_2-b_1 < \ell _2 + \sum_{j=t+2}^{t+c-1}a_j-\sum _{j=t}^{t+c-2} a_{j}
  = \ell _2 - a_{t} -a_{t+1} + a_{t+c-1} \ , $$
  i.e. from $a_{t+c-1} >a_t+a_{t+1}-b_1$ and we are done.
\end{proof}

\begin{remark} \label{improvement} \rm  (1) It is worthwhile to point out that
  Theorem \ref{std_det_case1} solves Conjecture 4.2 stated by  the authors in
  \cite{KM2011}. Also \cite{FF} and \cite{K2014} proves Conjecture 4.2.

  (2) For $n-c=1$ the hypothesis $a_{t+c-1}>a_t+a_{t+1}-b_1$ in Theorem
  \ref{std_det_case1} cannot be eliminated. In fact, as in \cite[Example
  4.1(ii)]{K2011} we consider the case $t=2$, $a_1=\cdots =a_{c}=1$,
  $a_{c+1}=2$ and $b_1=b_2=0$, i.e., $\cA$ is a $2\times (c+1)$ matrix with
  entries general linear forms everywhere except in the last column where we
  have general quadratic forms. The vanishing of the $2\times 2$ minors of
  $\cA$ defines a smooth irreducible curve $C\subset \PP^{c+1} $ of degree
  $\deg(C)=2c+1$ and genus $g(C)=c$. The degree and the genus of $C$ can be
  computed using the resolution of $I(C)$ given by the Eagon-Northcott complex
  (\ref{EN}):
$$
0 \longrightarrow R(-c-2)^c \longrightarrow R(-c)^{c-1}\oplus R(-c-1)^{c(c-1)}\longrightarrow \cdots $$
$$\longrightarrow R(-2)^{{c\choose 2}}\oplus R(-3)^c\longrightarrow R\longrightarrow R/I(C)\longrightarrow 0.
$$
Since $K_i=0$ for $3\le i \le c$, we have
 $$\dim \overline{W(0,0;1,\cdots , 1,2)}=\lambda  _c=c^2+7c+2.$$ On the other
 hand, $\dim _{(C)}\Hilb^{p_C(t)} (\PP^{c+1})$ is at least
 $$(c+2)\deg (C)+(c-2)(1-g(C))=c^2+8c.$$ Therefore, it follows that, for
 $c\ge 3$, $\overline{W(0,0;1,\cdots, 1,2)}$ is not an irreducible component
 of $\Hilb^{p_C(t)} (\PP^{c+1})$, whence not every deformation of $C$ comes
 from deforming its associated matrix $\cA$.

(3) Let $C\subset \PP^6$ be a smooth standard determinantal curve of
degree $21$ and arithmetic genus $15$ defined by the maximal minors of
a $3\times 7$ matrix $\cA$ with linear entries. The
closure of $W(\underline{b};\underline{a})$ inside $\Hilb
^{21t-14}(\PP^6)$ is not an irreducible  component. In fact, let $
H_{21,15}\subset \Hilb ^{21t-14}(\PP^6)$ be the open subset
parameterizing smooth connected curves of degree $d=21$ and
arithmetic genus $g=15$. It is well known that any irreducible
component of $H_{21,15}$ has dimension $\ge 7d+3(1-g)=105$; while
by Theorem \ref{std_det_case}, $\dim W(0,0,0;1,1,1,1,1,1,1)=90.$ Therefore, there exist deformations of $C$ which do not come from deforming $\cA$.
\end{remark}

To increase applications we will include a series of results that instead of
considering a matrix with general entries we consider standard determinantal
schemes defined by matrices with explicitly depth conditions. These results
could be used to treat for instance Example 2.2(ii). In this example we have
seen that the rational normal curve $C$ in $\PP^n$ is a standard determinantal
scheme whose associated matrix has no general linear entries but it is linear
and the singular locus has enough depth. Many results in our earlier papers,
as well as \cite{FF}, deal with general matrices, while \cite{K2014} proves
most results for determinantal schemes with depth conditions on the locus of
submaximal minors, allowing corollaries for determinantal schemes defined by
general matrices. In fact we prove in \cite{K2014} the results by directly
compare deformations of $R \twoheadrightarrow A$ to those of the $R$-module
$MI$ without deleting columns (except in \cite[Theorem 4.5]{K2014}), hence
entries which are units are allowed. To introduce these results we first
recall the definition of good determinantal scheme.

\begin{definition} \rm
 A codimension $c$ standard determinantal scheme $X\subset \PP^{n}$ is
called a \emph{good determinantal} scheme if its associated $t\times (t+c-1)$ homogeneous matrix $\cA$ contains a $(t-1)\times (t+c-1)$ submatrix
(allowing a change of basis if necessary) whose ideal of maximal
minors defines a scheme of codimension $c+1$.
 \end{definition}

\begin{remark}\label{good-standard} \rm
  It is well known that a good determinantal scheme $X\subset \PP^{n}$ is
  standard determinantal and the converse is true provided $X$ is a generic
  complete intersection (i.e. $\depth_{I_{t-1}(\cA)A}A \ge 1$), cf.
  \cite{KMNP}. Thus good or standard (determinantal) are
  equivalent assumptions in  Theorem~\ref{Amodulethm3} below.
\end{remark}

In \cite{K2014} the first author proved, for $c \ge 2$, the following theorem

\begin{theorem} \label{Amodulethm3} Let $X=\Proj(A) \subset \PP^{n}$,
  $A=R/I_{t}(\cA)$, be a good determinantal scheme of
  $ W(\underline{b};\underline{a})$.

\begin{itemize}
\item[(i)] If
  $\depth_{I_{t-1}(\cA)A}A \ge 3$, or if $n-c \ge 1$ and we
  get a local complete intersection (e.g. a smooth) scheme by deleting some column of $\cA$,
  then $$\dim W(\underline{b};\underline{a}) = \lambda_c + K_3+K_4+...+K_c \ .
 $$
In particular this equality holds if $n-c \ge 1$
 and $a_{i-1} \ge b_i$ for $
2\le i\le t$.
\item[(ii)] If $\depth_{I_{t-1}(\cA)A} A \ge 4$, or if $n - c\ge 2$ and we get
  a local complete intersection (e.g. a smooth) scheme by deleting some column of $\cA$, then the
  Hilbert scheme $\Hilb ^{p_X(t)}(\PP ^n)$ is smooth at $(X)$,
  $\dim_{(X)} \Hilb ^{p_X(t)}(\PP^n) = \dim W(\underline{b};\underline{a})$, and
  every deformation of $X$ comes from deforming $\cA$.
    \end{itemize}
\end{theorem}

\begin{proof} See\cite[Theorem 5.5, Corollary 5.6 and Theorem 5.8]{K2014}.
\end{proof}

Here we say that we get a local complete intersection (shortly, l.c.i.) scheme
by deleting some column if
\begin{equation} \label{delsmooth}
\mathfrak m = \sqrt{I_{t-1}(\cB)},
\end{equation}
where $\cB $ is the matrix obtained by deleting a column of $\cA$. It implies
$ \sqrt{ I_{t-1}(\cB)} = \sqrt{I_{t-1}(\cA)}$ and
$\depth_{I_{t-1}(\cA)A}A = \dim A$. 
Using this notion we succeeded in \cite{K2014},
with a rather complicated proof (see \cite[Theorem 4.5]{K2014}), to weaken the
assumption $\depth_{I_{t-1}(\cA)A}A \ge j$ as above in several theorems. In
this paper we  generalize Theorem~\ref{Amodulethm3} further and by
a much easier proof. Indeed, we will prove

\begin{theorem} \label{Amodulethm5} Let $X=\Proj(A) \subset \PP^{n}$,
  $A=R/I_{t}(\cA)$ be a standard determinantal scheme of
  $ W(\underline{b};\underline{a})$. 

\begin{itemize}
\item[(i)] If $\depth_{I_{t-1}(\cA)A}A \ge 2$ and $c \ge 2$, {\rm or}
  $\depth_{I_{t-1}(\cA)A}A \ge 3$ and $c = 1$, then
  $$\dim W(\underline{b};\underline{a}) = \lambda_c + K_3+K_4+...+K_c \ .
 $$

\item[(ii)] If $\depth_{I_{t-1}(\cA)A} A \ge 3$ and $c \ge 2$, then the
  Hilbert scheme $\Hilb ^{p_X(t)}(\PP ^n)$ is smooth at $(X)$,
  $$\dim_{(X)} \Hilb ^{p_X(t)}(\PP^n) = \dim W(\underline{b};\underline{a})\,
  ,$$ and
  every deformation of $X$ comes from deforming $\cA$.
    \end{itemize}
\end{theorem}

To show it we need to generalize the following result from \cite{K2014}
(avoiding \cite[Theorem 4.5]{K2014}) in a similar way, i.e. by replacing
``$\depth_{I_{t-1}(\cA)A}A \ge j$, or $\dim A \ge j-1$ provided we get an
l.c.i. (e.g. a smooth) scheme by deleting some column of $\cA$'' by
``$\depth_{I_{t-1}(\cA)A}A \ge j-1$'' for $j=3, 4$.

\begin{theorem} \label{Amodulecor1} Let $A=R/I_{t}(\cA)$ be a standard
  determinantal ring, let $MI= \coker (\varphi ^* )$ and suppose

\begin{itemize}
\item[(i)] either $\depth_{I_{t-1}(\cA)A}A \ge 3$, or $\dim A \ge 2$
  provided we get an l.c.i. (e.g. a smooth) scheme by deleting some column of
  $\cA$. Then $\Hom_A(MI,MI) \simeq A$ and
 $$ \Ext_A^1(MI,MI) = 0\ .$$
\item[(ii)]  either
   $\depth_{I_{t-1}(\cA)A}A \ge 4$, or $\dim A \ge 3$ provided we get an
   l.c.i. (e.g. a smooth) scheme by deleting some column of $\cA$. Then
   $\Hom_A(MI,MI) \simeq A$ and
 $$ \Ext_A^i(MI,MI) = 0 \ \ {\rm \ for \ } 1\le i\le 2\ .   $$
 In particular \ \ \ $ \Ext_R^1(MI,MI) \simeq \Hom_R(I_t(\cA ),A)\,.$ \\[-3mm]
\item[(iii)] $\depth_{I_{t-1}(\cA)A}A \ge 1$. Then the local deformation
  functor, ${\rm Def}_{MI/R}$, of $MI$ as a graded $R$-module is formally smooth
  (i.e. $MI$ is unobstructed), $\depth \Ext_R^1(MI,MI) \ge \dim A -1$, and
 $$\dim \ _0\! \Ext_R^1(MI,MI) = \lambda_c + K_3+K_4+...+K_c \ . $$
     \end{itemize}
\end{theorem}

\begin{proof} See \cite[Corollaries 4.7, 4.9 and Theorem 5.2]{K2014} for (i)
  and (ii) and \cite[Theorem 3.1]{K2014} for (iii). Note that $c \ge 2$ in
  \cite{K2014}, but the proof in \cite[Theorem 3.1]{K2014} easily applies to
  $c=1$, leading to
\begin{equation} \label{thm31c1}
\depth \Ext_R^1(MI,MI) \ge \dim A-2, \ \ {\rm and \ } \
 \dim\, _0\! \Ext_R^1(MI,MI) = \lambda_1\
\end{equation}
when $\depth_{I_{t-1}(\cA)A}A \ge 2$ because then $ \Hom_R(MI,MI) \simeq
A$ by \eqref{NM} with $Z=V(I_{t-1}(\cA)A)$.
 \end{proof}

 \begin{remark} \label{modulerem} \rm (i) By deformation theory
   $\ _0\! \Ext_R^1(MI,MI)$ is the tangent space of the local deformation
   functor ${\rm Def}_{MI/R}$. Moreover, using the Buchsbaum-Rim complex, we
   get that $MI$ is unobstructed and that every deformation of $MI$ comes from
   deforming $\cA$ (see the $1^{st}$ paragraph of the proof of \cite[Theorem
   3.1]{K2014}). The unobstructedness of $MI$ is also proved in Ile's thesis
   (see \cite{I01}, ch.\! 6).

   (ii) By observing that the map
   $$d_1: \wedge^{t+1}G^* \otimes S_{0}(F)\otimes \wedge^tF \to G^*$$ appearing
   in the Buchsbaum-Rim complex \eqref{BR} satisfies
   $\im d_1 \subset I_A \cdot G^*$ for  $c\ge 2$ ($d_1=0$ if $c=1$), we get
   $\ _0\! \Hom_R(d_1,MI)=0$ since $MI$ is an $A$-module. Applying
   $ \Hom_R(-,MI)$ onto \eqref{BR}, the definition of $\ _0\! \Ext_R^i(MI,MI)$
   yields, for $c \ge 1$, an exact sequence
\begin{equation} \label{homext}
  0 \to  \ _0\! \Hom_R(MI,MI) \to\ _0\! \Hom_R(F^*,MI) \to \ _0\!
  \Hom_R(G^*,MI) \to \ _0\! \Ext_R^1(MI,MI) \to 0 \ ,
\end{equation}
determining $\dim\, _0\! \Ext_R^1(MI,MI)$, and an injection
\begin{equation} \label{moduleremseq}
 \Ext_R^2(MI,MI) \hookrightarrow \  \Hom_R( \wedge^{t+1}G^* \otimes
S_{0}(F)\otimes \wedge^tF,MI)\ .
\end{equation}
\end{remark}

\vspace{3mm}
To see that the generalized version of Theorem~\ref{Amodulecor1} leads to
Theorem~\ref{Amodulethm5} we use the exact sequence
\begin{equation} \label{specseq} 0 \to \Ext_A^1(MI,MI) \to \Ext_R^1(MI,MI) \to
E_2^{0,1} \to \Ext_A^2(MI,MI) \to \Ext_R^2(MI,MI)\to E_2^{1,1} \to \
\end{equation}
associated to the spectral sequence
$$E_2^{p,q}:=
\Ext_A^p(\Tor_q^R(A,MI),MI) \ \ \Rightarrow  \ \ \Ext_R^{p+q}(MI,MI),$$
noting that $E_2^{p,0} \simeq \Ext_A^p(MI,MI)$ and that
$ \Tor_q^R(A,MI) \simeq \Tor_{q-1}^R(I({X}),MI)$ for $q \ge 1$. We get
$$E_2^{0,1} \simeq \Hom_A(I(X) \otimes_R MI,MI) \simeq \Hom_R(I(X),\Hom_R(MI,MI))
\simeq \Hom_R(I(X),A), $$
and looking carefully (cf. \cite[Lemma 3.3]{K2018}) we may identify
$\Ext_R^1(MI,MI) \to E_2^{0,1}$ with the tangent map from ${\rm Def}_{MI/R}$ into
the functor of graded deformations of $R \to A$ induced by Fitting's lemma,
cf. \cite[Corollary 20.4]{eise}. Hence, it suffices to prove

\begin{theorem} \label{Amodulecor2} Let $A=R/I_{t}(\cA)$ be a standard
  determinantal ring, let $MI= \coker (\varphi ^* )$ and suppose $c = 1$ and
  $\depth_{I_{t-1}(\cA)A}A \ge 3$, {\rm or } $c \ge 2$ and
  $\depth_{I_{t-1}(\cA)A}A \ge 2$ (resp. $\depth_{I_{t-1}(\cA)A}A \ge 3$).
  Then $\Hom_A(MI,MI) \simeq A$ and
 $$ \Ext_A^i(MI,MI) = 0 \ {\rm for} \ \ i=1 \ {\rm (resp.} \ \  i=1  {\rm \ and} \ 2) \ .$$
 In particular, if $\depth _{I_{t-1}(\cA )A}A\ge 3$ and $c \ge 2$, then
 $\Ext ^1_R(MI,MI)\simeq \Hom_R(I_t(\cA),A)$.
\end{theorem}

\begin{proof} For $c \ge 2$, note that $\Hom_A(MI,MI) \simeq A$ by \cite[Lemma
  3.2]{KM2005}, owing to $\depth_{I_{t-1}(\cA)A}A \ge 1$.

  Let $Z=V(I_{t-1}(\cA)A)$ and suppose $\depth_{I_{t-1}(\cA)A}A \ge 2$ and
  $c \ge 2$. Using Theorem~\ref{Amodulecor1}(iii) we get
  $$\depth_{I(Z)} \Ext_R^1(MI,MI) \ge \depth_{I(Z)}A -1 \ge 1 \ .$$
  Applying the local cohomology functor $H^0_{I(Z)}(-)$, which is left exact,
  onto \eqref{specseq} we therefore get
  $H^0_{I(Z)}( \Ext_A^1(MI,MI)) \hookrightarrow H^0_{I(Z)}( \Ext_R^1(MI,MI)) =
  0$,
  whence $\Ext_A^1(MI,MI)=0$ because $\widetilde MI$ is locally free of rank 1
  over $\Spec(A) \setminus Z$. Using \eqref{thm31c1} we get
  $\Ext_A^1(MI,MI)=0$ if $\depth_{I_{t-1}(\cA)A}A \ge 3$ and $c = 1$ by the
  same argument.

  Suppose $\depth_{I_{t-1}(\cA)A}A \ge 3$ and $c \ge 2$ and define
  $C \hookrightarrow \Ext_A^2(MI,MI)$ such that the sequence
$$0 \to \Ext_R^1(MI,MI) \to \Hom_R(I(X),A) \to C \to 0$$ obtained from
\eqref{specseq} and $\Ext_A^1(MI,MI)=0$, is exact. Applying $H^0_{I(Z)}(-)$ to
this sequence we get
$$H^0_{I(Z)}(C) \hookrightarrow H^1_{I(Z)}(\Ext_R^1(MI,MI)) = 0\,.$$ But
$C \subset \Ext_A^2(MI,MI)$ and $\widetilde MI$ is locally free over
$\Spec(A) \setminus Z$. So $\widetilde C = 0$ over $\Spec(A) \setminus Z$ and
we get $C=0$. Finally using \eqref{specseq} and \eqref{moduleremseq} we get
injections
$$ \ \Ext_A^2(MI,MI) \hookrightarrow \ \Ext_R^2(MI,MI) \hookrightarrow \
\Hom_R( \wedge^{t+1}G^* \otimes S_{0}(F)\otimes \wedge^tF,MI)\ ,$$
whence $ H^0_{I(Z)}(\Ext_A^2(MI,MI))=0$ since $ H^0_{I(Z)}(MI)=0$ and we have
proved that $\Ext_A^2(MI,MI)=0$.

Finally remark how easily we get Theorem~\ref{Amodulethm5}(ii) in the case
$\depth_{I_{t-1}(\cA)A}A \ge 3$. Indeed by \eqref{specseq}
$ \Ext_R^1(MI,MI) \cong \Hom_R(I(X),A)$ are isomorphic and then
Theorem~\ref{Amodulecor1}(iii) concludes the proof.
\end{proof}

To illustrate Theorems \ref{std_det_case} and \ref{std_det_case1} and Remark
\ref{improvement}, we consider the particular case of standard determinantal
schemes defined by the maximal minors of a matrix with {\em all} entries of
the same degree $d\ge 1$. We have

\begin{corollary} \label{same_deg} Fix integers $d\ge 1$, $t\ge 2$, $c\ge 2$
  and set $a_i=d$ for $1\le i\le t+c-1$ and $b_j=0$ for $1\le j \le t$. Assume
  $n-c\ge 2$. Then, $\overline{W(\underline{0};\underline{d})}$ is a
  generically smooth irreducible component of $\Hilb ^{p_X(t)}(\PP^n)$ of
  dimension $$\lambda _c=t(t+c-1){n+d\choose d}-t^2-(t+c-1)^2+1.$$
\end{corollary}

As another example we have:

\begin{example} \rm Let $m \ge 1$ and $n \ge 4$ be integers. In the case $n=4$
  we also suppose $m\ge 3$. Let $X\subset \PP^n$ be a standard determinantal
  scheme of codimension 3 given by the maximal minors of a $t\times (t+2)$
  matrix $[L,M]$ where $L$ is a matrix with linear entries and $M$ a column
  with entries homogeneous forms of degree $m$. Note that $X$ is a curve if
  $n=4$, but Theorem~\ref{std_det_case1}(ii) applies to any $n \ge 4$, to get
  that $\overline{W(\underline{0};1, \cdots,1,m)}$ is a generically smooth
  irreducible component of $\Hilb ^{p_X(t)}(\PP^n)$ of dimension
$$  \dim _{(X)}\Hilb ^{p_X(t)}(\PP^n)  =  \lambda _3+K_3 = t(t+2)(n+1)-2t^2-4t-3 \quad {\rm for}
\ m=1, \quad {\rm and } $$
$$
\begin{array}{rcl} \dim _{(X)}\Hilb ^{p_X(t)}(\PP^n) & = & \lambda _3+K_3 \\
& = & {m+n\choose n}t+t(t+1)(n-1)-1-(t+1){m+n-1\choose n}+{m+n-t-1\choose n}.
\end{array} \ .$$
Note that for  $n=4$ and $1 \le m \le 2$,  $\overline{W(\underline{0};1,
  \cdots,1,m)}$ is not always a component, see Remark~\ref{improvement}(2).
\end{example}

The goal of our work is to generalize all results on standard determinantal schemes collected in this section  and to study  the Hilbert scheme  $\Hilb ^{p_X(t)}(\PP^n)$ along the locus of determinantal schemes $
W(\underline{b};\underline{a};r)$ and to determine the dimension of $W(\underline{b};\underline{a};r)$ in terms of $a_i$ and $b_j$.


\section{Unobstructedness of quotients of zerosections}
\label{section}

The aim of this section is to demonstrate the following key result in
achieving our goals: Any determinantal subscheme $X=\Proj(A)\subset \PP^n$,
$A=R/I_{t+1-r}(\varphi ^*)$ 
is determined by a section $$\sigma^*: B\longrightarrow N\otimes B(a_{t+c-1}),$$
(as its degeneracy locus) where $B=R/I_{t+1-r}(\varphi_{t+c-2}^*)$ is the
determinantal ring of $(t+1-r)\times (t+1-r)$ minors of the matrix
obtained deleting the last column of $\varphi ^*$ and
$N=\coker (\varphi ^*_{t+c-2})$. This result will allow us to obtain results
for the Hilbert scheme $ \Hilb ^{p_X(t)}(\PP^{n})$ around $(X\subset \PP^{n})$
by using the deformation theory given in \cite{K2007} for deforming the
embedded scheme $Y=\Proj(B)\subset \PP^n$ together with deforming its regular
section, i.e. for deforming the pair $((N\otimes B)^*,\sigma)$,
$(-) ^*=\Hom_B(- ,B)$.

\vskip 2mm To fix a general setting for deforming $\Proj(B)\subset \PP^n$ and
its regular section $\sigma$, let $B=R/I_B$ be any graded quotient, $M$ a finitely
generated graded $B$-module and $p_Y(t)\in \QQ[t]$ the Hilbert polynomial of
$Y:=\Proj (B)\subset \PP^n$. Let $Z\subset Y$ be a closed subset such that
$\tilde{M}_{|U}$ is locally free of rank $r$ on $U:=Y\setminus Z$. Let
$\sigma _U\in \ H^0(U,\tilde{M}^*(s))$ be a regular section on $U$ inducing a
section $\sigma :M_1(-s)\longrightarrow B$ where
$M_{i}=H^0_*(U,\wedge ^{i}\tilde{M})$. Let $A=\coker (\sigma)$,
$X:=\Proj(A)$ and $I_{A/B}=\ker(B\twoheadrightarrow A)$. It holds
\begin{proposition}\label{propoA} Let $r=2$. With notations and conditions as
  above, suppose
\begin{itemize}
\item[{\rm (i)}] $\depth _{I(Z)}M_2\ge 3$, $\depth _{I(Z)}M\ge 2$,
  $\depth _{\mathfrak{m}}M_2\ge 4$ and $\depth _{\mathfrak{m}}M\ge 3$.

\item[{\rm (ii)}] $ _{(-s)}\!\Ext^2_B(M,M_2)=0$ and $ _s\!\Ext^1_B(M,B)=0$.

\item[{\rm (iii)}] $ _0\!\Ext^2_B(M,M)=0$, or $ _0\!\Ext^1_B(M,M)=0$ and $M$ is
  "liftable"  to any (graded) deformation of $B$.
  \end{itemize}
  Then the $2^{nd}$ projection
  $p_2: \Hilb ^{p_X(t),p_Y(t)}(\PP^{n})\longrightarrow \Hilb
  ^{p_Y(t)}(\PP^{n})$ is smooth at $(X \subset Y)$. Moreover if
\begin{itemize}
\item[{\rm (iv)}] $ _0\!\Ext^1_B(I_B/I^2_B, I_{A/B})=0$,
\end{itemize}
then the $1^{st}$ projection $p_1: \Hilb ^{p_X(t),p_Y(t)}(\PP^{n})\longrightarrow \Hilb
  ^{p_X(t)}(\PP^{n})$  is smooth at $(X \subset Y)$. Hence the
natural map $H^0(X,{\mathcal N}_X)\longrightarrow H^0(X,{\mathcal N}_{Y|X})$
is surjective, $X$ is $p_Y$-generic, and
$$h^0{\mathcal N}_X+hom_{{\mathcal O}_{\PP^n}}({\mathcal I}_Y,{\mathcal
  I}_{X/Y})=h^0{\mathcal N}_Y+hom_{{\mathcal O}_{Y}}({\mathcal
  I}_{X/Y},{\mathcal O}_{X})\ ,$$
$$\dim _{(X)}\Hilb ^{p_X(t)}(\PP^n)+hom_{{\mathcal O}_{\PP^n}}({\mathcal I}_Y,{\mathcal I}_{X/Y})=\dim _{(Y)}\Hilb ^{p_Y(t)}(\PP^n)+hom_{{\mathcal O}_{Y}}({\mathcal I}_{X/Y},{\mathcal O}_{X}).$$
In particular, $Y$ is unobstructed if and only if $X$ is unobstructed.
\end{proposition}

\begin{proof} The proposition follows from \cite[Proposition 13 and
  Theorem 47]{K2007}  after having remarked some details. Indeed, since $\sigma _U$ is a
  regular section on $U$, we have the Koszul resolution
$$0\longrightarrow \wedge ^2\tilde{M}(-2s) \longrightarrow \tilde{M}(s)  \stackrel{\sigma _U}{\longrightarrow} \tilde{B} \longrightarrow \coker \sigma _U\longrightarrow 0$$
which is exact over $U$. Splitting the Koszul resolution into two short exact
sequences and applying $H^0_*(U,-)$ we get that
\begin{equation} \label{liftsection}
 0\longrightarrow M_2(-2s)\longrightarrow M_1(-s)  \stackrel{\sigma
 }{\longrightarrow} B\longrightarrow H^0_*(U,\coker \sigma _U)
\end{equation}
is exact because $H^1_*(U,\wedge ^2\tilde{M})=0$ by the assumption (i). Since by
definition $A=B/ \im (\sigma)$ we have proved the exactness of
 $$ 0\longrightarrow M_2(-2s)\longrightarrow M_1(-s)\longrightarrow I_{A/B}
 \longrightarrow 0.$$

 If $\depth _{I(Z)}M_2\ge 4$ and $\depth _{I(Z)}M_1\ge 3$, we even get
 $H^1_*(U,\im (\sigma _U))=0$, whence $A\cong H^0_*(U,\coker (\sigma _U))$ and
 $\depth _{I(Z)}A\ge 2$, while the weaker assumptions of (i) only imply
 $\depth _{\mathfrak{m}}A\ge 2$. We also have $\depth _{\mathfrak{m}}B\ge 2$
 (since $\depth _{\mathfrak{m}}M_1\ge 3$), in which case the natural projection map
 $$p_2: \Hilb ^{p_X(t),p_Y(t)}(\PP^{n})\longrightarrow \Hilb
 ^{p_Y(t)}(\PP^{n})$$
 is isomorphic to the corresponding projection of the graded deformation
 functors at $(X \subset Y)$ by \eqref{Grad} and \eqref{GradPair}. Then
 \cite[Theorem 47]{K2007} (see its proof for the smoothness of $p_1$)
 applies since \cite[Proposition 13]{K2007} and (iii) above imply that the
 assumptions of \cite[Theorem 47]{K2007} are fulfilled.
\end{proof}

The proof in \cite[Theorem 47]{K2007} is actually just a proof of the fact
that any deformation $B_T \to A_T$, $T$ local artinian with residue field $k$,
fits into a lifting of $M_1(-s) \stackrel{\sigma }{\longrightarrow} B$ in
\eqref{liftsection} to $T$ (so to $B_T$) and that this lifting $\sigma _T$ can
be lifted further to any artinian local $T'$ with $T' \twoheadrightarrow T$.

\begin{remark} \rm (1) We have $M=M_1$ because $\depth _{I(Z)}M\ge 2$.
  Similarly if we assume $2\le \depth _{I(Z)}(\wedge ^2M)$ we get
  $M_2\cong \wedge ^2M$ because $M_2=H^0_*(U,\wedge ^2\widetilde{M})$. In this
  case we may in  Proposition \ref{propoA}(i) replace $M_2$ by
  $\wedge ^2M$.

(2) Due to the exact sequence
\begin{equation} \label{liftsectionI}
0\longrightarrow M_2(-2s)\longrightarrow
M_1(-s)\longrightarrow I_{A/B}\longrightarrow 0\,,
\end{equation}
 we get that
$$ {\rm (iv')} \quad \qquad \qquad  _0\!\Ext^1_B(I_B/I^2_B,M_1(-s))=0  \qquad \text{ and }
\qquad  _0\!\Ext^2_B(I_B/I^2_B, M_2(-s))=0 \qquad \qquad \qquad $$
imply (iv) of Proposition \ref{propoA}, i.e.  $ _0\!\Ext^1_B(I_B/I^2_B, I_{A/B})=0$. Using (iv') instead of (iv) in Proposition \ref{propoA} we get a result with  assumptions only on $B$ (and $M_{i}$, $i=1,2$).
\end{remark}

Our next goal will be to analyze if the section $\sigma$ exists in our
situation and whether the hypothesis (i)-(iv) of the above proposition are
satisfied. There are quite a lot of $\Ext$-groups that have to vanish, but
fortunately our applications are for determinantal schemes and for
determinantal schemes we will see that the assumptions are often fulfilled!
Note that if $c=1$, then \eqref{liftsectionI} leads to
\begin{equation} \label{introMs}
 {\footnotesize
0 \rightarrow K_B(t-2s)\rightarrow N_B(-s)  \stackrel{\sigma}{\rightarrow}  B
\rightarrow A \rightarrow 0 \,,
}
\end{equation}
where $ N_B:=\Hom_R(I_B ,B)$ is the normal module, i.e. exactly to the case we
studied in \cite{KM2009} in which many of the $\Ext$-groups above vanish
``almost for free'' because $K_B$ is the canonical module.

We keep the notation introduced in section \ref{preli}. So, we have a graded
morphism
$$\varphi: F:=\oplus _{i=1}^tR(b_{i})\longrightarrow \oplus _{j=1}^{t+c-1}R(a_j)=:G$$
and if $c \ge 3-r$, i.e. $t+c-2 \ge t+1-r$ we delete the last row to get
$$\varphi _{t+c-2}:F\longrightarrow \oplus _{j=1}^{t+c-2}R(a_j)=:G_{t+c-2}.$$
Let $B=R/I_{t+1-r}(\varphi _{t+c-2}^*)$, $A=R/I_{t+1-r}(\varphi ^*)$,
$N=\coker (\varphi _{t+c-2}^*)$, $MI:=\coker (\varphi ^*)$, $X=\Proj(A)$ and
$Y=\Proj(B)$. We are going to prove the main result of this section, namely
that $X$ is determined by a (twisted) regular section
$\sigma ^*$ of $N\otimes B$ where $\sigma ^*=\Hom_B(\sigma ,B)$. This result
generalizes \cite[Proposition 4.3]{KM2009} from $c=1$, $r=2$ to $c\ge 1$,
$r\ge 1$. With a slight change in the proof, mainly replacing $=\Proj(B)$ by
$=\Spec(B)$, the result also holds for $\dim B = r$. More precisely we have

\begin{theorem}\label{mainthm} With $B\twoheadrightarrow A$ and $N$ as above
  (with $B$ and $A$ determinantal),
  we suppose $\depth _{J_B}B\ge 2$, where $J_B=I_{t-r}(\varphi _{t+c-2}^*)$ and
  $\dim B > r$. Then the commutative diagram
\[
\begin{array}{cccccccccc}0 & & & & & & & \\
\downarrow & &  & & & & &
    \\
G_{t+c-2}^*\otimes B &  \longrightarrow & F^*\otimes B & \longrightarrow & N\otimes B & \longrightarrow & 0\\
\downarrow & & \| & & & & &
    \\
    G^*\otimes B &  \longrightarrow & F^*\otimes B & & & & & & \\
    \downarrow & &  & & & & & \\
    B(-a_{t+c-1}) & & & & & & &\\
    \downarrow & &  & & & & &
    \\
    0 & & & & & & &
\end{array}
\]
induces a homogeneous section
$\sigma ^*:B\longrightarrow N\otimes B(a_{t+c-1})$, regular on the open subset
$\Proj(B)\setminus V(J_B)$ where $\widetilde{N\otimes B}$ is locally free of
rank $r$, whose zero locus precisely defines $A$ as a quotient of $B$, i.e.
$A=B/\im (\sigma)$. Moreover a twist of $\sigma ^*$ fits into an exact
sequence
\begin{equation}\label{section}
0\longrightarrow B(-a_{t+c-1})\stackrel{\sigma ^* }{\longrightarrow} N\otimes
B \longrightarrow MI\otimes B\longrightarrow 0\, .
\end{equation}
\end{theorem}

\begin{proof}
We will start proving the theorem for $2\times 2$ minors by describing the maps in the diagram
\begin{equation}\label{exact*}
\begin{array}{cccccccccc}
  &  & & &  B(-a_{t+c-1}) & &\\
  &  & & & \sigma ^*\downarrow & &\\
G_{t+c-2}^*\otimes B &  \stackrel{\varphi^*_{t+c-2}\otimes id_B }{\longrightarrow} & F^*\otimes B & \longrightarrow & N\otimes B & \longrightarrow & 0
\end{array}
\end{equation}
over the open set $D(f_{11})\subset \Proj(B)$,  letting the matrix of $\varphi ^*$ be ${\mathcal A}=(f_{ij})_{1\le i\le t \atop 1\le j\le t+c-1}$.
We will use Gauss elimination as follows. Let
$$I(f_{11})=\begin{pmatrix}
1 & 0 & 0 &  \cdots & 0 \\
-f_{21} & f_{11} & 0 & \cdots & 0 \\
-f_{31} & 0 & f_{11} &  \cdots & 0 \\
\vdots & \vdots & \vdots & \ddots & \vdots \\
-f_{t1} & 0 & 0 &  \cdots & f_{11}
\end{pmatrix}.
$$
Then $$I(f_{11})\cdot {\mathcal A}=\begin{pmatrix}
f_{11} & f_{12} & f_{13} &  \cdots & f_{1t+c-1} \\
0 & f_{22}^2 & f_{23}^2 & \cdots & f_{2t+c-1}^2 \\
\vdots & \vdots & \vdots & \ddots & \vdots \\
0 & f_{t2}^2 & f_{t3}^2 &  \cdots & f_{tt+c-1}^2
\end{pmatrix},$$
where $f_{ji}^2=-f_{j1}f_{1i}+f_{11}f_{ji}$, $2\le i \le t+c-1$ and
$2\le j \le t$. These are all the $2\times 2$ minors involving $f_{11}$.
Moreover since $B=R/I_2(\varphi ^*_{t+c-2})$, all $2\times 2$ minors above
belong to $I_2(\varphi ^*_{t+c-2})$ and lead to a vanishing of a
corresponding map in (\ref{exact*}), except those involving the last
column. More
precisely, (\ref{exact*}) extended to $G^*$ and restricted to the open set
$D(f_{11})\subset \Proj(B)$ is given by

\[
\xymatrix{ &G^*\otimes B_{(f_{11})} \ar[r]^{\varphi^*\otimes id_{B_{(f_{11})}}} \ar[rd]  & F^*\otimes B_{(f_{11})} \ar[d]^{\simeq} \\
&   & (\oplus _{i=1}^tR(-b_{i}))\otimes B_{(f_{11})}  &
}
\]
where
$F^*\otimes B_{(f_{11})}\cong (\oplus _{i=1}^tR(-b_{i}))\otimes B_{(f_{11})} $
is given by multiplication with $I(f_{11})$, and the morphism
$$G^*\otimes B_{(f_{11})}\longrightarrow (\oplus _{i=1}^tR(-b_{i}))\otimes
B_{(f_{11})}$$ is induced by
$$\begin{pmatrix}
f_{11} & f_{12} &  \cdots & f_{1t+c-2} & f_{1t+c-1} \\
0 & 0 & \cdots & 0 &  f_{2t+c-1}^2 \\
\vdots & \vdots & \ddots & \vdots & \vdots \\
0 & 0 &  \cdots & 0 &  f_{tt+c-1}^2
\end{pmatrix} \ .
$$
So $G_{t+c-2}^*\otimes B_{(f_{11})} \stackrel{\begin{pmatrix} f_{1i} \\
    0 \end{pmatrix}}{\longrightarrow} \oplus _{i=1}^tB(-b_{i})_{(f_{11})}$
and we have $N\otimes B_{(f_{11})}\cong \oplus _{i=2}^t B(-b_{i})_{(f_{11})}$
since the ideal of $ B_{(f_{11})}$ generated by
$\{f_{11}, f_{12}, \cdots,f_{1t+c-2}\}$ contains $1=f_{11}/f_{11}$. Moreover
$\sigma ^*$ restricted to the degree zero part of $B(-a_{t+c-1})_{f_{11}}$ is
given by mapping $1\in B(-a_{t+c-1})_{f_{11}}$ onto the transpose of the
vector $$[f_{2t+c-1}^2 \ f_{3t+c-1}^2 \ \cdots \ f_{tt+c-1}^2]:=v_{11}^{tr} \,.$$
Note that this argument can be done for any entry $f_{ij}$, $j \ne t+c-1$ of
${\mathcal A}$, e.g. for $f_{1j}$, yielding a map
$B(-a_{t+c-1})_{f_{1j}} \to \oplus _{i=2}^t B(-b_{i})_{f_{1j}}$ taking
$1 \in B(-a_{t+c-1})_{f_{1j}}$ onto the corresponding $v_{1j}$ which for $j=2$
is $['f_{2t+c-1}^2 \ 'f_{3t+c-1}^2 \ \cdots \ 'f_{tt+c-1}^2]^{tr}$, where
$'f_{ji}^2=-f_{j2}f_{1i}+f_{12}f_{ji}$. Since
$f_{1j} \cdot v_{11} = f_{11} \cdot v_{1j}$ these local sections glue
together.

The argument above extends also to $f_{ij}$ for $i >1$. Indeed, taking the
$i$th row, we exchange it with the first row and proceed as above except for
being a little more careful with the gluing. Say we exchange the $1^{st}$ and
$2^{nd}$ row and we use Gauss elimination leaving the left upper corner, where now
$f_{21}$ sits, fixed. Then the section $\sigma ^*$ restricted to the degree
zero part of $B(-a_{t+c-1})_{f_{21}}$ is given by mapping
$1\in B(-a_{t+c-1})_{f_{21}}$ onto the transpose of
$[''f_{1t+c-1}^2 \ ''f_{3t+c-1}^2 \ \cdots \ ''f_{tt+c-1}^2]:=v_{21}^{tr}$,
where now $''f_{ji}^2=-f_{j1}f_{2i}+f_{21}f_{ji}$. Letting $F(f_{21})$ be the
following $(t-1) \times (t-1)$ matrix:
$$F(f_{21})=\begin{pmatrix}
-f_{11} & 0 & 0 &  \cdots & 0 \\
-f_{31} & f_{21} & 0 & \cdots & 0 \\
-f_{41} & 0 & f_{21} &  \cdots & 0 \\
\vdots & \vdots & \vdots & \ddots & \vdots \\
-f_{t1} & 0 & 0 &  \cdots & f_{21}
\end{pmatrix},
$$
the local sections glue because $F(f_{21}) \cdot v_{11} = f_{11} \cdot v_{21}$.
Thus we get $\sigma^* $ well defined over
$$\bigcup _{j\ne t+c-1}D(f_{ij})=\Proj (B)\setminus V(I_1(\varphi
^*_{t+c-2})) $$
and it extends to $\Proj (B)$ by the depth assumption. It follows that the
description of $\sigma $ in the Theorem holds noting that, with
$\{f^2_{ij}\}=\{f^2_{ij}\}^{2\le i \le t}_{2\le j \le t+c-2}$, we have the following
equalities of ideals 
$$I_2(\varphi ^*_{t+c-2})B_{(f_{11})}= B_{(f_{11})}(\{f^2_{ij}\})\, ,$$
and $I_2(\varphi ^*_{t+c-2})B_{(f_{21})}= B_{(f_{21})}(\{''f^2_{ij}\})$ and
correspondingly for all the $B_{(f_{ij})}$ having $j \ne t+c-1$.
%

Now we consider the case of $3\times 3$ minors and we will describe (\ref{exact*}) in the open set $D(f_{22}^2)$ of $\Proj(B)$, where $B=R/I_3(\varphi ^*_{t+c-2})$. To this end, we denote by $f^{i_1i_2\cdots i_k}_{j_1j_2\cdots j_k}$ the determinant of the submatrix of ${\mathcal A}$ consisting of the columns $i_1,i_2,\cdots, i_k$ and the rows $j_1,j_2,\cdots,j_k$. As above, to describe
(\ref{exact*}) in the open set $D(f_{12}^{12})\subset \Proj(B)$  we  use Gauss elimination. We multiply ${\mathcal A}$ by the $t\times t$ matrix:
$$I(f_{11},f_{12}^{12})=\begin{pmatrix}
1 & 0 & 0 &  0 &\cdots & 0 \\
-f_{21} & f_{11} & 0 & 0 & \cdots & 0 \\
f_{23}^{12} & -f_{13}^{12} & f_{12}^{12} & 0 &  \cdots & 0 \\
f_{24}^{12} & -f_{14}^{12} & 0 & f_{12}^{12} &   \cdots & 0 \\
\vdots & \vdots & \vdots & \vdots & \ddots & \vdots \\
f_{2t}^{12} & -f_{1t}^{12} &  0 & 0 &  \cdots & f_{12}^{12}
\end{pmatrix}
$$
and we get, using determinantal expansion along columns, that
$$I(f_{11},f_{12}^{12}){\mathcal A}=\begin{pmatrix}
f_{11} & f_{12} & f_{13} &  f_{14} & f_{15} & \cdots & f_{1t+c-1} \\
0 & f_{12}^{12} & f_{12}^{13} & f_{12}^{14} & f_{12}^{15} & \cdots & f_{12}^{1t+c-1} \\
0 & 0 & f_{123}^{123} & f_{123}^{124} & f_{123}^{125} &  \cdots & f_{123}^{12t+c-1} \\
0 & 0 & f_{124}^{123} & f_{124}^{124} & f_{124}^{125} &  \cdots & f_{124}^{12t+c-1} \\
\vdots & \vdots & \vdots & \vdots & \vdots & \ddots & \vdots \\
0 & 0 & f_{12t}^{123} & f_{12t}^{124} & f_{12t}^{125} &  \cdots & f_{12t}^{12t+c-1} \\
\end{pmatrix} ,
$$
where all $3\times 3$ minors belong to $I_3(\varphi ^*_{t+c-2})$, except those
involving the last column which define the section $\sigma ^*$ over the open
$D(f_{11})\cap D(f_{12}^{12})\subset \Proj(B)$. Note that all $3\times 3$
minors of the form $f_{12j}^{12i}$, $i,j\ge 3$ are obtained from the lower
$(t-2)\times (t+c-3)$ block. And, moreover, if we start the Gauss elimination
above by using the matrix $I(f_{12},f_{12}^{12})$ defined to be equal to
$I(f_{11},f_{12}^{12})$ for all rows except the 2nd row where
$[-f_{21} \ f_{11} \ 0 \ 0 \cdots ]$ is replaced by
$[-f_{22} \ f_{12} \ 0 \ 0 \cdots ]$, we see that $I(f_{12},f_{12}^{12})$ is
invertible over $D(f_{12})\cap D(f_{12}^{12})$ and that
$I(f_{12},f_{12}^{12}){\mathcal A}$ is equal to
$I(f_{11},f_{12}^{12}){\mathcal A}$ in all rows except the 2nd row where we get
$[-f_{12}^{12} \ 0 \ f_{12}^{23} \ f_{12}^{24} \ \cdots \ f_{12}^{2t+c-1}]$ for
$I(f_{12},f_{12}^{12}){\mathcal A}$. This defines $\sigma ^*$ in the open set
$D(f_{12})\cap D(f_{12}^{12})$ and since
$(f_{12}^{12})\subset (f_{11},f_{12})$ we have $\sigma ^*$ defined over
$$(D(f_{12})\cap D(f_{12}^{12}))\cup (D(f_{11})\cap D(f_{12}^{12}))=D(f_{12}^{12}).$$
Thus $$\sigma ^*:B(-a_{t+c-1})_{(f_{12}^{12})}\longrightarrow N\otimes B_{(f_{12}^{12})}=\oplus _{i=3}^tB(-b_{i})_{(f_{12}^{12})}$$
is given by the $3\times 3$ minors of the last column of $I(f_{11},f_{12}^{12}){\mathcal A}$, where $\im (\sigma) \otimes  B_{(f_{12}^{12})}$ is the ideal of $  B_{(f_{12}^{12})}$ generated by $\{ f_{123}^{12t+c-2},f_{124}^{12t+c-2},\cdots ,f_{12t}^{12t+c-2} \}$ and we conclude because a determinantal ideal does not change under elementary row/column operations.

It is now clear how to proceed to finish the proof. If $B=R/I_4(\varphi ^*_{t+c-2})$, we multiply ${\mathcal A}$ by
$$I_1=\begin{pmatrix}
1 & 0 & 0 &  0 & 0 & \cdots & 0 \\
-f_{21} & f_{11} & 0 & 0 & 0 & \cdots & 0 \\
f_{23}^{12} & -f_{13}^{12} & f_{12}^{12} & 0 &  0 & \cdots & 0 \\
-f_{234}^{123} & f_{134}^{123} & -f_{124}^{123} & f_{123}^{123} & 0 &   \cdots & 0 \\
-f_{235}^{123} & f_{135}^{123} & -f_{125}^{123} & 0 &  f_{123}^{123}  &   \cdots & 0 \\
\vdots & \vdots & \vdots & \vdots & \vdots & \ddots & \vdots \\

-f_{23t}^{123} & f_{13t}^{123} & -f_{12t}^{123}&  0 & 0 &  \cdots & f_{123}^{123}
\end{pmatrix}
$$
and determinantal expansions along columns imply that  $I_1 \cdot{\mathcal A}$ is equal to  $I(f_{11},f_{12}^{12}){\mathcal A}$
 in the first 3 rows and that the last $t-3$ rows of  $I_1 \cdot{\mathcal A}$  will be of the form
 $$[ 0  \ | \  f_{123j}^{123i}], \quad  4\le i \le t+c-1, \quad 4\le j\le t$$
 and $0$ a $(t-3)\times 3$ matrix of zeros which lead to
 $$\im (\sigma ) \otimes \Gamma (U_{i}, \tilde{B})\cong \Gamma(U_i,\widetilde{I_{A/B}})$$
 where $\{ U_{i} \}$ is an open covering of $\Proj(B)\setminus V(I_3(\varphi
 ^*_{t+c-2}))$.

 More precisely letting $I_2$ be $I_1$ except for the 2nd row which is replaced by $[-f_{22} \ f_{12} \ 0 \ 0 \cdots ]$ and proceeding as for $3\times 3$ minors above, the section $\sigma ^*$ will be defined over the open $D(f_{12}^{12})\cap D(f_{123}^{123})$. Then since
 \begin{equation} \label{f123}
 f_{123}^{123}=f_{31}f_{12}^{23}-f_{32}f_{12}^{13}+f_{33}f_{12}^{12}
 \end{equation}
 we replace the 3rd row of $I_1$
 by $[f_{23}^{13}
 \ -f_{13}^{13} \ f_{12}^{13} \ 0 \ 0 \ \cdots ]$ to get
 $I_3$,
 respectively $[f_{23}^{23}
 \ -f_{13}^{23} \ f_{12}^{23} \ 0 \ 0 \cdots ]$ to get $I_4$. Computing $I_{i}
 \cdot {\mathcal A}$ for $i=3,4$, we obtain $I_1 \cdot {\mathcal
   A}$ except in the 3rd row where we obtain $[0\ -f_{123}^{123} \ 0 \
 f_{123}^{134} \ f_{123}^{135} \ \cdots ]$ for $I_3 \cdot {\mathcal
   A}$; respectively, $[f_{123}^{123} \ 0 \ 0 \ f_{123}^{234} \ f_{123}^{235}
 \ \cdots ]$ for $I_4 \cdot {\mathcal
   A}$. Combining with replacing the 2nd row with $[-f_{22} \ f_{12} \ 0 \
 \cdots ]$, cf. $I_2$ above, we get the section $\sigma ^*$ defined over
 $$D(f_{12}^{13})\cap D(f_{123}^{123}), \ \text{respectively } D(f_{12}^{23})\cap D(f_{123}^{123}). $$
 Therefore,  $\sigma
 ^*$ is defined over
 $D(f_{123}^{123})$
 by (\ref{f123}). Continuing this process we get in general the existence of
 the section $\sigma ^*$ whose zero locus defines $A$ by $A=B/\im (\sigma)$.

Finally note that $\sigma ^*$ is injective because it is regular on
  $\Proj(B)\setminus V(J_B)$ Thus the sequence
\begin{equation*}
0\longrightarrow B(-a_{t+c-1})\stackrel{\sigma ^* }{\longrightarrow} N\otimes B \longrightarrow MI\otimes B\longrightarrow 0
\end{equation*}
is exact, and we are done.
\end{proof}

We will apply Proposition~\ref{propoA} to the section
$\sigma ^*:B\to N\otimes B(a_{t+c-1})$ given by Theorem~\ref{mainthm}. To do
so letting $M=(N\otimes B)^*:= \Hom_B(N\otimes B,B)$, we need to verify all
assumptions of  Proposition~\ref{propoA}.

\begin{proposition} \label{vanish} Let
  $M=(N\otimes B)^*$,
  $J_B:=I_{t-r}(\varphi _{t+c-2}^*)$ and assume  $c\ge 3-r$ and $r \ge 2$.  It holds:
\begin{itemize}
\item[{\rm (i)}] Suppose $\depth _{J_B}B\ge 1$ and $J_B \ne R$. If $c \ge 2$ then
  $M$ and $N\otimes B$ are maximal Cohen-Macaulay $B$-modules. If $c \le 1$,
  then $M$ is a maximal Cohen-Macaulay $B$-module while $N\otimes B$ has
  $\codepth 1$.
\item[{\rm (ii)}] If $\depth _{J_B}B\ge 3$ and $J_B \ne R$, then
  $\Ext^{i}_B(M,B)=0$ and $\Ext^{i}_B(N\otimes B,B)=0$ for $i=1,2$.
\item[{\rm (iii)}] Suppose $\depth _{J_B}A\ge 2$, $J_A:=I_{t-r}(\varphi ^*)\ne R$
  and if $3-r \le c \le 0$ we also suppose $\depth _{J_A}A\ge 3$. Then
  $\Ext^{1}_B(M,I_{A/B})=0$ and the normal module of $B\twoheadrightarrow A$
  is maximally Cohen-Macaulay (resp. of $\codepth 1$) for $c \ge 1$
  (resp. $3-r \le c \le 0$) and satisfies
  $$\Hom_B(I_{A/B},A) \cong\Hom_B(M,A)(a_{t+c-1})\cong MI\otimes A(a_{t+c-1})
  \, .$$
\item[{\rm (iv)}] Suppose $\depth _{J_B}A\ge 3$, $J_A \ne R$ and if
  $3-r \le c \le 0$ we also suppose $\depth _{J_A}A\ge 4$. Then
  $\Ext^{1}_B(M,A)=\Ext^1_A(I_{A/B}/I_{A/B}^2,A)=0$ and
  $\Ext^{i}_B(M,I_{A/B})=0$ for $i=1,2$.
  \end{itemize}
\end{proposition}

\begin{remark} \label{minormax} \rm
(1) Proposition~\ref{vanish} (i) and (ii) hold also for $c \ge 2-r$ and $r=1$.

(2) Note that we have $\dim B-\dim A=r$ and $\dim B-\dim B/J_B=c+2r-1$ if
$B/J_B$
is {\it determinantal} (allowing $J_B:=J_B \cdot B$). Considering $J_B$
as an ideal in $A$, then $A/J_B\cong B/J_B$ and we get:
$$\depth _{J_B}A=\dim A-\dim A/J_B=c+r-1 \, ,$$
while $\depth _{J_B}A \le c+r-1$ holds in general. So assuming
$\depth _{J_B}A\ge 3$ in Proposition \ref{vanish}(iv) we implicitly assume
$c \ge 4-r$. If the matrix $\cA$ is defined by general homogeneous polynomials
and $b_t < a_1$, then we may replace $\depth _{J_B}A\ge 3$ by $c \ge 4-r$ (and
$\dim A\ge 3$), and simultaneously $\depth _{J_A}A\ge 4$ by $\dim A\ge 4$.
\end{remark}

\begin{proof} (i)  It follows from \cite[ Theorems 1 and 2]{B}.

(ii) It follows from \cite[Theorems 4 and 5]{B}.

(iii) Let $J=J_B$ and $U_B:=\Proj(B)\setminus V(J)$. Sheafifying the exact
sequence
\begin{equation}\label{seq}
0\longrightarrow B(-a_{t+c-1})\stackrel{\sigma ^* }{\longrightarrow} N\otimes B \longrightarrow MI\otimes B\longrightarrow 0
\end{equation}
we see that $\tilde{\sigma ^*}_{|U_B}$ is given by minors belonging to
$\widetilde{I_{A/B}}|U_B$, whence $(\tilde{\sigma ^*}\otimes id_A)_{|U}=0$,
where $U:=U_B\cap \Proj(A)$. It follows that
$\widetilde{N\otimes A}_{|U}\cong \widetilde{MI\otimes A}_ {|U}$.

Dualizing the exact sequence  (\ref{seq}), we get the diagram
$$
\begin{array}{cccccccccc}
& &  &  & 0 \\
& &  & & \downarrow \\
& &   M & \twoheadrightarrow & I_{A/B}(a_{t+c-1}) \\
& &  \| & & \downarrow \\
\Hom_R(MI,B) & \longrightarrow & \Hom_R(N,B) &   \stackrel{\sigma }{\longrightarrow} &  B(a_{t+c-1} ).
\end{array}
$$
If we sheafify and restrict to $U_B$, we get again that
$\tilde{M}_{|U_B}\stackrel{\tilde{\sigma}_{|U} }{\twoheadrightarrow}
\widetilde{I_{A/B}(a_{t+c-1})}_{|U_B}$
is generated by the minors of $\widetilde{I_{A/B}}$, whence
$\widetilde{M\otimes A}_{|U}\stackrel{\cong }{\longrightarrow}
\widetilde{I_{A/B}(a_{t+c-1})\otimes A}_{|U}$.
Since $N\otimes B$ is locally free over $U_B$ we have
$$
\begin{array}{rcl} \widetilde{M\otimes A}_{|U} & \cong & {\mathcal H}om(\widetilde{N\otimes
  B},\tilde{B})_{|U_B}\otimes \tilde{A}_{|U} \\
  & \cong & {\mathcal
  H}om(\widetilde{N\otimes A},\tilde{A})_{|U} \\
  & \cong &  {\mathcal
  H}om(\widetilde{MI\otimes A},\tilde{A})_{|U}.
  \end{array}
  $$

Let $c \ge 1$. Applying $\Hom _R(-,A)$ once more, and using
$\depth _JA=\depth _J(MI\otimes A)\ge 2$ we get
$$
\begin{array}{rcl}\Hom_R(M,A)& \cong &  H^0_*(U,{\mathcal H}om(\tilde{M},\tilde{A}))\\
 & \cong & H^0_*(U,\widetilde{MI\otimes A})=MI\otimes A
 \end{array}
 $$
and, similarly
$$
\begin{array}{rcl}\Hom_R(M,A) & \cong & H^0_*(U,{\mathcal H}om(\widetilde{I_{A/B}(a_{t+c-1})}\otimes \tilde{A},\tilde{A})) \\
& \cong & \Hom_A(I_{A/B}/I_{A/B}^2,A)(-a_{t+c-1}).\end{array}
$$
These are maximal Cohen-Macaulay $A$-modules because $MI\otimes A$ is a
maximal Cohen-Macaulay $A$-module for $c \ge 1$ by (i). If, however,
$3-r \le c \le 0$, then the assumption $\depth _{J_A}A\ge 3$ implies
$\depth _{J_A}(MI\otimes A)\ge 2$ and letting $U_A:=\Proj(A)\setminus V(J_A)$,
we get
$$
\begin{array}{rcl} MI\otimes A & \cong & H^0_*(U_A,\widetilde{MI\otimes A}) \\
& \cong &
H^0_*(U,\widetilde{MI\otimes A})\end{array}
$$
because $MI\otimes A$ is locally free over $U_A \supset U$ and
$\depth _JA \ge 2$. Thus the two displayed formulas for $\Hom_R(M,A)$ above
hold in this case too. Moreover, since we have an exact sequence
\begin{equation}\label{diag}
\begin{array}{ccccccccccccc}
0 & \rightarrow & \Hom_B(M,I_{A/B}) & \rightarrow & \Hom_B(M,B) & \rightarrow & \Hom_B(M,A) &\rightarrow & \Ext^1_B(M,I_{A/B}) & \rightarrow \\
& & & & \| & & \| \\
& & & & N\otimes B & \twoheadrightarrow & MI\otimes A
\end{array}
\end{equation}
and $\Ext^1_B(M,B)=0$ by (ii), we get $\Ext^1_B(M,I_{A/B})=0$.

(iv) Let $c \ge 1$. Since  $\depth _JA=\depth _J(MI\otimes A)\ge 3$ and
$\widetilde{M\otimes A}_{|U}\cong \widetilde{I_{A/B}\otimes
  A(a_{t+c-1})}_{|U}$
are locally free  we get
$$
\begin{array}{rcl}\Ext^1_B(M,A) & \cong & H^1_*(U_B,{\mathcal H}om(\tilde{M}\otimes \tilde{A},\tilde{A})) \\
& \cong & H^1_*(U_B,\widetilde{MI\otimes A}) \\
& = & H^2_J(MI\otimes A) \\
& = & 0\end{array} $$
and up to twist;
$$
\begin{array}{rcl} \Ext^1_A(I_{A/B}/I_{A/B}^2,A) & \cong &  H^1_*(U,{\mathcal H}om(\tilde{I_{A/B}}\otimes \tilde{A},\tilde{A})) \\
& \cong &  H^1_*(U_B,\widetilde{MI\otimes A}) \\
& = & H^2_J(MI\otimes A)\\
& = & 0.\end{array}$$
Then (\ref{diag}) leads to
$$
\longrightarrow\Ext^1_B(M,I_{A/B})\longrightarrow \Ext^1_B(M,B) \longrightarrow \Ext^1_B(M,A)\longrightarrow \Ext^2_B(M,I_{A/B})\longrightarrow
$$
and using $\Ext^{2}_B(M,B)=0$ from (ii), we get
$\Ext^2_B(M,I_{A/B})=0$ which proves (iv) when $c \ge 1$.

Finally if $3-r \le c \le 0$, then the assumption $\depth _{J_A}A\ge 4$ implies
$H^2_{J_A}(MI\otimes A)=0$, and it suffices to show
$ H^1_*(U_B,\widetilde{MI\otimes A}) \cong H^1_*(U_A,\widetilde{MI\otimes A})$.
This, however, follows from $\depth _JA \ge 3$ and the fact that $MI\otimes A$ is locally free over $U_A \supset U$.
\end{proof}

Now we restrict to $r=2$, i.e. to $k$-algebras defined by submaximal minors. We
have that
\begin{equation}\label{num1}
0\longrightarrow M_2\longrightarrow M_1\longrightarrow B(a_{t+c-1})\longrightarrow A(a_{t+c-1})\longrightarrow 0
\end{equation}
is exact, whence $M_2=\Hom_B(MI,B)$ is a maximal Cohen-Macaulay $B$-module.
Indeed, since $M_1$ is a maximal Cohen-Macaulay $B$-module and $B$ and $A$ are
determinantal, hence Cohen-Macaulay, the exact sequence (\ref{num1})
implies that $M_2$ is also a maximal Cohen-Macaulay $B$-module. In this
case we have the vanishing of the following $\Ext$-groups.

\begin{proposition}\label{vanish2} Let $B=R/I_{t-1}(\varphi ^*_{t+c-2}) \twoheadrightarrow A:=R/I_{t-1}(\varphi ^*)$ be determinantal rings and suppose $c\ge 1$ and $I_{t-2}(\varphi ^*)\ne R$. Let $M=(N\otimes B)^*$ and  ${J_B}=I_{t-2}(\varphi ^*_{t+c-2})$. Then, it holds:
\begin{itemize}
\item[{\rm (i)}] If $\depth_{J_B}B\ge 3$ (resp. $\ge 4$), we have $\Ext^{i}_B(M,M_2)=0$ for $i=1$ (resp. $i=1,2$).
\item[{\rm (ii)}] Suppose $\depth_{J_B}A\ge 2$
(resp. $\depth_{J_B}A\ge 3$).
Then we have
  $$\Ext^{i}_B(M,M)=0 \ {\rm for} \ i=1  {\rm \ (resp.} \  i=1, 2  {\rm )} \, .$$
\end{itemize}
\end{proposition}

\begin{proof} (i) The exact sequence (\ref{num1}) is induced  from the Koszul resolution of a regular sequence of 2 elements, whence $\tilde{M_2}_{|U_B}\cong \wedge ^2\tilde{M}(-a_{t+c-1})_{|U_B}$. Moreover since the rank of $M$ is 2, we have isomorphisms
$$
\begin{array}{rcl}\tilde{M}^*\otimes \wedge ^2\tilde{M}_{|U_B} & \cong & \tilde{M}_{|U_B}\\
& \cong &  {\mathcal H}om(\tilde{M},\tilde{M_2}(a_{t+c-1}))_{|U_B}.\end{array}$$
Using $\depth _{J_B}B=\depth _{J_B}M=\depth _{J_B}M_2\ge 3$, we get
$$
\begin{array}{rcl}
\Ext^1_B(M,M_2) & \cong &  \Ext^1_{{\mathcal O}_{U_B}}(\tilde{M},\tilde{M_2}) \\
& \cong &  H^1_*(U_B,{\mathcal H}om(\tilde{M},\tilde{M_2})) \\
& \cong &  H^2_{J_B}(M(-a_{t+c-1})) \\
& = & 0 \end {array}
$$
and correspondingly, $\Ext^2_B(M,M_2)\cong H^3_{J_B}(M(-a_{t+c-1}))=0$ if $\depth_{J_B}B\ge 4$.

(ii) Using the exact sequence $$0\longrightarrow M_2\longrightarrow M_1=M\longrightarrow I_{A/B}(a_{t+c-1})\longrightarrow 0,$$ we deduce an exact sequence
$$
\longrightarrow\Ext^1_B(M,M_2)\longrightarrow\Ext^1_B(M,M)\longrightarrow \Ext^1_B(M,I_{A/B}(a_{t+c-1}))\longrightarrow\Ext^2_B(M,M_2)\longrightarrow.
$$
Since $\depth _{J_B}A\ge 2$ 
and $\depth _{J_B}B=\depth _{J_B}A+2\ge 4$, we get $\Ext^1_B(M,M)=0$ by
Proposition \ref{vanish} (iii) and Proposition \ref{vanish2} (i).

Similarly, if $\depth _{J_B}A\ge 3$,
we get $$\Ext^2_B(M,M)\cong \Ext^2_B(M,I_{A/B}(a_{t+c-1}))=0$$ by Proposition
\ref{vanish} (iv) and Proposition \ref{vanish2} (i), and we are done.
\end{proof}

\begin{corollary} \label{corSectRank2} Let
  $B=R/I_{t-1}(\varphi ^*_{t+c-2}) \twoheadrightarrow A:=R/I_{t-1}(\varphi
  ^*)$
  be determinantal rings defined by submaximal minors, let
  $J_B:=I_{t-2}(\varphi _{t+c-2} ^*)$ and suppose $c\ge 1$ and
  $I_{t-2}(\varphi ^*)\ne R$. Then:
\begin{itemize}
\item[{\rm (i)}] If $\depth_{J_B}A\ge 3$, then (i), (ii) and (iii) of
  Proposition \ref{propoA} hold.
\item[{\rm (ii)}] Suppose $\depth_{J_B}A = 2$ and that every deformation of
  $B$ comes from deforming its matrix, then (i), (ii) and (iii) of
  Proposition \ref{propoA} hold.
\end{itemize}
\end{corollary}
\begin{proof} Since $\depth _{J_B}B=\depth _{J_B}A+2\ge 4$ and $M_2$ is
  maximally Cohen-Macaulay by \eqref{liftsectionI} and we may take $Z=V(J_B)$,
  all conclusions follow from Propositions~\ref{vanish} and \ref{vanish2},
  except possibly $\Ext^2_B(M,M)=0$ which however is true if
  $\depth _{J_B}A \ge 3$. Only assuming $\depth _{J_B}A \ge 2$ we have at
  least $\Ext^1_B(M,M)=0$. Moreover in this case, $M$ is "liftable" to any
  deformation of $B$. The argument for this is given right before
  \eqref{auxdiag} in the next section; the short version of that argument is
  that there is by assumption a matrix $\cB_T$ and a corresponding morphism
  $(\varphi _{t+c-2})_T^*$ which defines any given deformation $B_T$ of $B$ to
  a local artinian $T$. Set $N_T = \coker( (\varphi _{t+c-2})_T^*)$. Then $N_T$
  is a deformation of $N$ and thus the $T$-flat $\Hom(N_T \otimes B_T,B_T)$
  lifts $M$ to $T$, and we are done.
\end{proof}

In \cite{KM2009} we studied the case $c=1$ of Proposition~\ref{propoA}. Now we
take an example where $c=2$.

\begin{example} \label{prop41ex} \rm (Determinantal quotients of
  $R=k[x_0, x_1, \cdots ,x_n]$, using Proposition~\ref{propoA})

  Let $\cA= [\cB,v]$ be a general $3 \times 4$ matrix with $\cB$ a linear
  matrix and $v$ a column with all entries of degree $m$. Thus the degree
  matrix of $\cA$ is
  $\left(\begin{smallmatrix}1 & 1  & 1  & m\\
      1 & 1 & 1 & m\\ 1 & 1 & 1 & m\end{smallmatrix}\right)$
  \ and set $b_i=0$ for $1 \le i \le 3$. The vanishing of all $2 \times 2$
  minors defines a determinantal ring that satisfies all conditions of
  Corollary~\ref{corSectRank2} for all $m\ge 1$ provided $n\ge 8$. Indeed
  $\codim_R B = 4$ and $r=2$ (submaximal minors), so $n \ge 8$ is equivalent
  to $\dim A \ge 2$. To avoid some details we suppose $\dim A \ge 3$. So
  $n \ge 9$ and since $\cB$ is general, we may, after a linear coordinate
  change, suppose that $\cB$ is the ``generic'' linear matrix with entries
  $x_0,x_1,...,x_{8}$. For such a matrix, by results of the next sections, one
  know that the closure of the determinantal locus
  $\overline{ W(0,0,0;1,1,1;2)}$ containing $(\Proj(B))$, is a generically smooth
  component of $ \Hilb(\PP^{n})$ by Proposition~\ref{deforminggenericcase} and
  Corollary~\ref{gendetrem} and of dimension $64+9(n-8)=9n-8$
  (Example~\ref{exgendet}(i) and Corollary~\ref{cordimRdetW}(ii)). In fact
  $\lambda_c$ defined in \eqref{lamda} for $c=1$ is $9n-8$.

  We {\it claim} that $\overline{ W(0,0,0;1,1,1,m;2)}$ is a generically
  smooth irreducible component of $ \Hilb(\PP^{n})$ for every $m \ge 3$ of
  dimension
  $$\dim \overline{ W(0,0,0;1,1,1;2)} +
  \dim (MI \otimes A)_m = 9n-8 + \dim (MI \otimes A)_m\ .$$
  Indeed, all assumptions of Proposition~\ref{propoA} hold except possibly
  $_0\!\Ext^1_B(I_B/I^2_B, I_{A/B})=0$. To show that it vanishes we consider
  the minimal free resolution of $R/I_B$ given by  Lascoux in \cite{L} or, even easier, the
  Gulliksen-Negaard free resolution (\cite{GN}) or you may for simplicity run Macaulay2 to
  see that the minimal free resolution is
    \begin{equation} \label{G33} 0 \longrightarrow R(-6) \longrightarrow
      R(-4)^{9} \longrightarrow R(-3)^{16} \longrightarrow
      R(-2)^{9}\longrightarrow I_B \rightarrow 0\ .
\end{equation}
The generators of $I_{A/B}$ are $2 \times 2$ minors involving the last column
of $\cA$. So all generators of  $I_{A/B}$ have degree $m+1 \ge 4$. Applying
$\ _0\!\Hom_R(-, I_{A/B})$ to \eqref{G33} we get that all terms in
   $$0 \to \ _0\!\Hom_R(I_B, I_{A/B}) \longrightarrow\  _0\!\Hom_R(R(-2)^9, I_{A/B})
   \longrightarrow\ _0\!\Hom_R(R(-3)^{16}, I_{A/B}) \to $$
   vanish. So, by definition of $_0\!\Ext^i_R(I_B, I_{A/B})$ for $i=1,2$, we have
   $$\ _0\!\Hom_R(I_B, I_{A/B})\  =  \ _0\!\Ext^1_R(I_B, I_{A/B})\  =  \ 0 \ .$$
   Finally, since
   $_0\!\Ext^1_B(I_B/I^2_B, I_{A/B}) \subset\ _0\!\Ext^1_R(I_B, I_{A/B})$, it
   follows that $_0\!\Ext^1_B(I_B/I^2_B, I_{A/B}) = 0$. Thus
   Proposition~\ref{propoA} applies and since smooth morphisms maps
   irreducible components to irreducible components we get the claim. Indeed,
   using also the exact sequences in Proposition~\ref{propoA} we get
   $$\dim \overline{ W(0,0,0;1,1,1,m;2)} = \dim \overline{ W(0,0,0;1,1,1;2)} +
   \dim (MI \otimes A)_m \,$$
   because $\Hom_B(I_{A/B},A) \cong MI\otimes A(a_{t+c-1})$ by
   Proposition~\ref{vanish}.
     \end{example}

     Example~\ref{prop41ex} illustrates how Corollary~\ref{corSectRank2}
     ``generalizes'' the main results of \cite{KM2009}, which only holds for
     submaximal minors with $c=1$. Indeed, Corollary~\ref{corSectRank2} also
     deals with submaximal minors but it allows any $c \ge 1$. However in
     \cite{KM2009} we also compute
     $\dim \overline{W(\underline{b};\underline{a};r)}$ for $c=1$, which
     really means that $\dim (MI \otimes A)_{a_{t+c-1}} \,$ is computed, under
     quite weak assumption (which we generalize further in
     Theorem~\ref{dimW1}). For $c\ge 2$, we have not been able to compute
     $\dim (MI\otimes A(a_{t+c-1}))_0$, in general, but as we will see in the
     next sections we succeed to compute it under some assumptions.

     \begin{remark} \label{rem38} \rm By Corollary~\ref{corSectRank2}, if
       $\depth_{J_B}A\ge 3$ (resp. $\depth _{J_B}A\ge 2$ under an additional
       assumption) and $r=2$, then all assumptions of Proposition \ref{propoA}
       are satisfied (except possibly (iv)), hence the first conclusion holds.
       The nice thing to remark now is that Proposition \ref{vanish} holds for
       {\it any} $r \ge 2$. In particular, if $\depth _{J_B}A\ge 3$, and
       $\depth _{J_A}A\ge 4$ in the case $c \le 0$, then 
       \begin{equation}\label{num2}
  \Ext^1_A(I_{A/B}/I_{A/B}^2,A)=0
\end{equation}
for any  $r\ge2$. Since $\depth _{J_B}A\ge 3$ implies
$${\Hl}^2(B,A,A)\cong \Ext^1_A(I_{A/B}/I_{A/B}^2,A)$$
(because $\Hom({\Hl}_2(B,A,A),A)=0$, see \cite[Section1.1]{K04} and its
references for bagkground on algebra (co)homology) we get, also for $r\ge 3$,
{\it all conclusions} of Proposition \ref{propoA} using \cite[Proposition
4(ii)]{K04}. 
More precisely, recalling $J_A:=I_{t-r}(\varphi^*)$, we have
\end{remark}

\begin{proposition} \label{everydef} Let $B\twoheadrightarrow A$ be as in
  Theorem \ref{mainthm} and suppose $\depth _{J_B}A\ge 3$,
  $J_A:=I_{t-r}(\varphi ^*) \ne R$, $r \ge 2$ and $c\ge 3-r$. Suppose also
  $\depth _{J_A}A\ge 4$ if $c \le 0$. Then $ _0{\Hl}^2(B,A,A)=0$, and the
  $2^{nd}$ projection
  $p_2: \Hilb ^{p_X(t),p_Y(t)}(\PP^{n})\longrightarrow \Hilb ^{p_Y(t)}(\PP^{n})$ is
  smooth at $(X \subset Y)$.

  Moreover, if Proposition~\ref{propoA}(iv) holds,  or
 \begin{itemize}
 \item[{\rm (iv'')}]
  $ _0\!\Ext^1_B(I_B/I^2_B, I_{A/B})\hookrightarrow \ _0\!\Ext^1_B(I_B/I^2_B,
  B)$ is
  injective and $B$ unobstructed (as a graded algebra),
  \end{itemize}
  \noindent then the $1^{st}$ projection $p_1:\Hilb
  ^{p_X(t),p_Y(t)}(\PP^{n})\longrightarrow \Hilb ^{p_X(t)}(\PP^{n})$ is smooth
  at $(X \subset Y)$.
  Hence the natural map
  $H^0(X,{\mathcal N}_X)\twoheadrightarrow H^0(X,{\mathcal N}_{Y|X})$ is
  surjective, $X$ is $p_Y$-generic, and
$$h^0{\mathcal N}_X+hom_{{\mathcal O}_{\PP^n}}({\mathcal I}_Y,{\mathcal
  I}_{X/Y})=h^0{\mathcal N}_Y+hom_{{\mathcal O}_{Y}}({\mathcal
  I}_{X/Y},{\mathcal O}_{X}) \,,$$
$$\dim _{(X)}\Hilb ^{p_X(t)}(\PP^n)+hom_{{\mathcal O}_{\PP^n}}({\mathcal
  I}_Y,{\mathcal I}_{X/Y})=\dim _{(Y)}\Hilb ^{p_Y(t)}(\PP^n)+hom_{{\mathcal
    O}_{Y}}({\mathcal I}_{X/Y},{\mathcal O}_{X}).$$
In particular, $Y$ is unobstructed if and only if $X$ is unobstructed.
\end{proposition}

\begin{proof} By Remark \ref{rem38}, $ _0{\Hl}^2(B,A,A)=0$, whence we get the
  smoothness of the projection
$$q:\GradAlg(H_B,H_A)\longrightarrow \GradAlg(H_B)$$
given by $q((B'\rightarrow A'))=(B')$ (cf. \cite[p.\,234]{K2007}) which implies
the surjectivity below
\begin{equation} \label{ZJ}
{\small
\begin{array}{ccccccccccccc}
 \Hom_B(I_{A/B},A) & \hookrightarrow & \Hom_R(I_A,A) & \rightarrow & _0\!
                                                                     \Hom_R(I_B,A)
  &\rightarrow &  _0{\Hl}^2(B,A,A) & \rightarrow . \\
 & & \| & & \| &  & \| \\
 & & H^0(X,{\mathcal N}_X) & \twoheadrightarrow &  H^0(X,{\mathcal Hom}({\mathcal I}_Y,{\mathcal O}_{X}))  & & 0
\end{array}
}
\end{equation}
Indeed, the smoothness of $$q:\GradAlg(H_B,H_A)\longrightarrow \GradAlg(H_B)$$
is a consequence of the fact that $ _0{\Hl}^2(B,A,A)$ is the obstruction group
of deforming $A_S$ to $B_T$ in the diagram
$$
\begin{array}{ccccccccccccc}
B_T & & & \\
\downarrow &  &  & & \\
B_S &  \longrightarrow &  A_S \\
\downarrow &  &  & & \\
0 & & & &
\end{array}
$$
where $T\twoheadrightarrow S$ is a small artinian surjection and $B_T$, $B_S$
and $A_S$ are deformations of $B$ and $A$ (\cite[Remark 3]{K04}).

Finally either (iv): $_0\! \Ext^1_B(I_B/I_B^2,I_{A/B})=0$ and \cite[Proposition
4(ii)]{K04}, or the statement given in
(iv'') and  \cite[Proposition 4(iii)]{K04}, imply that the projection
$$p:\GradAlg(H_B,H_A)\longrightarrow \GradAlg(H_A)$$ given by
$q((B'\rightarrow A'))=(A')$ is smooth at $(B\rightarrow A)$. The smoothness
of $p$ and $q$ imply the smoothness of $p_1:\Hilb ^{p_X(t),p_Y(t)}(\PP^{n})\longrightarrow \Hilb ^{p_X(t)}(\PP^{n})$ and $p_2:\Hilb ^{p_X(t),p_Y(t)}(\PP^{n})\longrightarrow \Hilb ^{p_Y(t)}(\PP^{n})$ by \eqref{Grad} and
\eqref{GradPair} and hence all conclusions of Proposition \ref{propoA}.
\end{proof}


\section{Deformation of  minors}

Our next goal is to analyze whether the deformation of a determinantal scheme
comes from deforming the associated matrix. To this end, let ${\mathcal A}$
(resp. ${\mathcal B}$) be the homogeneous matrix representing $\varphi ^*$
(resp. $\varphi^*_{t+c-2}$ if $c \ge 3-r$). Let $\underline{\ell}$ be the
category of local
artinian $k$-algebras $T$ with residue field $k=T/{\mathfrak m}_T$ and with
morphisms inducing the identity over $k$. Set
$A=R/I_{i}(\varphi ^*)$ where $i = t+1-r$ with $2\le i\le t$ and $X=\Proj(A)$.

\begin{lemma} \label{lemma2} Let $\varphi ^*_T$ be the map induced by a
  lifting $\cA _T$ of $\cA$ to $T \in ob(\underline{\ell})$, let
  $MI_T:=\coker (\varphi ^*_T)$ and $A_T=R/I_{i}(\varphi ^*_T)$ and suppose
  $A=R/I_{i}(\varphi ^*)$ is determinantal. Then 
\begin{itemize}
\item[{\rm (i)}] $MI_T$ (resp. $A_T$) are deformations of $MI$ (resp. $A$). In particular, there exists well defined maps
$$\psi:\ _0\!\Ext^1_R(MI,MI)\longrightarrow \ _0\!\Hom(I_{i}(\varphi ^*),A), \text{ and }$$
$$Def(\psi )(-):Def_{MI}(-)\longrightarrow Def_{A}(-)\,,$$
the latter of local functors over $\underline{\ell}$, deforming $MI$ (resp. $A$) as an $R$-module (resp. $R$-quotient).
\item[{\rm (ii)}] The following statements are equivalent for a fixed $i= t+1-r$.
\begin{itemize}
\item[(a)] Every deformation of $A$ comes from deforming the matrix $\cA$
  associated to $A$.
\item[(b)] The morphism $Def(\psi )$ of local graded deformation functors is smooth.
\item[(c)] The map $\psi$ is surjective.
\end{itemize}
\end{itemize}
\end{lemma}

\begin{proof} (i) Let $\cA_T$ be a lifting of $\cA$ (see
  Definition~\ref{defor}). Recall that $A_T$ is flat over $T$ if not only
  generators of $I_{i}(\varphi ^*)$ lift to polynomials in
  $I_{i}(\varphi ^*_T) \subset R_T$, but also relations lift to relations of
  $I_{i}(\varphi ^*_T)$. Hence we need to lift all relations since lifting of
  generators follows at once by taking appropriate minors of $\cA_T$. But the
  relations are given by the Lascoux resolution of $I_{i}(\varphi ^*)$ (see
  \cite{L}, or see Lemma~\ref{mdr_mdg} for details of {\it all} syzygies).
  Indeed a relation is either given by computing the determinant of any
  $(i+1) \times (i+1)$ matrix arising from a submatrix of $\cA$ of size
  $i \times (i+1)$ (resp. $(i+1) \times i$) by repeating one of its rows
  (resp. columns) and expanding it along the mentioned row (resp. column), or by computing the
  determinant of an $(i+1) \times (i+1)$ submatrix of $\cA$ by column and row
  expansions and taking their differences. The latter gives also non-trivial
  relations. But we can do exactly the same computations using $\cA_T$ instead
  of $\cA$. Since the corresponding determinant (i.e. relation) or difference
  of determinants map via $T\twoheadrightarrow k$ to the corresponding
  relation over $k$, we get that every relation lifts to a relation for
  $I_{i}(\varphi ^*_T)$. This shows that $A_T$ is $T$-flat, and it follows
  that $A_T$ is a deformation of $A$.

  For similar (and even easier) reasons $MI_T$ is $T$-flat. Indeed since
  $R/I_i(\varphi ^*)$ is determinantal, it follows by \cite[Corollary 1]{Br}
  that $R/I_t(\varphi ^*)$ is standard determinantal, whence that the
  Buchsbaum-Rim complex of $MI$, using $\cA$ is exact. Moreover this
  resolution of $MI$ commutes with the Buchsbaum-Rim resolution of $MI_T$
  using $\cA _T$, i.e. we can lift any relation of $MI$ to a relation of
  $MI _T$, whence $MI _T$ is a deformation of $MI$ (cf. \cite[Lemma
  4.2]{K2011}). Since $MI=\coker (\varphi ^*)$ it is clear that every element
  $(MI)_T\in Def _{MI}(T)$ is given by $MI_T=\coker (\varphi _T^*) $ for some
  matrix $\cA _T$ that lifts $\cA $ to $T$ where $\varphi ^*_T$ is the map
  induced by $\cA _T$. As proved above this matrix $\cA _T$ defines a
  deformation, i.e. an element $A_T\in Def_A(T)$. Since different matrices
  representing the same cokernel define the same determinantal ideal by e.g.
  Fitting's lemma, $Def(\psi )(T)$ is well-defined. Then we define $\psi $ as
  the tangent map of $Def(\psi )$, i.e. we let
  $\psi =Def(\psi )(k[\varepsilon]/\varepsilon^2)$, and (i) is proved.

  (ii) Since smooth maps are surjective on tangent spaces, the following
  implications are straightforward from Definition \ref{defor}:
  $(b)\Rightarrow (a) \Rightarrow (c)$. It only remains to prove
  $(c)\Rightarrow (b)$. To prove it, let us first describe $\psi $ more
  concretely. Take $\overline{\eta}\in\, _0\!\Ext^1_R(MI,MI)$, let
  $\eta'\in \, _0\!\Hom(G^*,MI)$ represent $\overline{\eta}$ and let
  $\eta\in \, _0\!\Hom(G^*,F^*)$, with matrix $\cD $, map to $\eta '$. Let
  $D=k[\varepsilon]/(\varepsilon^2)$ and let $\cA _1+\varepsilon \cD _1$ be
  some $i \times i$ submatrix of $\cA+\varepsilon\cD $ representing a
  corresponding composition
  $$G^*_1\hookrightarrow G^*
  \stackrel{\varphi^*+\varepsilon\eta}{\longrightarrow} F^*\twoheadrightarrow
  F^*_1 \ .$$
  Then
  $$\psi (\overline{\eta})(\det \cA_1)=[\det (\cA_1+\varepsilon
  \cD_1)-\det(\cA_1)]/\varepsilon =\det(\cA_1^{adj}\cdot \cD_1)\otimes 1_A$$
  is the image of $\det \cA_1$ via
  $\psi (\overline{\eta})\in\, _0\!\Hom(I_{i}(\varphi ^*),A)$ in $A$. If,
  say $i=3$, $\cA_1=[A^1 \ A^2 \ A^3]$ and $\cD_1=[D^1 \ D^2 \ D^3]$, then
  $$\psi (\overline{\eta})(\det \cA_1)=\det[D^1 \ A^2 \ A^3]+\det[A^1 \ D^2 \
  A^3]+\det[A^1\ A^2\ D^3] \,.$$

To prove the smoothness, let $T:=k[t]/(t^{n+1})\twoheadrightarrow
S:=k[t]/(t^n)$ and consider the diagram
\[
\begin{array}{cccccc}
Def_{MI}(S) &  \longrightarrow & Def_{A}(S) & \longleftarrow & Def_A(T) .\\
MI_S & \mapsto & A_S & \leftarrow & A_T
\end{array}
\]
Here $MI_S$ is a deformation of $MI$ given as the cokernel of $\varphi _S^*$
with matrix $\cA _S$ such that its $i$-minors define $A_S$. Moreover
$A_T:=R_T/I_T$ is an arbitrary deformation of $A_S$ to $T$.

Let $\cA _T$ be a lifting of the matrix $\cA _S$ to $T$. Since
$R_T \twoheadrightarrow R_S$ is surjective,  it exists and defines a
deformation $A'_T$ (resp $MI'_T$) of $A_S$ (resp. $MI_S$) to $T$. By
deformation theory since $A'_T=R_T/I_{i}(\varphi _T^*)$ and
$A_T=R_T/I_{T}$ are deformations of the same algebra
$A_S=R_S/I_{i}(\varphi _S^*)$ to $T$, then the difference of corresponding
generators of $I_{T}$ and $I_{i}(\varphi _T^*)$ maps to zero via
$(-) \otimes _TS$, i.e. these differences "belong" to
$$ \Hom _{R_S}(I_{i}(\varphi _S^*), A_S)\otimes _S(t^n)/(t^{n+1})\cong \Hom_R(I_{i}(\varphi ^*),A)\otimes _k(t^n)/(t^{n+1}).$$

By the surjectivity assumption (b) and the description of the tangent map
$\psi $ above there exists $\overline{\eta}\in\, _0\!\Ext^1_R(MI,MI)$ and
$\eta\in\, _0\!\Hom(G^*,F^*)$, with matrix $\cD $, that maps to
$\overline{\eta}$ and such that
$$t^n \cdot tr(\cA_1 ^{adj}\cD_1)\otimes 1_A \quad mod(t^{n+1})$$
is equal to the difference of the corresponding generators of
$I_{T}-I_{i}(\varphi _T^*)$ (and similarly for the other
$i$-minors).

Now look at the matrix $\cA_T+t^n\cD$ ($mod \ (t^{n+1})$). By the same
calculation as done for the map $\psi $, only replacing $\varepsilon $ by
$t^n$, one shows that
$$[\det((\cA_1)_T+t^n\cD_1)-\det((\cA_1)_T)]=t^n\det(\cA_1^{adj}\cD_1)\otimes
1_A.$$
It follows that the generators of $A_T$ are defined by the matrix
$\cA_T+t^n\cD$, i.e. given by its $(t+1-r)$-minors. Since the cokernel of the
map
given by $\cA_T+t^n\cD$ defines a deformation $MI_T$ to $T$ that maps to $A_T$
and reduces to $MI_S$ via $(-) \otimes _TS$, we have proved that $Def(\psi )(-)$
is (formally) smooth.
\end{proof}

\begin{corollary} \label{unobstr} Let $A=R/I_{t+1-r}(\cA)$ with $1\le r < t$
  be a determinantal ring, let $X=\Proj(A)$ and suppose that $A$ has the
  following property:

  \quad {\rm (*)} \quad every deformation of $A$ comes from deforming its matrix
  $\cA$.

  Then $A$ is unobstructed and the property {\rm (*)} is an open property in
  $\GradAlg(H_A)$. Hence if $\dim X \ge 1$, then $\Hilb ^{p_X(t)}(\PP^n)$ is
  smooth at $(X)$ and the property (*) is open in $\Hilb ^{p_X(t)}(\PP^n)$.
\end{corollary}
\begin{proof} To see that $A$ is unobstructed, let $T \longrightarrow S$ be a
  surjection of artinian local rings whose kernel ${\mathfrak a}$ satisfies
  ${\mathfrak a} \cdot \mathfrak{m}_T=0$, and let $A_S$ be a deformation of
  $A$ to $S$. By assumption, $A_S=R_S/I_{t+1-r}(\cA_S)$ for some matrix
  $\cA_S=(f_{ij,S})$. Since $T \longrightarrow S$ is surjective, we can lift
  each $f_{ij,S}$ to a polynomial $f_{ij,T}$ with coefficients in $T$ such
  that $f_{ij,T} \otimes_T S=f_{ij,S}$. By Lemma~\ref{lemma2}, it follows that
  $A_T:=R_T/I_{t+1-r}(\cA_T)$ is flat over $T$. Since $A_T \otimes_T S = A_S$
  we get the unobstructedness of $A$, and by \eqref{Grad} that $\Hilb
  ^{p_X(t)}(\PP^n)$ is  smooth at $(X)$.

  Alternatively one may prove that $MI$ is unobstructed as an $R$-module by a
  similar argument and then get that $A$ is unobstructed as a consequence of
  $Def(\psi )$ being smooth by Lemma~\ref{lemma2}. The smoothness also implies
  that the property (*) is open in $\GradAlg(H_A)$, as well as in
  $\Hilb ^{p_X(t)}(\PP^n)$, cf.  the text accompanying
  \eqref{Grad}.
\end{proof}

It is worthwhile to point out that it is not always true that every
deformation of $A=R/I_{i}(\varphi ^*)$ comes from deforming the homogeneous
matrix $\cA$ associated  to $\varphi ^*$ (see, for instance, Remark
\ref{improvement}(2) and Examples \ref{ex1dimW} and \ref{examples712}).

We will now show for so-called {\em generic determinantal rings} $A_{(s)}$,
that the tangent map $\psi $ in Lemma \ref{lemma2} is surjective. Thus, we can
conclude that every deformation of $A_{(s)}$ comes from deforming its
associated matrix  $\cA$. Indeed, for $1\le i \le t$, $1\le j \le t+c-1$, let
$R=k[x_{ij}]$, $\cA=(x_{ij})$ be the $t\times (t+c-1)$ matrix of
indeterminates of $R$ and let $\varphi :F=R^t\longrightarrow G=R(1)^{t+c-1}$
be the morphism induced by the transpose $\cA ^T$ of $\cA$. Then
$A_{(s)}:=R/I_s(\cA)$, $s=t+1-r$, is called a {\em generic determinantal
  ring}. By \cite[Theorem 15.15]{b-v} $A_{(s)}$ is rigid for every $r$, $1 \le
r\le t$, $(r,c)\ne (1,1)$, i.e. the algebra cohomology group ${\rm H}^1(k,A_{(s)},A_{(s)})=0$ or, equivalently, the sequence
\begin{equation}\label{rigid}
0\longrightarrow \Der _k(A_{(s)},A_{(s)})\longrightarrow \Der _k(R,A_{(s)}) \stackrel{\gamma }{\longrightarrow} \Hom _R(I_s(\cA),A_{(s)}) \longrightarrow 0
\end{equation}
is exact. Here  $\Der _k(S,L)$ is the set of $k$-derivations from a
$k$-algebra $S$ to an $S$-module $L$.

\begin{proposition} \label{deforminggenericcase}
Let $A_{(s)}$ be a generic determinantal ring, $2 \le s \le t$. Then every deformation of $A_{(s)}$ comes from deforming the matrix $\cA$ above provided $(s,c)\ne (t,1)$.
\end{proposition}

\begin{proof} Let $MI=\coker \varphi ^*$. We claim that there exist morphisms fitting into the commutative diagram
$$\begin{array}{cccccc}
\Hom(G^*\otimes F, A_{(t)}) & \cong & \Hom(G^*,F^*\otimes A_{(t)}) & \twoheadrightarrow &
 \Ext^1_R(MI,MI) \\
\parallel & & & & \downarrow \psi \\
\Der _k(R, A_{(t)}) & \twoheadrightarrow & \Der _k(R, A_{(s)}) & \stackrel{\gamma }{\rightarrow} & \Hom_R(I_s(\cA),A_{(s)})
\end{array}
$$
for every $2\le s\le t$. By Lemma \ref{lemma2} this will prove the result because $\gamma $ is surjective by (\ref{rigid}).

To see the upper horizontal surjection we combine morphisms appearing in the two exact sequences
$$G^*\otimes A_{(t)}  \longrightarrow F^*\otimes  A_{(t)} \longrightarrow
MI\longrightarrow 0, \text{ and } $$
$$\Hom(F^*,MI)\longrightarrow \Hom(G^*,MI) \longrightarrow \Ext^1_R(MI,MI)
\longrightarrow 0$$
induced from
$ \cdots \stackrel{\epsilon}{\longrightarrow} G^*\longrightarrow
F^*\longrightarrow MI \longrightarrow 0$
($\epsilon$ is the splice map in the Buchsbaum-Rim resolution),
recalling that $\Hom(G^*,MI) \longrightarrow \Ext^1_R(MI,MI)$ is well-defined
and surjective by \cite[(3.1)]{K2014} (mainly because $\Hom(\epsilon,MI)=0$)
and that $MI\otimes A_{(t)}\cong MI.$ Moreover, the first lower surjection is
induced from the natural surjection $A_{(t)}\twoheadrightarrow A_{(s)}$ since
$I_t (\cA)\subset I_s(\cA)$. Finally the leftmost vertical isomorphism is a
natural identification of $\Hom(G^*\otimes F,A_{(t)})$ and
$\Der _k(R,A_{(t)})$ because the matrix of $G^*\longrightarrow F^*$ is
$\cA=(x_{ij})$ and $R=k[x_{ij}]$.

With this identification we now check that
\begin{equation}\label{tocheck}
  \psi(\overline{\eta})(f)=\gamma(D)(f) \end{equation}
for every $\eta \in \Hom(G^*\otimes F^*,A_{(t)})$ and every $s\times s$ minor $f$, where $\overline{\eta}$ (resp. $D$) is the image of $\eta $ in $\Ext^1_R(MI,MI)$ (resp. $\Der _k(R,A_{(s)})$). Running over the standard basis of the free $A_{(t)}$-module  $\Hom(G^*\otimes F,A_{(t)})$ it suffices to check (\ref{tocheck}) for each element of the basis. Indeed, we may just take $\eta=\begin{pmatrix}  1 & 0 & \cdots 0 \\
  0 & 0 & \cdots 0 \\
  \vdots & \vdots & \cdots  \vdots \\
  0 & 0 & \cdots 0 \end{pmatrix}$. For similar reasons we may let $f$ be the
determinant of the $s\times s$ matrix appearing in the left-upper corner of $\cA$. Then we expand this minor along its first column and we get  $f=x_{11}A_{11}-x_{21}A_{21}+\cdots $ where $A_{ij}$ is the ($s-1$)-minor corresponding to $x_{ij}$. We get
$$\gamma(D)(f)=\sum \frac{\partial f}{\partial x_{ij}}D(x_{ij})=\frac{\partial f}{\partial x_{11}}=A_{11}$$
because $D\in \Der _k(R,A_{(s)})$, the derivation corresponding to
$\overline{\eta }$, is given by $D(x_{11})=1$ and $D(x_{ij})=0$ otherwise. Moreover,
by the proof of Lemma \ref{lemma2}, we get
$$\psi (\overline{\eta})(f)=\det[D^1A^2A^3\cdots]+\det[A^1D^2A^3\cdots]+\cdots =\det \begin{pmatrix}1 & & & &  \\ 0 & & & &
\\ 0 & A^2 & A^3 \cdots  \\
 \vdots & & & & \\
0 & & & &
\end{pmatrix}= A_{11},$$
where  $\eta =[D^1D^2 \cdots ]$ and $\cA=[A^1A^2A^3\cdots ]$ are expressed by their columns. To check that also the signs corresponds, let $\eta =\begin{pmatrix}  0 & 1 & \cdots 0 \\
0 & 0 & \cdots 0 \\
\vdots & \vdots & \cdots  \vdots \\
0 & 0 & \cdots 0 \end{pmatrix}$. With the same $f$ as above, it is easy to see that $\gamma (D)(f)=\frac{\partial f}{\partial x_{12}}=-A_{12}$ and $\psi (\overline{\eta})(f)=\det[A^1D^2A^3\cdots]=-A_{12}$ and we are done.
\end{proof}

\begin{corollary} \label{gendetrem} \rm Let $A=R/I_s(\cA)$, $2 \le s \le t$ be
  a generic determinantal ring, or more generally a determinantal ring for which
  every deformation comes from deforming its matrix $\cA$. Let $S$ be a flat
  $R$-algebra and a polynomial ring over $k$, e.g.
  $S=R \otimes_k k[\underline y]$ where $y_1, \cdots, y_e$ are indeterminates.
  Then the matrix $\cA=(f_{ij})$ induces a corresponding matrix $\cA_S$ whose
  entries consists of the images of $f_{ij}$ in $S$ and we let
  $A_S=S/I_s(\cA_S)$. Then every deformation of $A_S$ comes from deforming its
  matrix $\cA_S$.
\end{corollary}

\begin{proof} Let $\varphi _S$ be the morphism corresponding to the transpose
  of $\cA _S$ and let $MI_S:=\coker \varphi _S^*$. Then $MI_S=MI\otimes _R S$
  by the right-exactness of the tensor product. Moreover, if we apply
  $(-) \otimes _R S$ onto the first terms of the Lascoux resolution of $A$, the
  $R$-flatness of $S$ easily implies that $S/I:=A\otimes _R S$ satisfies
  $I=I_s(\cA _S)$, and that we have a commutative diagram
$$\begin{array}{cccc}
\Ext^1_R(MI,MI)\otimes _R S  &  \stackrel{\psi \otimes 1_S }{\longrightarrow} & \Hom(I_s(\cA),A)\otimes _R S \\
\downarrow \simeq & & \downarrow \simeq \\
\Ext^1_S(MI_S,MI_S)  &  \stackrel{\psi _S }{\longrightarrow} & \Hom(I_s(\cA _S),S/I)
\end{array}
$$
where $\psi _S$ is the map $\psi $ of Lemma \ref{lemma2} for the module
$MI_S$. In particular, $\psi _S$ is surjective as $\psi$ is surjective by
Proposition~\ref{deforminggenericcase} or assumption. By Lemma \ref{lemma2} we get that
every deformation of $S/I$ comes from deforming its matrix $\cA _S$.
\end{proof}

\vskip 2mm
The main theorem of this section is:

\begin{theorem}\label{teo3} Let $B=R/I_B\twoheadrightarrow A$ be
  determinantal algebras defined by $t+1-r$ minors of matrices ${\mathcal A}$
  (resp. ${\mathcal B}$) representing $\varphi ^*$ (resp.
  $\varphi^*_{t+c-2}$), and suppose $r\ge 2$, $c\ge 3-r$,
  $J_A:=I_{t-r}(\cA)\ne R$ and $\depth _JB\ge r+2$ where
  $J=I_{t-r}({\mathcal B})$. If $c \le 0$ we also suppose
  $\depth _{J_A}A\ge 3$. Moreover, suppose
  \begin{itemize}
  \item[{\rm (i)}]
  $_0\!\Ext^1_B(I_B/I_B^2,I_{A/B})\longrightarrow \ _0\!\Ext^1_B(I_B/I_B^2,B)$
  is injective,  and
\item[{\rm (ii)}] every deformation of $B$ comes from deforming
  ${\mathcal B}$.
  \end{itemize}
  Then every deformation of $A$ (or $\Proj(A)$) comes from deforming $\cA$.
\end{theorem}

\begin{remark} \rm \label{rem53} Computations with Macaulay2 suggest for
  $\dim A \ge 3$ (resp. $\dim A \ge 4$), that
  $_0\!\Ext^1_B(I_B/I_B^2,I_{A/B})=0$ if $c\ne 1$ (resp. $c=1$), but the
  dimension assumption seems important (see, for instance, Example
  \ref{examples712}). This vanishing of $_0\!\Ext^1_B(I_B/I_B^2,I_{A/B})$ is
  also true under some restrictions on $a_i$ (cf.
  Corollary~\ref{corWsmooth3} and Proposition~\ref{conj3ext}). So, quite often we
can eliminate the hypothesis (i) in the above theorem.
\end{remark}
\begin{proof} Let $M=(N\otimes B)^*$ and note that $\depth _JB\ge r+2$ implies
  $\depth _JA\ge 2$, and that we have
\begin{equation}\label{propo3}
\Ext^1_B(N\otimes B,B)=\Ext^1_B(M,B)=\Ext^1_B(M,I_{A/B})=0
\end{equation}
by Proposition \ref{vanish}. Since it suffices by Lemma \ref{lemma2} to show
that any deformation of $A$ to the dual numbers
$T:=k[\varepsilon]/( \varepsilon ^2)$ comes from deforming $\cA $, we only
consider deformations to $T$.

Let $A_T$ be any deformation of $A$ to $T$. By assumption (i) we have a surjection  in the diagram
\begin{equation} \label{auxdiag1}
\begin{array}{cccccc}
_0\!\Hom(I_B,B) & \twoheadrightarrow & _0\!\Hom(I_B,A) & \leftarrow & _0\!\Hom_R(I_A,A) & \leftarrow \\
\exists (B_T) & \dashrightarrow  & \bullet & \leftarrow & (A_T) \\
\end{array}
\end{equation}
Hence there exists a deformation $B_T$ of $B$ to $T$ and a morphism
$B_T\longrightarrow A_T$ reducing to $B\longrightarrow A$ via
$(-) \otimes _Tk$ (cf.\,Remark~\ref{rem0thm42}). Moreover by assumption (ii),
there exists a $t \times (t+c-2)$-matrix ${\mathcal B}_T$ such that
$B_T=R_T/I_{t-r+1}({\mathcal B}_T)$, $R_T=R\otimes _k T$. Let $N_T$ be the
cokernel of the morphism determined by ${\mathcal B}_T$. Since $B_T$ and $N_T$
are flat over $T$ by Lemma \ref{lemma2}, we get that $N_T\otimes B_T$ is
$T$-flat and $N_T\otimes B_T\otimes k=N\otimes B$. Then
$M_T:=\Hom _{B_T}(N_T\otimes B_T,B_T)$ is also $T$-flat because
$\Ext^1_B(N\otimes B,B)=0$ (see \cite[Proposition (A1)]{dJ}). Moreover,
$I_{A_T/B_T}:=\ker (B_T\longrightarrow A_T)$ is $T$-flat, and hence we have
deformations of $M$, $I_{A/B}$ and $B$ fitting into a diagram
\begin{equation}\label{auxdiag}
\begin{array}{ccccccccc}
  M_T(-a_{t+c-1}) & \stackrel{\exists \sigma _T}{\dashrightarrow} &  I_{A_T/B_T} &\hookrightarrow & B_T \\
  \downarrow &  & \downarrow  &  &\downarrow \\
  M(-a_{t+c-1}) & \stackrel{ \sigma' }{\twoheadrightarrow} &  I_{A/B} & \stackrel{ i }{\hookrightarrow} & B \\
  \downarrow &  & \downarrow  &  &\downarrow \\
  0 & & 0  & &  0 \\
\end{array}
\end{equation}
where $i \cdot \sigma' = \sigma$ and where the dotted arrow exists due to the fact that
$_0\!\Ext^1_B(M(-a_{t+c-1}) ,I_{A/B})=0$ and deformation theory. Note that
$\sigma _T$ is surjective by Nakayama's lemma. Dualizing (horizontal
compositions) once more, we get
\begin{equation}\label{diag44}
\begin{array}{ccccccccc}
B_T(-a_{t+c-1}) & \stackrel{ \sigma _T^*}{\longrightarrow} & \Hom_{B_T}(M_T,B_T) & \stackrel{ \alpha }{\leftarrow }& N_T\otimes B_T \\
\downarrow &  & \downarrow  &  &\downarrow \\
B(-a_{t+c-1}) &  \stackrel{ \sigma ^*}{\longrightarrow} &  \Hom_{B}(M,B) &\cong  & N\otimes B \\
& &  \downarrow  &  &\downarrow \\
&  & 0  & &  0 \\
\end{array},
\end{equation}
where $\alpha $ must be an isomorphism since $\Hom _{B_T}(M_T,B_T)$ is a deformation of $\Hom _{B}(M,B)$, due to $\Ext^1_B(M,B)=0$. Indeed taking the cokernel $C$ of $\alpha $, we get $C\otimes _Tk=0$, i.e. $C_{\nu }\otimes _T k=0$ where $C=\oplus C_{\nu }$ and hence $C_{\nu}=0$ by Nakayama's lemma for all $\nu $. In the same way the kernel of $\alpha $ vanishes. Since $\sigma ^*$ defines the last column of $\cA $ (restricted to $B$) we can use $\sigma_T^*$ to define $s_T$ in the diagram
\begin{equation}\label{diag4*}
\xymatrix{& & &   B_T(-a_{t+c-1})\ar[d]^{\sigma_T^*} \ar[dl]_{\exists s_T} \\
&G_{t+c-2}^*\otimes B_T   \ar[r] &  F^*\otimes B_T  \ar[r] & N_T\otimes B_T \ar[r] &0
}
\end{equation}
and to lift the column $s_T(1)$ to a column
$v_T \in  (F^*\otimes R_T)_{(a_{t+c-1})}$ such that $v_T$, by putting
$\varepsilon =0$, becomes exactly equal to the last column
$v \in \oplus _{i=1}^t R(a_{t+c-1}-b_{i})_0$, of $\cA$, i.e. $v=v_T\otimes _Tk.$
Set
\begin{equation} \label{diag45} \cA_T=[{\mathcal B}_T , v_T] \ .
\end{equation}
It remains to see that the determinantal ring given by the $t+1-r$ minors of
$\cA _T$ is $A_T$ or, equivalently, that $I_{A_T/B_T}$ is given by the $t+1-r$
minors involving the last column since $\cB _T$ is already given, i.e. it
suffices to show
$$ I_{t+1-r}(\cA _T)/I_{t+1-r} ({\mathcal B}_T)=I_{A_T/B_T}.$$
But Theorem \ref{mainthm} and its proof imply that
$ I_{t+1-r}(\cA _T)/I_{t+1-r} ({\mathcal B}_T) =\im (\sigma _T)$. Indeed in
the proof of Theorem \ref{mainthm} we nowhere used that $R$ was a
polynomial ring over $k$ as long as the invertible matrices used in the proof
are still invertible with entries in $R \otimes_k T$, e.g. that they map to
invertible matrices via $(-) \otimes_T k$. Since the diagram \eqref{auxdiag}
shows that $\im (\sigma _T)=I_{A_T/B_T}$ we are done.
\end{proof}

\begin{remark} \label{rem0thm42} \rm The proof shows also that every
  deformation $B_T \longrightarrow A_T$ of $B \longrightarrow A$ to the dual
  numbers $T$ is given by $(t+1-r)$-minors of the matrix
  $\cA_T =[{\mathcal B}_T , v_T]$ of \eqref{diag45} where ${\mathcal B}_T$
  defines $B_T$, without supposing (i) of Theorem \ref{teo3}. And, moreover,
  fixing $B_T$ and some matrix ${\mathcal B}_T$ defining $B_T$ and $N_T$,
  ``the family of choices'' of the last column $v_T$ which via \eqref{diag45}
  leads to deformations of $A$, corresponds precisely to the variation of
  $\sigma_T^*(1)_{(a_{t+c-1})}$ in the diagram \eqref{diag4*} whose dimension
  is $\dim (MI \otimes A)_{(a_{t+c-1})}$. Note that by
  Proposition~\ref{vanish}, $MI\otimes A\cong \Hom_B(I_{A/B}(a_{t+c-1}),A)$.
  Thus, every deformation of the fiber of the natural projection
  $$p_2: \Hilb ^{p_X(t),p_Y(t)}(\PP^{n})\longrightarrow \Hilb
  ^{p_Y(t)}(\PP^{n})$$  comes from deforming (the last column of) the matrix
  $ \cA.$
\end{remark}

\begin{remark} \label{remthm42} \rm It is possible to weaken the
  assumption (i) of Theorem \ref{teo3} and still get the same conclusion.
  Indeed, the assumption (i) is equivalent to the surjectivity of the morphism
  $$\,_0\!\Hom(I_B,B)  \longrightarrow \ _0\!\Hom(I_B,A)$$ in the diagram
  \eqref{auxdiag1}. We can enlarge this diagram and get the cartesian square
  (i.e. pullback diagram):
\begin{equation}\label{cartesiansquare}
\begin{array}{cccccc}
A^1_{(B\rightarrow A)} &  \stackrel{pr_2}{\longrightarrow}& _0\!\Hom _R(I_B ,B) \\
 \downarrow pr_1 & & \downarrow p \\
_0\!\Hom _R(I_A ,A) & \longrightarrow & _0\!\Hom _R(I_B ,A)
\end{array}
\end{equation}
where $A^1_{(B\rightarrow A)}$ is the tangent space of $\GradAlg(H_B,H_A)$ at $(B\rightarrow A)$. If we instead of the assumption (i) in Theorem \ref{teo3} assume the weaker assumption;
\begin{equation*}
  pr_1: \ A^1_{(B \to A)}\twoheadrightarrow\ _0\!\Hom_R(I_A ,A)
\end{equation*}
is surjective, we still get the existence of $B_T\longrightarrow A_T$ reducing
to
$B\longrightarrow A$ via $ (-)\otimes _T k$, and the rest of the proof holds.
Note
also that the surjectivity of
$pr_1: A^1_{(B \to A)}\longrightarrow\ _0\!\Hom_R(I_A ,A)$ is equivalent to
$\gamma =0$ where $\gamma$ is the composition
$$_0\!\Hom_R(I_A ,A)\longrightarrow\ _0\!\Hom_R(I_B ,A)\longrightarrow\
_0\!\Ext^1_B(I_B/I_B^2,I_{A/B})$$
of natural maps. This is easily seen since the map $p:\ _0\!\Hom_R(I_B ,B)
\longrightarrow \ _0\!\Hom _R(I_B ,A)$ in the diagram
\eqref{cartesiansquare} is part of a long exact sequence
 \begin{equation}\label{longp}
   \longrightarrow\ _0\!\Hom_R(I_B ,B) \stackrel{p}{\longrightarrow}\ _0\!\Hom
   _R(I_B ,A) \longrightarrow \
   _0\!\Ext^1_B(I_B/I_B^2,I_{A/B}) \longrightarrow \ _0\!\Ext^1_B(I_B/I_B^2,B)
   \longrightarrow .
\end{equation}

\noindent Note that $\gamma =0$ whenever $_0\!\Ext^1_B(I_B/I_B^2,I_{A/B})=0$.
See Remark \ref{rem53} for the vanishing of this Ext group.
\end{remark}

The notion ``every deformation of $B$ (or $\Proj(B))$ comes from deforming its
associated matrix'', has another important consequence:

\begin{proposition}\label{thmext3} Let $B=R/I_B\twoheadrightarrow A$ be
  determinantal algebras defined by $t+1-r$ minors of matrices ${\mathcal A}$
  (resp. ${\mathcal B}$) representing $\varphi ^*$ (resp. $\varphi^*_{t+c-2}$)
  and let $X=\Proj(A)$ and $Y=\Proj(B)$, and suppose $r\ge 2$, $c\ge 3-r$,
    $J_A:=I_{t-r}(\cA)\ne R$ and $\depth _JB\ge r+2$ where
  $J=I_{t-r}({\mathcal B})$. If $c \le 0$ we also suppose
  $\depth _{J_A}A\ge 3$. Moreover, suppose that every deformation of $Y$ comes
  from deforming ${\mathcal B}$. Then the natural projection
  $$p_2: \Hilb ^{p_X(t),p_Y(t)}(\PP^{n})\longrightarrow \Hilb ^{p_Y(t)}(\PP^{n})$$ given by
  $p_2((X' \subset Y'))=(Y')$, is smooth at $(X \subset Y)$.
\end{proposition}
\begin{proof}
  To prove the smoothness of  $p_2: \Hilb ^{p_X(t),p_Y(t)}(\PP^{n})\longrightarrow \Hilb ^{p_Y(t)}(\PP^{n})$, or equivalently the smoothness at
  $(B \to A)$ of the projection
  $$\GradAlg(H_B,H_A) \longrightarrow \GradAlg^{H_B}(R) \ $$
  given by $(B' \to A') \mapsto (B')$ we can rather closely follow some of the
  arguments in the proof of Theorem~\ref{teo3}. But instead of only
  considering deformations to the dual numbers, we take a surjection of
  artinian local rings
  $(T,\mathfrak{m}_T) \longrightarrow (S, \mathfrak{m}_S)$ whose kernel
  ${\mathfrak a}$ satisfies ${\mathfrak a} \cdot \mathfrak{m}_T=0$. Let
  $B_S \longrightarrow A_S$ be a deformation of $B \longrightarrow A$ to $S$
  and let $B_T$ be a deformation of $B_S$ to $T$. By definition of smoothness (see \cite[Definition 2.2 and Remark 2.3]{Sc})
  it suffices to find a deformation $A_T$ of $A_S$ to $T$ and a map
  $B_T \longrightarrow A_T$ reducing to $B_S \longrightarrow A_S$ via
  $(-)\otimes_T S$. Let $I_{A_S/B_S} = \ker(B_S \longrightarrow A_S)$. Since
  $B_S$ is defined by a matrix ${\mathcal B}_S$ that lifts the matrix
  ${\mathcal B}$ by assumption, we can use the arguments leading to
  \eqref{auxdiag} and \eqref{diag44}, which relies on the vanishing of three
  $\Ext_B^1$-groups, to get $M_S$ and $\sigma_S$ so that \eqref{auxdiag} and
  \eqref{diag44} hold, only replacing there $T$ by $S$. Using that $B_T$ is
  defined by some matrix ${\mathcal B}_T$ that lifts ${\mathcal B}_S$ to $T$
  and letting $N_T$ be the cokernel of the map induced by ${\mathcal B}_T$ and
  $M_T:= \Hom_{B_T}(N_T,B_T)$, the vanishing of $\Ext_B^1$-groups yields a
  diagram as in \eqref{diag44}, only indexing the lower line there by $S$
  (i.e. $B_S$, $\sigma^*_S$ instead of $B$, $\sigma^*$ etc.) and replacing
  $\sigma^*_T$ by some deformation of $\sigma^*_S$. Then we argue as in
  Theorem~\ref{teo3} to get the column $v_T$, and we define the $T$-flat
  quotient $A_T$ by the minors of the matrix in \eqref{diag45} of size
  $t+1-r$. Thus we have found a deformation $A_T$ of $A_S$ and a map $B_T \longrightarrow A_T$ reducing to
  $B_S \longrightarrow A_S$ and we are done.
\end{proof}

\begin{remark} \rm If we suppose $\depth _JA\ge 3$ (i.e. $c \ge 4-r$ for $\cA$
  general), and $\depth _{J_A}A\ge 4$ if $c \le 0$, then the smoothness of
  $$p_2: \Hilb ^{p_X(t),p_Y(t)}(\PP^{n})\longrightarrow \Hilb
  ^{p_Y(t)}(\PP^{n})$$ is much easier to prove because Proposition
  \ref{vanish} implies that
  $\Ext_A^1(I_{A/B}/I_{A/B}^2,A)=0$, whence ${\rm H}^2(B,A,A)=0$, see
  Proposition \ref{everydef}. It is important to point out  that  in Proposition
  \ref{thmext3}  we ``only'' assume $c \ge 3-r$ because one of main theorems
  of this
  monograph recursively proves that the closure
  $ \overline{ W(\underline{b};\underline{a};r)}$ is a generically smooth
  component of $\Hilb ^{p_X(t)}(\PP^n)$ under some assumptions by starting the
  induction from a case where $c = 3-r$, i.e. the case where $B$ is defined by
  maximal minors.
 \end{remark}

Combining Lemma~\ref{lemma2}, Theorem \ref{teo3} and
Remark~\ref{remthm42} 
we get

\begin{corollary} \label{evdef} Let $B=R/I_B\twoheadrightarrow A$ be
  determinantal algebras defined by $t+1-r$ minors of matrices ${\mathcal A}$
  (resp. ${\mathcal B}$) representing $\varphi ^*$ (resp.
  $\varphi^*_{t+c-2}$). Suppose $\depth _JB\ge r+2$ where
  $J=I_{t-r}({\mathcal B})$, $r\ge 2$, $c\ge 3-r$ and
  $J_A:=I_{t-r}(\cA)\ne R$. Moreover suppose that every deformation of $B$
  comes from deforming ${\mathcal B}$. If $c \le 0$, suppose also
  $\depth _{J_A}A\ge 3$. Then the following statements are equivalent
\begin{itemize}
\item[(a)] Every deformation of $A$ comes from deforming the matrix $\cA$.
\item[(b)] The map $pr_1:A^1_{(B \to A)}\longrightarrow\ _0\!\Hom_R(I_A ,A)$ of
  \eqref{cartesiansquare} is surjective.
\item[(c)] The composition $\gamma$ of natural maps in
  Remark~\ref{remthm42} is zero.
\item[(d)] The map
  $\psi_A:\ _0\!\Ext^1 _R(MI,MI)\longrightarrow\ _0\!\Hom _R(I_A ,A)$ of
  Lemma~\ref{lemma2} is surjective.
\end{itemize}
If furthermore $\depth _JA\ge 3$, and in case $c \le 0$; $\depth _{J_A}A\ge 4$,
they are also equivalent to
   $$ {\rm (e)} \quad  _0\!\Ext^1_B(I_B/I_B^2,I_{A/B})\longrightarrow\
   _0\!\Ext^1_B(I_B/I_B^2,B) {  \ is \ injective}\,. \qquad \qquad
   \qquad \qquad \qquad \qquad $$
\end{corollary}
\begin{proof} Since $(a) \Leftrightarrow (d)$ by Lemma~\ref{lemma2} and
  $(c)\Leftrightarrow (b) \Rightarrow (a)$ by Theorem \ref{teo3} and
  Remark~\ref{remthm42}, we must show $(a)\Rightarrow (b)$. To prove it take
  any $(A_T)\in\ _0\!\Hom_R(I_A ,A)$. By assumption (a), $A_T$ is of the form
  $A_T=R_T/I_{t+r-1}(\cA _T)$ where $T:=k[\varepsilon]/(\varepsilon^2)$. Now
  if we delete the last column of $\cA _T$, we get a matrix ${\mathcal B} _T$.
  By Lemma~\ref{lemma2}, $B_T:=R_T/I_{t+r-1}({\mathcal B} _T)$ is a
  deformation of $B$ to $T$. Hence we get an element
  $(B_T\longrightarrow A_T)\in A^1_{(B \to A)}$ that maps to $(A_T)$ via the
  map $$pr_1:A^1_{(B \to A)}\longrightarrow\ _0\!\Hom_R(I_A ,A)$$ of diagram
  (\ref{cartesiansquare}), i.e. $(b)$ is proved.

  Finally to see $(e) \Leftrightarrow (b)$ we notice that the lower horizontal
  map in the diagram \eqref{cartesiansquare} is surjective by
  Proposition~\ref{everydef}; indeed ${\rm H}^2(B,A,A)=0$. It follows that the
  maps $$p: \ _0\!\Hom_R(I_B ,B) \longrightarrow \ _0\!\Hom _R(I_B ,A)$$ and
  $pr_1:A^1_{(B \to A)}\longrightarrow \ _0\!\Hom_R(I_A ,A)$ of
  (\ref{cartesiansquare}) are surjective simultaneously and we conclude the
  argument by using the exact sequence \eqref{longp}.
\end{proof}

\section{The dimension of the determinantal locus}

Given integers $\underline{a}=(a_1,a_2,...,a_{t+c-1})$ and
$\underline{b}=(b_1,...,b_t)$ recall that 
$W(\underline{b};\underline{a};r)\subset \Hilb ^{p_X(t)}(\PP^{n})$ is the
locus parameterizing determinantal schemes $X\subset \PP^{n}$ of codimension
$r\cdot (c+r-1)$ defined as the vanishing of the $(t-r+1)\times (t-r+1)$
minors of a $t\times (t+c-1)$ matrix
$\cA=(f_{ij})^{i=1,...,t}_{j=1,...,t+c-1}$ where
$f_{ij}\in k[x_0,x_1,...,x_{n}]$ is a homogeneous polynomial of degree
$a_j-b_i$. Correspondingly, set $\underline{a'}=(a_1,a_2,...,a_{t+c-2})$ and
let $W_{(\underline{b};\underline{a};r)}^{(\underline{b};\underline{a'};r)}$
be the locus of determinantal flags $X\subset Y\subset \PP^n$ of the
Hilbert-flag scheme with $X\in W(\underline{b};\underline{a};r)$ and
$Y\in W(\underline{b};\underline{a'};r)$ defined by the
$(t-r+1)\times (t-r+1)$ minors of the $t\times (t+c-2)$ matrix ${\mathcal B}$
that we obtain deleting the last column of $\cA$.

As previously
$\varphi : F=\oplus _{i=1}^tR(b_i)\longrightarrow G=\oplus _{j=1}^{t+c-1}
R(a_j)$
is a graded morphism, $\varphi _{t+c-2}$ is obtained deleting the last row,
$B:=R/I_{t-r+1}(\varphi ^*_{t+c-2})\twoheadrightarrow A:=R/I_{t-r+1}(\varphi
^*)$,
$MI=\coker \varphi ^*$ and $N=\coker \varphi ^*_{t+c-2}$. We also assume that the integers
$\{a_j\}$ and $\{b_i\}$ satisfy the weak
conditions \eqref{paperassump1}.

\vskip 2mm
In this section, we address the first of the following three fundamental problems:

\begin{pblm}\label{pblmsagain} \begin{itemize}\item[{\rm (1)}]  To determine  the dimension of
$W(\underline{b};\underline{a};r)$ in terms of $a_j$ and $b_i$.
\item[{\rm (2)}] To determine the unobstructedness of a generic point of $W(\underline{b};\underline{a};r)$.

\item[{\rm (3)}] To determine whether  $W(\underline{b};\underline{a};r)$ fills in an open dense subset of the corresponding component of the Hilbert scheme.
\end{itemize}\end{pblm}


Our proof will go by induction and use the following diagram:

\begin{equation}\label{flag}
\begin{array}{cccccc}
\HH om(F,G) & & & & & \\
& \searrow q & & & & \\
& & W_{(\underline{b};\underline{a};r)}^{(\underline{b};\underline{a'};r)} &  \stackrel{p_{2W}}{\longrightarrow}& W(\underline{b};\underline{a'};r) \\
& & \downarrow p_{1W} & &  \\
& & W(\underline{b};\underline{a};r)
\end{array}
\end{equation}
where the maps $p_{iW}$ is the restriction to
$W_{(\underline{b};\underline{a};r)}^{(\underline{b};\underline{a'};r)}$ of
the projection $p_{i}$ of the Hilbert-flag scheme into its Hilbert schemes.
Since $q$ is an algebraic rational morphism (rational as $q$ is only defined
for $\varphi $ with $\coker \varphi ^*$ and $\coker \varphi ^*_{t+c-2}$ of
maximal codimensions), there is an open subset of $\HH om(F,G)$ that via $q$
(resp. $p_{1W}\cdot q$) maps surjectively onto an open dense subset of
$ W_{(\underline{b};\underline{a};r)}^{(\underline{b};\underline{a'};r)}$
(resp. $ W(\underline{b};\underline{a};r)$). Thus we have the first of the
following lemmas:

\begin{lemma}\label{lemaaux2}
  $ W(\underline{b};\underline{a};r)$ and
  $W_{(\underline{b};\underline{a};r)}^{(\underline{b};\underline{a'};r)}$ are
  irreducible algebraic sets.
\end{lemma}

\begin{lemma} With the above notation, if the last column of $\cA $ does not
  contain any unit (e.g. if $a_1>b_t$), then $p_{1W}$ is surjective.
\end{lemma}

\begin{proof} The result immediately follows from the well-known fact that if
  the $i\times i$ minors of an homogeneous matrix $\cA$ defines a
  determinantal scheme $X$ (i.e. $X$ has the expected codimension) the same is
  true for the scheme defined by the $i\times i$ minors of the matrix that we
  obtain deleting a column of $\cA$ without units, see \cite{Br}.
\end{proof}

Our first goal is to show that we have an upper bound for the dimension of
$ W(\underline{b};\underline{a};r)$. Indeed if $r=1$ we set
$W(\underline{b};\underline{a}):=W(\underline{b};\underline{a};1)$
and we have
$$\dim W(\underline{b};\underline{a}) = \lambda _c+K_3+K_4+\cdots +K_c$$ for
$n-c \ge 1$ and $c \ge 2$, with
$\dim W(\underline{b};\underline{a}) = \lambda _c$ for $c=2$, provided
$a_{i-1} \ge b_i$ for $i \ge 2$ by 
Theorem~\ref{std_det_case}. Here $ \lambda _c$ and the
non-negative numbers $K_i$ are defined in \eqref{lamda}, and as long as
$c \ge 1$, i.e. $\cA$ is a homogeneous $t \times (t+c-1)$ matrix with
$ t \le t+c-1$ the expression $\lambda _c+K_3+K_4+\cdots +K_c$ turns out to be
an upper bound for $\dim W(\underline{b};\underline{a};r)$ for every $r$,
$1 \le r \le t-1$. If $c \le 1$, we can find an upper bound for
$\dim W(\underline{b};\underline{a};r)$ from the $c \ge 1$ statement by
transposing the matrix $\cA$ because $(\cA)^{tr}$ is a $(t+c-1) \times t$
matrix with $(t+c-1) \le t$. Thus, other non-negative numbers, $K'_i$, given
in our next lemma, come into play and are needed to bound
$\dim W(\underline{b};\underline{a};r)$. These $K'_i$ are really the numbers
$K_i$ defined by \eqref{lamda} using $(\cA)^{tr}$, up to equivalent matrices,
instead of $\cA$. With more details we have

\begin{lemma} \label{tr} Let $\cA$ be a homogeneous $t \times (t+c-1)$ matrix
  with degree matrix $\cD=(d_{ij})^{i=1,...,t}_{j=1,...,t+c-1}$ where
  $d_{ij}=a_j-b_i$, and let $\lambda(D):=\lambda_c$ be the associated number
  given by \eqref{lamda}. Then the transposed matrix $(\cA)^{tr}$ with
  $(t+c-1) \times t$ degree matrix $\cD^{tr}$ and associated number
  $\lambda(\cD^{tr}):=\lambda'_{2-c}$ given by \eqref{lamda} satisfies
  $\lambda(\cD^{tr})=\lambda(\cD)$, i.e. $\lambda'_{2-c}=\lambda_{c}$.
  Moreover, for the loci of $(t-r+1)$-minors we have
$$W(-a_{t+c-1},-a_{t+c-2}, \cdots,-a_1;-b_t, \cdots,-b_1;c+r-1)=W(b_1,
\cdots, b_t;a_1,a_2, \cdots,a_{t+c-1};r).$$
Finally, if $c=2-r < 1$, these loci are defined by maximal minors and if
$a_{i-r} \ge b_i$ for $r+1 \le i \le t$ we have
 $$\dim W(\underline{b};\underline{a};r)=\lambda_{c} +K'_3+K'_4+\cdots
+K'_r \ \ {\text for \ }  n-r \ge 1 \quad {\text  where}$$ \\[-11mm]
\begin{equation}\label{lamdaprime}
\begin{array}{rcl}
\ell'_i & :=  & \sum_{j=1}^{t-r+1}a_j-\sum_{k=r-i+1}^tb_k, \\
h'_{i-3} & := &
  -2b_{r-i+1}-\ell'_i +n, \text{ for } 1\le i \le c,\\
    K'_3 & := & \binom{h'_0}{n}, \\
    K'_4 &:= & \sum_{j=r-2}^{t} \binom{h'_1-b_j}{n}-
\sum_{i=1}^{t-r+1} \binom{h'_1-a_i}{n}, \text{ and, in general,} \\
 K'_{i} &: = &\sum _{x+y=i-3
  \atop x , y \ge 0} \sum _{r-i+2\le i_1< ...< i_{x}\le t \atop 1\le
  j_1\le...\le j_y \le t-r+1} (-1)^{i-1-x} \binom{h'_{i-3}-b_{i_1}-\cdots
  -b_{i_x}-a_{j_1}-\cdots -a_{j_y} }{n}.  \end{array}
\end{equation}
\end{lemma}
\begin{proof} In the definition \eqref{lamda} of $\lambda(\cD):=\lambda_c$ we
  may put $a_j-a_i = d_{1j} -d_{1i}$ for $i,j \in \{1,...,t+c-1\}$ and
  $b_i-b_j = d_{i1} -d_{j1}$ for $ i,j \in \{i=1,...,t\}$. Applying
  \eqref{lamda} similarly onto $\lambda(\cD^{tr})$ with degree matrix given by
  $d^{tr}_{i,j}$ where $d^{tr}_{i,j}=d_{j,i}$, we get
  $d^{tr}_{1j} -d^{tr}_{1i}=b_j-b_i$ and $d^{tr}_{i1} -d^{tr}_{j1}=a_i-a_j$
  and so we see that the sum of the negative terms in \eqref{lamda} defining
  $\lambda$ is the same for $\lambda(\cD^{tr})$ and $\lambda(\cD)$. The
  positive terms in the definition of $\lambda$ in \eqref{lamda} are obviously
  equal, and we get $\lambda(\cD^{tr})=\lambda(\cD)$.

  Moreover, since it is clear that the locus of the $(t-r+1) \times (t-r+1)$
  minors of the $t \times (t+c-1)$ matrix $\cA$ is exactly the same as the
  locus of $(t-r+1) \times (t-r+1)$ minors of the $(t+c-1) \times t$
  transposed matrix $\cA^{tr}$ and $t+c-1-(c+r-1)+1=t-r+1$, we get the
  equality of the loci, noticing that the inequalities
  $$b_1 \le b_2\le \cdots \le b_t, a_1 \le a_2 \le \cdots \le a_{t+c-1}$$ of the
  numbers attached to $\cA$, are reversed for $\cA^{tr}$, i.e. they correspond
  to
  $$-b_1 \ge -b_2\ge \cdots \ge -b_t, -a_1 \ge -a_2 \ge \cdots \ge -a_{t+c-1}.$$
  Indeed, transforming $\cA$ by keeping the entry of degree $(a_1-b_t)$ in the
  lower left corner fixed (instead along the usual corner $(1,1)$) we
  get a $(t+c-1) \times t$ matrix $\cA'$ with degree matrix
  $(d'_{ji})^{j=1,...,t+c-1}_{i=1,...,t}$,
\begin{equation} \label{trrevdeg}
  d'_{ji}=a'_i-b'_j \ {\rm where}\ a'_i=-b_{t+1-i} \ {\rm for}\ {i=1,...,t} \
  {\rm and}\ b'_j=-a_{t+c-j}\ {\rm for}\ {j=1,...,t+c-1}
 \end{equation}
  which is equivalent to $\cA^{tr}$ (up to row- and column-equivalence). This
  $\cA'$ has the same property as $\cA$ with respect to inequalities between
  the entries, i.e. the one of the smallest degree sits in the lower left
  corner, and the degree of entries increases in all directions from there.
  Note also that we easily get $\lambda(\cD^{tr})=\lambda(\cD')$ by
  \eqref{trrevdeg} and the first part of the proof.

  In particular, for $c=2-r$ then $\cA^{tr}$, or rather its equivalent matrix
  $\cA'$, is a $(t-r+1) \times t$ matrix and since the inequalities
  $b_1 \le b_2\le \cdots \le b_t, a_1 \le a_2 \le \cdots \le a_{t-r+1}$
  attached to $\cA$ corresponds to
  $-a_{t-r+1} \le -a_{t-r} \le \cdots \le -a_1$,
  $ -b_t\le -b_{t-1} \le \cdots \le -b_1$ attached to $\cA'$, we set
  $b'_i = -a_{t-r+2-i}$ and $a'_i=-b_{t+1-i}$ using primes (') for $\cA'$.
  Then we get the lemma from \eqref{lamda} by substitution and the fact that
  $$\dim W(\underline{b'};\underline{a'}) = \lambda'_r+K'_3+K'_4+\cdots +K'_r$$ for
  $n-r \ge 1$ provided $a'_{i-1} \ge b'_i$ for $2 \le i \le t+1-r$ by
  Theorem~\ref{std_det_case}. 
\end{proof}

\begin{theorem}\label{ineqdimW} \begin{itemize}
\item[{\rm (i)}] Let  $c \ge 1$ and suppose $a_{i-1} \ge
  b_i$
  (and $a_{i-1} > b_i$ if $c=1$) for $2 \le i \le t$. Then for any $r$,
  $1 \le r < t$ we have
   $$\dim W(\underline{b};\underline{a};r) \le \lambda_c +K_3+K_4+\cdots +K_c
   \quad {\text for \ }\ n-r(c+r-1) \ge 1\, .$$
   In particular 
   if   
   $a_{t+c-1}< \ell_2=\sum_{i=1}^{t+1}a_i - \sum_{i=1}^tb_i$, then $K_i=0$ for
   $3 \le i \le c$, whence \
   $$\dim W(\underline{b};\underline{a};r) \le \lambda_c\, . $$
   \item[{\rm (ii)}] Let $c < 1$ and suppose $a_{i-r} \ge b_i$ for $r+1 \le i \le t$. Then
   for any $r$, $2-c \le r < t$, we have
   $$\dim W(\underline{b};\underline{a};r) \le \lambda_{c} +K'_3+K'_4+\cdots
   +K'_r \ \ {\text for \ }\ n-r(c+r-1) \ge 1 \, .$$
   In particular if $-b_1 < \ell'_2 = \sum_{j=1}^{t-r+1}a_j-\sum_{k=r-1}^tb_k$
   (e.g. if $r=2$), then $K'_i=0$ for $3 \le i \le r$ and
   $$\dim W(\underline{b};\underline{a};r) \le \lambda_{c}\, .$$
 \item[{\rm (iii)}] If $r=1$, $(r,c,n) \ne (1,1,2)$ in {\rm (i)}, resp.
   $r=2-c$ in {\rm (ii)}, then we have equalities in both displayed formulas
   of {\rm (i)}, resp. {\rm (ii)}.
   \end{itemize}
 \end{theorem}

 \begin{remark} \rm Since the inequality $a_{t+c-1}< \ell_2$ in (i) is always
   satisfied for $c \in \{1,2\}$ the second displayed formula holds for
   $c \in \{1,2\}$ in (i), and similarly for $c=0$ in (ii). The assumptions on
   ${a_j}$ and ${b_i}$ are at least needed in (iii). Moreover, since in this
   paper we almost always work with determinantal schemes
   $X\in W(\underline{b};\underline{a};r)$ of dimension at least one and
   $r(c+r-1)$ is the codimension of $X$ in $\PP^n$, the assumption
   $n \ge r(c+r-1)+1$ in Theorem~\ref{ineqdimW} is mostly a reminder of
   the settings in which we work.
\end{remark}

\begin{proof} (i) To prove the first inequality we observe that the morphism
  $p_{1W} \cdot q$ which maps, say a sufficiently general
  $\varphi \in \HH om(F,G)$, onto
  $(\Proj(A)) \in \ W(\underline{b};\underline{a};r) $, $A=R/I_{t+1-r}(\cA)$,
  factors via $MI:=\coker (\varphi ^*)$ by the Fitting's lemma
  (\cite[Corollary-Definition 20.4]{eise}). Thus
  $p_{1W} \cdot q$ factors through the moduli space of such $MI$, a moduli
  space which at least exits at a very local level, say at this $MI$. Since
  $MI$ is unobstructed by \cite[Theorem 3.1]{K2014}, the dimension of the
  moduli space at $(MI)$ is equal to the dimension of its tangent space \
  $ _0\!\Ext^1_R(MI,MI)$. This argument implies
  $\dim W(\underline{b};\underline{a};r) \le \ _0\!\ext^1_R(MI,MI)$, and since
  we have $$ \ _0\!\ext^1_R(MI,MI)= \lambda_c+K_3+K_4+\cdots +K_c$$ by
  \cite[Theorem 3.1]{K2014}, we get the first inequality. Note that the part
  of the proof of \cite[Theorem 3.1]{K2014} that concerns the unobstructedness
  of $MI$ and the dimension of $ \ _0\!\Ext^1_R(MI,MI)$ hold for $c=1$ as well
  (with $E=0$ there, whence all $K_i=0$), cf. \eqref{thm31c1}.

  In \cite[Remark 3.4]{KM2005} we noticed that all the binomials in $K_i$
  vanish under the numerical assumption
  $\ell_c >2a_{t+c-1}+a_{t+c-2}+\cdots +a_{t+2}$ (when $c > 3$, cf. the
  published version). This assumption is
  equivalently to $a_{t+c-1}< \ell_2$. Indeed one may show that $K_i=0$ using the
  minimal free resolution of $\varphi _{t+c-2}^*$ given by the Buchsbaum-Rim
  complex (see (\ref{BR}))
 \begin{equation}\label{xxx}
 \wedge ^{t+1}G^*_{t+c-2}\otimes \wedge ^tF\longrightarrow
 G^*_{t+c-2}\longrightarrow F^*\longrightarrow\coker (\varphi
 _{t+c-2}^*)\longrightarrow 0\ ,
 \end{equation}
 and recalling that $K_c=\ _0\!\hom (B_{c-1},R(a_{t+c-1}))$ where
 $B_{c-1}=\coker (\varphi _{t+c-2})$, cf. \cite[(3.13)]{KM2005}. Then we get
 that $a_{t+c-1}< \ell_2$ implies $K_c=0$, also when $c=3$ (and all $K_i=0$) because
 $$_0\!\Hom ( \wedge ^{t+1}G^*_{t+c-2}\otimes \wedge ^tF, R(a_{t+c-1}))=0$$ by
 only considering the degree of the generators of this $\Hom$-group. This
 completes the proof when $c \ge 1$.

 (ii) If $c < 1$ and we look at the locus of $(t-r+1) \times (t-r+1)$ minors
 of the $(t+c-1) \times t$ transposed matrix $\cA^{tr}$ (or rather its
 equivalent matrix $\cA'$ to which \eqref{trrevdeg} belongs), we get a
 $t' \times (t'+(2-c)-1)$ matrix with $2-c > 1$ and we have
 $\dim W(\underline{b'};\underline{a'};r) \le \dim
 W(\underline{b'};\underline{a'};2-c)$
 by \eqref{trrevdeg} and the proven part of the theorem. Then we get the
 theorem from Lemma~\ref{tr}, noticing that $-b_1 < \ell'_2$ implies every
 $K'_i=0$ by the proven part and that $r=2$ implies $-b_1 < \ell'_2$.

 (iii) It follows from  Lemma~\ref{tr} and Theorem~\ref{std_det_case}.
\end{proof}

\begin{remark} \label{differentproofs} \rm (Alternative proof of
  Theorem~\ref{ineqdimW}(i))

The map $q$ in diagram \eqref{flag} factors through
  the orbit set $\HH om(F,G)//G$ since $G:=Aut(F) \times Aut(G)$ acts on
  $\HH om(F,G)$ in a natural way such that every element of an orbit maps to
  the same determinantal scheme, cf. \cite{KMMNP}, p. 97. Indeed the proof
  there works also for $ W(\underline{b},\underline{a};r)$ only replacing the
  Eagon-Northcott complex by the Lascoux resolution in the argument. Thus we
  obtain a dominating rational morphism
  $\HH om(F,G)//G \to W(\underline{b},\underline{a};r)$ and
  $\dim \HH om(F,G)//G $ is an upper bound for
  $\dim W(\underline{b},\underline{a};r)$. In \cite{KMMNP}, pp.\! 97-98 we
  compute $\dim G$ (see Propositions 10.2 and  10.3) and get
 $$ W(\underline{b},\underline{a};r) \le \lambda_c -1  +\  _0\!\hom(B_c,B_c)$$
 where $B_{c}=\coker (\varphi)$. Since we get
 $ _0\!\hom(B_c,B_c)= 1+K_3+K_4+\cdots +K_c$ by \cite[Proposition
 3.12]{KM2005} and Remark \ref{aug2023}(3), we have proved the first dimension
 formula.
\end{remark}

Note that we have equalities in Theorem~\ref{ineqdimW} if $r=1$ and $c >1$, e.g.
$\dim W(\underline{b};\underline{a})=\lambda_c $ provided $a_{t+c-1}< \ell_2$.
One of the main goals in this work is to generalize this formula and to show
$\dim W(\underline{b};\underline{a};r)=\lambda_c$ for any $r$ under some 
numerical assumptions. We will also analyze when the inequality
$\dim W(\underline{b};\underline{a};r)\le \lambda_c+K_3+\cdots +K_c$ (resp.
$\dim W(\underline{b};\underline{a};r)\le \lambda_c+K'_3+\cdots +K'_r$) turns
out to be an equality up to a correction term. To this end, we define
$$s_r:=\sum _{i=1}^{t-r+1}a_i-\sum
_{i=1}^{t-r}b_{r+i}\ ,$$
and note that $s_0=\ell_2$. Thus
$\dim W(\underline{b};\underline{a};r)=\lambda_c$ for $r=1$ and $c >1$ when
$a_{t+c-1}<s_0$ because
\begin{equation} \label{Kis0} K_i= 0 \quad {\rm for \ all} \quad 3 \le i \le c
  \quad {\rm provided} \quad a_{t+c-1}< s_0:=\ell_2
\end{equation}
by the proof of Theorem~\ref{ineqdimW}. In what follows we will try to prove that
$\dim W(\underline{b};\underline{a};r)=\lambda_c$ provided
$a_{t+c-1}<s_r-b_r+b_1$ which for $r=1$ is just $a_{t+c-1}<s_1$. This is a
slightly stronger assumption that implies that all $K_i=0$ because, in general, we
have $b_j \le a_j$ for $1 \le j \le t$, $b_{j_0} < a_{j_0}$ for some
$j_0$ by \eqref{paperassump1} which implies $ s_1 < s_0 $, as well as
\begin{equation} \label{Ki0}
s_r-b_r+b_1 < s_{r-1}-b_{r-1}+b_1 \le s_1 = s_0-a_{t+1}+b_1 \ \ {\rm for }\ 2 \le r \le t\ .
\end{equation}
Finally note that if we have $a_{t+c-1} \ge b_{r-1}$ (a very weak assumption),
we get that $a_{t+c-1} < s_r-b_r+b_1$ implies that all $K'_i=0$ by
Theorem~\ref{ineqdimW} because
$-b_1 < \ell'_2 = \sum_{j=1}^{t-r+1}a_j-\sum_{k=r-1}^tb_k=s_r-b_r-b_{r-1}$.

Now we consider the diagram of infinitesimal deformations that corresponds to
diagram (\ref{flag}) at a given point
$(X\subset Y)\in
W_{(\underline{b};\underline{a};r)}^{(\underline{b};\underline{a'};r)}$:
\begin{equation}\label{flagdef}
\begin{array}{cccccc}
A^1_{(B\rightarrow A)} &  \stackrel{pr_2}{\longrightarrow}& _0\!\Hom _R(I_B ,B) \\
 \downarrow pr_1 & & \downarrow  \\
_0\!\Hom _R(I_A ,A) & \longrightarrow & _0\!\Hom _R(I_B ,A)
\end{array} ,
\end{equation}
where $X=\Proj(A)$ and $Y=\Proj(B)$. Here the tangent space of the fiber,
$p_1^{-1}((X))$, of
$$p_1: \Hilb ^{p_X(t),p_Y(t)}(\PP^{n})\longrightarrow \Hilb
^{p_X(t)}(\PP^{n})$$
at $(X\subset Y)$ corresponds to the kernel, $_0\!\Hom _R(I_B ,I_{A/B})$, of
$pr_1$ 
and the fiber $p_2^{-1}((Y))$ of
$$p_2: \Hilb ^{p_X(t),p_Y(t)}(\PP^{n})\longrightarrow \Hilb ^{p_Y(t)}(\PP^{n})
$$
to the kernel $_0\!\Hom
_R(I_{A/B},A)\cong (MI\otimes A)_{a_{t+c-1}}$ of
$pr_2$ by Proposition \ref{vanish}(iii). 

Let us compute the dimensions of these tangent spaces to fibers under
appropriate assumptions, and we start by finding
$\dim (MI\otimes A)_{a_{t+c-1}}$.

\begin{lemma}\label{dimMI} Assume $r\ge 2$ and $\depth _JA\ge 2$ where
  $J:=I_{t-r}(\varphi _{t+c-2}^*)$.
\begin{itemize}
\item[\rm (i)] If $c\ge 2-r$ and $a_{t+c-1}+v< s_r-b_r+b_1$, then
   $$\dim (MI\otimes A)_{a_{t+c-1}+v} = \dim (MI)_{a_{t+c-1}+v} \ .$$
\item[\rm (ii)] Assume  $c > 2-r$. If $v$ is any
   integer such that $(MI\otimes I_{A/B})_{a_{t+c-1}+v}=0$ (e.g. $v < s_r-b_r-a_{t-r+1}+b_1$), then we have:
 $$\dim (MI\otimes A)_{a_{t+c-1}+v}=\dim (N\otimes B)_{a_{t+c-1}+v}-\dim
 B_v.$$
 \end{itemize}
 Moreover in the case $v < s_r-b_r-a_{t-r+1}+b_1$, we also have
 $\dim B_v = \dim R_v.$
 \end{lemma}

 \begin{proof} (i) Let $v'= a_{t+c-1}+v$. Since  the smallest degree of a
generator of $I_A:=\ker (R\twoheadrightarrow A)$ is
$s_r-b_r$ (cf. \eqref{digclassic} below) and since $b_1 \le a_1$ it follows that
$$(G^*\otimes I_A)(v')_0=(F^*\otimes I_A)(v')_0=0$$ by the assumption
$v'<s_r-b_r +b_1$. Considering the diagram
$$
\begin{array}{cccccccccc}
 &  &  G^* & \longrightarrow & F^* & \longrightarrow & MI & \longrightarrow &  0 \\
 &  & \downarrow & & \downarrow \\
  & &  G^* \otimes A & \longrightarrow & F^* \otimes A & \longrightarrow & MI\otimes A & \longrightarrow &  0 \\
\end{array}
$$
in degree $v'$ we easily conclude the proof of (i).

 (ii) Again set $v'= a_{t+c-1}+v$. We recall the exact sequence  \eqref{section}:
\begin{equation*}
0  \longrightarrow  B(-a_{t+c-1})  \longrightarrow  N\otimes B
\longrightarrow  MI\otimes B  \longrightarrow  0 \ .
\end{equation*}
Since  by assumption we have $(MI\otimes I_{A/B})_{v'}=0$ which implies
$(MI\otimes B)_{v'} \cong (MI\otimes A)_{v'}$, we get
$$\dim (MI\otimes A)_{v'}=\dim (N\otimes B)_{v'}-\dim B_v.$$ Moreover
$F^* \longrightarrow MI$ is surjective and we get $(MI\otimes I_{A/B})_{v'}=0$ if we can
show $(F^* \otimes I_{A/B})_{v'}=0$. To show it we consider the
following diagram
\begin{equation}\label{digclassic}
\begin{array}{cccccccc}
\wedge ^{t-r+1}G_{t+c-2}^*\otimes \wedge ^{t-r+1}F & \longrightarrow&  R & \longrightarrow &  B & \longrightarrow & 0\\
\downarrow &  & \| & & \downarrow  \\
\wedge ^{t-r+1}G^*\otimes \wedge ^{t-r+1}F & \longrightarrow & R & \longrightarrow &  A & \longrightarrow & 0\\
\end{array}
\end{equation}
of exact horizontal sequences.
The summand of $\wedge ^{t-r+1}G^*\otimes \wedge ^{t-r+1}F$ of the smallest
possible degree is $\sum _{i=1}^{t-r+1}a_{i}-\sum _{i=0}^{t-r}b_{i+r}=s_r-b_r$
and correspondingly for $\wedge ^{t-r+1}G_{t+c-2}^*\otimes \wedge ^{t-r+1}F$
while the smallest degree of a generator of $I_{A/B}$ is
$s_r-b_r-a_{t-r+1}+a_{t+c-1}$ because it must involve $a_{t+c-1}$. Hence we get $(F^*\otimes I_{A/B})(a_{t+c-1}+v)_0=0$
because $v<s_r-b_r-a_{t-r+1}+b_1$.
\end{proof}

\begin{remark} \rm (1) In the linear case  (i.e. $b_i=0$ and $a_j=1$ for all
  $i,j$) we have $s_r=t-r+1$, and Lemma \ref{dimMI} (i) and (ii) applies for $v<t-r$.

  (2) By \cite[ pgs 162-166]{b-v} we have an exact sequence:
\begin{equation*} 
\wedge ^{t-r+1}G^*\otimes \wedge ^{t-r}F\otimes A  \longrightarrow  G^*\otimes
A  \longrightarrow F^*\otimes A  \longrightarrow  MI\otimes A  \longrightarrow
0 \ .
\end{equation*}
Since the smallest degree of a summand of
$\wedge ^{t-r+1}G^*\otimes \wedge ^{t-r}F$ is $s_r$, we get the injectivity of
the morphism $G^*\otimes A \to F^* \otimes A$ in degree
$a_{t+c-1}+v$ only assuming $v<s_r-a_{t+c-1}$. For such $v$ we get 
$$\dim (MI\otimes A)_{a_{t+c-1}+v}=\dim (F^*\otimes A)_{a_{t+c-1}+v}-\dim
(G^*\otimes A)_{a_{t+c-1}+v}$$
which may be used to improve upon Lemma~\ref{dimMI} in the case $b_r >
b_1$.
\end{remark}

The problem of finding the dimension of the other fiber will mainly be
considered in Section 8. There we successively delete $c+r-2$ columns from the
right-hand side of the $t \times (t+c-1)$ homogeneous matrix $\cA$, and taking
the
$(t-r+1)\times (t-r+1)$ minors, we get a flag of determinantal rings:
$$A_{2-r}\twoheadrightarrow \cdots \twoheadrightarrow A_0\twoheadrightarrow
A_1\twoheadrightarrow A_2\twoheadrightarrow A_3\twoheadrightarrow \cdots
\twoheadrightarrow A_c =A, \quad X_{i}=\Proj(A_{i})$$
where e.g. $A_1$ (resp. $A_{2-r}$) is Gorenstein (resp. standard
determinantal) defined by the $(t-r+1)\times (t-r+1)$ minors of a $t\times t$
(resp. $t\times (t-r+1)$) matrix. Let $I_{i}=I_{A_{i+1}/A_{i}}$ be the ideal
defining $A_{i+1}$ in $A_{i}$ and let
$I_{A_{i}}=I_{t-r+1}(\varphi ^*_{t+i-1})$. We will prove that
$\dim W(\underline{b};\underline{a};r)=\lambda_c $ provided
$\dim W(\underline{b};\underline{a'};r)=\lambda _{c-1}$ and to check it we
will need
$$_0\! \hom_R(I_{A_{c-1}},I_{c-1})=\sum _{j=1}^{t+c-2} {a_j-a_{t+c-1}+n\choose
  n}.$$
In Section 8, we will show the last equality under the numerical
restrictions given in Corollary~\ref{fiberpr2}. In particular, we prove there

\begin{corollary} \label{prefiberpr2} Let $\cA $ be a general homogeneous $t\times (t+c-1)$ matrix with entries homogeneous forms of degree $a_j-b_i$. Let $c \ge 3-r$
  and suppose $a_1>b_t$, $r \ge 2$ and $\dim A \ge 2$ if $c>0$ and $\dim A \ge 3$ if $c \le 0$. If
  $b_r-b_1 < \sum _{i=1}^{t-r}(a_i-b_{r+i})+a_{t+c-1}- a_{t+c-2}$ and
  $a_{t-r+1} < a_{t+c-1} -\sum _{i=1}^{t-r+1}(b_{r+i-1}-b_{i})$, then
$$ _0\! \hom_R(I_{A_{c-1}},I_{c-1})=\sum
_{j=1}^{t+c-2}{a_j-a_{t+c-1}+n\choose n} \, .$$
Note that all the above inequalities (for $a_j, b_i$) hold if
 $b_t = b_1 < a_1$ and $a_{t-r+1} < a_{t+c-1}$ (and $r < t$).
\end{corollary}

 Let us come back to the diagram of infinitesimal deformations at a point $(X\subset Y)\in W_{(\underline{b};\underline{a};r)}^{(\underline{b};\underline{a'};r)}$:
$$
\begin{array}{cccccc}
A^1_{(B\rightarrow A)} &  \stackrel{pr_2}{\longrightarrow}& _0\!\Hom_R(I_B ,B) \\
 \downarrow pr_1 & & \downarrow  \\
_0\!\Hom_R(I_A ,A) & \longrightarrow & _0\!\Hom_R(I_B ,A)
\end{array}
$$
where $X=\Proj(A)$ and $Y=\Proj(B)$. By Proposition \ref{everydef}, if
$\depth _{J_B}A\ge 3$, and $\depth _{J_A}A\ge 4$ for $c \le 0$, then
$_0\!\, {\rm H}^2(B,A,A)\cong \ _0\! \Ext^1(I_{A/B}/I_{A/B}^2,A)=0$ which
implies that
$$pr_2: A^1_{(B\rightarrow A)} \longrightarrow\ _0\!\Hom_R(I_B ,B)$$ is
surjective and the corresponding projection
$$p_2: \Hilb ^{p_X(t),p_Y(t)}(\PP^{n})\longrightarrow \Hilb ^{p_Y(t)}(\PP^{n})$$ is smooth at
$(X\subset Y)$. Since  in our next proposition we need the property ``every
deformation of $B$ comes from deforming its matrix'' at  several places in the
proof, we rather use Proposition \ref{thmext3}, which then
implies that
$p_2$ 
is smooth at $(X\subset
Y)$ only assuming $\depth _{J_B}A\ge 2$ with $\depth _{J_A}A\ge
3$ if $c \le 0$. The map called
$$p_{2W}:
W_{(\underline{b};\underline{a};r)}^{(\underline{b};\underline{a'};r)} \longrightarrow
W(\underline{b};\underline{a'};r)$$
in diagram (\ref{flag}) is essentially the pullback of $p_2$ to
$ W(\underline{b};\underline{a'};r) \subset \Hilb ^{p_Y(t)}(\PP^n)$ and is
therefore smooth at $(X\subset
Y)$.
Invoking also Remark~\ref{remthm42} we get

\begin{proposition}\label{propo8}
  Let $c
  \ge 3-r$, let $B \twoheadrightarrow
  A$ be determinantal algebras defined by $(t+1-r)$-minors of matrices
  $\cA$
  and $\cB$
  representing $\varphi
  ^*$ and $\varphi^*_{t+c-2}$, respectively, and suppose $\depth _{J_B}B\ge
  r+2$, $J_B:=I_{t-r}(\cB)$, $J_A:=I_{t-r}(\cA) \ne
  R$ and that every deformation of $B$ comes from deforming $\cB$. If $c
  \le0$, we also suppose $\depth _{J_A}A\ge 3$. Letting $W :=
  W_{(\underline{b};\underline{a};r)}^{(\underline{b};\underline{a'};r)}$,
  then the closures $\overline
  W \subset \Hilb ^{p_X(t),p_Y(t)}(\PP^{n})$ and $
  \overline{W(\underline{b};\underline{a'};r)} \subset \Hilb
  ^{p_Y(t)}(\PP^{n})$ are generically smooth irreducible components and the
  restriction $p_{2W}$
  of the $2^{nd}$
  projection $p_2:
  \Hilb ^{p_X(t),p_Y(t)}(\PP^{n})\longrightarrow \Hilb
  ^{p_Y(t)}(\PP^{n})$ to $W$ is smooth at $(X\subset
  Y)$ with fiber of dimension $\dim (MI \otimes
  A)_{a_{t+c-1}}$. Supposing that the matrix $\cA$ is general, we have
  $$\dim W(\underline{b};\underline{a};r) = \dim
  W(\underline{b};\underline{a'};r)+\dim_k(MI \otimes A)_{(a_{t+c-1})}\ -\
  _0\! \hom_R(I_{B},I_{A/B}).$$
\end{proposition}

First we prove the following

\begin{lemma} \label{unobst1} Let $c$ and
  $B=R/I_{t+1-r}(\cB) \twoheadrightarrow A$ be as in Proposition~\ref{propo8}
  with $\cA$ not necessarily general. Let $X=\Proj(A)$ and $Y=\Proj(B)$. Then
  $\Hilb ^{p_Y(t)}(\PP^{n})$ is smooth at $(Y)$, and
  $ \overline{ W(\underline{b};\underline{a'};r)}$ is an irreducible component
  of $\Hilb ^{p_Y(t)}(\PP^n)$. Moreover $\Hilb ^{p_X(t),p_Y(t)}(\PP^{n})$ is
  smooth at $(X\subset Y)$, and $\overline W $ is an irreducible component of
  $\Hilb ^{p_X(t),p_Y(t)}(\PP^n)$.
\end{lemma}
\begin{proof} By Corollary~\ref{unobstr} we get that $\Hilb ^{p_Y(t)}(\PP^n)$
  is smooth at $(Y)$. By Proposition \ref{thmext3},
  $p_2: \Hilb ^{p_X(t),p_Y(t)}(\PP^{n})\longrightarrow \Hilb
  ^{p_Y(t)}(\PP^{n})$ is smooth at $(X\subset Y)$,
  hence $\Hilb ^{p_X(t),p_Y(t)}(\PP^{n})$, is smooth at $(X\subset Y)$.

  To show that $\overline{ W(\underline{b};\underline{a'};r)}$ is an
  irreducible component of $ \Hilb ^{p_Y(t)}(\PP^n)$, one may use
  Corollary~\ref{unobstr} or argue as follows (we will need the argument
  later). Let $(T,{\mathfrak m}_T)$ be the local ring of
  $\Hilb ^{p_Y(t)} (\PP^{n})$ at $(Y)$ and let $B_{T_2}$ be the pullback of
  the universal object of $\Hilb ^{p_Y(t)}(\PP^n)$ to $\Spec(T_2)$ where
  $T_2= T/{\mathfrak m}_T^2$. Since $T$ is a regular local ring, it suffices
  to show
  $\dim \overline{ W(\underline{b};\underline{a'};r)}=\dim {\mathfrak
    m}_T/{\mathfrak m}_T^2$,
  i.e. that the ``universal object'' $B_{T_2}$ is defined by some matrix
  $\cB_{T_2}$; $B_{T_2}=R_{T_2}/I_t(\cB_{T_2})$. This is, however, true by
  assumption, and even more is true (see Remark~\ref{remMthm61} for an
  extension).

  The proof of $\overline W $ being an irreducible component of
  $\Hilb ^{p_X(t),p_Y(t)}(\PP^n)$ is very similar, only changing
  $(T,{\mathfrak m}_T)$ to be the local ring of $\Hilb ^{p_X(t),p_Y(t)} (\PP^{n})$
  at $(X \subset Y)$ and noticing that the ``universal object''
  $B_{T_2} \twoheadrightarrow A_{T_2}$ is defined in terms of a matrix
  $\cA_{T_2}= [\cB_{T_2},v_{T_2}]$ lifting $\cA= [\cB,v]$ to $T_2$ such that
  $\cA_{T_2}$ defines $A_{T_2}$ and $\cB_{T_2}$ defines $B_{T_2}$ by
  Remark~\ref{rem0thm42}.
  \end{proof}

  \begin{proof} [Proof of Proposition~\ref{propo8}] By Lemma~\ref{unobst1},
    $ \overline{W{(\underline{b};\underline{a'};r)}} \subset \Hilb
    ^{p_Y(t)}(\PP^n)$
    and $\overline W \subset \Hilb ^{p_X(t),p_Y(t)}(\PP^{n})$ are generically smooth
    irreducible components, and $p_2: \Hilb ^{p_X(t),p_Y(t)}(\PP^{n})\longrightarrow \Hilb ^{p_Y(t)}(\PP^{n})$, as well as its restriction $p_{2W}$
    ($W$ is dense in a component), is smooth at $(X\subset Y)$ by Proposition
    \ref{thmext3}. Moreover the fiber $p_{2W}^{-1}((Y))$ of $p_{2W}$ is also
    smooth in some neighbourhood of $(X\subset Y)$, whence its dimension at
    $(X\subset Y)$ is given by the dimension of its tangent space which by
    Proposition \ref{vanish}(iii) is
    $ \ _0\! \hom_B(I_{A/B},A)=\dim_k(MI \otimes A)_{(a_{t+c-1})}$.

    It remains to see the dimension formula. To show it we endow
    $\overline{ W(\underline{b};\underline{a};r)}$ with its reduced scheme
    structure, and we use generic smoothness (since $char k = 0$) onto the
    restriction map $p_{1W}:W \longrightarrow \overline{
      W(\underline{b};\underline{a};r)}$ of
    the $1^{st}$ projection $p_1: \Hilb
    ^{p_X(t),p_Y(t)}(\PP^{n})\longrightarrow \Hilb ^{p_X(t)}(\PP^{n})$ to $W$.
    Since $p_{1W}$ is dominating, there
    is an open $ U \subset \overline{ W(\underline{b};\underline{a};r)}$ such
    that $p_{1}$ restricted to $p_{1W}^{-1}(U)$ is smooth, with fiber dimension
    $\ _0\! \hom_R(I_{B},I_{A/B})$. Hence we get
    $$ \dim W(\underline{b};\underline{a};r) + \ _0\! \hom_R(I_{B},I_{A/B}) =
    \dim W\,.$$
    Then we conclude the proof by using the smoothness of $p_{2W}$ which
    implies
    $$\dim W(\underline{b};\underline{a'};r)+\dim_k(MI \otimes
    A)_{(a_{t+c-1})}\ = \dim W\,.$$
\end{proof}
\begin{remark} \label{genericfib} \rm In general the dimension of fibers may be
  larger than the dimension of the ``generic'' fiber made explicit in the last
  paragraph of the proof above; whence we have
  $$\ _0\! \hom_R(I_{B},I_{A/B}) \ \le \ _0\! \hom_R(I_{B'},I_{A'/B'})$$
  with $(\Proj(A') \subset \Proj(B'))$ any point of $W$ and $B \to A$ as in
  Proposition~\ref{propo8}.
\end{remark}
Our goal is to compute $\dim W(\underline{b};\underline{a};r)$ in terms of
$\dim W(\underline{b};\underline{a'};r)$. Due to Proposition~\ref{propo8} we
need to find the difference,
$\dim (MI \otimes A)_{a_{t+c-1}}-\ _0\! \hom_R(I_B,I_{A/B})$.
Using the definition of $\lambda _c$ and $K_c$, and the exactness of the
Buchsbaum-Rim complex for $c \ge 2$, it is not difficult to show that
\begin{equation}
\label{dimensioMI}
\dim (MI)_{a_{t+c-1}}-\sum _{j=1}^{t+c-2}{a_j-a_{t+c-1}+n\choose n}=\lambda
_c-\lambda_ {c-1} + K_c\ ,
\end{equation}
cf. \cite[proof of Theorem  4.5]{KM2005}  for some details. Since, for
$2-r \le c \le 1$ and $\cA$ general, the sequence
$$0 \longrightarrow G^* \longrightarrow F^* \longrightarrow MI \longrightarrow
0$$
is exact, \eqref{dimensioMI} also holds for $2-r \le c \le 1$ letting $K_c:=0$
for $c \le 2$. Below $A=R/I_{t-r+1}(\cA)$ and $B=R/I_{t-r+1}(\cB)$ where $\cB$
is obtained by deleting the last column of $\cA$.

\begin{theorem}\label{dimW}  Let $\Proj(A) \in W(\underline{b};\underline{a};r)$
  be general with $\dim A \ge 2$ and suppose $3-r \le c$, $b_t < a_1$ and that
  every deformation of $\Proj(B)\in W(\underline{b};\underline{a'};r)$ comes
  from deforming its matrix. For $c \le0$, we also suppose $\dim A\ge 3$. Let
  $\kappa$ be an integer satisfying
  $ \dim W(\underline{b};\underline{a'};r)=\lambda _{c-1} -\kappa$ and suppose
  $a_{t+c-1} < s_r-b_r+b_1$. Moreover suppose \
  \begin{equation} \label{fibp1}_0\! \hom_R(I_B,I_{A/B}) \le \sum
    _{j=1}^{t+c-2} {a_j-a_{t+c-1}+n\choose   n}
  \end{equation}
  (resp. with equality in \eqref{fibp1}, e.g. that
  $b_t = b_1$ and $a_{t-r+1} < a_{t+c-1}$). Then $K_i=0$ for
  $3 \le i \le c$ and we have
   $$\dim W(\underline{b};\underline{a};r) \ge \lambda_c - \kappa \ \qquad
   ({\rm resp.} \quad \ \dim W(\underline{b};\underline{a};r) = \lambda_c -
   \kappa \ ) \ . $$ In particular if $\kappa=0$ and
   the {\rm inequality} in  \eqref{fibp1} hold,
   then  $$\dim W(\underline{b};\underline{a};r) = \lambda_c \ . $$
 \end{theorem}


 \begin{proof} Since we assume $b_t < a_1$ 
   we have $s_r \le s_0$ and hence $K_i=0$ for $3 \le i \le c$ by
   \eqref{Kis0}. We also have
   $\depth _{I_{t-r}(\cB)A}A\ge \min\{\dim A,c+r-1\} \ge 2$ by
   Remark~\ref{minormax}, and similarly
   $\depth _{I_{t-r}(\cA)A}A\ge \min\{\dim A,c+2r\} \ge 3$ for $c \le 0$. Note
   that $b_t < a_1$ implies $I_{t-r}(\cA) \ne R$. Since we by
   Lemma~\ref{dimMI} and the assumption $a_{t+c-1} < s_r-b_r+b_1$ may replace
   $\dim (MI)_{a_{t+c-1}}$ by $\dim (MI \otimes A)_{a_{t+c-1}}$ in
   (\ref{dimensioMI}), it follows from (\ref{dimensioMI}), Proposition
   \ref{propo8} and assumptions that we have the inequality
$$
\begin{array}{rcl} \dim  W(\underline{b};\underline{a};r) & = & \dim (MI
                                                                \otimes A)_{a_{t+c-1}}- \dim\, _0\!\Hom (I_B ,I_{A/B}) +\dim  W(\underline{b};\underline{a'};r) \\
& \ge &  \lambda_c-\lambda_{c-1}+K_c+\lambda_{c-1}- \kappa = \lambda_c -\kappa
\end{array} \ $$  which, moreover, becomes an equality if we have equality in
\eqref{fibp1}. Since the e.g.-assumptions imply
$_0\! \hom_R(I_B,I_{A/B})=\sum _{j=1}^{t+c-2} {a_j-a_{t+c-1}+n\choose n}$
taking $A_{c-1}=:B=R/I_B$, $A_{c}=:A=R/I_A$ and $I_{A/B}=I_A/I_B$ in
Corollary~\ref{prefiberpr2}, we are done except for the final statement.

We have $\dim W(\underline{b};\underline{a};r) \ge \lambda_c $ by the first
part of the proof. Then we conclude the proof by Theorem~\ref{ineqdimW}
because $a_{t+c-1} < \ell_2$ and $-b_1 < \ell'_2$ holds by \eqref{Ki0} and the
text accompanying \eqref{Ki0}.
\end{proof}

Below we denote $\lambda_i$ and $W(\underline{b};\underline{a};r)$ by
$\lambda_i(R)$ and $W(\underline{b};\underline{a};r,R)$ respectively since
they obviously both depend on $R$. Moreover, as always, $\cB$ is obtained by
deleting the last column of $\cA$.

\begin{corollary} \label{cordimRdetW} {\rm (i)} Let
  $\Proj(A) \in W(\underline{b};\underline{a};r,R)$, $A:=R/I_{s}(\cA)$ with
  $s=t+1-r$, be a general determinantal scheme and suppose $\dim A \ge 2$,
  $c \ge 3-r$, $a_1 > b_t$ and that every deformation of
  $\Proj(B)\in W(\underline{b};\underline{a'};r,R)$ comes from deforming its
  matrix $\cB$. For $c \le 0$, we also suppose $\dim A\ge 3$. Moreover suppose
  that $ \dim W(\underline{b};\underline{a'};r,R)=\lambda _{c-1}(R)$ and
  $a_{t+c-1} < s_r-b_r+b_1$ and that
  $_0\! \hom_R(I_B,I_{A/B}) \le \sum _{j=1}^{t+c-2} {a_j-a_{t+c-1}+n\choose
    n}\, .$
  Then $$\dim W(\underline{b};\underline{a};r,R) = \lambda_c(R) \ .$$

  {\rm (ii)} Suppose $a_1=a_{t+c-2}$ and $b_1=b_t$ or more generally that
  $ _{v}\! \hom_R(I_B,I_{A/B}) = 0$ for all $v \le -1$. Taking
  $s \times s$-minors of the matrices $\cA$ and $\cB$ considered belonging to
  a larger polynomial ring $R':=R[\underline y]$ where
  $\underline y := y_1,\cdots,y_e$ are $e$ indeterminates and supposing
  $ \dim W(\underline{b};\underline{a'};r,R')=\lambda _{c-1}(R')$, then
  $A':=R'/I_{s}(\cA)$ and $B':=R'/I_{s}(\cB)$ satisfy all assumptions of {\rm
    (i)}, replacing $A, B, R$ there by $A', B', R'$, except that $A'$ is
  general. Indeed we have
  $\ _0\! \hom_R(I_{B},I_{A/B})= \ _0\! \hom_R(I_{B'},I_{A'/B'})$ and we
  get $$\dim W(\underline{b};\underline{a};r, R') = \lambda_c(R')\, .$$
\end{corollary}

\begin{remark} \label{minma} \rm (1) Note that in the case $c = 3-r$ of
  Corollary~\ref{cordimRdetW}, the assumptions ``every deformation of
  $\Proj(B)\in W(\underline{b};\underline{a'};r,R)$ comes from deforming
  $\cB$'', ``$ \dim W(\underline{b};\underline{a'};r,R)=\lambda _{c-1}(R)$''
  and ``$ \dim W(\underline{b};\underline{a'};r,R')=\lambda _{c-1}(R')$'' hold
  by Theorem~\ref{ineqdimW}(iii) and Theorem~\ref{Amodulethm5}(ii).

  (2) Since the smallest degree of a generator of $I_{A/B}$ is
  $s_r-b_r-a_{t-r+1}+a_{t+c-1}$,
  $s_r-b_r:=\sum _{i=1}^{t-r+1}(a_i-b_{r-1+i})$, and the largest degree of a
  generator of $I_B$ is
  $mdg(I_B):=\sum ^{t+c-2}_{j=c+r-2}a_j -\sum _{i=1}^{t+1-r}b_i=\sum
  _{i=1}^{t+1-r}(a_{c+r-3+i}-b_i)$,
  we get $ _{v}\! \hom_R(I_B,I_{A/B}) = 0$ for $v \le -1$ if
  $s_r-b_r-a_{t-r+1}+a_{t+c-1} \ge mdg(I_B)$, which holds if $a_1=a_{t+c-2}$
  and $b_1=b_t$ because $a_{t+c-1} \ge a_{t-r+1}$ always holds.
\end{remark}
\begin{proof}
  (i) This is really the main statement of Theorem~\ref{dimW} because we only
  need the inequality in \eqref{fibp1} to get
  $\dim W(\underline{b};\underline{a};r,R) = \lambda_c(R)$.

  (ii) Suppose
  $ \dim W(\underline{b};\underline{a'};r,R')=\lambda _{c-1}(R')$. Then we
  claim that all assumptions of (i) holds for $A', B', R'$ instead of
  $A, B, R$. Indeed by Corollary~\ref{gendetrem} we get that
  $\Proj(B')\in W(\underline{b};\underline{a'};r,R')$ comes from deforming its
  matrix $\cB$. The other assumptions are straightforward provided we can
  prove $\ _0\! \hom_R(I_{B},I_{A/B})= \ _0\! \hom_R(I_{B'},I_{A'/B'})$. However, this
  will follow if we can show
 \begin{equation} \label{kflatS} \Hom_R(I_{B},I_{A/B}) \otimes_k k[\underline
   y] \cong \
 \Hom_R(I_{B'},I_{A'/B'})
\end{equation}
 because we can then take the degree zero part to conclude using
 $ _{v}\! \hom_R(I_B,I_{A/B}) = 0$ for $v < 0$ and $\deg(y_i)= 1$ for all $i$.
 Hence, since $k[\underline y]$ is obviously $k$-flat, it  suffices  to prove
  $$I_{B}\otimes_k k[\underline y] = I_{B'} \quad {\rm and} \quad
  I_{A/B}\otimes_k k[\underline y] = I_{A/B} \ .$$
  To prove it we use $R'=R \otimes_k k[\underline y]$ and the first terms of
  the Lascoux resolution of $I_B$ as an ideal in $R$ and apply
  $(-) \otimes_k k[\underline y]$ to it. Since the matrices defining $I_B$ and
  $I_{B'}$ are the same, we get $I_{B}\otimes_k k[\underline y] = I_{B'}$. The
  same argument applies to see $I_{A}\otimes_k k[\underline y] = I_{A'}$ and
  then \eqref{kflatS} follows, and hence the claim, by applying
  $(-) \otimes_k k[\underline y]$ onto $0 \to I_B \to I_A \to I_{A/B} \to 0$.

  We remark that $\Proj(A')$ is a special element of
  $W(\underline{b};\underline{a};r, R')$. To get (ii) from (i) we still have
  to see that the general element of $W(\underline{b};\underline{a};r, R')$
  satisfies the assumptions in (i). Again this is not so difficult due to
  Corollary~\ref{unobstr} and Remark~\ref{genericfib}, noting that the latter
  gives us exactly the inequality the right way to be consistent with
  \eqref{fibp1}. Thus we have proved
  $\dim W(\underline{b};\underline{a};r, R') = \lambda_c(R')$.
\end{proof}

\begin{example} \label{ex1dimW}  \rm (Determinantal quotients of
  $R=k[x_0, x_1, \cdots ,x_n]$, using Theorem~\ref{dimW} for $r=2$)

  Let $\cA= [\cB,v]$ be a general $4 \times 4$ matrix with linear (resp.
  quadratic) entries in the first, second and third (resp. fourth) column,
  i.e. where $\cB$ is a linear $4 \times 3$ matrix. We {\it claim} the
  vanishing of all $3 \times 3$ minors of $\cA$ defines a determinantal ring
  that satisfies the assumption of Theorem~\ref{dimW} for $r=2$, $c=1$ and
  $n \ge 5$. Indeed since the $3 \times 3$ minors of $\cB$ are maximal minors
  defining $B$, one knows that every deformation of $B$ comes from deforming
  $\cB$ by Theorem~\ref{Amodulethm5} or more directly by \cite[Theorem
  5.16]{K2014}, cf. \cite{E}. By Theorem~\ref{ineqdimW} we
  get that the dimension of
  $$\overline{ W(0^4;1^3;2)}:= \overline{ W(0,0,0,0;1,1,1;2)}$$ is $ \lambda_0$.
  Since the assumption $a_t < s_2$ and the e.g. statement of
  Theorem~\ref{dimW} clearly hold we get the claim and hence
  that $$\dim \overline{ W(0^4;1^3,2;2)}=\lambda_1=2n^2+15n-12$$ by
  Theorem~\ref{dimW}. We have checked the answer by using Macaulay2 in the
  case $\cB$ is the generic linear matrix with entries $x_0,x_1,...,x_{11}$
  and the transpose of $v$ is $x_{12}^2,x_{13}^2,x_{14}^2,x_{15}^2$. We have
  got $\ \Ext_A^1(I_A/I_A^2,A) = 0$ and that the dimension of the tangent
  space of $ \Hilb^{p_X(t)}(\PP^{15})$ at $X:=\Proj(A)$ was $663$, whence
  $\dim_{(X)} \Hilb(\PP^{15}) =663$, coinciding with our formula when $n=15$.

  In fact the Macaulay2 computation also implies that
  $\overline{ W(0^4;1^3,2;2)}$ is a generically smooth irreducible component
  of $\Hilb^{p_X(t)}(\PP^{15})$, a problem which we will study closely in
  Section 7 and 8. 
  At this stage we remark that this is no longer true for small values of $n$.
  Indeed in \cite[Example 5.1]{KM2009} we considered an arithmetically
  Gorenstein scheme defined by the submaximal minors of a general $4\times 4$
  matrix with $\cA= [\cB,v]$ as above, and
  got that $\dim_{(X)}\Hilb ^{p_X(t)}(\PP^5)=125$ and
   $$\codim_{\Hilb ^{p_X(t)}(\PP^5)} W(0^4;1^3,2;2)=12$$  by direct
   calculations using the theory developed there. 
   Clearly the irreducible locus $\overline{ W(0^4;1^3,2;2)}$ is not a
   component of $ \Hilb(\PP^{15})$. Moreover note that we get
   $$\dim \overline{ W(0^4;1^3,2;2)}=113\, , $$ confirming our formula above
   with $n=5$.
\end{example}

\begin{example} \label{exgendet} \rm (Verifying the conditions of
  Theorem~\ref{dimW} by using Macaulay2)

  In this example we consider determinantal rings where the entries $x_{i,j}$
  of the $t \times (t+c-1)$ matrix $\cA$ are the indeterminates of $R$, i.e.
  so-called generic determinantal schemes $X=\Proj(A)$, and we use
  Macaulay2 to compute $ _0\! \hom_R(I_B,I_{A/B})$
  and check that
\begin{equation} \label{codgen}\begin{array}{rcl}
  _0\! \hom_R(I_B,I_{A/B}) & = & \sum _{j=1}^{t+c-2} {a_j-a_{t+c-1}+n\choose
    n}=t+c-2 
  \end{array}
  \end{equation}
  where $B=R/I_{t-r+1}(\cB)$ and $\cB$ is obtained by deleting the last column
  of $\cA= [\cB,v]$. Note that by Corollary~\ref{gendetrem}, every deformation
  of $B$ comes from deforming $\cB$, whence all
  assumptions of Theorem~\ref{dimW} hold and we conclude
  that
  $$ \dim W(\underline{b};\underline{a};r)= \lambda_c\ . $$

\vskip 2mm (i) Submaximal minors, i.e. $r=2$. Let $t=3$, $\cA= (x_{i,j})$ the
generic $3 \times (c+2)$ matrix and let $1 \le c \le 7$. The vanishing of all
$2 \times 2$ minors of $\cA$ defines a determinantal ring. We successively
show that all conditions of Theorem~\ref{dimW} are satisfied starting with the
case $c=1$. Due to Proposition~\ref{deforminggenericcase} and
Corollary~\ref{gendetrem} one knows that every deformation of $B$ comes from
deforming $\cB$, and by Theorem~\ref{ineqdimW}(iii) that
$$\dim \overline{ W(0^3;1^2;2)}= \lambda_0=3\cdot 2\cdot 9 -3^2-2^2+1=42.$$
Then \eqref{codgen} is
verified for $c=1$ by Macaulay2, and we get that
$$\dim \overline{ W(0^3;1^3;2)}= \lambda_1=64 \ $$  by Theorem~\ref{dimW}.
We get even more by Corollary~\ref{cordimRdetW}(ii), i.e. we get that
$$\dim \overline{ W(0^3;1^3;2;R')}= \lambda_1(R')$$ where $R':=R[\underline y]$
may contain more indeterminates that those of $R$. By
Corollary~\ref{gendetrem} we also know that every deformation of
$A':=R'/I_2(\cA)$ comes from deforming $\cA$. 

Now for every $c$, $2 \le c \le 7$ we verify \eqref{codgen}. Then we claim
that
$$\dim \overline{ W(0^3;1^{c+1};2)}= \lambda_{c-1}$$
and $\dim \overline{ W(0^3;1^{c+2};2)} = \lambda_c$ by induction on $c \ge 2$.
Indeed when we delete a column of the matrix $\cA$ of indeterminates defining
the ring $R$, the resulting matrix $\cB$ consists of indeterminates belonging
to a ring which has more variables that those appearing in $\cB$. So $B$ is as
$A'$ above and we get
$$\dim \overline{ W(0^3;1^{c+2};2)} = \lambda_c = 8(c+2)^2-8 \ $$
by Theorem~\ref{dimW} and in fact
$\dim \overline{ W(0^3;1^{c+2};2,R')} = \lambda_c(R')$ by
Corollary~\ref{cordimRdetW}(ii) and moreover that every deformation of
$R'/I_2(\cA)$ comes from deforming $\cA$ by Corollary~\ref{gendetrem}. Hence
we conclude by induction. Finally it is worthwhile to point out that
$\overline{ W(0^3;1^{c+2};2)}$ is a generically smooth irreducible component
of $\Hilb ^{p_X(t)}(\PP^{3c+5})$ for $1 \le c \le 6$ by Lemma~\ref{unobst1}, cf.
Lemma~\ref{unobst} which also covers the case $c=7$.

\vskip 2mm
(ii) Submaximal minors, i.e. $r=2$ with $t=4$. Let $\cA= (x_{i,j})$ be the
generic $4 \times (c+3)$ matrix and let $1 \le c \le 3$. The vanishing of all
$3 \times 3$ minors of $\cA$ defines a determinantal ring and we verify
\eqref{codgen} by using Macaulay2 for every $c$, $1 \le c \le 3$. Then we can
argue exactly as in (i) to get that
$$\dim \overline{ W(0^{4};1^{c+3};2)} = \lambda_c:=15(c+3)^2-15\, .$$
By Lemma~\ref{unobst1}, $\overline{ W(0^{4};1^{c+3};2)}$ is a generically
smooth irreducible component of $\Hilb ^{p_X(t)}(\PP^{4c+11})$ for
$1 \le c \le 2$ .

\vskip 2mm
(iii) ``Subsubmaximal'' minors, i.e. $r=3$. Let $t=4$, let $\cA= (x_{i,j})$ be
the generic $4 \times (c+3)$ matrix and let $1 \le c \le 4$. The vanishing of
all $2 \times 2$ minors of $\cA$ defines a determinantal ring $A$ and we
verify \eqref{codgen} by Macaulay2 for every $c$, $1 \le c \le 4$. Starting
with $(A) \in W(0^{4};1^4;3)$ where $c=1$, then $B$ belongs to
$W(0^{4};1^3;3)$, a locus which we considered in (i) above. Indeed by
Lemma~\ref{tr}, $ W(0^{3};1^4;2)= W(-1^{4};0^3;3)$ and obviously,
$ W(-1^{4};0^3;3)= W(0^{4};1^3;3)$. Hence we get
$$\dim \overline{ W(0^{4};1^3;3)} =\lambda_0=120$$ and that every deformation
of $B$ comes from deforming $\cB$. Similarly if $R':=R[y_1, \cdots y_e]$ where
the $y_i$ are indeterminates, we get from (i) that
$\dim \overline{ W(0^{4};1^3;3;R')} =\lambda_0(R')$ and that every deformation
of $R'/I_2(\cB)$ comes from deforming $\cB$. It follows that
$\dim \overline{ W(0^{4};1^{4};3)} = \lambda_1$ by Theorem~\ref{dimW}, that
$\dim \overline{ W(0^{4};1^{4};3,R')} = \lambda_1(R')$ by
Corollary~\ref{cordimRdetW}(ii) and that every deformation of $R'/I_2(\cA)$
comes from deforming $\cA$ by Corollary~\ref{gendetrem}. Then we can as
previously use induction and Corollary~\ref{cordimRdetW} to get
$$\dim \overline{ W(0^{4};1^{c+3};3)} = \lambda_c:=15(c+3)^2-15\, .$$
By Lemma~\ref{unobst1} we get that $\overline{ W(0^{4};1^{c+3};3)}$ for
$1 \le c \le 3$, is a generically smooth irreducible component of
$\Hilb ^{p_X(t)}(\PP^{4c+11})$. 
Note that Macaulay2 is used to verify \eqref{codgen} for
$1 \le c \le 4$.

\vskip 2mm (iv) Let $r=3$ (resp. $r=4$) and $t=5$. The vanishing of all
$3 \times 3$ (resp. $2 \times 2$) minors of a generic $5 \times 5$ matrix
$\cA$ defines a determinantal ring and we verify \eqref{codgen} by using
Macaulay2. Moreover by Lemma~\ref{tr},
$ W(0^{4};1^{5};2) = W(-1^{5};0^4;3)= W(0^{5},1^{4};3)$ (resp.
$ W(0^{4};1^{5};3) = W(0^{5};1^{4};4)$) and combining with (ii) (resp. (iii))
above, we get that a general $B$ of the $\lambda_0$-dimensional
$ W(0^{5};1^4;3)$ (resp. $ W(0^{5};1^4;4)$) comes from deforming $\cB$. Then
Theorem~\ref{dimW}  imply that
$$\dim \overline{ W(0^{5};1^{5};3)} = \dim \overline{ W(0^{5};1^{5};4)} =
\lambda_1=576\, .$$
\end{example}

Since we expect the condition on $_0\! \hom_R(I_B,I_{A/B})$ of
Theorem~\ref{dimW} to be true in general provided $\dim A \ge 3$ (and not only
under the e.g.- assumptions there), we suggest the following:

\begin{conj} \label{conjdimW} Let $\cA$ be a general homogeneous
  $t \times (t+c-1)$ matrix, let $A=R/I_{t-r+1}(\cA)$ where $1 \le r \le t-1$,
  $2-r \le c$ 
  and suppose that $\dim A\ge 2$ for $c \ne 1$ and  $\dim A\ge 3$ for $c= 1$.
  Moreover suppose
  $\Proj(A) \in W(\underline{b};\underline{a};r)$, $a_1 >b_t$ and
  $a_{t+c-1} < s_r-b_r+b_1$. Then
    $$\dim W(\underline{b};\underline{a};r) = \lambda_c \ .$$
  \end{conj}

  \begin{remark} \label{remconjdimW} \rm The conjecture is true for $r=1$ (and
    $c=2-r)$ by Theorem~\ref{ineqdimW}(iii), \eqref{Kis0} and the text
    accompanying \eqref{Ki0}. We have found examples in the case $c=1$ and
    $\dim A = 2$ (and none when $\dim A \ge 3$) where the conclusion of the
    conjecture is not true. Indeed, to support the conjecture we have for
    $r \ge 2$ and $\dim A \ge 2$ computed quite a lot of examples, mostly with
    linear matrices, to see if the assumption
    $$_0\! \hom_R(I_B,I_{A/B})=\sum _{j=1}^{t+c-2} {a_j-a_{t+c-1}+n\choose n}$$
    of Theorem~\ref{dimW} and hence Conjecture~\ref{conjdimW} hold, and we got
    $\dim W(\underline{b};\underline{a};r) = \lambda_c$ except for one in
    Example \ref{afterdeg}(1) below where $\dim A = 2$. It may be true that the
    conclusion of Conjecture~\ref{conjdimW} even holds for
    $\dim A = 2, c=1, r \ge 2$ and $t \ge 4$, see Example \ref{afterdeg}(2)
    and (3).
\end{remark}

\begin{example}\label{afterdeg}\rm
  \vskip 2mm (1) Let $R=k[x_0, x_1, \cdots ,x_{5}]$, $\cA=[\cB,v]$ a general
  linear $3 \times 3$ matrix, $v$ a column, and let $A$ and $B$ be the
  quotients of $R$ defined by their $2 \times 2$ minors. So $t=3$, $r=2$,
  $c=1$ and $\dim A = 2$. A Macaulay2 computation shows that
  $_0\! \hom_R(I_B,I_{A/B})=3$ and not $2$ as Theorem~\ref{dimW} assumes, and
  the proof of Theorem~\ref{dimW} yields
  $$\dim W({0^3};1^3;2) = \lambda_c -1 \ .$$ So
  $\dim W(\underline{b};\underline{a};r) < \lambda_c$ may occur (even though
  it turns out that $\overline{ W({0^3};1^3;2)}$ above is an irreducible
  component, cf. Example~\ref{examples712})! The ``same'' example with one
  more variable in $R$ yields $\dim A = 3$, $_0\! \hom_R(I_B,I_{A/B})=2$ and
  using Remark~\ref{minma}(1)
  we get $\dim W({0^3};{1^3};2) = \lambda_c$ by Theorems~\ref{dimW} or
  Corollary~\ref{cordimRdetW}(i). Also the ``same'' example with two more
  variable in $R$, as well as the generic one where $\dim R=9$, yields
  $_0\! \hom_R(I_B,I_{A/B})=2$ and hence $\dim W(0^3;1^3;2) = \lambda_c$. In
  fact since $_0\! \hom_R(I_B,I_{A/B})=2$ for $\dim R = 7$, we do not need to
  use Macaulay2 for further computations because
  Corollary~\ref{cordimRdetW}(ii) and Remark~\ref{minma}(1) apply and we
  conclude that $\dim W({0^3};{1^3};2) = \lambda_c$ for determinantal
  quotients $A$ of $R$ as above provided $\dim R \ge 7$.

  \vskip 2mm (2) If $R=k[x_0, x_1, \cdots ,x_{n}]$ and we take $\cA=[\cB,v]$
  to be a general linear $4 \times 4$ matrix, so $t=4$, $c=1$ and we let
  $r=3$, i.e. we define the determinantal rings $A$ and $B$ by the
  $2 \times 2$ minors of $\cA$ and $\cB$, then a Macaulay2 computation shows
  that $_0\! \hom_R(I_B,I_{A/B})=3$ in the case $n=10$ ($\dim A = 2$), as
  Theorem~\ref{dimW} or Corollary~\ref{cordimRdetW}(i) requires to get
  $\dim W({0^4};1^4;3) = \lambda_1=145$. But we still have to show the other
  assumptions of Theorem~\ref{dimW}. Therefore letting $\cB=[\cC,w]$ and $C$
  be the ring defined by the $2 \times 2$ minors of $\cC$ we verify
  $_0\! \hom_R(I_C,I_{B/C})=2$ and $_0\! \hom_R(I_B,B)=108$ by using
  Macaulay2. Hence Corollary~\ref{cordimRdetW}(i) and Remark~\ref{minma}(1)
  applies and we get $\dim W({0^4};1^3;3) = \lambda_{0}$. And since
  $\lambda_{0}=108$ we get $\dim W({0^4};1^3;3)=\ _0\! \hom_R(I_B,B)$. It
  follows that every deformation of $B$ comes from deforming $\cB$ by using
  Theorem~\ref{teo3}, Remark~\ref{remthm42} for $C \to B$, instead of
  $B \to A$ (this is a consequence of Theorem~\ref{Wsmooth} as explained in
  detail in Remark~\ref{remthm61}). Thus we have
  $$\dim W({0^4};1^4;3) = \lambda_1=145 \quad {\rm for} \quad n=10 \ (\dim A =
2) \ .$$
  Then using Corollary~\ref{cordimRdetW}(ii) twice, first for $C \to B$
  instead of $B \to A$, we get $\dim W({0^4};1^3;3) = \lambda_{0}$ for any
  $n \ge 10$, and then 
  for $B \to A$, and we conclude that $\dim W({0^4};1^4;3) = \lambda_1$ for
  every $n \ge 10$ ($\dim A \ge 2$).

\vskip 2mm
(3) Let $R=k[x_0, x_1, \cdots ,x_{n}]$ and let
  $\cA=[\cB,v]$ be a general linear $4 \times 4$ matrix, so $t=4$, $c=1$ but
  now we take $r=2$. Then a Macaulay2 computation shows that
  $_0\! \hom_R(I_B,I_{A/B})=3$, as Theorem~\ref{dimW} requires, for every
  $n \in \{5,6,7\}$
  and we get $\dim W({0^4};1^4;2) = \lambda_c$ in each of the cases
  $\dim A = 2, 3, 4$. But again Macaulay2 is only needed in the case
  $\dim A = 2$ ($n=5$) because then Remark~\ref{minma}(1) allows to use
  Corollary~\ref{cordimRdetW}(ii) to get $\dim W({0^4};1^4;2) = \lambda_c$ for
  every $n \ge 5$.
\end{example}

The case $r=2$ and $c=1$ of Theorem~\ref{dimW} is considered in \cite[Theorem
4.6]{KM2009} where a correction term $\kappa$ to the dimension formula is
introduced. Since \cite[Theorem 4.6]{KM2009} assumes
$a_{t} > a_{t-1}+a_{t-2}-b_1$ which is unnecessary for getting
$\dim W(\underline{b};\underline{a};2)$, we take the opportunity to generalize
the dimension formula of $ W(\underline{b};\underline{a};2)$ given there.
Recalling that we get $\underline{a'}$ by deleting $a_{t}$ in
$\underline{a}:=(a_1,...,a_{t})$ we have

\begin{theorem}\label{dimW1} Suppose  that $\Proj(A) \in
  W(\underline{b};\underline{a};2)$
  is general with $c=1$, $I_{t-2}(\cA) \ne R$,
  $\Proj(B)\in W(\underline{b};\underline{a'};2)$ and $\dim A \ge 2$. Moreover
  let $s:=\sum_{i=1}^t(a_i-b_i)$ and suppose $a_i \ge b_{i+3}$ for
  $1 \le i \le t-3$ ($a_1 \ge b_{t}$ for $t=3$) and that
  $_0\! \hom_R(I_B,I_{A/B})=\sum _{j=1}^{t-1} {a_j-a_{t}+n\choose n}$ (e.g.
  $t \ge 3$, $b_t = b_1 < a_1$ and $a_{t-1} < a_{t}$). Then we have
   $$\dim W(\underline{b};\underline{a};2) = \lambda_1 - \kappa_1 \qquad {\rm
     where} $$
$$\kappa_1=
\sum _{1\le j \le t \atop 1\le i \le k \le
t}{a_t-s-b_i-b_k+a_j+n\choose n}- \sum _{1\le i, j \le t \atop 1\le k \le t-1}{a_t-s-b_i-a_k+a_j+n\choose n}$$
$$ + \sum _{1\le i < k \le t-1\atop 1\le j \le t}{a_t-s-a_i-a_k+a_j+n\choose
n}-\sum _{2\le i\le t}{a_t-s+b_i-2b_1+n\choose n}.$$
  \end{theorem}
  \begin{remark} \label{remthm622} \rm The first sentence implicitly implies
    that $W(\underline{b};\underline{a};2) \ne \emptyset $ and
    $W(\underline{b};\underline{a'};2) \ne \emptyset$, while the assumption
    $a_i \ge b_{i+3}$ for $1 \le i \le t-3$ ($a_1 \ge b_{t}$ for $t=3$)
    implies that $\depth _{I_{t-2}(\cB)}B \ge 4$ by \cite[Remark 2.7]{KM2005}.
    Then $I_{t-2}(\cA) \ne R$ is equivalent to $a_i > b_{i+2}$ for some $i$,
    $1 \le i \le t-2$.
\end{remark}
  \begin{proof} Note that $B$ is defined by the maximal minors of the matrix
    $\cB$ associated to $B$. By \cite{E} it follows that every deformation of
    $B$ comes from deforming $\cB$ and that
    $\dim W(\underline{b};\underline{a'};2) = \dim (N_B)_0$ where
    $N_B:= \Hom(I_{B},B)$. We also have
    $\depth _{I_{t-2}(\cB)}B = \depth _{J_B}A +2 \ge 4$, $J_B=I_{t-r}(\cB)A$ by
    Remark~\ref{remthm622}. Hence, using Proposition~\ref{propo8} we get
    that $pr_2$ is surjective and that
$$
\begin{array}{rcl} \dim  W(\underline{b};\underline{a};2) & =
& \dim (MI \otimes A)_{a_{t}}- \dim\, _0\!\Hom (I_B ,I_{A/B}) +\dim  W(\underline{b};\underline{a'};2)
\end{array}
$$
and note that the assumptions of Proposition \ref{vanish}(iii) ($c+r \ge3$,
$\depth _{J_B}A\ge 2$) hold, which implies
$MI\otimes A(a_{t+c-1}) \cong \Hom(I_{A/B},A)$. Applying \cite[Lemma
28]{K2007} in our situation where $M=N_B$ and $M^*=I_B/I_B^2$, cf. \cite[(33)
and Corollary 41]{K2007} and using $s:=\sum_{i=1}^t(a_i-b_i)=s_0-a_{t+1}$ we
have
$$ _0\! \hom(I_{A/B},A) = \dim (I_B/I_B^2)_{s}-\ _0\!
\hom(I_B/I_B^2,I_B/I_B^2)+\dim(K_B)_{n+1-2s}$$
because
$\Hom(I_B/I_B^2,I_B/I_B^2) \cong \Hom_{\mathcal{O}_U}(\widetilde {I_B/I_B^2}^*
\arrowvert_U, \widetilde {I_B/I_B^2}^* \arrowvert_U)) \ \cong \
\Hom_B(N_B,N_B)$
by \eqref{NM} and $\ \Ext_B^1(N_B,B)=0$ and $\Ext_B^1(N_B,N_B)=0$ for similar
reasons, cf. \cite[proof of Corollary 41]{K2007} for more details. By
\cite[Remark 35]{K2007} or \cite[section 2.3]{KM2009}, the exact sequence
$$0\longrightarrow \oplus_{j=1}^{t-1} R(-a_j+a_t-s)\longrightarrow \oplus_{i=1}^t R(-b_i+a_t-s)\longrightarrow
I_B\longrightarrow 0$$
induces exact sequences
$$ 0 \longrightarrow R \longrightarrow \oplus_i I_B(s-a_t+b_i) \longrightarrow \oplus_j
I_B(s-a_t+a_j) \longrightarrow N_B \longrightarrow 0,$$
$$
0 \longrightarrow \Hom_B(I_B/I_B^2,I_B/I_B^2) \longrightarrow \oplus_i
I_B/I_B^2(s-a_t+b_i) \longrightarrow \oplus_j I_B/I_B^2 (s-a_t+a_j) \longrightarrow
N_B \longrightarrow 0$$
and $$ ..\longrightarrow \oplus_j B(s-a_t+a_j) \longrightarrow K_B(n+1) \longrightarrow 0.$$
 Thus if
$\eta(v)=\dim (I_B/I_B^2)_v$ we have
$$ \dim
W(\underline{b};\underline{a};2)=\eta(s)-\sum_{i=1}^{t-1}\eta(a_t-s-b_i)+\sum_{j=1}^{t}\eta(a_t-s-a_j)-
\sum _{j=1}^{t} {a_j-a_{t}+n\choose n}+1 $$
because $ 0=B(s-a_t+a_j-2s)_0 \twoheadrightarrow K_B(n+1-2s)_0$ and
$_0\! \hom_R(I_B,I_{A/B})=\sum _{j=1}^{t-1} {a_j-a_{t}+n\choose n}$ by
assumption. To compute $\eta(v)$ we use the exact sequences
$$0 \longrightarrow I_B^2 \to I_B \longrightarrow I_B/I_B^2 \longrightarrow 0 $$ and
\begin{equation} \label{ImultI}
0\longrightarrow \oplus _{1\le i<j\le
t-1}R(-a_i-a_j+2a_t-2s) \longrightarrow
\oplus _{1\le i \le t \atop 1\le j \le
t-1}R(-b_i-a_j+2a_t-2s)\end{equation}
$$ \longrightarrow \oplus _{1\le
i\le j\le t}R(-b_i-b_j+2a_t-2s)\longrightarrow
I_B^2\longrightarrow 0
$$
and we refer to \cite{KM2009} to see that
$$\eta(s)-\sum_{i=1}^{t-1}\eta(a_t-s-b_i)+\sum_{j=1}^{t}\eta(a_t-s-a_j) =
\sum _{1\le i \le t \atop 1\le j \le
  t}{a_j-b_i+n\choose n}-$$
$$\sum _{1\le i\le t-1 \atop 1\le j \le
  t}{a_j-a_i+n\choose n}-\sum _{1\le i \le t \atop 1\le j \le
  t}{b_i-b_j+n\choose n}+\sum _{1\le i\le t \atop 1\le j \le
  t-1}{b_i-a_j+n\choose n}-\kappa_1\! , \quad \rm{where}$$
$$\kappa_1= \sum _{1\le j \le t \atop 1\le i \le k \le
  t}{a_t-s-b_i-b_k+a_j+n\choose n}- \sum _{1\le i, j \le t \atop 1\le k \le
  t-1}{a_t-s-b_i-a_k+a_j+n\choose n}$$
$$ + \sum _{1\le i < k \le t-1\atop 1\le j \le t}{a_t-s-a_i-a_k+a_j+n\choose
  n}-\sum _{2\le i\le t}{a_t-s+b_i-2b_1+n\choose n}.$$
To get $ \dim W(\underline{b};\underline{a};2)$ we only need to subtract
$\sum _{j=1}^{t-1} {a_j-a_{t}+n\choose n}$ which amounts to change the
indices of the second sum of binomials from ${1\le i\le t-1, 1\le j \le t}$ to
${1\le i\le t, 1\le j \le t}$. We can do exactly the same change for the
indices of the $4^{\rm th}$ sum of binomials because the extra binomials we
add are zero provided $a_t > b_t$. Hence we get
$ \dim W(\underline{b};\underline{a};2)=\lambda_1-\kappa_1$ and we are done.
\end{proof}

\begin{example} \label{dimW2}  \rm (Determinantal quotients of
   $R=k[x_0, x_1, \cdots ,x_n]$, using  Theorem~\ref{dimW1} with $r=2$)

   Let $\cA= [\cB,v]$ be a general $3 \times 3$ matrix with linear (resp.
   quadratic) entries in the first and second (resp. third) column. The degree
   matrix of $\cA$ is
   $\left(\begin{smallmatrix}1 & 1  & 2 \\
       1 & 1 & 2 \\ 1 & 1 & 2 \end{smallmatrix}\right)$
   and $b_i=0$ for $1 \le i \le 3$. The vanishing of all $2 \times 2$ minors
   of $\cA$ defines a determinantal ring that satisfies all conditions of
   Theorem~\ref{dimW} (with $r=2$, $t=3$, $c=1$, noting that
   $$\dim \overline{ W(0^3;1^2;2)}= \lambda_0$$ by Lemma~\ref{tr}), except
   $a_t < s_2 = 2$. However, the weak assumption of Theorem~\ref{dimW1},
   $a_{t-1} < a_t$ holds, and it follows that
   $$\dim \overline{ W(0^3;1^2,2;2)}= \lambda_1-\kappa_1=(3n^2+17n-24)/2$$ by
   Theorem~\ref{dimW1}. Indeed $\kappa_1=6$. To check the answer using
   Macaulay2, let $\cB$ be the generic linear matrix with entries
   $x_0,x_1,...,x_{5}$ and let $v^{tr}=(x_{6}^2,x_{7}^2,x_{8}^2)$. We get
   $\ \Ext_A^1(I_A/I_A^2,A) = 0$ and $\dim_{(X)} \Hilb^{p_X(t)}(\PP^{8}) =152$
   for $X:=\Proj(A)$, coinciding with our formula for
   $ \dim \overline{ W(0^3;1^2,2;2)}$ when $n=8$. We have also checked the
   cases $5 \le n \le7$ using Macaulay2, and for $n=7$ they coincide while in
   the cases $5 \le n \le 6$,
   $ \dim \overline{ W(0^3;1^2,2;2)} < \dim_ {(X)} \Hilb^{p_X(t)}(\PP^{n})$,
   showing that $\overline{W(0^3;1^2,2;2)}$ is not an irreducible component of
   $ \Hilb^{p_X(t)}(\PP^{n})$ for $n < 7$, i.e. when $1 \le \dim X \le 2$, so
   deforming $X$ is not equivalent to deforming its associated homogeneous
   matrix. Note that the curve case ($n=5$) of this example was thoroughly
   analysed in \cite[Example 23]{K2007}.
 \end{example}


\section{Generically smooth components of the Hilbert scheme}

The goal of this section is to address Problems \ref{pblmsagain} (2) and (3)
and to examine when $\overline{ W(\underline{b};\underline{a};r)}$ is a
generically smooth irreducible component of $ \Hilb ^{p_X(t)}(\PP^n)$. Letting
$\underline{a'}=a_1,a_2, \cdots, a_{t+c-2}$ as previously we have

\begin{theorem}\label{Wsmooth}  Let  $c > 2-r$, let $B=R/I_B\twoheadrightarrow
  A$
  be determinantal algebras defined by $(t+1-r)$-minors of matrices $\cA$ and
  $\cB$ representing $\varphi ^*$ and $\varphi^*_{t+c-2}$, respectively, and
  suppose that $X:=\Proj(A) \in W(\underline{b};\underline{a};r)$ is general
  and satisfies $\dim A\ge 2$ and $a_1 > b_t$. If $c \le0$, we also
  suppose $\dim A\ge 3$. Moreover let $\gamma$ be the composition
  $$_0\!\Hom_R(I_A ,A)\longrightarrow\ _0\!\Hom_R(I_B ,A)\longrightarrow\
  _0\!\Ext^1_B(I_B/I_B^2,I_{A/B})$$
  and suppose $\gamma =0$ (e.g.\
  $_0\!\Ext^1_B(I_B/I_B^2,I_{A/B})\hookrightarrow\ _0\!\Ext^1_B(I_B/I_B^2,B)$
  is injective) and that every deformation of $B$ comes from deforming $\cB$.
  Then $ \overline{W(\underline{b};\underline{a};r)}$ is a generically smooth
  irreducible component of $\Hilb ^{p_X(t)}(\PP^n)$ and every deformation of $A$
  comes from deforming $\cA$. Moreover
  $$\dim W(\underline{b};\underline{a};r) = \dim
  W(\underline{b};\underline{a'};r)+\dim_k(MI \otimes A)_{(a_{t+c-1})}\ -\ _0\!
    \hom_R(I_{B},I_{A/B}).$$
  \end{theorem}

  Since we have $\depth _{I_{t-r}(\cB)A}A\ge \min\{\dim A,c+r-1\} \ge 2$ by
  Remark~\ref{minormax}, the theorem is an immediate consequence of
  Theorem~\ref{teo3}, Remark~\ref{remthm42}, Proposition~\ref{propo8} and the
  following

  \begin{lemma} \label{unobst} Set $A=R/I_{t+1-r}(\varphi ^*)$, let
    $X:=\Proj(A) \in W(\underline{b};\underline{a};r)$, and suppose $\dim
    X \ge 1$ and
    that every deformation of $A$ comes from deforming $\cA$. Then
    $\overline{ W(\underline{b};\underline{a};r)}$ is a generically smooth
    irreducible component of $\Hilb ^{p_X(t)}(\PP^n)$.
\end{lemma}
\begin{proof} By Corollary~\ref{unobstr}, $\Hilb ^{p_X(t)}(\PP^n)$
  is smooth at $(X)$. Then the proof in the $2^{nd}$ paragraph of
  Lemma~\ref{unobst1}, with $A$ instead of $B$, applies, or see
  Remark~\ref{remMthm61} which proves even more.
 \end{proof}

 \begin{remark} \label{remMthm61} \rm Let $(T,{\mathfrak m}_T)$ be the local
   ring of $\Hilb ^{p_X(t)} (\PP^{n})$ at $(X)$, $X=\Proj(A)$ of dimension
   $\ge 1$ (cf. \eqref{Grad}) and let $\Proj(A_{T_2})$ be the pullback of the
   universal object of $\Hilb ^{p_X(t)}(\PP^n)$ to $\Spec(T_2)$ where
   $T_m= T/{\mathfrak m}_T^m$, $m \ge 2$ and $T_2=k\{t_1, \cdots,t_k\}$. In
   the proof of Lemma~\ref{unobst1} we observed that the ``universal object''
   $A_{T_2}$ is defined by some matrix by assumption. Even more is true. In
   fact we can extend the pullback of the universal object of
   $\Hilb ^{p_X(t)}(\PP^n)$ to $\Spec({\hat T})$ where ${\hat T}$ is the
   completion of the regular local ring $T$ with respect to ${\mathfrak m}_T$.
   And since any deformation of the ``universal quotient''
   $R\otimes_k T_2 \to A_{T_2}$ to ${\hat T}$ suffices to define the
   prorepresenting object by \cite{La}, proof of Theorem 4.2.4, up to
   isomorphism, we may take the matrix $\cA_{\hat T}$ of $A_{\hat T}$ as
   defined by some lifting (e.g. take the entries to be of degree one in the
   $t_i$) of $\cA_{T_2}$ to ${\hat T}$, whence the generators of
   $I_{t+1-r}(\cA_{\hat T})$ are polynomials, and not power series in $t_i$.
   Thus we can further extend the entries of $\cA_{\hat T}$ to polynomials
   $f_{ij,D}$ with coefficients in $D$ where $\Spec(D)$ is a small enough open
   set of $ \Hilb ^{p_X(t)}(\PP^n)$ containing $(X)$ for which the Lascoux
   complex associated to the matrix $\cA_D=(f_{ij,D})$ is exact at any
   $(X') \in \Spec(D)$.
\end{remark}

\begin{remark} \label{remthm61} \rm (1) We often use Theorem~\ref{Wsmooth}
  verifying\ $ _0\!\Ext^1_B(I_B/I_B^2,I_{A/B})=0$, but the corresponding
  injectivity assumption in Theorem~\ref{Wsmooth} is a priori weaker and, in
  fact, equivalent to the vanishing of the connecting map
  $ _0\!\Hom _R(I_B ,A) \to \ _0\!\Ext^1_B(I_B/I_B^2,I_{A/B})$, so equivalent to
  the exactness of
\begin{equation} \label{thm61cond1}
 0 \longrightarrow\ _0\!\Hom _R(I_B ,I_{A/B}) \longrightarrow\ _0\!\Hom
    _R(I_B ,B) \longrightarrow\ _0\!\Hom _R(I_B ,A) \longrightarrow 0 \ .
\end{equation}
Since the left-exactness always holds, the exactness may be verified by
Macaulay2 by computing dimensions of these $\Hom$-groups which may be faster
than computing $ _0\!\Ext^1_B(I_B/I_B^2,I_{A/B})$.

    (2) Suppose all assumptions of Theorem~\ref{Wsmooth}, except $\gamma=0$,
    hold. Then the map $pr_2$ in diagram \eqref{cartesiansquare} is surjective
    by Proposition~\ref{thmext3}. By Corollary~\ref{evdef} the map $pr_1$ in
    \eqref{cartesiansquare} is surjective if and only if $\gamma=0$.
    Hence we get from diagram \eqref{cartesiansquare} that one may verify the
    condition
    $\gamma = 0$ by showing
    \begin{equation} \label{thm61cond}\ _0\! \hom_R(I_{A},A) = \ _0\!
      \hom_R(I_{B},B)+\dim_k(MI \otimes A)_{(a_{t+c-1})}\ -\ _0\!
      \hom_R(I_{B},I_{A/B})\, ,
\end{equation}
and conversely that $\gamma=0$ implies \eqref{thm61cond}. Moreover $\gamma=0$
is further equivalent to
$\overline{ W(\underline{b};\underline{a};r)} \subset \Hilb ^{p_X(t)}(\PP^n)$
being a generically smooth irreducible component. Indeed we have one way by
Theorem~\ref{Wsmooth}, and conversely, we use the dimension formula of
Proposition~\ref{propo8} to see that \eqref{thm61cond} holds. Thus $\gamma=0$
is also equivalent to
$\dim \overline{ W(\underline{b};\underline{a};r)} =\ _0\! \hom_R(I_{A},A)$.
Finally we notice that if also the assumptions of Theorem~\ref{dimW} hold with
$\kappa=0$, then \eqref{thm61cond} is equivalent to
$\ _0\! \hom_R(I_{A},A) = \lambda_c$.

(3) Suppose all assumptions of Theorem~\ref{Wsmooth}, except $\gamma=0$, hold.
Then the a priori weaker injectivity assumption of (1) is expected to be
weaker only when $\dim A = 2$ and in the cases $(c,r)=(2,2)$ and $r=1$ because
$ _0\!\Ext^1_B(I_B/I_B^2,B)=0$ is expected for $\dim B \ge 4$, $c \ne 2$ by
Conjecture~\ref{conjdepthNb}(i) and Proposition~\ref{conj3ext}(i), see
\cite[Theorem 5.11]{K2014} for $r=1$. Moreover it is clear that
\eqref{thm61cond1} implies $\gamma=0$ by the definition of $\gamma$. And
conversely due to Remark~\ref{minormax} and (2) above we get that $\gamma=0$
implies \eqref{thm61cond1} provided $c \ge 4-r$. However, if $c = 3-r$ they
are not equivalent, even in the generic case as Example~\ref{rem74ex} shows.
\end{remark}

\begin{example} \label{rem74ex} \rm Here is an example in the case $c=3-r$
  where $ _0\!\Ext^1_B(I_B/I_B^2,I_{A/B})\ne 0$ and
  $ _0\!\Ext^1_B(I_B/I_B^2,B)= 0$, but $\gamma = 0$. Indeed let $A$ be the
   generic determinantal ring defined by the ideal of all
  $2 \times 2$ minors of the $3 \times 3$ matrix $\cA= (x_{i,j})$ and let
  $B=R/I_{2}(\cB)$ where $\cA= [\cB,v]$ and $R=k[x_{i,j}]$. In this case
  $ \ _0\! \hom_{R}(I_{B},B)=42$, $ \ _0\! \hom_{R}(I_{B},I_{A/B})=2$ and
  $ \ _0\! \hom_{R}(I_{A},A)=\lambda_{1}=64$ by Example~\ref{exgendet}(i) with
  $c=1$. Using Macaulay2 we get $ \ _0\! \hom_R(I_B,A)=48$ and
  $\dim_k(MI \otimes A)_{(1)}=24$, 
  whence $\gamma
  = 0$, i.e. \eqref{thm61cond} holds while \eqref{thm61cond1} is not true.
\end{example}

\begin{corollary} \label{corgendetW} Let $X=\Proj(A) \in W(0^t;1^{t+c-1};r)$,
  $1 \le r < t$, $(r,c) \ne (1,1)$ be a generic determinantal scheme defined
  by $s \times s$-minors of a $t\times (t+c-1)$ matrix $\cA=(x_{ij})$ of
  indeterminates of the polynomial ring $R:=k[x_{ij}]$, i.e. $A:=R/I_{s}(\cA)$
  with $s={t+1-r}$. Then $ \overline{W(0^t;1^{t+c-1};r)}$ is a generically
  smooth irreducible component of $\Hilb ^{p_X(t)}(\PP^n)$, $n:=t(t+c-1)-1$
  and every deformation of $A$ comes from deforming $\cA$.

  More generally taking $s \times s$-minors of a matrix $\cA=(x_{ij})$ of
  indeterminates belonging to a larger polynomial ring
  $R[\underline y]:=k[x_{ij},y_k]$, $1 \le k \le e$ and letting
  $A':=R[\underline y]/I_{s}(\cA)$ and $W(0^t;1^{t+c-1};r;R[\underline y])$ be
  the locus of all determinantal rings defined by $s \times s$-minors of
  linear $t\times (t+c-1)$ matrices with coefficients in $R[\underline y]$,
  then $ \overline{W(0^t;1^{t+c-1};r;R[\underline y])}$ is a generically
  smooth irreducible component of $\Hilb ^{p_X(t)}(\PP^{n+e})$ and every
  deformation of $A$ comes from deforming $\cA$.
   \end{corollary}

\begin{proof}
  This follows from Theorem~\ref{Wsmooth}, Remark~\ref{remthm61}(2),
  Corollary~\ref{evdef} and Proposition~\ref{deforminggenericcase} (or more
  directly by combining Lemma~\ref{unobst} and
  Proposition~\ref{deforminggenericcase}). The final statement follows from
  Lemma~\ref{unobst} and Corollary~\ref{gendetrem} because
  $(\Proj(A')) \in W(0^t;1^{t+c-1};r;R[\underline y])$.
 \end{proof}

  Since $ _0\!\Ext^1_B(I_B/I_B^2,I_{A/B}) \subset\ _0\!\Ext^1_R(I_B,I_{A/B})$
  we see that if the degree of all generators of $I_{A/B}$ is larger than the
  maximum of the degree of the relations, $mdr(I_B)$, of $I_B$ appearing in
  the Lascoux resolution, the mentioned \,$ _0\!\Ext^1$-groups vanish. Here we
  define $mdr(I)$ by $mdr(I)=\max\{n_{2,j}\}$ and $mdg(I)$ by
  $mdg(I)=\max\{n_{1,i}\}$ where
  $$ \longrightarrow \oplus_j R(-n_{2,j}) \longrightarrow \oplus_i R(-n_{1,i})
  \longrightarrow I \longrightarrow 0$$ is a minimal
  presentation of a graded ideal $I$. To apply Theorem~\ref{Wsmooth} it is of
  interest to compute both $mdg(I)$ and  $mdr(I)$. We have:

  \begin{lemma} \label{mdr_mdg} Let $\cA$ be a $t \times (t+c-1)$ matrix and
    let $A=R/I_{t+1-r}(\cA)$ be determinantal. We have:

\begin{itemize}
\item[{\rm (i)}] $mdg(I_{t+1-r}(\cA))=\sum ^{t+c-1}_{j=c+r-1}a_j -\sum _{i=1}^{t+1-r}b_i$.

\vskip 2mm

 \item[  {\rm (ii)}] $mdr(I_{t+1-r}(\cA))=\sum_{j=c+r-2}^{t+c-1}a_j-\sum_{i=1}^{t+2-r}b_i+m$ where $m=\max \{ b_{t+2-r}-b_1, a_{t+c-1}- a_{c+r-2} \}$.
 \end{itemize}
  \end{lemma}

  \begin{proof} (i) The generators
    $f_{i_1\cdots i_{t+1-r}}^{j_1\cdots j_{t+1-r}}$ of $I_{t+1-r}(\cA)$ are
    the determinant of the ${t\choose t+1-r}{t+c-1\choose t+1-r}$ minors of
    size $(t+1-r)\times (t+1-r)$ that we obtain choosing $t+1-r$ rows
    ($1\le i_1<i_2<\cdots <i_{t+1-r}\le t+c-1$) and $t+1-r$ columns
    ($1\le j_1<j_2<\cdots <j_{t+1-r}\le t$) of $\cA$.
    Therefore,
    $$mdg(I_{t+1-r}(\cA))=\sum ^{t+c-1}_{j=c+r-1}a_j -\sum
    _{i=1}^{t+1-r}b_i.$$

  (ii)  According to \cite{Ma} the first syzygies of $I_{t+1-r}(\cA)$ are given by one of the following constructions:
\begin{itemize}
\item[(a)] We choose $t+1-r$ rows $1\le i_1<i_2<\cdots <i_{t+1-r}\le t+c-1$ of
  $\cA$ and we construct a $(t+2-r) \times (t+c-1)$ matrix
  $\cA_{i_1\cdots i_{t+1-r}i_v}^{1\cdots t+c-1}$ using all these rows $i_v$
  and repeating one of the $i_v$. The determinant of the
  ${t+c-1\choose t-r+2}$ minors of size $(t+2-r)\times (t+2-r)$ of the matrix
  $\cA_{i_1\cdots i_{t+1-r}i_v}^{1\cdots t+c-1}$ gives us a syzygy of
  $I_{t+1-r}(\cA)$ .
\item[(b)] We choose $t+1-r$ columns $1\le j_1<j_2<\cdots <j_{t+1-r}\le t$ of
  $\cA$ and we construct a $t \times (t+2-r)$ matrix
  $\cA_{1\cdots t}^{j_1\cdots j_{t+1-r}j_v}$ using all these columns $j_v$ and
  repeating one of them. The determinant of the
  ${t\choose t-r+2}$ minors of size $(t+2-r)\times (t+2-r)$ of the matrix
  $\cA_{1\cdots t}^{j_1\cdots j_{t+1-r}j_v}$ gives us a syzygy of
  $I_{t+1-r}(\cA)$ .
\item[(c)] We consider a $(t+2-r)\times (t+2-r)$ minor
  $\cA _{i_1\cdots i_{t+2-r}}^{j_1\cdots j_{t+2-r}}$ of $\cA $ choosing
  $t+2-r$ rows ($1\le i_1<i_2<\cdots <i_{t+2-r}\le t+c-1$) and $t+2-r$ columns
  ($1\le j_1<j_2<\cdots <j_{t+2-r}\le t$) of $\cA$. Call
  $D _{i_1\cdots i_{t+2-r}}^{j_1\cdots j_{t+2-r}}(p,q)$ the relation that we
  get taking the difference of the expansion of the determinant of
  $\cA _{i_1\cdots i_{t+2-r}}^{j_1\cdots j_{t+2-r}}$ along the $p$-th row and
  along the $q$-th column. This also gives us a syzygy of $I_{t+1-r}(\cA)$.
\end{itemize}

From the above description of the first syzygies we immediately get that
$$mdr(I_{t+1-r}(\cA))=\sum_{j=c+r-2}^{t+c-1}a_j-\sum_{i=1}^{t+2-r}b_i +m
\text{ where } m=\max \{ b_{t+2-r}-b_1, a_{t+c-1}- a_{c+r-2} \}.$$
  \end{proof}

  Due to results in Section 8 (cf.
  Theorem~\ref{corWsmooth}) it suffices to compute $mdr(I_B)$ when $I_B$ is
  defined by submaximal (or maximal) minors of a homogeneous matrix $\cA$. In
  these particular cases we have

\begin{corollary} \label{maxrel}  Let  $\cA$ be a $t \times (t+c-1)$ matrix
  and let
  $A=R/I_{t+1-r}(\cA)$ be determinantal.
\begin{itemize}
\item[{\rm (i)}] If $r = 2$, $c \ge 1$, then
  $mdr(I_A)=\sum _{i=c}^{t+c-1}a_i-\sum_{i=1}^{t}b_{i}+m$ where
  $m=\max\{a_{t+c-1}-a_c,b_{t}-b_1\}$.
\item[{\rm (ii)}] If $c=3-r$ and $2 \le r \le t$ then $A$ is defined by
  submaximal
  minors and
\begin{equation}
mdr(I_A)=\sum _{i=1}^{t-r+2}a_i-\sum_{i=1}^{t-r+2}b_{i}+m \quad where \quad
m=\max\{a_{t-r+2}-a_1,b_{t-r+2}-b_1\}\,.
\end{equation}
\item[{\rm (iii)}] If $c=2-r$ and $2 \le r \le t-1$ then $A$ is defined by
  maximal
minors and $$mdr(I_A)= a_{t-r+1}+\sum _{i=1}^{t-r+1}a_i-\sum_{i=1}^{t-r+2}b_{i}.$$
\end{itemize}
\end{corollary}

\begin{proof} (i) It follows from Lemma \ref{mdr_mdg} (ii).

  (ii) If we transpose the matrix $\cA$ then we get a $(t-r+2) \times t$
  matrix which fits into the set-up of (i), but note that the inequalities
  $b_1 \le b_2\le \cdots \le b_t, a_1 \le a_2 \le \cdots \le a_{t-r+2}$
  attached to $\cA$ are reversed, i.e we have the inequalities
  $-b_1 \ge -b_2\ge \cdots \ge -b_t, -a_1 \ge -a_2 \ge \cdots \ge -a_{t-r+1}$,
  attached to $\cA^{tr}$. Thus if we for $\cA^{tr}$ also ``transpose'' both
  rows and columns or, more precisely, using primes (') for $\cA^{tr}$ we
  have $-b_i=a'_{t+1-i}$ and $-a_i=b'_{t-r+3-i}$. Then (i) implies
  $$mdr(I_A)=\sum _{i=1}^{t-r+2}(-b_i)-\sum_{i=1}^{t-r+2}(-a_{i})+m$$ where
  $m=\max\{-b_1-(-b_{t-r+2}),-a_1-(-a_{t-r+2})\}$.

  (iii) This is rather immediate to see using the Eagon-Northcott resolution.
\end{proof}

\begin{corollary}\label{corWsmooth3}
  Let $B=R/I_B\twoheadrightarrow A$ be as in (the first two sentences of)
  Theorem~\ref{Wsmooth}, 
  and suppose $2 \le r \le t-1$ and that every deformation of $B$ comes from
  deforming $\cB$.
  If \begin{itemize}
  \item[{\rm (i)}] $c=3-r$ and $a_{t-r+2} > 2a_{t-r+1} +
  \sum_{i=r}^{t}b_{i}-\sum_{i=1}^{t-r+2}b_{i}$,  or
  \item[{\rm (ii)}] $c=4-r$ and $a_{t-r+3} > \sum _{i=t-r+1}^{t-r+2}a_i+
  \sum_{i=r}^{t}b_{i}-\sum_{i=1}^{t-r+2}b_{i}+\max\{a_{t-r+2}-a_1,b_{t-r+2}-b_1\}$
  \end{itemize}
  then $ _0\!\Ext^i_R(I_B,I_{A/B})=0$ for $i=0,1$. In particular
  $ \overline{W(\underline{b};\underline{a};r)}$ is a generically smooth
  irreducible component of $\Hilb ^{p_X(t)}(\PP^n)$ and every deformation of $A$
  comes from deforming $\cA$. Furthermore
  $$\dim W(\underline{b};\underline{a};r) = \dim
  W(\underline{b};\underline{a'};r)+\dim_k(MI \otimes A)_{(a_{t+c-1})}\,.$$
\end{corollary}

\begin{proof} (i) It is easy to see that the smallest degree of the minimal
  generators of $I_{A/B}$ is
  $$ s(I_{A/B}): =a_{t-r+2}+\sum _{i=1}^{t-r}a_i-\sum _{i=r}^{t}b_{i}$$ while
  the maximum degree of the relations, $mdr(I_B)$, of $I_B$ is
  $$mdr(I_B)=\sum _{i=1}^{t-r+1}a_i-\sum_{i=1}^{t-r+2}b_{i} + a_{t-r+1}$$ by
  Corollary~\ref{maxrel}(iii). Indeed $B$ is defined by maximal minors. Thus by
  the definition of $ _0\!\Ext^1_R(I_B,I_{A/B})$ this $ _0\!\Ext^1_R$-group
  (as well as $ _0\!\Hom_R(I_B,I_{A/B})$) vanish if
  $ mdr(I_{B}) < s(I_{A/B})$, which is equivalent to assumption (i) and
  Theorem~\ref{Wsmooth} applies.

(ii) Now the smallest degree of a generator of $I_{A/B}$ is
$$ s(I_{A/B}): =a_{t-r+3}+\sum _{i=1}^{t-r}a_i-\sum _{i=r}^{t}b_{i}.$$ Since $B$
is defined by submaximal minors, we have by  Corollary~\ref{maxrel}(ii)  a formula
for the maximum degree of relations of $I_B$. Again it is clear
by definition of $ _0\!\Ext^1_R(I_B,I_{A/B})$ that this group vanishes if
$ mdr(I_{B}) < s(I_{A/B})$. Since this is equivalent to the assumption of
(ii), we get $ _0\!\Ext^i_R(I_B,I_{A/B})=0$ for $i=1$ and certainly also for
$i=0$, and we conclude the corollary by Theorem~\ref{Wsmooth}.
\end{proof}

\begin{example} \label{Wsmooth2} \rm (i) For any $r$, $2 \le r \le t-1$, let
  $\cA= [\cB,v]$ be a general $t \times (t-r+2)$ matrix with $\cB$ linear and
  $v$ a column of cubic entries, let $B=R/I_B\twoheadrightarrow A$ be defined
  by $(t-r+1)$-minors of $\cB$, resp. $\cA$, and suppose $R$ is large enough
  so that $\dim A \ge 3$. Since $B$ is defined by maximal minors and
  $\dim B \ge 4$ one know that every deformation of $B$ comes from deforming
  $\cB$ by Theorem~\ref{Amodulethm5}. Since the numerical conditions of
  Corollary~\ref{corWsmooth3}(i) are satisfied,
  $ \overline{W(\underline{b};\underline{a};r)}$ is a generically smooth
  irreducible component of $\Hilb ^{p_X(t)}(\PP^n)$ and every deformation of $A$
  comes from deforming its associated matrix $\cA$. Note that for $r=2$,
  $ \dim \overline{W(\underline{b};\underline{a};r)}$ is given by
  Theorem~\ref{dimW1} and that this case was considered in \cite{KM2009}.

\vskip 2mm
  (ii) Let $\cA= [\cB,w]$ be a general $t \times (t-r+3)$ where the
  $t \times (t-r+2)$ matrix $\cB$ is exactly the matrix $\cA$ in (i) above and
  $w$ a column whose entries are of degree $7$. Let
  $B=R/I_B\twoheadrightarrow A$ be defined by $(t-r+1)$-minors of $\cB$, resp.
  $\cA$, and suppose that $R$ is large enough so that $\dim A \ge 3$. Then
  every deformation of $B$ comes from deforming $\cB$ by (i) above and since
  $B$ is defined by submaximal minors and the numerical conditions of
  Corollary~\ref{corWsmooth3}(ii) are satisfied,
  $ \overline{W(\underline{b};\underline{a};r)}$ is a generically smooth
  irreducible component of $\Hilb ^{p_X(t)}(\PP^n)$ and every deformation of $A$
  comes from deforming $\cA$.
 \end{example}

 Examples seem to indicate that $ _0\!\Ext^1_B(I_B/I_B^2,I_{A/B})=0$ or more
 generally that \eqref{thm61cond} hold for $\dim A$ large enough, and
 further evidence to this observation is given in Section 8. We expect

 \begin{conj} \label{conjWsmooth} Let $r \ge 1$, $c \ge 2-r$,
   $(r,c) \ne (1,1)$ and let $X=\Proj(A) \in W(\underline{b};\underline{a};r)$
   be defined by the vanishing of the $(t-r+1)\times (t-r+1)$ minors of a
   general $t\times (t+c-1)$ matrix $\cA$ with $a_1 > b_t$ and suppose that
   $\dim A \ge 4$ for $c = 1$ and $\dim A\ge 3$ for $c \ne 1$. Then,
   $ \overline{W(\underline{b};\underline{a};r)}$ is a generically smooth
   irreducible component of $\Hilb ^{p_X(t)}(\PP^n)$ and every deformation of
   $A$ comes from deforming $\cA$.
  \end{conj}

  \begin{remark} \label{remconjWsmooth} \rm The conjecture is true for $r=1$
    and $c=2-r$ by Theorem~\ref{Amodulethm5} and Lemma~\ref{tr}, and for the
    component $\overline{W(0^t;1^{t+c-1};r)}$, $r \ge 1$ containing a generic
    determinantal scheme, whence for components
    $\overline{W(0^t;1^{t+c-1};r)}$ with $\dim R \ge t(t+c-1)$, by
    Corollary~\ref{corgendetW}. For $c \in \{0,2\}$ the conclusion of
    Conjecture~\ref{conjWsmooth} may even be true for $\dim A\ge 2$, as it is
    for $(r,c)=(1,2)$. We have considered many examples and used Macaulay2
    to check if \eqref{thm61cond} or $\gamma=0$ (or
    $ _0\!\Ext^1_B(I_B/I_B^2,I_{A/B})=0$) of Theorem~\ref{Wsmooth} hold, and
    it seems that only the Gorenstein case $(c=1)$ requires $\dim A \ge 4$
    while in all other cases, $\dim A \ge 3$ suffices. We now list some
    examples, mostly where the conclusion of Conjecture~\ref{conjWsmooth}
    fails.
\end{remark}

\begin{example} \rm \label{examples712} (i) Let
  $R=k[x_0, x_1, \cdots ,x_{9}]$, let $\cA=[\cB,v]$, $v$ a column, be a
  general $3 \times 5$ matrix with linear entries, and let $A$ and $B$ be the
  quotients of $R$ defined by the $2 \times 2$ minors of $\cA$ and $\cB$,
  respectively. So $t=3$, $r=2$, $c=3$ and $\dim A = 2$. A Macaulay2
  computation shows that $_0\! \hom_R(I_B,I_{A/B})=4$,
  $_0\! \hom_R(I_B,B)=96$, $_0\! \hom_R(I_A,A)=120$ and
  $\dim_k(MI \otimes A)_{(1)}=25$ and \eqref{thm61cond} is not satisfied. Let
  us check that every deformation of $B$ comes from deforming $\cB=[\cC,w]$.
  Indeed if $C$ is defined by the $2 \times 2$ minors of the $3 \times 3$
  matrix $\cC$, then every deformation of $C$ comes from deforming $\cC$ owing
  to Remark~\ref{remconjWsmooth} and $\dim R \ge 9$. Then
  $ _0\!\Ext^1_C(I_C/I_C^2,I_{B/C})=0$ is easily checked using Macaulay2.
  Applying Theorem~\ref{Wsmooth}, replacing $B \to A$ there by $C \to B$, we
  get that every deformation of $B$ comes from deforming $\cB$. It follows
  from Remark~\ref{remthm61}(2) that $ \overline{W(0^3;1^5;2)}$ is not a
  generically smooth irreducible component of $\Hilb ^{p_X(t)}(\PP^9)$.

\vskip 2mm
    (ii) If we try to treat example (i) by deleting a row, one may
    transpose the matrix and instead delete a column, i.e. we look at a
    general $5 \times 3$ matrix $\cA=[\cB,v]$ with linear entries from
    $R=k[x_0, x_1, \cdots ,x_{9}]$, $v$ a column, and we let $A$ and $B$ be
    the quotients of $R$ defined by their $2 \times 2$ minors, respectively.
    So $t=5$, $r=4$, $c=-1$ and $\dim A=2$. The only problem with this
    approach is that Theorem~\ref{Wsmooth} and its corollaries require
    $\dim A \ge 3$ when $c=-1$, i.e. they do not apply.

    \vskip 2mm (iii) Let $R=k[x_0, x_1, \cdots ,x_{n}]$ with $n=6$,
    $\cA=[\cB,v]$ a general linear $3 \times 3$ matrix, $v$ a column, and let
    $A$ and $B$ be the quotients of $R$ defined by their $2 \times 2$ minors.
    So $t=3$, $r=2$, $c=1$ and $\dim A = 3$. Using Macaulay2 we get that
    $_0\! \hom_R(I_B,I_{A/B})=2$ and $_0\! \hom_R(I_A,A)=46$. Hence
    $\dim W(0^3;1^3;2) = \lambda_1$ by Theorem~\ref{dimW}, and since
    $\lambda_1=46$, $ \overline{W(0^3;1^3;2)}$ is a generically smooth
    irreducible component of $\Hilb ^{p_X(t)}(\PP^6)$, cf.
    Remark~\ref{remthm61}(2). The ``same'' example with $n=7$ yields
    $_0\! \hom_R(I_B,I_{A/B})=2$ and $_0\! \hom_R(I_A,A)=55$, and arguing as
    above, we get that $ \overline{W(0^3;1^3;2)}$ is a generically smooth
    irreducible component of $\Hilb ^{p_X(t)}(\PP^7)$ of dimension
    $\lambda_1=55$. For $n \ge 8$ with $n=8$ for the generic case, we do not
    need to use Macaulay2 because it follows from Corollary~\ref{corgendetW}
    that $ \overline{W(0^3;1^3;2)}$ is a generically smooth irreducible
    component of $\Hilb ^{p_X(t)}(\PP^n)$.

    The ``same'' example with $n=5$ yields $_0\! \hom_R(I_B,I_{A/B})=3$, and
    $$ \dim \overline{W(0^3;1^3;2)}= \lambda_c-1=36\,,$$ as noticed in
    Example~\ref{afterdeg}(1). Since we have $_0\! \hom_R(I_A,A)=36$ by
    Macaulay2, $ \overline{W(0^3;1^3;2)} \subset \Hilb ^{p_X(t)}(\PP^5)$ is a
    generically smooth component of dimension $ \lambda_c-1$.

    \vskip 2mm (iv) We repeat (iii) above with one change, namely we let
    $\cA=[\cB,v]$ be a general $3 \times 3$ matrix, where $\cB$ is linear
    while all entries of the column $v$ are of degree $2$. This example was
    considered in Example~\ref{dimW2} where we computed
    $\dim \overline{W(0^3;1^2,2;2)}$ in terms of $n$. Using Macaulay2 we show
    $_0\! \hom_R(I_A,A)=94$ (resp. $71$) for $n=6$ (resp. $n=5$), which is
    different (for $n \le 6$ only) from the values of
    $ \dim \overline{W(0^3;1^2,2;2)}$ we found in Example~\ref{dimW2}. This
    shows that $ \overline{W(0^3;1^2,2;2)}$ is not a generically smooth
    irreducible component of $\Hilb ^{p_X(t)}(\PP^n)$. Note that $c=1$ and
    $\dim A=3$ (resp. $\dim A=2$) for $n=6$ (resp. $n=5$), cf.
    Conjecture~\ref{conjWsmooth} and Remark~\ref{remconjWsmooth}.

    \vskip 2mm (v) More ``submaximal minors'': Let
    $R=k[x_0, x_1, \cdots ,x_{n}]$ with $n=6$, let $\cA=[\cB,v]$ a general
    linear $4 \times 4$ matrix, $v$ a column, and let $A$ and $B$ be the
    quotients of $R$ defined by their $3 \times 3$ minors. So $t=4$, $r=2$,
    $c=1$ and $\dim A = 3$. A Macaulay2 computation shows that
    $_0\! \hom_R(I_B,I_{A/B})=3$, $_0\! \hom_R(I_B,B)=60$,
    $_0\! \hom_R(I_A,A)=88$ and $\dim_k(MI \otimes A)_{(1)}=24$ and since
    \eqref{thm61cond} is not satisfied, $ \overline{W(0^4;1^4;2)}$ is not a
    generically smooth irreducible component of $\Hilb ^{p_X(t)}(\PP^6)$ by
    Remark~\ref{remthm61}(2). Corresponding calculations of the ``same''
    example, only changing $n$ to $n=5$, so $\dim A=2$ yields
    $_0\! \hom_R(I_B,I_{A/B})=3$, $\dim \overline{W(0^4;1^4;2)}=65$ and
    $_0\! \hom_R(I_A,A)=80$ and since \eqref{thm61cond} is not satisfied,
    $ \overline{W(0^4;1^4;2)} \subset \Hilb ^{p_X(t)}(\PP^5)$ is not a
    generically smooth irreducible component, cf. \cite[Example 5.1]{KM2009}.

    \vskip 2mm (vi) Finally we consider the case $t=4$, $r=3$, $A$ Gorenstein
    with $\dim A = 3$ of ``subspecifical minors'', i.e. let
    $R=k[x_0, x_1, \cdots ,x_{11}]$ and $\cA=[\cB,v]$ a general linear
    $4 \times 4$ matrix, $v$ a column, and let $A$ and $B$ be the quotients of
    $R$ defined by the $2 \times 2$ minors of $\cA$ and $\cB$ respectively. By
    Example~\ref{afterdeg}(2), $\dim W(0^4;1^4;3) = \lambda_c=161$ while we have
    $_0\! \hom_R(I_A,A)=162$ by Macaulay2, whence
    $ \overline{W(0^4;1^4;4)} \subset \Hilb ^{p_X(t)}(\PP^{11})$ is not a
    generically smooth irreducible component.
\end{example}

\vskip 2mm
\begin{example} \label{qmdeg} \rm {\em Varieties of quasi-minimal degree. } It
  is a well known result in Algebraic Geometry that for any non-degenerate
  (reduced and irreducible) variety $X\subset \PP^n$ we have
  $\deg (X)\ge \codim X+1$. The classification of varieties of minimal degree
  is well understood (see, for instance, \cite{EH}) and as an attempt to
  classify varieties of quasi-minimal degree (i.e. $\deg(X)=\codim X+2$) Hoa
  described their minimal graded free resolution and proved (cf. \cite[Theorem
  1]{H}):

  Let $X\subset \PP^n$ be a non-degenerate ACM variety of $\dim X\ge 1$ and
  quasi-minimal degree, i.e. $\deg X=\codim X+2$. Then, $I(X)$ has a minimal
  free resolution of the following type:
    $$0\longrightarrow R(-c-2)\longrightarrow R(-c)^{\alpha
      _{c-1}}\longrightarrow\cdots \longrightarrow R(-2)^{\alpha
      _1}\longrightarrow R\longrightarrow R/I(X)\longrightarrow 0$$
    where $c=\codim(X)$ and
    $$\alpha _{i}=i{c+1\choose i+1}-{c\choose i-1} \text{ for } 1\le i\le c-1.$$
    In particular, $X$ is arithmetically Gorenstein and generated by hyperquadrics. To better
    understand the structure of $X$ we can ask whether $I(X)$ is generated by
    the $2\times 2$ minors of a $t\times t$ matrix with linear entries.
    Looking at the graded Betti numbers this is possible if and only if $t=3$.
     So, in these cases it is natural to ask whether any
    codimension 4 AG subscheme $X\subset \PP^n$ of degree 6
    can be defined by the $2\times 2$ minors of a $3\times 3$ matrix with
    linear entries.

    The answer is yes. Indeed, by Example \ref{examples712}(iii) we have
    $\dim W(0^3;1^3;2)=\dim_{(X)} \Hilb ^{p_X(t)}(\PP^n)$
    and we conclude that a general non-degenerate ACM variety
    $X\subset \PP^n$ of quasi-minimal degree is arithmetically
    Gorenstein, generated by 9 hyperquadrics and these hyperquadrics are the
    $2\times 2$ minors of a $3\times 3$ matrix with linear entries.
\end{example}

  An algebraic related conjecture is concerned with the depth of the following
  ``normal  modules'':
  \begin{conj} \label{conjdepthNb} Let $r \ge 1$, $c \ge 2-r$ and let
    $A=R/I_A$ (resp. $B=R/I_B$ if $c > 2-r$) be defined by the vanishing of
    the $(t-r+1)\times (t-r+1)$ minors of a general $t\times (t+c-1)$ matrix
    $\cA=[\cB,v]$ with $a_1 > b_t$, $v$ a column. Let $N_A:=\Hom_R(I_A,A)$ and
    suppose that $\dim A \ge 3$. Set $I_{A/B} = I_A/I_B$.
\begin{itemize}
\item[{\rm (i)}] If $c \notin \{0,1,2\}$ then the $A$-module $N_A$ satisfies
  \ $$\codepth(N_A)= 1\,.$$ If $c=1$ (resp. $c =0, 2$), then
  $\codepth(N_A)= \min\{ 2,\dim A - 2\} $ (resp. $\codepth(N_A)= 0$).\\[-4mm]
  \item[{\rm (ii)}] Let $(r,c)\ne (1,2)$. If $c\ge 2$ (resp. $3-r\le c\le 1$)
    then the $B$-module $\Hom_R(I_B,I_{A/B})$ satisfies \
    $$\codepth\,(\Hom_R(I_B,I_{A/B})) = r \ ({\it resp.}\ r+1)\,.$$
   \end{itemize}
    \end{conj}

    \begin{remark} \label{conjdepthNbdim2} \rm If $\dim A = 2$ and otherwise
      with notations and assumptions as in Conjecture~\ref{conjdepthNb}, it
      seems that we often have $\codepth \Hom_R(I_B,I_{A/B})=r$ and certainly
      $\codepth(N_A)=0$ because in general any $\Hom_A(-,A)$-group has
      $\depth$ at least $2$ if $\depth A \ge 2$.
\end{remark}

Conjecture~\ref{conjdepthNb}(i) is related to the smoothness
    of $\Hilb ^{p_X(t)}(\PP^n)$ along
    $ W(\underline{b};\underline{a};r)$ while
    Conjecture~\ref{conjdepthNb}(ii) is concerned with
    $ \overline{W(\underline{b};\underline{a};r)}$ being an irreducible
    component of
    $\Hilb ^{p_X(t)}(\PP^n)$ as well as to the property ``every deformation of
    $A$ comes from deforming its matrix''. Moreover, under the assumptions of
    Theorem~\ref{Wsmooth}, we also get generically smoothness of
    $ \overline{W(\underline{b};\underline{a};r)}$ from
    Conjecture~\ref{conjdepthNb}(ii), so that conjecture is really what we
    need in this paper.

    To see the connection between Conjecture~\ref{conjdepthNb} and
    Conjecture~\ref{conjWsmooth}, note the following

    \begin{proposition} \label{conj3ext}  Let $d$ and $e \ge -2$ be integers,
      let $A=R/I_A$ (resp. $B=R/I_B$) be defined by the vanishing of the
      $(t-r+1)\times (t-r+1)$ minors of a general $t\times (t+c-1)$ matrix
      $\cA=[\cB,v]$  with $a_1 > b_t$ (resp. $\cB$), $v$ a column, and let
      $J_A:= I_{t-r}(\varphi ^*)$ and $J_B:=I_{t-r}(\varphi _{t+c-2}^*)$.
\begin{itemize}
\item[{\rm (i)}] Let $\dim A \ge c+2r$. If \ $\codepth(N_A)=d$ then \
  $\Ext^i_A(I_A/I_A^2,A)=0$ for $1 \le i \le c+2r-d-2$. In particular if
  $d \le c+2r-3$ (i.e. $\depth_{J_A}A \ge d+3$),
  then $\Ext^1_A(I_A/I_A^2,A)=0$. Hence if
  Conjecture~\ref{conjdepthNb}{\rm (i)} holds and $\dim A \ge 4$ for
           $c \ne 1$ (resp. $\dim A \ge 5$ for $c = 1$) then
           $$\Ext^1_A(I_A/I_A^2,A)=0\,.$$
         \item[{\rm (ii)}] Let $\dim A \ge c+r-1$. If \
           $\codepth\,(\Hom_R(I_B,I_{A/B})) = r+e$, then
           $\Ext^i_B(I_B/I_B^2,I_{A/B})=0$ for $1 \le i \le c+r-e-3$. In
           particular, if $e \le c+r-4$ (i.e. $\depth_{J_B}A \ge e+3$) then
           $\Ext^1_B(I_B/I_B^2,I_{A/B})=0$. Hence if
           Conjecture~\ref{conjdepthNb}{\rm (ii)} holds and $c+r \ge 4$ for
           $c \ne 1$ (resp. $r \ge 4$ for $c = 1$) then
           $$\Ext^1_B(I_B/I_B^2,I_{A/B})=0\,.$$
           \end{itemize}
\end{proposition}

\begin{proof} (i) Let  $X=\Proj(A)$. Since
  $\cA$ is general, it is well known that $\depth_{J_A}A = c+2r$. It follows
  that $\depth_{J_A}N_A = c+2r-d$ by assumption. Since
  $U_A:=X\setminus V(J_A)$ is a local complete intersection, $\widetilde{N_A}$
  is locally free on $U_A$, and using \eqref{NM} we get natural isomorphisms
  \begin{equation*} \label{NMA} \Ext^{i}_{A}(I_A/I_A^2,A )\cong
    \Hl_{*}^{i}(U_A,{\mathcal H}om_{\odi{X}}(\widetilde{I_A/I_A^2},\odi{X}))
  \cong  \Hl_{J_A}^{i+1}(\Hom_{A}(I_A/I_A^2,A))
\end{equation*}
for $1 \le i \le c+2r-2$. Since the latter group vanishes for
$i+1 \le c+2r-d-1$, we get (i).

(ii) If   $Y=\Proj(B)$, then $\depth_{J_B}B = c-1+2r$.
Taking local cohomology $ \Hl_{J_B}^{i} (-)$ of the sequence
$$ 0 \longrightarrow I_{A/B} \longrightarrow B \longrightarrow A
\longrightarrow 0$$  it follows that  $\codepth(I_{A/B})=r-1$ which
implies $\depth_{J_B}(I_{A/B})=c+r$. Using \eqref{NM} we get a natural
isomorphism (resp. an injection)
\begin{equation*} \label{NMB} \Ext^{i}_{B}(I_B/I_B^2,I_{A/B} ) \hookrightarrow
  \Hl_{*}^{i}(U_B,{\mathcal
    H}om_{\odi{Y}}(\widetilde{I_B/I_B^2},\widetilde{I_{A/B}})) \cong
  \Hl_{J_B}^{i+1}(\Hom_{B}(I_B/I_B^2,I_{A/B}))
\end{equation*}
for $1 \le i \le c+r-2$ (resp. $i = c+r-1$). Since we by assumption have
$$\depth_{J_B}(\Hom_{B}(I_B/I_B^2,I_{A/B})) = c-1+2r-(r+e),$$ we get that the
latter group vanishes for $i+1 \le c-1 +r-e-1$, i.e. for
$1 \le i \le c+r-e-3$, cf. Remark~\ref{minormax}. Since the
other statements are straightforward we are done.
\end{proof}

\begin{proposition} \label{genericdetring} Let $A$ be a generic determinantal
  ring. Then Conjecture \ref{conjdepthNb}(i) holds. Moreover, Conjecture
  \ref{conjdepthNb}(ii) holds for $c\ge 0$, $c\ne 2$ while for
  $4-r\le c\le -1$, resp. $c=2$, we have
$$\codepth \Hom _R(I_B,I_{A/B})\le r+1 \text{ (resp. r)}.$$
\end{proposition}

\begin{proof} Conjecture \ref{conjdepthNb}(i) holds by \cite[Theorem
  15.10]{b-v}, see also Supplement to
  Theorem 15.10.

  To see that Conjecture \ref{conjdepthNb}(ii) almost holds for $c\ge 4-r$, we
  use Proposition \ref{vanish} (iv) which implies that
  $\Ext^1_A(I_{A/B}/I_{A/B}^2,A)=0$. Thus there exists an exact sequence:
 \begin{equation}
  \label{seqex1}
  0\longrightarrow \Hom _R(I_{A/B},A)\longrightarrow \Hom _R(I_A,A)\longrightarrow \Hom _R(I_B,A)\longrightarrow 0.
  \end{equation}
  Moreover, by Proposition \ref{deforminggenericcase} and Corollary
  \ref{evdef} it follows that $pr_1$ in the following diagram
\begin{equation}\label{cartesiansquare1}
\begin{array}{cccccc}
A^1_{(B\rightarrow A)} &  \stackrel{pr_2}{\longrightarrow}& _0\!\Hom _R(I_B ,B) \\
 \downarrow pr_1 & & \downarrow p \\
_0\!\Hom _R(I_A ,A) & \longrightarrow & _0\!\Hom _R(I_B ,A)
\end{array}
\end{equation}
is surjective. Hence, taking degree zero in (\ref{seqex1}), we get that
$p:\ _0\!\Hom _R(I_B ,B)\longrightarrow\ _0\!\Hom _R(I_B ,A)$ is surjective.
We can argue similarly for the surjectivity of the corresponding
$\tilde{p} :\Hom_R (I_B,B)\longrightarrow \Hom_R (I_B,A)$ using the entire
$\Hom$-groups. Indeed, the proof of Proposition \ref{deforminggenericcase}
shows that $\Ext^1_R(MI,MI)\longrightarrow \Hom_R (I_A,A)$, and not only the
degree zero part of this morphism, is surjective. Also the argument of
Corollary 5.9 holds for non-graded deformations. Thus there exists a diagram
similar to (\ref{cartesiansquare1}) where the lower index zero is removed and
$A^1_{(B\rightarrow A)}$ correspondingly redefined. That diagram implies that
$\tilde{p}$ is surjective. It follows that there is an exact sequence
 \begin{equation} \label{seqex3}
  0\longrightarrow \Hom _R(I_{B},I_{A/B})\longrightarrow
  \Hom _R(I_B,B)\stackrel{\tilde{p}}{\longrightarrow } \Hom
  _R(I_B,A)\longrightarrow 0.
  \end{equation}

  Using (\ref{seqex1}), (\ref{seqex3}), and say \cite[Corollary 18.6]{eise} or
  rather the mapping cone construction, we can quite closely determine the
  depth of $H_1:=\Hom _R(I_{B},I_{A/B})$ by first using (\ref{seqex1}) for
  finding the depth of $H_2:=\Hom _R(I_B,A)$. Indeed note that since $A$ is a
  generic determinantal ring, Conjecture \ref{conjdepthNb}(i) holds for the
  $A$-module $N_A:=\Hom_R (I_A,A)$, whence also applies for the $B$-module
  $N_B$, due to Corollary~\ref{gendetrem} and taking into account that
  $\dim B- \dim A=r$. Moreover, Proposition 4.5(iii) implies that
  $H_3:=\Hom _R(I_{A/B},A)$ is maximally CM for $c\ge 1$ and of $\codepth $ 1
  if $c\le 0$. Letting $\pd (H)$ being the length of an free $R$-resolution of
  $H$, we can find $\pd (H_1)$ by considering different cases of $c$ in the
  table:

$$
  \begin{tabular}{ c | c | c| c |c | c | c}
   c & $\pd (H_3)$ & $\pd (N_A)$ & $\pd (H_2)$ & $\pd (N_B)$ & $\pd (H_1)$ & $\codepth H_1$ \\
     \hline
    $\ge $ 3 & $\ell $ & $\ell +1$ & $\ell +1$ & $\le \ell -r+1$ & $\ell$ & $ r$ \\
     2 & $\ell $ &  $\ell $  & $\le \ell + 1$ & $\ell -r+2$ & $\le \ell $ & $\le r $ \\
     1 & $\ell $ & $\ell +2 $& $\ell +2$ &$ \ell -r $& $ \ell +1$ &$ r+1$ \\
     0 & $\ell +1$ &$ \ell $ & $\ell +2$ & $\ell -r+1$ & $ \ell +1$ &$ r+1$ \\
     $\le -1$ & $\ell +1$ &$ \ell +1$ & $\le \ell +2$ &$ \ell -r +1$& $\le \ell +1$ & $\le r+1 $
    \end{tabular}
$$

\vskip 2mm
\noindent
where $\ell :=\pd (R)$ and the column of $\pd (H_2)$, resp. $\pd (H_1)$, is determined by 2 columns to the left of $\pd (H_2)$, resp. $\pd (H_1)$. Then the codepth column is just obtained by transferring the $\pd (H_1)$ column to depth and then to codepth by using the Auslander-Buchsbaum formula, $\pd (H_1)+\depth H_1=\dim R$, and we are done.
\end{proof}

\begin{remark} \label{conjr=1} \rm Conjecture \ref{conjdepthNb}(i) and (ii)
  essentially holds for $r=1$, i.e. we have $\codepth N_A\le 1$ (resp. 0) for
  $c\ge 3$ (resp. $c=2$) by Theorem 3.11 and Theorem 3.9(iii) (resp. e.g.
  \cite[Corollary 3.7]{KM2015}). Moreover,
  $\codepth (\Hom _R(I_B,I_{A/B}))\le 1$ by \cite[Corollary 3.8 and Remark 3.6
  of latest arXiv version]{KM2015}, using these results for $B$ instead of
  $A$. We have for $r \ge 2$ considered several examples using Macaulay2 to
  check if the depth of the modules satisfies the conjectures. We list some of
  them below in Example \ref{exbcn}. The computations with Macaulay2 were
  time-consuming (or aborted), and we should have liked to check more
  examples, e.g. in the range where $\dim A \in \{3, 4\}$ and especially for
  Conjecture~\ref{conjdepthNb}(ii). 
\end{remark}
\begin{example} \label{exbcn}\rm (1) Submaximal minors, i.e. $r=2$. Let $t=3$.
  We have checked Conjecture~\ref{conjdepthNb}(i) and (ii) for the following
  determinantal rings $R/I_{t-r+1}(\cA)$: For every $c$, $1 \le c \le 4$ let
  $\cA = (x_{ij})$ by the generic $3 \times (c+2)$ matrix, or a non-linear
  matrix of the same size whose entries are powers of the corresponding
  entries of the generic one (so $R$ contains at least $3(c+2)$ variables and
  various such non-linear matrices are considered), and the conjectures hold.
  For $c=1$ we have also checked non-generic (for every $\dim A \in \{3,4\}$)
  as well as some non-linear determinantal rings and all satisfy
  Conjectures~\ref{conjdepthNb}(i) and (ii). The same pattern is valid for
  $c=2$ for non-generic linear determinantal rings with
  $\dim A \in \{3,4, 5\}$.

  \vskip 2mm (2) Let $\cA$ be the generic $t \times (t-r+2)$ matrix, i.e. $A$ a
  generic determinantal ring and let $t-r=1$. Then we have checked that
  Conjecture~\ref{conjdepthNb}(ii) holds for each $t \in \{4,5,6\}$ confirming
  Proposition~\ref{genericdetring} and that we have
  equality in the codepth formula there. 
  For $t=4$ we have checked non-generic linear determinantal rings for every
  $\dim A \in \{3,4\}$, and also some non-linear determinantal rings with
  $\dim A = 6$ and both conjectures hold.

  \vskip 2mm (3) Let $\cA$ be the generic $4 \times 4$ matrix, or some
  non-linear matrix whose entries are powers of the corresponding entries of
  the generic one (so $R$ contains at least $16$ variables and various such
  non-linear matrices are considered). Let $r=2$, i.e. $A$ is defined by
  submaximal minors. In this case Conjecture~\ref{conjdepthNb}(i) and (ii)
  holds. Then we consider the ``subsubmaximal case'' of the same $4 \times 4$
  matrices, i.e. we let $r=3$. Again both conjectures hold.
\end{example}

\begin{remark} \rm For a ``generic'' $5 \times 5$ matrix $\cA$, taking $r=3$
  we have checked using Macaulay2 that Conjecture~\ref{conjdepthNb}(ii) holds.
  In
  this case we have also got $ \Ext^1_B(I_B/I_B^2,I_{A/B})=0$ indicating that the
  case $c=1$ of Proposition~\ref{conj3ext}(ii) may be improved to $r \ge 3$.
\end{remark}

\section{Computing dimensions by deleting columns}

Since Theorem~\ref{Wsmooth} assumes $_0\!\Ext^1_B(I_B/I_B^2,I_{A/B})=0$ or the
related assumption $\gamma =0$ and Theorems~\ref{dimW} and ~\ref{Wsmooth} need
the numbers $\ _0\!\hom_B(I_B/I_B^2,I_{A/B})$ and
$\dim (MI\otimes A)_{(a_{t+c-1})}$ to find
$\dim W(\underline{b};\underline{a};r)$, the goal of this section is to
compute the dimension of $_0\!\Ext^i_B(I_B/I_B^2,I_{A/B})$ for $i=0,1$ and
$ (MI\otimes A)_{(a_{t+c-1})}$ more effectively. Indeed this $_0\!\Ext^0_B$\;-
(resp. $ (MI\otimes A)_{(a_{t+c-1})}$)-group is important because it is
isomorphic to the tangent space of the fiber of the projection
$$p_1: \Hilb ^{p_X(t),p_Y(t)} (\PP^{n}) \longrightarrow \Hilb ^{p_X(t)}(\PP^n)
\quad ({\rm resp.}\ p_2: \Hilb ^{p_X(t),p_Y(t)} (\PP^{n}) \longrightarrow \Hilb ^{p_Y(t)}(\PP^n))$$
 at
$(B \longrightarrow A)$, cf. \eqref{flagdef} while the vanishing of
$_0\!\Ext^1_B(I_B/I_B^2,I_{A/B})$ not only gives the smoothness of this fiber,
but in fact the smoothness of $p_1$ at $(B \longrightarrow A)$.

To find $\dim\, _0\!\Ext^i_B(I_B/I_B^2,I_{A/B})$ for $i=0,1$, we consider a flag
of determinantal rings:
\begin{equation} \label{fullflag} A_{2-r}\twoheadrightarrow \cdots
  \twoheadrightarrow A_0\twoheadrightarrow A_1\twoheadrightarrow
  A_2\twoheadrightarrow A_3\twoheadrightarrow \cdots \twoheadrightarrow A_c
  =A, \quad B=A_{c-1}, \quad X_{i}=\Proj(A_{i})
\end{equation}
with corresponding cokernels $$N_{2-r} \longrightarrow \cdots \longrightarrow N_{j}\longrightarrow \cdots \longrightarrow N_c=MI$$ of $\varphi ^*_{t+j-1}$, obtained by successively deleting $c+r-2$ columns from the right-hand side of
the $t \times (t+c-1)$ matrix $\cA$, and letting the $(t-r+1)\times (t-r+1)$
minors $I_{t-r+1}(\varphi ^*_{t+i-1})$ define $A_i$. Then e.g. $A_1$ (resp.
$A_{2-r}$) is Gorenstein (resp. standard determinantal) defined by the
$(t-r+1)\times (t-r+1)$ minors of a $t\times t$ (resp. $t\times (t-r+1)$)
matrix. Now recall that when we in Proposition \ref{vanish}(iii) proved
$$\Hom _B(M,A)\cong MI\otimes A\cong \Hom(I_{A/B}(a_{t+c-1}),A)$$
under some assumptions on $\depth _{J_i}A$ where
$J_i=I_{t-r}(\varphi ^*_{t+i-1})$, we first remarked that we have
\begin{equation}\label{MA}
\widetilde{M\otimes A}_{|U}\cong\widetilde{I_{A/B}(a_{t+c-1})\otimes A}_{|U},
\text{ and}
\end{equation}
\begin{equation}\label{NA}
\widetilde{N\otimes A}_{|U}\cong\widetilde{MI\otimes A}_{|U}
\end{equation}
where $U:=\Proj(A)\setminus V(J_{c-1}A)$. These isomorphisms are also needed for
describing the fiber of $pr_1:A^1_{(B \to A)}\longrightarrow\ _0\!\Hom_R(I_A ,A)$. Define
\begin{equation}
M_i:=\Hom_{A_{i}}(N_{i}\otimes _RA_{i}, A_{i}) \text{ where }
N_{i}:=\coker (\varphi ^*_{t+i-1}).
\end{equation}
 Let $I_{i}=I_{A_{i+1}/A_{i}}$ be the
ideal defining $A_{i+1}$ in $A_{i}$ and let
$I_{A_{i}}=I_{t-r+1}(\varphi ^*_{t+i-1})$. Assuming
$b_t<a_1$  and $\cA $ general; and using \cite[Theorem 2.7 and Example 2.5]{KM2011}, we obtain
\begin{equation}\label{aux5}
\dim R/(J_i+I_{A_{c-1}})=\dim A_{c-1}-r-i \quad {\rm for} \quad 2-r\le i \le c-2 .
\end{equation}
Thus, we have
$$\codim _{A_{c-1}} R/(J_i+I_{A_{c-1}})=\depth _{J_{i}A_{c-1}}A_{c-1}=r+i
\quad \text{ for } \quad 2-r\le i\le c-2$$
and $$\codim _{A_c} R/(J_i+I_{A_{c}})=\depth _{J_{i}A_{c}}A_{c}=r+i$$ for the
same reason (provided $\dim A_c \ge r+i$). Letting
$U_i=X_i\setminus V(J_iA_i)$ and assuming $\dim A_c \ge 2$, we get
$$
\Hom _{A_{i}}(I_i,A_j)\cong \Hom _{{\mathcal O}_{U_{i}}}(\widetilde{I_i},\widetilde{A_j})\quad \text{ for }\quad c-1\le j\le c.
$$
Then the left exact sequence
$$
0 \longrightarrow \Hom _{A_i} (I_i,I_{c-1})\longrightarrow \Hom  _{A_i} (I_i,A_{c-1})\longrightarrow \Hom _{A_i} (I_i,A_c)
$$
and the corresponding one for the global $\Hom $ of sheaves imply
\begin{equation} \label{59}
\Hom _{A_i} (I_i,I_{c-1}) \cong \Hom _{{\mathcal O}_{U_{i}}}(\widetilde{I_i},\widetilde{I_{c-1} })\quad \text{ for }\quad 2-r\le i\le c-1,
\end{equation}
and we get
$\Hom(M_i,I_{c-1})\cong \Hom _{{\mathcal
    O}_{U_{i}}}(\widetilde{M_i},\widetilde{I_{c-1} }) $
by the same argument. Continuing these left exact sequences by including
$\Ext^1(-,-)$ to the right and using the five-lemma we get that
\begin{equation} \label{510} \Ext^1 _{A_i}(M_i,I_{c-1}) \cong \Ext^1
  _{{\mathcal O}_{U_{i}}}(\widetilde{M_i},\widetilde{I_{c-1} })\quad \text{
    for }\quad 2-r < i\le c-1,
\end{equation}
and correspondingly for $\Ext^2(-,-)$ if $3-r < i\le c-1$. This leads to

\begin{proposition} \label{fibdim2} Let $\cA $ be general and suppose
  $a_1>b_t$, $r \ge 2$ and $\dim A \ge 2$. If $c \le 0$ suppose also
  $\dim A \ge 3$. Let $j$ be any integer satisfying $2-r \le j \le c-2$. Then
  $\Hom _{A_j}(I_j,I_ {c-1})$ is an $A_{c-1}$-module of codepth $r-1$ and we
  have isomorphisms
  $$\Hom _{A_j}(M_j,I_{c-1}) \cong \Hom
  _{A_{c-1}}(M_{c-1},I_{c-1}) \, , \quad {\rm and} \ $$
  $$\begin{array}{rcl} \Hom _{A_j}(M_j(-a_{t+j}),I_{c-1}) &  \cong & \Hom _{A_j}(I_j,I_{c-1}) \\
  & \cong &
  \Hom _{A_{c-2}}(I_{c-2},I_{c-1})(a_{t+j}-a_{t+c-2}) \, .\end{array} $$
  Moreover,
  $ _v\! \hom_R(M_{j},I_{c-1}) = \dim (N_{c-1}\otimes A_{c-1})_v - \dim
  (N_{c}\otimes A_{c})_v$ for any integer $v$, and if \ $ 2-r < i\le c-2$, then
 $$\Ext^1_{A_i}(M_i,I_ {c-1})=\Ext^1_{A_{c-1}}(M_{c-1},I_ {c-1})=0 \quad and \quad
 \Ext^1_{A_{i+1}}(I_i/I_i^2,I_ {c-1})=0\, .$$
Furthermore, we have
 $$_{\mu}\! \Hom _{A_j}(I_j,I_{c-1})\cong R(a_{t+j}-a_{t+c-1})_{\mu} \ \ \text{
   provided } \ \ a_{t+j}-a_{t+c-1}+ \mu <s_r-b_r-a_{t-r+1}+b_1.$$
\end{proposition}

\begin{proof} Using (\ref{MA}) - (\ref{59}), letting $U_i=X_i\setminus V(J_iA_i)$, we have isomorphisms
$$\begin{array}{rcll}
\Hom _{A_j}(I_j,I_ {c-1}) & \cong & \Hom _{{\mathcal O}_{U_j}}(\widetilde{I_j}_{|U_j},\widetilde{I_{c-1}}_{|U_j}) & \text{ (by } (\ref{59}) \text{)} \\
 & \cong & \Hom _{{\mathcal O}_{U_j}}(\widetilde{M_j}(-a_{t+j}),\widetilde{I_{c-1}})) & \text{ (by (\ref{MA}))  } \\
 & \cong & \Hom _{{\mathcal O}_{U_{j}}}({\mathcal
           Hom}(\widetilde{N_{j}},\widetilde{A_{j}}) \otimes
           \widetilde{A_{c-1}},\widetilde{I_{c-1}}(a_{t+j})) & (I_{c-1} \
                                                               \text{an} \ A_{c-1}\text{-module}) \\
 & \cong & \Hom _{{\mathcal O}_{U_{c-1}}}(\widetilde{M_{c-1}},\widetilde{I_{c-1}}(a_{t+j})) & (\text{by (\ref{NA})) }\\
& \cong &  \Hom _{A_{c-1}}(M_{c-1},I_{c-1})(a_{t+j}) & \text{ (because } \depth_{J_{c-1}}I_{c-1}\ge 2 \text{) }
\end{array}
$$
and also
$ \Hom _{A_j}(M_j,I_ {c-1}) \cong \Hom _{{\mathcal
    O}_{U_{j}}}(\widetilde{M_{j}},\widetilde{I_{c-1}}) \cong \Hom
_{A_{c-1}}(M_{c-1},I_{c-1})$.
Then the exact sequence
$$0\longrightarrow I_{c-1}\longrightarrow A_{c-1} \longrightarrow A\longrightarrow 0$$ and
Proposition \ref{vanish}(iii) yield the exact sequence
$$
0 \longrightarrow \Hom _{A_{c-1}}(M_{c-1},I_{c-1}) \longrightarrow  \Hom _{A_{c-1}}(M_{c-1},A_{c-1})\cong N_{c-1}\otimes A_{c-1}$$ $$ \longrightarrow \Hom _{A_{c-1}}(M_{c-1},A)\cong N_{c}\otimes A_{c}.
$$
By Proposition \ref{vanish}(i), $ N_{c}\otimes A_{c}$ is a maximal
Cohen-Macaulay $A$-module for $c > 0$ and of codepth $1$ for $c \le 0$, and
similarly for $ N_{c-1}\otimes A_{c-1}$, and since
$N_{c-1}\otimes A_{c-1} \stackrel{p}{\twoheadrightarrow} N_{c}\otimes A_{c}$
is surjective, we get the formula for $ _v\! \hom_R(M_{j},I_{c-1})$ and the
codepth of $A_{c-1}$-module $\Hom _{A_{c-1}}(M_{c-1},I_{c-1})$.
Moreover, the dimension of $\ker p$ in degree
$v+a_{t+c-1}$ is determined in Lemma \ref{dimMI}. Indeed, we have
$$R_v\cong \Hom _{A_{c-1}}(M_{c-1},I_{c-1})_{a_{t+c-1}+v}\cong \Hom
_{A_{j}}(I_{j},I_{c-1})(a_{t+c-1}-a_{t+j})_{v}$$
provided $v < s_r-b_r-a_{t-r+1}+b_1$, and we get the last statement of
Proposition \ref{fibdim2}.

Finally let $2-r < j \le c-1$. Then we have $\depth_{J_{j}A_{c-1}}A_{c-1}-1 \ge 3$
by \eqref{aux5}. Applying $H^0_*(U_j,-)$ to the exact sequence
$$0 \longrightarrow {\mathcal Hom} _{{\mathcal
    O_U}_{j}}(\widetilde{M_{j}},\widetilde{I_{c-1}}) \longrightarrow
\widetilde{N_{c-1}}\otimes \widetilde{A_{c-1}}_{|U_j} \longrightarrow
\widetilde{N_{c}}\otimes \widetilde{A_{c}}_{|U_j} \longrightarrow 0$$
and noticing that
$H^0_*(U_j, \widetilde{ N_{i}}\otimes \widetilde{ A_{i}}) \cong { N_{i}}
\otimes {A_{i}}$
for $i=c-1,c$ and that
$p:N_{c-1}\otimes A_{c-1} \longrightarrow N_{c}\otimes A_{c}$ is surjective,
we get an injection
$$
H^1_*(U_j,{\mathcal Hom} _{{\mathcal
    O_U}_{j}}(\widetilde{M_{j}},\widetilde{I_{c-1}})) \hookrightarrow
H^1_*(U_j,\widetilde{N_{c-1}}\otimes \widetilde{A_{c-1}}) \cong H^2
_{J_{j}A_{c-1}}( N_{c-1} \otimes {A_{c-1}}).$$
If $c > 1$ (resp. $j > 3-r$) the latter group vanishes because
$$\depth _{J_{j}A_{c-1}}( N_{c-1} \otimes {A_{c-1}})=\depth
_{J_{j}A_{c-1}}A_{c-1}\ge 3$$
$$({\rm resp.} \quad
\depth _{J_{j}A_{c-1}}( N_{c-1} \otimes {A_{c-1}}) \ge \depth
_{J_{j}A_{c-1}}A_{c-1}-1 \ge 4-1= 3$$
by \eqref{aux5}, noting that $\dim A_{c-1} = \dim A_c+r \ge 4$). Recalling that $ \widetilde{M_{j}}$ is a locally free
${{\mathcal O_U}_{j}}$-Module, we get $\Ext^1_{A_j}(M_j,I_ {c-1})=0$ by
\eqref{510}. For $j\le c-2$ we also get
$\Ext^1_{A_{j+1}}(I_j/I_j^2,I_ {c-1})=0$ because
$I_j \otimes A_{j+1} \cong I_j/I_j^2$ is locally free over
$U_j \cap \Proj(A_{j+1})$ by \eqref{MA}.

It remains to consider the case $c \le 1$ and $j=3-r \le c-1$ where we
unfortunately only have
$\depth _{J_{j}A_{c-1}}(N_{c-1} \otimes {A_{c-1}}) \ge 2$.
If $j \le c-2$, we can, however, apply \eqref{aux5} onto the larger ideal
$J_{j+1}A_{c-1} \supset J_{j}A_{c-1}$. Indeed, we get
$\depth _{J_{j+1}A_{c-1}}(N_{c-1} \otimes {A_{c-1}}) \ge 3$ by \eqref{aux5}
which implies
$ H^1_*(U_{j+1},\widetilde{N_{c-1}}\otimes \widetilde{A_{c-1}}) = 0$.
Observing that
${\mathcal Hom}_{{\mathcal
    O_U}_{j+1}}(\widetilde{M_{j+1}},\widetilde{A_{c-1}}) \cong
\widetilde{N_{c-1}}\otimes \widetilde{A_{c-1}}$
and that $\widetilde{M_{j+1}}$ is a locally free over $U_{j+1}$, whence the
depth of $ \widetilde{N_{c-1}}\otimes \widetilde{A_{c-1}}$ and
$\widetilde{A_{c-1}}$ coincide at every point of $U_{j+1}$, we get
$$ H^1_*(U_{j},\widetilde{N_{c-1}}\otimes \widetilde{A_{c-1}}) \cong
H^1_*(U_{j+1},\widetilde{N_{c-1}}\otimes \widetilde{A_{c-1}}) = 0\,.$$
It follows that
$H^1_*(U_j,{\mathcal Hom} _{{\mathcal
    O_U}_{j}}(\widetilde{M_{j}},\widetilde{I_{c-1}})) =0$
and we get $$\Ext^1_{A_j}(M_j,I_ {c-1})=\Ext^1_{A_{j+1}}(I_j/I_j^2,I_ {c-1})=0$$
by \eqref{510} and \eqref{MA} as previously. Since, in the special case
$j=3-r = c-1$, we directly get $\Ext^1_{A_j}(M_j,I_ {c-1})=0$ from Proposition \ref{vanish}(iii), we are done.
\end{proof}

\begin{remark} \label{remprop81} \rm Suppose $\dim A \ge 3$, and if
  $c \le 0$ we also suppose $\dim A \ge 4$. Then using \eqref{aux5}, we can
  mainly argue as in the last part of the proof above to get, for
  $ 3-r < i\le c-2$, that
 $$\Ext^2_{A_i}(M_i,I_ {c-1})=\Ext^2_{A_{c-1}}(M_{c-1},I_ {c-1})=0 \quad {\rm and} \quad
 \Ext^2_{A_{i+1}}(I_i/I_i^2,I_ {c-1})=0\, .$$
\end{remark}

\begin{corollary} \label{fiberpr2} Let $\cA $ be general, let $c \ge 3-r$ and
  suppose $a_1>b_t$, $r \ge 2$, $\dim A \ge 2$ and
  $a_{t-r+1}+ a_{t+c-2}-a_{t+c-1} < s_r-b_r+b_1$. If $c \le 0$ suppose also
  $\dim A \ge 3$. Then it holds
 \begin{itemize} \item[{\rm (i)}] If $c > 3-r$,  then $$  _0\!
   \hom_R(I_{A_{c-1}},I_{c-1})= \ _0\!
  \hom_R(I_{A_{3-r}},I_{c-1})+\sum _{j=t-r+3}^{t+c-2} {a_j-a_{t+c-1}+n\choose
    n}$$  and $$
 _0\! \hom_R(I_{A_{3-r}},I_{c-1}) \le \ _0\!
 \hom_R(I_{A_{2-r}},I_{c-1})+{a_{t-r+2}-a_{t+c-1}+n\choose n}. $$
 Moreover, for $3-r \le j \le c-2$, we have
 $$ \ _0\! \ext^1_{A_{j+1}}(I_{A_{j+1}}/I_{A_{j+1}}^2,I_{c-1}) \le \ _0\!
 \ext^1_{A_j}(I_{A_{j}}/I_{A_{j}}^2,I_{c-1}).$$
\item[{\rm (ii)}] If $ \ _0\! \hom_R(I_{A_{i}},I_{c-1})=0$ for  some $i$, $2-r
< i \le c-1$, or  $_0\! \hom_R(I_{A_{2-r}},I_{c-1}) =0$ and $a_{t-r+1}
< a_{t+c-1}$, then
 $$_0\! \hom_R(I_{A_{3-r}},I_{c-1})= {a_{t-r+2}-a_{t+c-1}+n\choose n}=\sum
 _{j=1}^{t-r+2}{a_j-a_{t+c-1}+n\choose n},$$
 whence
 $$ _0\! \hom_R(I_{A_{c-1}},I_{c-1})=\sum
 _{j=1}^{t+c-2}{a_j-a_{t+c-1}+n\choose n} \, .$$
 \item[{\rm (iii)}] If
 $a_{t-r+1} < a_{t+c-1} -\sum _{i=1}^{t-r+1}b_{r+i-1}+\sum
 _{i=1}^{t-r+1}b_{i}$
 (e.g. $a_{t-r+1} < a_{t+c-1}$ if $b_1=b_t$), then
 $_0\! \hom_R(I_{A_{2-r}},I_{c-1}) =0$ and we have
$$ _0\! \hom_R(I_{A_{c-1}},I_{c-1})=\sum
 _{j=1}^{t+c-2}{a_j-a_{t+c-1}+n\choose n} \, .$$
\end{itemize}
\end{corollary}

\begin{proof} (i) Since by Proposition \ref{fibdim2},
  $\Ext^1_{A_{j+1}}(I_{{j}}/I_{{j}}^2,I_{c-1})=0$
  we get an exact sequence
\begin{equation}\label{aux*}
0\longrightarrow\ _0\! \Hom(I_j,I_{c-1})\longrightarrow\ _0\! \Hom(I_{A_{j+1}},I_{c-1})\longrightarrow\ _0\! \Hom(I_{A_{j}},I_{c-1})\longrightarrow  0
\end{equation}
induced by $$0\longrightarrow I_{A_j}\longrightarrow I_{A_{j+1}}\longrightarrow
I_j\longrightarrow 0.$$
Indeed a long exact sequence of algebra cohomology, often called the
Jacobi-Zariski sequence, implies (\ref{aux*}), because we can continue to the
right the left-exact part of (\ref{aux*}) by
 $$ \longrightarrow \ _0\! \Ext^1_{A_{j+1}}(I_{{j}}/I_{{j}}^2,I_{c-1}) \longrightarrow \ _0\!
 \Ext^1_{A_{j+1}}(I_{A_{j+1}}/I_{A_{j+1}}^2,I_{c-1}) \ \longrightarrow \ _0\!
 \Ext^1_{A_j}(I_{A_{j}}/I_{A_{j}}^2,I_{c-1}) \longrightarrow \,.$$
 Note that e.g. (\ref{aux5}) shows that the algebra cohomology groups are just
 these $\Ext^1$-groups above because their $\Hom(-,I_{c-1})$ terms in a well
 known spectral sequence relating algebra cohomology to algebra homology vanish. 
 Since $\Ext^1_{A_{j+1}}(I_{{j}}/I_{{j}}^2,I_{c-1})=0$, we also get the
 inequality
 $$ \ _0\! \ext^1_{A_{j+1}}(I_{A_{j+1}}/I_{A_{j+1}}^2,I_{c-1}) \le \ _0\!
 \ext^1_{A_j}(I_{A_{j}}/I_{A_{j}}^2,I_{c-1}) \text{ for } 3-r \le j \le c-2.$$
  From (\ref{aux*}) and Proposition
 \ref{fibdim2} which implies
 $\dim\!\, _0\! \Hom(I_j,I_{c-1})={a_{t+j}-a_{t+c-1}+n\choose n}$ for
 $2-r\le j\le c-2$, we get
$$_0\! \hom_R(I_{A_{j+1}},I_{c-1})=\ _0\! \hom_R(I_{A_{j}},I_{c-1})+{a_{t+j}-a_{t+c-1}+n\choose n} \text{ for } 2-r<j<c-1$$
and hence the equality of (i). Finally since $\Hom(-,-)$ is
left-exact, \eqref{aux*} holds also for $j=2-r$ except for the surjectivity to the right. Since
Proposition \ref{fibdim2} holds for $j$ in the range $2-r \le j \le c-2$ we
get
$$_0\! \hom_R(I_{A_{3-r}},I_{c-1}) \le \ _0\! \hom_R(I_{A_{2-r}},I_{c-1}) +
{a_{t-r+2}-a_{t+c-1}+n\choose n}$$ and (i) is proved.

(ii) From the last inequality we get that $a_{t-r+1} < a_{t+c-1}$ and
$_0\! \hom_R(I_{A_{2-r}},I_{c-1})=0$ imply
$$ _0\! \hom_R(I_{A_{3-r}},I_{c-1}) \le {a_{t-r+2}-a_{t+c-1}+n\choose n} = \sum
_{j=1}^{t-r+2}{a_j-a_{t+c-1}+n\choose n}$$
while the left-exactness of $\Hom(-,-)$ in \eqref{aux*} for $j=2-r$ implies
equality. Moreover, for every $j$ such that $2-r < j < i$, the exact sequence
\eqref{aux*} shows that if $ _0\! \hom_R(I_{A_{j+1}},I_{c-1})=0$ then we have
$_0\! \hom_R(I_{A_{j}},I_{c-1})=0$ as well as \
$0=\ _0\! \hom_R(I_j,I_{c-1})={a_{t+j}-a_{t+c-1}+n\choose n},$ i.e.
$a_{t+j} < a_{t+c-1}$ and hence $a_{t-r+2} < a_{t+c-1}$. Using this repeatedly
we get $_0\! \hom_R(I_{A_{3-r}},I_{c-1})=0$ and we are done since
$ \sum
_{j=1}^{t-r+2}{a_j-a_{t+c-1}+n\choose n}=0$.

(iii) By (ii) it suffices to verify that \
$_0\! \hom_R(I_{A_{2-r}},I_{c-1}) =0$ and $a_{t-r+1} < a_{t+c-1}$. Since
$ -\sum _{i=1}^{t-r+1}b_{r+i-1}+\sum _{i=1}^{t-r+1}b_{i} \le 0$, the
inequality follows from the assumption of (iii). To see that
$_0\! \hom_R(I_{A_{2-r}},I_{c-1}) =0$, we look at the degree of the minimal
generators of $I_{A_{2-r}}$. By Proposition~\ref{mdr_mdg} the largest degree
of the minimal generators of $I_{A_{2-r}}$ is
$$ mdg(I_{A_{2-r}}):= \sum _{j=1}^{t-r+1}a_j-\sum _{i=1}^{t-r+1}b_{i},$$ while the smallest degree of a generator of $I_{c-1}$ is just
$$s(I_{c-1}): =a_{t+c-1}+\sum _{i=1}^{t-r}a_i-\sum _{i=1}^{t-r+1}b_{r+i-1}.$$
Hence $ mdg(I_{A_{2-r}}) < s(I_{c-1})$, i.e. the assumption of (iii), implies
$ \ _0\! \hom_R(I_{A_{2-r}},I_{c-1})=0$.
\end{proof}

\begin{remark} \rm \label{remcor83}
  {\rm (1)} The general conditions of Corollary~\ref{fiberpr2} are quite
  weak (and is not needed for sequence \eqref{aux*}, nor for the inequality of
  $ \ext^1$ in (i)). Indeed writing the
  condition on $s_r$ as $$b_r-b_1 < (s_r -a_{t-r+1})+ (a_{t+c-1}-a_{t+c-2})$$
  and noting that the right hand side is $ \ge t-r$, due to
  $ s_r-a_{t-r+1}=\sum _{i=1}^{t-r}(a_i-b_{r+i})$ and $a_1>b_t$, we see that the
  condition is at least satisfied if $ b_r-b_1\le t-r-1$.

  (2) Under the assumptions in Remark \ref{remprop81}, most importantly
assuming $j> 3-r$, we have
 $$ \ _0\! \ext^1_{A_{j+1}}(I_{A_{j+1}}/I_{A_{j+1}}^2,I_{c-1}) = \ _0\!
 \ext^1_{A_j}(I_{A_{j}}/I_{A_{j}}^2,I_{c-1}) \, .$$
 So if we in Corollary~\ref{fiberpr2}(i) increase the dimension assumptions by
 $1$ it is only for $j=3-r$ where strict inequality in
 $$ \ _0\! \ext^1_{A_{4-r}}(I_{A_{4-r}}/I_{A_{4-r}}^2,I_{c-1}) \le \ _0\!
 \ext^1_{A_{3-r}}(I_{A_{3-r}}/I_{A_{3-r}}^2,I_{c-1})\, $$
 may occur. To give an example where this happens, let $A=R/I_A$ be defined by
 the $2 \times 2$ minors of a general linear $3 \times 5$ matrix
 $\cA= [\cB,v]$, let $B=R/I_B (=A_{4-r})$, resp. $C=R/I_C (=A_{3-r})$, be
 correspondingly defined by $\cB$, resp. $\cC$, where $\cB = [\cC,w]$ is a
 $3 \times 4$ matrix and $w$ a column. Set $I_{B/C}:=I_B/I_C$. In this case we
 have used Macaulay2 to check that $\Ext^1_{C}(I_C/I_C^2,I_{A/B}) \ne 0$ and
 even $_0\! \Ext^1_{C}(I_C/I_C^2,I_{A/B}) \ne 0$, while we get
 $$\Ext^1_{C}(I_C/I_C^2,I_{B/C})= \Ext^1_{B}(I_B/I_B^2,I_{A/B})= 0\, ,$$
cf. Proposition~\ref{conj3ext}(ii). Deleting one more column of $\cC$ we get
 $\cA_{2-r}$, and letting $I_D:=I_{2}(\cA_{2-r})$ then computations
 show \ $_0\! \Ext^1_{D}(I_{D}/I_{D}^2,I_{j})\ne 0$ for
 $2-r \le j \le 4-r$ ($r=2$), see Remark~\ref{remthm61}(3).
\end{remark}

To give further evidence to Conjecture~\ref{conjdimW} we remark that
Corollary~\ref{fiberpr2} reduces the computation of
$ _0\! \hom_R(I_{A_{c-1}},I_{c-1})$ to computing
$ _0\! \hom_R(I_{A_{3-r}},I_{c-1})$. Using Macaulay2 this at least allows a much
faster verification of the conditions of Theorem~\ref{dimW} which
imply Conjecture~\ref{conjdimW}.

\begin{example} \label{exgendetcor83} \rm  In Example~\ref{exgendet} we
  considered many generic determinantal schemes $X=\Proj(A)$ for which
  Conjecture~\ref{conjdimW} holds. We now extend the result to cover many
  more cases by only verifying
 \begin{equation} \label{exgenassump}
   \ _0\! \hom_R(I_{A_{3-r}},I_{c-1})=\sum _{j=1}^{t-r+2}
  {a_j-a_{t+c-1}+n\choose n} = t-r+2
\end{equation}
   because then Corollary~\ref{fiberpr2}(i) implies that $ _0\!
   \hom_R(I_{A_{c-1}},I_{c-1})=  \sum _{j=1}^{t+c-2} {a_j-a_{t+c-1}+n\choose n}$.
   Hence we can argue as in Examples~\ref{exgendet} to get
   $\dim W(\underline{b};\underline{a};r) = \lambda_c$ and
   $\dim W(\underline{b};\underline{a};r;R') = \lambda_c(R')$ for
   $R'=R[\underline y]$, only using Corollary~\ref{cordimRdetW}(ii). In
   (i) of Examples~\ref{exgendet} we have considered the additional cases
   $8 \le c \le 18$ of $2 \times 2$ minors of a generic $3 \times (c+2)$
   matrix, in (ii) the cases $4 \le c \le 5$ of $3 \times 3$ minors of the
   generic $4 \times (c+3)$ matrix, in (iii) the cases $5 \le c \le 9$ of
   $2 \times 2$ minors of the generic $4 \times (c+3)$ matrix and in (iv) the
   cases of $2 \times 2$ minors of the generic $5 \times (c+4)$ matrix for
   $c \in \{2,3\}$, as well as the case of $3 \times 3$ minors of the generic
   $5 \times 6$ matrix, and everyone satisfies \eqref{exgenassump}, whence
   Conjecture~\ref{conjdimW} holds in all these cases.
\end{example}

To describe the fiber at $(B \longrightarrow A)$ of the other projection,
$p_2: \Hilb ^{p_X(t),p_Y(t)} (\PP^{n})\longrightarrow \Hilb ^{p_Y(t)}(\PP^n) $ with tangent map
$pr_2: A^1_{(B\rightarrow A)} \longrightarrow\ _0\!\Hom _R(I_B ,B)$, cf. diagram \eqref{flagdef}, we consider the flag
$$ A_{2-r}\twoheadrightarrow \cdots
  \twoheadrightarrow A_0\twoheadrightarrow A_1\twoheadrightarrow
  A_2\twoheadrightarrow A_3\twoheadrightarrow \cdots \twoheadrightarrow A_c
  =A, \quad B=A_{c-1}, \quad X_{i}=\Proj(A_{i})
$$
 of
determinantal rings
obtained by successively deleting columns from the right-hand side of a
general matrix $\cA$ with $a_1 > b_t$.
As usual $$N_{2-r} \longrightarrow \cdots \longrightarrow N_{j}\longrightarrow
\cdots \longrightarrow N_c=MI$$ is the corresponding sequence of cokernels.
Applying Lemma~\ref{dimMI}(ii) onto
$A_j \longrightarrow A_{j+1}$ we get
$$ \dim (N_{j+1} \otimes A_{j+1})_{a_{t+j}+v} = \dim (N_{j} \otimes
A_{j})_{a_{t+j}+v} - \dim R_v$$ provided $ v < s_r-b_r-a_{t-r+1}
+b_1$
where $ s_r-a_{t-r+1}:=\sum _{i=1}^{t-r}(a_i-b_{r+i})$. Letting
$v=a_{t+k} -a_{t+j}$ where $j \le k\le c-1$ and assuming
$ a_{t+c-1} < s_r-b_r+a_{t+j}-a_{t-r+1} +b_1$ we obtain
\begin{equation} \label{dimMIAflag}
 \dim (N_{j+1} \otimes A_{j+1})_{a_{t+k}} = \dim
  (N_{j} \otimes A_{j})_{a_{t+k}} -
  {a_{t+k}-a_{t+j}+n\choose n}\
\end{equation}
because $a_{t+k} \le a_{t+c-1}$. We also know that
$$0 \longrightarrow D_j(-a_{t+j}) \longrightarrow N_j \longrightarrow N_{j+1} \longrightarrow 0$$ is exact where
$D_j=R/I_t(\varphi^*_{t+j-1})$ for $j>0$ (e.g. see the text after (3.1) of
\cite{KMMNP}) and $D_j=R$ for $j \le 0$ (because $0 \longrightarrow G_{t+j-1}^*
   \longrightarrow  F^*   \longrightarrow N_{j} \longrightarrow 0$ is exact for $j \le 1$). Since
   $I_t(\varphi^*_{t+j-1})_v =0$ for $v < s_1-b_1$ we get
\begin{equation} \label{dimMImaxFlag} \dim (N_{j+1})_{a_{t+k}} = \dim (N_{j}
  )_{a_{t+k}} - {a_{t+k}-a_{t+j}+n\choose n}\
\end{equation}
provided $a_{t+c-1} < s_1-b_1$. Hence assuming
$ a_{t+c-1} < s_r-b_r+a_{t+j}-a_{t-r+1} +b_1$,  both \eqref{dimMIAflag}
and \eqref{dimMImaxFlag} holds by \eqref{Ki0}. 
Note that if we apply Lemma~\ref{dimMI}(i) onto $N_j$ and $A_j$ and assume
$ a_{t+c-1} < s_r-b_r+b_1$ we get
$\dim (N_{j} \otimes A_{j})_{a_{t+c-1}} = \dim (N_{j})_{a_{t+c-1}}$, as
previously. But the approach above using (ii) instead of (i) in
Lemma~\ref{dimMI}, leads to better results if we are able to find $R$-free
resolutions of $N_{j} \otimes A_{j}$. For $r=2$ and $c=1$ this is somehow done
in Theorem~\ref{dimW1} where $N_0 \cong I_B(s-a_t)$ and
$$\dim (N_{0})_{a_{t+c-1}}-\dim (N_{0} \otimes B)_{a_{t+c-1}}=\dim
I_B^2(s-a_t+a_{t+c-1})_0\, ,$$
which one may compute using the free resolution of $I_B^2$ given in
\eqref{ImultI}. More generally for $c=3-r$, $r \ge 2$ we have by
\eqref{dimMIAflag} that
\begin{equation}  \label{NoA}
\dim (N_{3-r} \otimes A_{3-r})_{a_{t-r+2}} = \dim (N_{2-r} \otimes
A_{2-r})_{a_{t-r+2}} - 1
\end{equation}
provided $b_r-b_1 < s_r-a_{t-r+1}$.
Hence a method to compute $\dim (N_{2-r} \otimes A_{2-r})_{a_{t-r+2}}$ is needed.
\begin{lemma} \label{kapp} Let $r \ge 2$, $B_i:= \coker \varphi _{t+i-1}$,
  $J_{A_i}:=I_{t-r}(\varphi _{t+i-1}^*) \ne R$ and suppose
  $\depth _{J_{A_i}}A_i \ge 1$ ($\cA$ not necessarily general). Then for
  $i \le 1$ the following sequence of maximal Cohen-Macaulay $A_i$-modules
 \begin{equation*}
   0  \longrightarrow \Hom_{A_i}(B_i \otimes A_i,A_i) \longrightarrow
   G_{t+i-1}^* \otimes A_i
   \longrightarrow  F^* \otimes A_i  \longrightarrow   N_i \otimes A_i
   \longrightarrow 0
\end{equation*}
is exact. Moreover, if $i=2-r$  then $A_i$ is defined by
maximal minors in which case we have the following minimal $R$-free resolution
of $(B_i\otimes A_i)^*:=\Hom_{A_i}(B_i \otimes A_i,A_i)$:
$$ \  0 \longrightarrow S_{t-1}( G_{t+i-1}^*)\longrightarrow
F^* \otimes S_{t-2}( G_{t+i-1}^*)\longrightarrow \ldots $$
$$\longrightarrow \wedge ^{t-2}
F^* \otimes G_{t+i-1}^* \longrightarrow \wedge ^{t-1} F^* \longrightarrow (B_i\otimes A_i)^*(\ell_{2-r}) \longrightarrow 0 \,
.$$
\end{lemma}
\begin{proof} Note that Proposition~\ref{vanish} applies to
  $N_i=\coker \varphi _{t+i-1}^*$ as well as to
  $B_i:= \coker \varphi _{t+i-1}$. In the latter case we get, for $i \le 1$,
  that $B_i\otimes A_i$ and $\Hom_{A_i}(B_i \otimes A_i,A_i)$ are maximal
  Cohen-Macaulay $A_i$-modules. 
  Moreover applying $\Hom_{A_i}(-,A_i)$ onto
  $$ \longrightarrow F \otimes A_i \longrightarrow G_{t+i-1} \otimes A_i
  \longrightarrow B_i \otimes A_i \longrightarrow 0$$ we
  get the
  leftmost part of the first exact sequence while applying $(-) \otimes A_i$
onto $$ \longrightarrow G_{t+i-1}^* \longrightarrow F^* \longrightarrow N_i
\longrightarrow 0$$ takes care of the rightmost part.
  Finally in the case of maximal minors, the minimal $R$-free resolution of
  $B_{2-r}^*$ is well known (\cite[pg. 595]{eise}), and that we may put it on
  the form above.
\end{proof}

By applying Theorem~\ref{Wsmooth} to every surjection
$A_{i-1}\twoheadrightarrow A_{i}$, $i > 2-r$ in the flag \eqref{fullflag} and
using Corollary~\ref{fiberpr2} which implies that
$ _0\!\Ext^1_B(I_B/I_B^2,I_{A/B})=0$ provided
$ _0\!\Ext^1_{A_{3-r}}(I_{A_{3-r}}/I_{A_{3-r}}^2,I_{c-1})=0$ where
$I_{j}=I_{A_{j+1}/A_j}$, $B=A_{c-1}$ and $A=A_c$, we are able to prove the
main results of this section. First we consider the case $c=3-r$, $B=A_{2-r}$
and $A=A_{3-r}$ which we need to start the induction and is of interest in
itself. Here Lemma~\ref{kapp} allows us to compute the invariants involved in
(ii).

\begin{theorem}\label{corWsmoothc}  Suppose  that $\Proj(A) \in
  W(\underline{b};\underline{a};r)$
  is general with $c = 3-r$, $r \ge 2$ and $\dim A \ge 2$. If $c \le 0$, we
  also suppose $\dim A\ge 3$. Moreover, let $\gamma$ be the composed map
  $\gamma:\ _0\!\Hom_{R}(I_{A_{}},A_{})\longrightarrow\ _0\!\Hom_R(I_{B},A_{})
  \longrightarrow\ _0\!\Ext^1_{B}(I_{B}/I_{B}^2,I_{A/B})$
  and suppose $a_1 > b_t$. \begin{itemize}
  \item[{\rm (i)}] If\ \ $\gamma =0$ or equivalently if \eqref{thm61cond} holds, then
  $ \overline{W(\underline{b};\underline{a};r)}$ is a generically smooth
  irreducible component of $\Hilb ^{p_X(t)}(\PP^n)$ and every deformation of $A$
  comes from deforming its matrix $\cA$.
\item[{\rm (ii)}] Let $MI=\coker \varphi^*_{t-r+2}$ and
    $N=\coker \varphi^*_{t-r+1}$. If
      $_0\! \hom_R(I_{B},I_{A/B})=\sum _{i=1}^{t-r+1} {a_i-a_{t-r+2}+n\choose
        n}$
      (e.g. $b_t = b_1$ and $a_{t-r+1} < a_{t-r+2}$), then
   $$\dim W(\underline{b};\underline{a};r) = \lambda_{c} +K'_3+K'_4+\cdots
   +K'_r -\kappa' $$
   where $K'_i$ is defined in \eqref{lamdaprime} and $\kappa' \ge 0$ is given
   by:
   $$\kappa'= \dim_k(MI)_{(a_{t-r+2})}- \dim_k(MI \otimes A)_{(a_{t-r+2})}.$$ In
   particular, if also $a_{t-r+2} < s_r-b_r+b_1$ then every $K'_i=0$,
   $\kappa' =0$ and $$\dim W(\underline{b};\underline{a};r) = \lambda_{c}\, .$$
   More generally, if $a_{t-r+1} < s_r-b_r+b_1$ or, equivalently,
   $b_r-b_1<\sum _{i=1}^{t-r}(a_i-b_{r+i})$ then $K'_i=0$,
   $\kappa'= \dim_k(N)_{(a_{t-r+2})}- \dim_k(N \otimes B)_{(a_{t-r+2})}$ and
   $\dim W(\underline{b};\underline{a};r) = \lambda_{c}-\kappa'$ where
   $\kappa'$ may be expressed in terms of binomials using Lemma~\ref{kapp}
   for $i=2-r$, i.e. with $N=N_{2-r}$ and $B = A_{2-r}$.
   \end{itemize}
    \end{theorem}

    \begin{proof} Let $\cB$ be the matrix of $\varphi _{t+1-r}^*$, i.e. the
      matrix whose $(t+1-r)$-minors define $B$. Since we know that every
      deformation of $B$ comes from deforming $\cB$ by
      Theorem~\ref{Amodulethm5} and $I_{t-r}(\varphi^*) \ne R$ by $a_1 > b_t$,
      we get (i) from Theorem~\ref{Wsmooth} and Remark~\ref{remthm61}.
      Theorem~\ref{Wsmooth} also implies that
 $$\dim W(\underline{b};\underline{a};r) = \dim
 W(\underline{b};\underline{a'};r)+\dim_k(MI \otimes A)_{(a_{t+2-r})}\ -\ _0\!
 \hom_R(I_{B},I_{A/B})$$
 where $MI=\coker(\varphi^*_{t+2-r})$,
 $\varphi _{t+2-r}=\varphi$. By \eqref{dimensioMI} and assumption we get
 $$\dim_k(MI \otimes A)_{(a_{t+2-r})}\ - \ _0\!
 \hom_R(I_{B},I_{A/B}) = -\kappa' + \lambda_{3-r}-\lambda_{2-r}$$
 while Lemma~\ref{tr} implies
 $$\dim W(\underline{b};\underline{a'};r) = \lambda_{2-r} +K'_3+K'_4+\cdots
 +K'_r,$$ and we get the displayed dimension formula.

 Finally if $a_{t-r+2} < s_r-b_r+b_1$, we get $\kappa'=0$ by
 Lemma~\ref{dimMI}(i), and every $K'_i=0$ by Lemma~\ref{tr} because
 $-b_1 < \ell'_2 = \sum_{j=1}^{t-r+1}a_j-\sum_{k=r-1}^tb_k=s_r-b_r-b_{r-1}$
 and $a_{t-r+2} \ge b_{r-1}$. Also the assumption $a_{t-r+1} < s_r-b_r+b_1$
 implies that every
 $K'_i=0$. Moreover since the map $\varphi^*_{t+i-1}$ is injective for
   $i \le 1$ we get an exact sequence
   $$0 \longrightarrow B(-a_{t+2-r}) \longrightarrow N_{2-r} \longrightarrow
   N_{3-r} \longrightarrow 0$$ where $ N_{3-r}=MI$,
   $ N_{2-r}=N$, by the snake lemma. Combining with \eqref{NoA}, we see that
   $$\kappa'= \dim_k(N)_{(a_{t-r+2})}- \dim_k(N \otimes B)_{(a_{t-r+2})}$$ and
   we conclude the theorem by using Lemma~\ref{kapp}.
\end{proof}

\begin{remark} \rm If $r=2$ in Theorem~\ref{corWsmoothc}, we get
  $\dim W(\underline{b};\underline{a};r) = \lambda_{3-r} -\kappa'$. In this
  case we proved in Theorem~\ref{dimW1} that
  $\dim W(\underline{b};\underline{a};r) = \lambda_{1} -\kappa_1$ under
  weaker assumptions. It follows that $\kappa'= \kappa_1$. In particular we
  have a very explicit formula for $\kappa'$ given in Theorem~\ref{dimW1}
  when $r=2$ and $c=1$.
\end{remark}

\begin{example} \label{dimW3}  \rm (Determinantal quotients of
   $R=k[x_0, x_1, \cdots ,x_n]$, using  Theorem~\ref{corWsmoothc}(ii))

   Let $\cA= [\cB,v]$ be a general $4 \times 3$ matrix with linear (resp.
   quadratic) entries in the first and second (resp. third) column. The degree
   matrix of $\cA$ is
   $\left(\begin{smallmatrix}1 & 1  & 2 \\
       1 & 1 & 2 \\ 1 & 1 & 2 \\ 1 & 1 & 2 \end{smallmatrix}\right)$
   and $b_i=0$ for $1 \le i \le 4$. The vanishing of all $2 \times 2$ minors
   of $\cA$ (resp. $\cB$) defines a determinantal ring $A$ (resp. $B$) with
   $r=3$, $t=4$, $c=0$ (resp. $c=-1$). Our goal is to find
   $\dim \overline{ W(0^4;1^2,2;2)}$. Since the condition $a_{t-1} < s_3 = 2$
   of Theorem~\ref{dimW} does not hold, we will use the generalization in
   Theorem~\ref{corWsmoothc}(ii) to find the dimension. Indeed since
   $a_{t-2} < a_{t-1}$ and $a_{t-2} < s_3 = 2$ we have
   $\dim \overline{ W(0^4;1^2,2;2)}= \lambda_0-\kappa'$ by
   Theorem~\ref{corWsmoothc}(ii) where
   $ \kappa'= \dim_k(N)_{(a_{3})}- \dim_k(N \otimes B)_{(a_{3})}$ and
   $N= \coker (G_{2}^* \hookrightarrow F^*)$, $G_{2}^*=R(-1)^2$ and $F^*=R^4$.
   Moreover by Lemma~\ref{kapp} the sequences
   $$0 \longrightarrow (B_{-1} \otimes B)^* \longrightarrow G_{2}^* \otimes B \longrightarrow F^* \otimes B \longrightarrow N
   \otimes B \longrightarrow 0$$ and
$$ \  0 \longrightarrow S_{3}( G_{2}^*)\longrightarrow
F^* \otimes S_{2}( G_{2}^*) \longrightarrow \wedge ^{2} F^* \otimes G_{2}^* \longrightarrow
\wedge ^{3} F^* \longrightarrow (B_{-1}\otimes B)^*(2) \longrightarrow 0
$$
are exact. The first of these sequences together with the definition of $N$
imply that $ \kappa'=24-\dim (B_{-1}\otimes B)^*_2$, taking into account that
$I_B$ has 6 minimal generators of degree $2$. Then the next displayed sequence
implies that $\dim (B_{-1}\otimes B)^*_2=4$. It follows that
$$\begin{array}{rcl} \dim \overline{ W(0^4;1^2,2;2)} & = & \lambda_0-\kappa' \\
& = & 8(n+1) + 4 {n+2\choose
  2}-16 -(5+2(n+1))+1 - \kappa' \\
  & = &  2n^2+12n-30\end{array}
  $$
  for $n \ge 8$ by definition of $ \lambda_0$. To check the answer using
  Macaulay2, let $\cB$ be the generic linear matrix with entries
  $x_0,x_1,...,x_{7}$ and let $v^{tr}=(x_{8}^2,x_{9}^2,x_{10}^2,x_{11}^2)$.
  Computations show that $\ \Ext_A^1(I_A/I_A^2,A) = 0$ and
  $\dim_{(X)} \Hilb^{p_X(t)}(\PP^{11}) =344$ at $X:=\Proj(A)$, coinciding with
  our formula for $ \dim \overline{ W(0^4;1^2,2;2)}$ when $n=11$. It also
  implies that $\overline{ W(0^4;1^2,2;2)}$ is a generically smooth
  irreducible component of $ \Hilb^{p_X(t)}(\PP^{11})$. We also checked the
  case $n=8$ where $\dim A=3$ by using Macaulay2, and we have got that
  \eqref{thm61cond} holds with $_0\! \hom_R(I_{A},A)=194$. Thus
  $\overline{ W(0^4;1^2,2;2)}$ is a generically smooth irreducible component
  of $ \Hilb^{p_X(t)}(\PP^{8})$ of dimension $194$, coinciding with our
  formula for $\dim \overline{ W(0^4;1^2,2;2)}$ and confirming
  Conjectures~\ref{conjdimW} and \ref{conjWsmooth} in this case.
 \end{example}

 To state our next result, let us index the compositions, previously called
 just $\gamma$,
 as follows: \\[-4mm]

 $\quad \gamma_{3,2}: \ _0\!\Hom_{R}(I_{A_{3-r}},A_{3-r})\longrightarrow\
 _0\!\Hom_R(I_{A_{2-r}},A_{3-r}) \longrightarrow \
 _0\!\Ext^1_{A_{2-r}}(I_{A_{2-r}}/I_{A_{2-r}}^2,I_{2-r})$ ,  and
  $$\gamma_{4,3}: \ _0\!\Hom_{R}(I_{A_{4-r}},A_{4-r})\longrightarrow\
 _0\!\Hom_R(I_{A_{3-r}},A_{4-r}) \longrightarrow \
 _0\!\Ext^1_{A_{3-r}}(I_{A_{3-r}}/I_{A_{3-r}}^2,I_{3-r})\quad  {\rm etc.}$$

\begin{theorem}\label{corWsmooth}  Suppose  that $\Proj(A) \in
  W(\underline{b};\underline{a};r)$
  is general with $c \ge 4-r$, $r \ge 2$ and $\dim A \ge 2$. If $c \le 0$, we
  also suppose $\dim A\ge 3$. Moreover suppose that $a_1 > b_t$ and that the
  composed maps $ \gamma_{3,2}$ and $\gamma_{4,3}$ are both zero. If\ \
  $ _0\!\Ext^1_{A_{4-r}}(I_{A_{4-r}}/I_{A_{4-r}}^2,I_{j})=0$ for
  $4-r \le j \le c-1$, then $ \overline{W(\underline{b};\underline{a};r)}$ is
  a generically smooth irreducible component of $\Hilb ^{p_X(t)}(\PP^n)$ and every
  deformation of $A$ comes from deforming its matrix.
     \end{theorem}

     \begin{proof}
       We will use induction on $c \ge 4-r$. In the initial case $c=4-r$ we
       know that every deformation of $A_{3-r}$ comes from deforming its
       matrix by Theorem~\ref{corWsmoothc}. Then we get the same conclusion
       for $A_{4-r}$, and also that
       $ \overline{W(\underline{b};\underline{a};r)}$ is a generically smooth
       irreducible component of $\Hilb ^{p_X(t)}(\PP^n)$, by
       Theorem~\ref{Wsmooth}.

       If $c>4-r$, we have by induction that every deformation of $A_{c-1}$
       comes from deforming its matrix. Since, for $3-r \le j \le c-2$, we have
       $$ \ _0\! \ext^1_{A_{j+1}}(I_{A_{j+1}}/I_{A_{j+1}}^2,I_{c-1}) \le \ _0\!
       \ext^1_{A_j}(I_{A_{j}}/I_{A_{j}}^2,I_{c-1})$$
       by Corollary~\ref{fiberpr2}(i) and Remark~\ref{remcor83} we get
       $ \ _0\! \Ext^1_{A_{c-1}}(I_{A_{c-1}}/I_{A_{c-1}}^2,I_{c-1})=0$ from
       the assumption
       $ _0\!\Ext^1_{A_{4-r}}(I_{A_{4-r}}/I_{A_{4-r}}^2,I_{c-1})=0$. Then we
       conclude the proof by Theorem~\ref{Wsmooth}.
\end{proof}

\begin{remark} \label{remthm89} \rm (1) An analogues result with the same
  conclusion, where we assume $ \gamma_{32}=0$ and
  $ _0\!\Ext^1_{A_{3-r}}(I_{A_{3-r}}/I_{A_{3-r}}^2,I_{j})=0$ for
  $3-r \le j \le c-1$, is true. The reason for not stating this result is seen
  in Remark~\ref{remcor83}(2) which implies
  $$ _0\!\Ext^1_{A_{4-r}}(I_{A_{4-r}}/I_{A_{4-r}}^2,I_{c-1})=\
  _0\!\Ext^1_{A_{c-1}}(I_{A_{c-1}}/I_{A_{c-1}}^2,I_{c-1})$$
  for $\dim A > 3$. Since we expect the latter group to vanish by
  Proposition~\ref{conj3ext}(2), it follows that
  $ _0\!\Ext^1_{A_{4-r}}(I_{A_{4-r}}/I_{A_{4-r}}^2,I_{c-1})$ should also
  vanish while $ _0\!\Ext^1_{A_{3-r}}(I_{A_{3-r}}/I_{A_{3-r}}^2,I_{j})$ for
  $j > 3-r$ and $ _0\!\Ext^1_{A_{2-r}}(I_{A_{2-r}}/I_{A_{2-r}}^2,I_{2-r})$ may
  be non-vanishing as Remark~\ref{remcor83}(2) shows. The compositions
  $ \gamma_{3,2}$ and $\gamma_{4,3}$ seem, however, to
  vanish. 
  All this is supported by Macaulay2 computations. Moreover the vanishing
  above is indeed a main reason for expecting Conjecture~\ref{conjWsmooth},
  and to a certain degree Conjecture~\ref{conjdepthNb} to be true.

  (2) Let $B=A_{c-1}$ and $A=A_{c}$ and suppose $\dim A \ge 3$, and
  $\dim A\ge 4$ in case $c \le 0$. If $A$ is general and $c \ge 4-r$ then by
  Proposition~\ref{vanish}(iv), $\Ext_A^1(I_{A/B}/I_{A/B}^2,A)=0$, whence
  $ _0\!\Hom_R(I_{A},A) \longrightarrow \ _0\!\Hom_R(I_{B},A)$ is surjective
  by Remark~\ref{rem38} and \eqref{ZJ}. It follows that the assumption
  $\gamma_{c+r,c+r-1}=0$ for $c \ge 4-r$ is equivalent to the surjectivity of
  $ _0\!\Hom_R(I_{B},B) \longrightarrow \ _0\!\Hom_R(I_{B},A)$ by
  Remark~\ref{remthm42}, cf. Remark~\ref{remthm61}(3). In particular this
  applies to $\gamma_{43}$ of Theorem~\ref{corWsmooth}, but not to
  $ \gamma_{32}$.
\end{remark}

 In the same way as the flag \eqref{flag} allows us to simplify the
 calculation of $\dim\, _0\!\Ext^i_B(I_B/I_B^2,I_{A/B})$ for $i=0,1$, very
 similar arguments lead to a simplification of the normal modules.

 \begin{proposition} \label{normdim} Let $\cA $ be general and suppose
   $a_1>b_t$, $r \ge 2$ and $\dim A \ge 3$. If $c \le 0$ suppose also
   $\dim A \ge 4$. Let $j$ be any integer satisfying $2-r \le j \le c-1$. Then
   $\Hom _{A_j}(I_j,A)$ is a maximal Cohen-Macaulay $A$-module for $c > 0$ and
   of codepth $1$ for $c \le 0$, and we have
   isomorphisms 
  $$\Hom_{A_{c}}(M_{c},A) \cong \Hom _{A_j}(M_j,A)   \cong \Hom
  _{A_j}(I_j(a_{t+j}),A) \cong  N_{c}\otimes A_{c} \, . $$
   In
  particular, for $c \ge 3-r$ we have
$$ \ _0\! \hom_R(I_{A},A)= \ _0\!
\hom_R(I_{A_{3-r}},A)+ \sum_{j=3-r}^{c-1} \dim (N_{c} \otimes A_c)_{(a_{t+j})}\, .$$
Moreover for $3-r \le j \le c-1$, we have
$$\Ext^1_{A_j}(M_j,A)=\Ext^1_{A_{j+1}}(I_j/I_j^2,A)=\Ext^1_{A_c}(M_c,A)=0
\quad {\rm and} $$
 $$ \ _0\! \ext^1_{A_{j+1}}(I_{A_{j+1}}/I_{A_{j+1}}^2,A) \le \ _0\!
 \ext^1_{A_j}(I_{A_{j}}/I_{A_{j}}^2,A)\, ;$$
thus we get that $ \ _0\! \Ext^1_{A_j}(I_{A_j}/I_{A_j}^2,A)= 0$ for some $j
\ge 3-r$ implies $\ _0\!
 \Ext^1_{A}(I_{A}/I_{A}^2,A)=0\, .$
\end{proposition}

\begin{proof} To use the proof of Proposition~\ref{fibdim2} effectively we
  remark that it suffices to prove Proposition~\ref{normdim} for $A_{c-1}$
  (instead of $A_c$) under the assumption $2-r \le j \le c-2$ and the
  dimension assumptions: $\dim A_{c-1} \ge 3$, and $\dim A_{c-1} \ge 4$ if
  $c-1 \le 0$. Then we get the first displayed formula using the arguments in
  the very first part of Proposition~\ref{fibdim2}, replacing $I_{c-1}$ there
  by $A_{c-1}$ and noticing that
  $ \Hom _{A_{c-1}}(M_{c-1},A_{c-1})\cong N_{c-1}\otimes A_{c-1}$ and that
  $N_{c-1}\otimes A_{c-1}$ is a maximal Cohen-Macaulay $A_{c-1}$-module for
  $c > 1$ and of codepth $1$ for $c \le 1$ by Proposition \ref{vanish}.
  Moreover following the proof of Proposition~\ref{fibdim2} where we got
  $\Ext^1_{A_j}(M_j,I_ {c-1})=\Ext^1_{A_{j+1}}(I_j/I_j^2,I_ {c-1})=0$ for
  $j \ge 3-r$ by showing
  $$H^1_*(U_{j},\widetilde{N_{c-1}}\otimes \widetilde{A_{c-1}}) \cong
  H^1_*(U_j,{\mathcal Hom} _{{\mathcal
      O_U}_{j}}(\widetilde{M_{j}},\widetilde{A_{c-1}})) =0\, ,$$
  we get the latter also now. We only need to be a little careful with the
  dimension of $A_{c-1}$ to get large enough depth, and one checks that our
  dimension assumptions suffice. Hence we get the vanishing of
  $\Ext^1_{A_j}(M_j,A_ {c-1})$ and $\Ext^1_{A_{j+1}}(I_j/I_j^2,A_ {c-1})$. We
  also get $\Ext^1_{A_c}(M_c,A)=0$ by Proposition \ref{vanish}(ii).

  Using $\Ext^1_{A_{j+1}}(I_j/I_j^2,A_{c-1})$ for $j \ge 3-r$ we get an exact
  sequence
\begin{equation}\label{auxNo}
  0\longrightarrow\ _0\! \Hom(I_j,A_{c-1})\longrightarrow\ _0\!
  \Hom(I_{A_{j+1}},A_{c-1})\longrightarrow\ _0\!
  \Hom(I_{A_{j}},A_{c-1})\longrightarrow  0
\end{equation}
induced by
$$0\longrightarrow I_{A_j}\longrightarrow I_{A_{j+1}}\longrightarrow I_j\longrightarrow 0.$$
Indeed the long exact Jacobi-Zariski sequence implies (\ref{auxNo}) since we
can continue  (\ref{auxNo}) to the right by
 $$ 0=\ _0\! \Ext^1_{A_{j+1}}(I_{{j}}/I_{{j}}^2,A_{c-1}) \longrightarrow \ _0\!
 \Ext^1_{A_{j+1}}(I_{A_{j+1}}/I_{A_{j+1}}^2,A_{c-1}) \ \longrightarrow \ _0\!
 \Ext^1_{A_j}(I_{A_{j}}/I_{A_{j}}^2,A_{c-1}) \longrightarrow \,.$$
 The latter sequence shows the inequality of $\ _0\! \ext^1(-,-)$ of
 Proposition \ref{normdim} while repeatedly using (\ref{auxNo}) and
 $\Hom(I_j,A_{c-1}) \cong N_{c-1}\otimes A_{c-1}(a_{t+j})$ for every $j$,
 $3-r \le j \le c-2$ we get a dimension formula for
 $ \ _0\! \hom_R(I_{A_{c-1}},A_{c-1})$ which precisely corresponds to the
 displayed dimension formula for $ \ _0\! \hom_R(I_{A},A)$ of
 Proposition~\ref{normdim}, and we are done.
\end{proof}

\begin{corollary} \label{cornormdim} Let $\cA $ be general and suppose
  $a_1>b_t$, $r \ge 2$, $c \ge 4-r$ and $\dim A \ge 3$. If $c \le 0$ suppose
  also $\dim A \ge 4$. In the flag \eqref{fullflag}, let $B=A_{c-1}$, $A=A_{c}$
  and $N_c=MI$. Then we have
\begin{equation} \label{a3r} \ _0\! \hom_R(I_{A_{3-r}},A_{c-1})= \
_0\!\hom_R(I_{A_{3-r}},A_{c})+ \ _0\!\hom_R(I_{A_{3-r}},I_{c-1})\,
\end{equation}
if and only if \eqref{thm61cond} holds, i.e.
\begin{equation*} \ _0\! \hom_R(I_{A},A) = \ _0\! \hom_R(I_{B},B)+\dim_k(MI
  \otimes A)_{(a_{t+c-1})}\ -\ _0\! \hom_R(I_{B},I_{A/B})\, .
\end{equation*}
\end{corollary}

\begin{proof} Using Proposition~\ref{normdim}, we get that
  $$\ _0\! \hom_R(I_{A},A) =\ _0\!\hom_R(I_{A_{3-r}},A_{c})+ \sum_{j=3-r}^{c-1} \dim (N_{c} \otimes
  A_c)_{(a_{t+j})} .$$
  Since $c-1 \ge 3-r$,  Proposition~\ref{normdim} also implies
  $$_0\! \hom_R(I_{B},B)=\ _0\!\hom_R(I_{A_{3-r}},A_{c-1})+ \sum_{j=3-r}^{c-2}
  \dim (N_{c-1} \otimes A_{c-1})_{(a_{t+j})}.$$
  Inserting these expressions into \eqref{thm61cond} and using $MI \otimes A =
  N_c \otimes A_c$ we get that \eqref{thm61cond} is equivalent to
$$\ _0\!\hom_R(I_{A_{3-r}},A_{c})= \ _0\!\hom_R(I_{A_{3-r}},A_{c-1})+
\sum_{j=3-r}^{c-2}\, _0\! \hom_R(I_{j},I_{c-1}) -\ _0\! \hom_R(I_{B},I_{A/B})\,$$
because we have
$$_0\! \hom_R(I_{j},I_{c-1})= \dim (N_{c-1} \otimes A_{c-1})_{(a_{t+j})}- \dim (N_{c} \otimes
A_c)_{(a_{t+j})}$$ by Proposition~\ref{fibdim2}. Now note that the exact
sequence \eqref{aux*}
holds for $3-r \le j \le c-2$ by  Proposition~\ref{fibdim2}. Using
\eqref{aux*} repeatedly for every such $j$, we get
$$ \sum_{j=3-r}^{c-2}\ _0\! \hom_R(I_{j},I_{c-1})=\ _0\!
\hom_R(I_{A_{c-1}},I_{c-1})-\ _0\!\hom_R(I_{A_{3-r}},I_{c-1})\, .$$
Since $ _0\!\hom _R(I_B ,I_{A/B})=\ _0\!
\hom_R(I_{A_{c-1}},I_{c-1})$ we are done.
\end{proof}

\begin{remark} \label{remcornormdim} \rm Let $j_0$ be some fixed integer
  satisfying $c-1 \ge j_0 \ge 3-r$. Using the proof above replacing
  $I_{A_{3-r}}$ by $I_{A_{j_0}}$ and $ \sum_{j=3-r}^{c-2}$ by
  $ \sum_{j=j_0}^{c-2}$ several times we get that \eqref{thm61cond} is also
  equivalent to
$$ \ _0\! \hom_R(I_{A_{j_0}},A_{c-1})= \
_0\!\hom_R(I_{A_{j_0}},A_{c})+ \ _0\!\hom_R(I_{A_{j_0}},I_{c-1})\, .$$
\end{remark}

\begin{theorem}\label{corWsmoothcnew}  Suppose  that $\Proj(A) \in
  W(\underline{b};\underline{a};r)$
  is general with $c \ge 4-r$, $r \ge 2$, $a_1 > b_t$ and $\dim A \ge 3$. If
  $c \le 0$, we also suppose $\dim A\ge 4$. Moreover in the flag
  \eqref{fullflag}, let $B=A_{c-1}$ belong to
  $W(\underline{b};\underline{a'};r)$ 
  and suppose that every
  deformation of $B$ comes from deforming its matrix $\cB$. \\[-3mm]
\begin{itemize}
\item[{\rm (i)}] If \eqref{a3r} holds, then
  $ \overline{W(\underline{b};\underline{a};r)}$ is a generically smooth
  irreducible component of $\Hilb ^{p_X(t)}(\PP^n)$ and every deformation of $A$
  comes from deforming its matrix $\cA$.
  \item[{\rm (ii)}] If
  $_0\! \hom_R(I_{A_{3-r}},I_{A/B})=\sum _{i=1}^{t-r+2}
  {a_i-a_{t+c-1}+n\choose n}$
  (e.g. $b_t = b_1$ and $a_{t-r+1} < a_{t-r+2}$),
  $a_{t+c-1} < s_r-b_r+b_1$ (e.g. $\cA$ linear) and
  $\dim W(\underline{b};\underline{a'};r) = \lambda_{c-1}$, then
   $$\dim W(\underline{b};\underline{a};r) = \lambda_{c}\, .$$
  \end{itemize}
  \end{theorem}

  \begin{remark} \label{remcornorm} \rm (1) In Remark~\ref{remthm61} we
    noticed that the injectivity assumption
    $_0\!\Ext^1_B(I_B/I_B^2,I_{A/B})\hookrightarrow\
    _0\!\Ext^1_B(I_B/I_B^2,B)$
    in Theorem~\ref{Wsmooth} is equivalent to the exactness of
    $$ 0 \longrightarrow\ _0\!\Hom _R(I_B ,I_{A/B}) \longrightarrow\ _0\!\Hom
    _R(I_B ,B) \longrightarrow\ _0\!\Hom _R(I_B ,A) \longrightarrow 0\,
    ,$$ which in the case $c \ge
    4-r$ is equivalent to \eqref{thm61cond} (or to
    $\gamma=0$)
    by Proposition~\ref{vanish}, as explained in Remark~\ref{remthm89}(2). By
    Corollary~\ref{cornormdim} we now see that \eqref{thm61cond} is further
    equivalent to the exactness of
$$ 0 \longrightarrow\ _0\!\Hom _R(I_{A_{3-r}} ,I_{A/B}) \longrightarrow\ _0\!\Hom
_R(I_{A_{3-r}} ,B) \longrightarrow\ _0\!\Hom _R(I_{A_{3-r}} ,A)
\longrightarrow 0\, .$$
The exactness of the latter sequence is often faster to verify by Macaulay2
than using \eqref{thm61cond} by computing dimensions of the $\Hom$-groups
involved.

(2) Replacing $a_{t+c-1} < s_r-b_r+b_1$ in Theorem~\ref{corWsmoothcnew}(ii)
by the weaker assumption $$a_{t+c-1} < s_r -a_{t-r+1} +a_{t+2-r}-b_r+b_1$$ one
may use the proof of Theorem~\ref{corWsmoothnew2} to see that
$\dim W(\underline{b};\underline{a};r) = \lambda_{c}-\kappa_c$ where
$$\kappa_c:= \dim_k(N_{2-r})_{(a_{t+c-1})}- \dim_k(N_{2-r} \otimes
A_{2-r})_{(a_{t+c-1})}.$$
Moreover $\kappa_c$ may be expressed in terms of binomials using
Lemma~\ref{kapp} for $i=2-r$.
\end{remark}

\begin{proof}({\it of theorem~\ref{corWsmoothcnew}\! }) (i) This follows from
  Theorem~\ref{Wsmooth}, Remark~\ref{remthm61} and Corollary~\ref{cornormdim}.

    (ii) Since every deformation of $B$ comes from deforming $\cB$ and
    $a_{t+c-1} < s_r-b_r+b_1$, we have
         $$\dim W(\underline{b};\underline{a};r) = \dim
       W(\underline{b};\underline{a'};r)+ \dim_k(N_c)_{(a_{t+c-1})} - \ _0\!
       \hom_R(I_{A_{c-1}},I_{c-1})\, $$
       by Proposition~\ref{propo8} and Lemma~\ref{dimMI}(i). Then by
       Corollary~\ref{fiberpr2}(i) and assumption we get that
       $ _0\! \hom_R(I_{A_{c-1}},I_{c-1})= \sum
       _{i=1}^{t+c-1}{a_i-a_{t+c-1}+n\choose n}$.
       Using \eqref{dimensioMI} it follows that
 $$\dim_k(N_c)_{(a_{t+c-1})} - \ _0\! \hom_R(I_{A_{c-1}},I_{c-1}) =
 \lambda_{c}-\lambda_{c-1} +K_c \ .$$
 Since $K_c=0$ by Theorem~\ref{ineqdimW} and \eqref{Ki0} and we are done.
\end{proof}

To give more evidence to Conjectures~\ref{conjdimW} and \ref{conjWsmooth}, we
have considered new examples using Theorem~\ref{corWsmoothcnew} and Macaulay2
because it is much faster to verify \eqref{a3r} than \eqref{thm61cond}.

\begin{example} \label{exgendetWit} \rm (i) An aspect in the conjectures is
  how low we can take $\dim A$ and still expect the conjectures to be true.
  Since the computations are time consuming in low dimensional cases, we have
  only checked a few examples in addition to those considered in Sections 6
  and 7 (Remark~\ref{remconjdimW}, Example~\ref{afterdeg},
  Remark~\ref{remconjWsmooth} and Example~\ref{examples712}). Indeed we have
  checked that the quotient $A$ of dimension 3 defined by the $2 \times 2$
  minors of a general $3 \times 6$ matrix of linear entries in a polynomial
  ring $R$ with 13 variables satisfies \eqref{a3r}, as well as the assumptions
  of Theorem~\ref{corWsmoothcnew}(ii), provided we can show
  $\dim \overline{W(0^3;1^{5};2)} = \lambda_{3}$ and that every deformation of
  $B$ comes from deforming its matrix. $B$ is, however, defined by the
  $2 \times 2$ minors of a general $3 \times 5$ matrix of entries in $R$ and
  can be treated similarly, i.e. by deleting a column to get a quotient $C$
  defined in the same manner. Since we may suppose that the entries of the
  $3 \times 4$ matrix of $C$ are different indeterminates of $R$, we can use
  Example~\ref{exgendet} to get $\dim \overline{W(0^3;1^{4};2)} = \lambda_{2}$
  and Remark~\ref{remconjWsmooth} to see that every deformation of $C$ comes
  from deforming its matrix. Since we have used Macaulay2 to check \eqref{a3r}
  and the first assumption of Theorem~\ref{corWsmoothcnew}(ii) for $B$, $C$,
  we conclude $\dim \overline{W(0^3;1^{5};2)} = \lambda_{3}$ and that every
  deformation of $B$ comes from deforming its matrix. Then we get the
  corresponding conclusion for $A$ by Theorem~\ref{corWsmoothcnew}.

  \vskip 2mm (ii) Finally we have checked the quotient $A$ of dimension 3
  defined by the $2 \times 2$ minors of a general $3 \times 7$ matrix of
  entries in a polynomial ring $R$ with 15 variables. We delete one (resp. 2)
  columns to get $B$ (resp. $C$). Using Macaulay2 we show \eqref{a3r} and
  that the first assumption of Theorem~\ref{corWsmoothcnew}(ii) hold for $B$,
  $C$ as well as for $A$, $B$. Now since $C$ is defined by a general
  $3 \times 5$ matrix of linear entries in a polynomial ring $R$ with 15
  variables, we may suppose that $C$ is a generic determinantal ring, whence
  every deformation of $C$ comes from deforming its matrix by Proposition
  \ref{deforminggenericcase} and
  $\dim \overline{W(0^3;1^{5};2)} = \lambda_{3}$ by Examples~\ref{exgendet}.
  Then using Theorem~\ref{corWsmoothcnew} for $B$, $C$ and then for $A$, $B$
  we get that Conjectures~\ref{conjdimW} and \ref{conjWsmooth} hold.
  \end{example}

  In Example~\ref{exgendetWit} we successively deleted columns to get a flag
  and used \eqref{a3r} several times until we got a ring $A_{j_0}$ for which
  all assumptions of Theorem~\ref{corWsmoothcnew} were satisfied. We could
  even have relaxed a little upon showing \eqref{a3r} each time, due to

\begin{corollary} \label{corWsmoothnew3}
Let $\Proj(A) \in
  W(\underline{b};\underline{a};r)$
  be general with $c \ge 4-r$, $r \ge 2$, $a_1 > b_t$,
  $a_{t+c-1} < s_r-b_r+b_1$ (e.g. $\cA$ linear) and suppose that
  $\dim A \ge 3$, and $\dim A\ge 4$ if $c \le 0$. Let
  $A_{j_0}\twoheadrightarrow A_{j_0+1}\twoheadrightarrow \cdots
  \twoheadrightarrow A_c =A$,
  $j_0 \ge 3-r$ be the flag we get by deleting columns from the right-hand
  side, and suppose that every deformation of $A_{j_0}$ comes from deforming
  its matrix. Moreover suppose that
  $_0\! \hom_R(I_{A_{3-r}},I_{j})=\sum _{i=1}^{t-r+2} {a_i-a_{t+j}+n\choose
    n}$ for every $j$, $j_0 \le j \le c-1$, and
  that $\dim W(\underline{b};\underline{a''};r) = \lambda_{j_0}$ where
  $\Proj(A_{j_0}) \in W(\underline{b};\underline{a''};r)$. If also
 \begin{equation} \label{a3rr} \ _0\! \hom_R(I_{A_{3-r}},A_{j_0})= \
    _0\!\hom_R(I_{A_{3-r}},A_{c})+ \sum_{j=j_0}^{c-1}\
    _0\!\hom_R(I_{A_{3-r}},I_{j})\, .
\end{equation}
then $ \overline{W(\underline{b};\underline{a};r)}$ is a generically smooth
irreducible component of $\Hilb ^{p_X(t)}(\PP^n)$ of dimension $\lambda_c$ and
every deformation of $A$ (and of each $A_j$) comes from deforming its matrix.
     \end{corollary}
\begin{proof}
  By the left-exactness of $\ _0\!\Hom_R(I_{A_{3-r}},-)$ we always have an
  inequality
$$  \ _0\! \hom_R(I_{A_{3-r}},A_{j})\le \
_0\!\hom_R(I_{A_{3-r}},A_{j+1})+ \ _0\!\hom_R(I_{A_{3-r}},I_{j})\, , \quad j_0
\le j \le c-1$$
in \eqref{a3r} that adds to a corresponding inequality the same way in
\eqref{a3rr}. Thus assuming equality in \eqref{a3rr} we get equality in
\eqref{a3r} for every $j$. Then we conclude by using
Theorem~\ref{corWsmoothcnew}.
\end{proof}

By applying Theorem~\ref{corWsmoothcnew} to every surjection
$A_{i-1}\twoheadrightarrow A_{i}$, $i > 3-r$ in the flag \eqref{fullflag} and
Theorem~\ref{corWsmooth} to start the induction, we are able to prove the
following main results related to the flag \eqref{fullflag}. Note that using
Remark~\ref{remcornormdim} for $j_0=4-r$ we get that the following corollary
to Theorem~\ref{corWsmoothcnew} generalizes Theorem~\ref{corWsmooth} (provided
we increase the dimension assumptions by 1).

\begin{corollary}\label{corWsmoothnew1}  Suppose  that $\Proj(A) \in
  W(\underline{b};\underline{a};r)$
  is general with $c \ge 4-r$, $r \ge 2$, $a_1 > b_t$ and $\dim A \ge 3$. If
  $c \le 0$, we also suppose $\dim A\ge 4$. Moreover, suppose that the
  composition
  $\gamma_{32}$ of Theorem~\ref{corWsmooth}
  is zero. If, for every $j$,
  $
  \ 3-r \le j \le c-1$, the maps $$
  _0\!\Hom_R(I_{A_{3-r}},A_{j+1})\longrightarrow \
  _0\!\Ext^1_{A_{3-r}}(I_{A_{3-r}}/I_{A_{3-r}}^2,I_{j})$$ are zero, {\rm or}
  equivalently if \eqref{a3r} holds, {\rm or} the maps $$
  _0\!\Hom_R(I_{A_{3-r}},A_{j})\longrightarrow \
  _0\!\Hom_R(I_{A_{3-r}},A_{j+1})$$ are surjective, then $
  \overline{W(\underline{b};\underline{a};r)}$ is a generically smooth
  irreducible component of $\Hilb
  ^{p_X(t)}(\PP^n)$ and every deformation of
  $A$ comes from deforming its matrix.
     \end{corollary}

\begin{proof}
 We will use induction on
  $c \ge 4-r$. The initial case $c=4-r$ follows from
  Theorem~\ref{corWsmooth} and Remark~\ref{remthm89}(2).

  If $c>4-r$, we have by induction that every deformation of $B:=A_{c-1}$
  comes from deforming $\cB$.  Then we conclude the proof by
  Theorem~\ref{corWsmoothcnew}.
\end{proof}

Note that all assumptions in Corollary~\ref{corWsmoothnew1} are fulfilled if
$ _0\!\Ext^1_{A_{2-r}}(I_{A_{2-r}}/I_{A_{2-r}}^2,I_{2-r})=0$ and
$ _0\!\Ext^1_{A_{3-r}}(I_{A_{3-r}}/I_{A_{3-r}}^2,I_{j})=0$ for every $j$,
$ \ 3-r \le j \le c-1$. Moreover since
$ _0\!\Ext^1_B(I_B/I_B^2,I_{A/B}) \subset\ _0\!\Ext^1_R(I_B,I_{A/B})$ we see
that if the degree of all generators of $I_{A/B}$ is larger that the maximum
of the degree of the relations of $I_B$ appearing in the Lascoux resolution,
both \ $ _0\!\Ext^1$-groups vanish. Also the corresponding \
$ _0\!\Ext^0$-groups vanish. The latter allows us to find
$\dim W(\underline{b};\underline{a};r)$, and we get a result which extends the
$c=4-r$ case of Corollary~\ref{corWsmooth3} substantially. But first we will
prove our final theorem in this section which finds
$\dim W(\underline{b};\underline{a};r)$ under some assumptions when a flag
\eqref{fullflag} is given. This somehow completes Theorem~\ref{ineqdimW}.

\begin{theorem}\label{corWsmoothnew2}
  Suppose that $\Proj(A) \in W(\underline{b};\underline{a};r)$ is general with
  $c \ge 4-r$, $r \ge 2$ and $\dim A \ge 2$. If $c \le 0$, we also suppose
  $\dim A\ge 3$. Moreover suppose $a_1 > b_t$, $b_r-b_1 < s_r -a_{t-r+1}$ and
  that the composed map $\gamma_{32}$ of Theorem~\ref{corWsmooth}
   is zero. If $c \ge 5-r$ suppose also that the maps
  $ _0\!\Hom_R(I_{A_{3-r}},A_{j+1})\longrightarrow \
  _0\!\Ext^1_{A_{3-r}}(I_{A_{3-r}}/I_{A_{3-r}}^2,I_{j})$
  are zero for every $j$, $ \ 3-r \le j \le c-2$. If\ \
  $b_t = b_1$ and $a_{t-r+1} < a_{t-r+2}$, or more
  generally, if  \\[3mm]
  {\rm (*):}
  $_0\! \hom_R(I_{A_{2-r}},I_{2-r})=\sum _{i=1}^{t-r+1}
  {a_i-a_{t-r+2}+n\choose n}$
  \quad and \quad
  $ _0\! \hom_R(I_{A_{3-r}},I_{j})= \sum _{i=1}^{t-r+2}{a_i-a_{t+j}+n\choose
    n}$ \\[-5mm]

 \quad for every $j$, $ \ 3-r \le j \le c-1$, then
   $$\dim W(\underline{b};\underline{a};r) = \lambda_c +K_3+K_4+\cdots
   +K_c -\kappa $$
   where $\kappa$ is the following non-negative integer:
   $$\kappa= \sum_{j=3-r}^c( \dim_k(N_j)_{(a_{t+j-1})}- \dim_k(N_j \otimes
   A_j)_{(a_{t+j-1})}).$$ Moreover if $a_{t+c-1} < s_r -a_{t-r+1}
   +a_{t+2-r}-b_r+b_1$ then every $K_i=0$ and $\dim
   W(\underline{b};\underline{a};r) = \lambda_{c}-\kappa$ where
   $$\kappa= \sum_{j=3-r}^c(\dim_k(N_{2-r})_{(a_{t+j-1})}- \dim_k(N_{2-r} \otimes
   A_{2-r})_{(a_{t+j-1})}),$$
    whence
   $\kappa$ may be expressed in terms of binomials using Lemma~\ref{kapp} for
   $i=2-r$.
   In particular, if also $a_{t+c-1} < s_r-b_r+b_1$ then every $K_i=0$,
   $\kappa =0$ and $$\dim W(\underline{b};\underline{a};r) = \lambda_c\, .$$
     \end{theorem}

  \begin{remark} \rm By Corollary~\ref{fiberpr2}, assumption (*) above holds
    in the following cases:

  (1) 
  $a_{t-r+1} < a_{t-r+2} -\sum _{i=1}^{t-r+1}b_{r+i-1}+\sum
  _{i=1}^{t-r+1}b_{i}$.

%
   (2) $ \ _0\! \hom_R(I_{A_{2-r}},I_{j}) =0$ for $2-r \le j \le c-1$ and $a_{t-r+1}
< a_{t-r+2}$
\end{remark}

\begin{proof} Due to assumptions we know that every deformation of
  $B:=A_{c-1}$ comes from deforming $\cB$ by Corollary~\ref{corWsmoothnew1}. In
  Corollary~\ref{corWsmoothnew1} we used Theorem~\ref{Wsmooth} at each step of
  the induction, but there is also a dimension formula attached to
  Theorem~\ref{Wsmooth} (the same formula applies in
  Theorems~\ref{corWsmoothc} and ~\ref{corWsmooth} because these results are
  directly deduced from Theorem~\ref{Wsmooth}). Thus letting
  $\underline{a}_{(j)}=a_1, \cdots,a_{t+j-1}$ and using that the
  $\Ext^1$-vanishing for (ii) only holds for $3-r \le j \le c-2$, we get
  that
  $$\dim W(\underline{b};\underline{a}_{(j)};r) = \dim
  W(\underline{b};\underline{a}_{(j-1)};r)+ \dim_k(N_j \otimes
  A_j)_{(a_{t+j-1})} - \ _0\! \hom_R(I_{A_{j-1}},I_{j-1}), \ 4-r \le j < c$$
  Since every deformation of $B$ comes from
  deforming its matrix, it follows from Proposition~\ref{propo8} that the above
  dimension formula holds for $j=c$ as well. From Theorem~\ref{corWsmoothc} we
  get that
  $$\dim W(\underline{b};\underline{a}_{(3-r)};r)= \lambda_{3-r}-\kappa'$$
  because $a_{t-r+1} < s_r-b_r+b_1$ implies $-b_1 < \ell'_2$, whence $K'_i=0$
  for all $i$ by Lemma~\ref{tr}. Summing all these dimensions  we
  get
 $$\dim W(\underline{b};\underline{a};r) = \lambda_{3-r}-\kappa'
 +\sum_{j=4-r}^c \dim_k(N_j \otimes A_j)_{(a_{t+j-1})} \ - \sum_{j=4-r}^c \
 _0\! \hom_R(I_{A_{j-1}},I_{j-1}).$$
 In this formula Corollary~\ref{fiberpr2} and assumption (*) imply that
 $ _0\! \hom_R(I_{A_{j}},I_{j})= \sum _{i=1}^{t+j-1}{a_i-a_{t+j}+n\choose n}$
 for every $j$, $ \ 2-r \le j \le c-1$. Using \eqref{dimensioMI} it follows
 that
 $$\dim_k(N_j)_{(a_{t+j-1})} - \ _0\! \hom_R(I_{A_{j-1}},I_{j-1}) =
 \lambda_{j}-\lambda_{j-1} +K_j \ , \ \ {\rm for } \ \ 4-r \le j \le c$$
 where $K_j=0$ for $j < 3$. Inserting this formula and recalling that
 $$\kappa'= \dim_k(N_{3-r})_{(a_{t-r+2})}- \dim_k(N_{3-r} \otimes
 A_{3-r})_{(a_{t-r+2})}$$
 by Theorem~\ref{corWsmoothc} we get the $1^{st}$ displayed dimension formula
 of Theorem~\ref{corWsmooth}. Moreover if
 $a_{t+c-1} < s_r -a_{t-r+1} +a_{t+2-r}-b_r+b_1$, then \eqref{dimMIAflag} and
 \eqref{dimMImaxFlag} applies with $j=2-r$ and we get the last expression for
 $\kappa$ which we may compute using Lemma~\ref{kapp}. We also get $K_i=0$ for
 $3 \le i \le c$ by Theorem~\ref{ineqdimW} and \eqref{Ki0}. Finally if
 $a_{t+c-1} < s_r-b_r+b_1$, we get $a_{t+j} < s_r-b_r+b_1$ for $j \le c-1$ and
 hence $\kappa=0$ by Lemma~\ref{dimMI}(i) and we are done.
\end{proof}

\begin{corollary}\label{corWsmooth4}
  Let $\Proj(A) \in W(\underline{b};\underline{a};r)$ be general with
  $c \ge 4-r$ and suppose $\dim A \ge 2$. If $c \le 0$, we also suppose
  $\dim A\ge 3$. Set $b:=\sum_{i=r}^{t}b_{i}-\sum_{i=1}^{t-r+2}b_{i}$ and
  suppose $a_1 > b_t$, $r \ge 2$ and
  $b_r-b_1 < s_r -a_{t-r+1}:=\sum _{i=1}^{t-r}(a_i-b_{r+i})$. If
  $$a_{t-r+2} > 2a_{t-r+1} + b \ \ and \ \
  a_{t-r+3} > \sum _{i=t-r+1}^{t-r+2}a_i+
  b+\max\{a_{t-r+2}-a_1,b_{t-r+2}-b_1\}$$
  then $ \overline{W(\underline{b};\underline{a};r)}$ is a generically smooth
  irreducible component of $\Hilb ^{p_X(t)}(\PP^n)$ and every deformation of $A$
  comes from deforming $\cA$.
  Moreover$$\dim W(\underline{b};\underline{a};r) = \lambda_c +K_3+K_4+\cdots
  +K_c -\kappa $$
  where $\kappa$ is the following non-negative integer:
  $$\kappa= \sum_{j=3-r}^c( \dim_k(N_j)_{(a_{t+j-1})}- \dim_k(N_j \otimes
  A_j)_{(a_{t+j-1})}).$$
  In particular, if also $a_{t+c-1} < s_r-b_r+b_1$ then every $K_i=0$,
  $\kappa =0$, whence $\dim W(\underline{b};\underline{a};r) = \lambda_c$.
\end{corollary}

\begin{proof} Consider the flag \eqref{fullflag} where
  $I_{j}=I_{A_{j+1}/A_{j}}$. By Theorem~\ref{corWsmooth} and
  Corollary~\ref{fiberpr2}(i), cf. Remark~\ref{remcor83}(1), it suffices to show
  $ _0\!\Ext^1_{A_{2-r}}(I_{A_{2-r}}/I_{A_{j}}^2,I_{2-r})=0$ and that
  $ _0\!\Ext^1_{A_{j}}(I_{A_{3-r}}/I_{A_{3-r}}^2,I_{j})=0$ for every
  $ \ 3-r \le j \le c-1$ under the given assumptions on $\dim A$ (which is one
  less than Corollary~\ref{corWsmoothnew1} requires). We have
  $ _0\!\Ext^1_{A_{2-r}}(I_{A_{2-r}}/I_{A_{2-r}}^2,I_{2-r})=0$ by
  Corollary~\ref{corWsmooth3} and assumption. Since the smallest degree of the
  minimal generators of $I_{j}$ is
  $$ s(I_{j}): =a_{t+j}+\sum _{i=1}^{t-r}a_i-\sum _{i=r}^{t}b_{i}$$ and the
  maximum degree of the relations of $I_{A_{3-r}}$ is
  $$mdr(I_{A_{3-r}})=\sum _{i=1}^{t-r+1}a_i-\sum_{i=1}^{t-r+2}b_{i} +
  a_{t-r+1}$$
  by Corollary~\ref{maxrel}(iii), we get the vanishing of the
  $ _0\!\Ext^1_R$-groups (as well as for the corresponding $ _0\!\Hom$-groups
  which means that (*) of Theorem~\ref{corWsmoothnew2} holds)
  if $ mdr(I_{A_{3-r}}) < s(I_{j})$, or equivalently if
  $ a_{t+j} > \sum _{i=t-r+1}^{t-r+2}a_i+
  b+\max\{a_{t-r+2}-a_1,b_{t-r+2}-b_1\}$
  for every $j$. Since, in general, we assume $a_{t+j} \ge a_{t-r+3}$ for
  $j \ge 3-r$, we are done by an assumption in Corollary~\ref{corWsmooth4}
\end{proof}

Let us extend Example~\ref{Wsmooth2} a lot.

\begin{example} \label{Wsmooth2f} \rm For any $r$, $2 \le r \le t-1$ and any
  $c \ge 4-r$, let $\cA= [\cB,\cC]$ be a general $t \times (t+c-1)$ matrix
  where $\cB$ is a $t \times (t-r+2)$ matrix that is exactly equal to the
  matrix $\cA$ in Example~\ref{Wsmooth2}(i), i.e. all entries in the last
  column are of degree $3$ and all other entries are of degree one (and take
  all $b_i=0$, thus $b_r-b_1 < s_r -a_{t-r+1}$ holds). Moreover let all
  entries of the 
  $t \times (c+r-3)$ matrix $\cC$ be of degree at least $7$, and otherwise
  arbitrary, i.e. $ 7 \le a_{t-r+3} \le a_{t-r+4}, \le \cdots, a_{t+c-1}$, and
  suppose that $\dim R$ is large enough so that $\dim A \ge 3$. Then
  $ \overline{W(\underline{b};\underline{a};r)}$ is a generically smooth
  irreducible component of $\Hilb ^{p_X(t)}(\PP^n)$ and every deformation of $A$
  comes from deforming $\cA$ by Corollary~\ref{corWsmooth4}. Moreover
  $\dim W(\underline{b};\underline{a};r) = \lambda_c$ holds if $t-r$ is large
  enough. More precisely, $\dim W(\underline{b};\underline{a};r) = \lambda_c$
  provided $a_{t+c-1} < s_r-b_r+b_1$, i.e. if $t-r \ge a_{t+c-1}$.
 \end{example}

 \section{Deformations of exterior powers of modules over determinantal schemes}

  We keep the notation of the previous sections. So,
 $\varphi: F:=\oplus _{i=1}^tR(b_i)\longrightarrow G:=\oplus _{j=1}^{t+c-1}
 R(a_j)$
 is a graded $R$-morphism between two free $R$-modules of rank $t$ and
 $t+c-1$, respectively, and we suppose $a_1 > b_t$ in this section. Set
 $MI=\coker (\varphi ^*)$, $I_A:=I_{t-r+1}(\varphi ^* )$, $A=R/I_A$ and
 $X:=\Proj(A)$. If we simultaneously work with $A$ for different $r$, we use
 the notation $I_{A_r}=I_{t-r+1}(\varphi ^* )$, $A_r=R/I_{A_r}$ and
 $X_r:=\Proj(A_r)$. We assume that $A$ has codimension $r(r+c-1)$ in $R$, i.e.
 that $A$ is determinantal. In this section we also suppose $c \ge 1$ so that
 $MI$ is an $A_1$-module, i.e. supported at $X_1$. By \cite[Proposition 5]{V} and \cite[Theorems 1 and 2]
{B}, $ \bigwedge ^rMI$ is a maximal Cohen-Macaulay $A$-module and
 letting $J_A=I_{t-r}(\varphi ^* )$ we also suppose that $\bigwedge ^rMI$ is
 invertible in $X\setminus V(J_A)$, e.g. that the matrix $\cA$ of $\varphi ^*$
 is general.
%

\begin{lemma}\label{42} If $\depth _{J_A}A\ge j+1$ for some
  $j\ge 1$, then $\Hom _{A}(\bigwedge ^rMI,\bigwedge ^rMI)\cong A$
  and
  $$\Ext^{i}_{A}(\bigwedge ^rMI,\bigwedge ^rMI)=0 \quad {\rm for} \quad 1\le i
  \le j-1 \ .$$
  In particular, if $\depth _{J_A}A\ge 4$, then
  $$\Hom _{A}(\bigwedge ^rMI,\bigwedge ^rMI)\cong A \quad {\rm and} \quad
  \Ext^{i}_{A}(\bigwedge ^rMI,\bigwedge ^rMI)=0 \quad {\rm for} \quad i=1,2\,
  .$$
\end{lemma}

\begin{remark} \rm \label{rema92} (1) The case $r=1$ in Lemma \ref{42} was considered in
  \cite{K2014} and slightly generalized in Theorem~\ref{Amodulecor2}. Indeed
  the hypothesis $\depth _{J_A} A\ge j+1$, for $\varphi $ general, is weakened
  to $\depth _{J_A} A \ge j$ in Theorem~\ref{Amodulecor2} and Lemma
  \ref{42} still works. Examples computed with Macaulay2 indicate that this
  also holds for $r>1$.

  (2) If $c=2$, then the conclusions of Lemma \ref{42} hold, without assuming
  $\depth _{J_A}A\ge 4$. Indeed, $\bigwedge ^r MI$ is a twist of the canonical
  module $K_{A}$ by \cite[Proposition 3.5]{KM2017} and it is well known that
  $ A\cong \Hom (K_{A},K_{A})$ and $\Ext^{i}_{A}(K_{A},K_{A})=0$ for $i\ge 1$
  hold in general for Cohen-Macaulay rings.
\end{remark}

\begin{proof} Since $\depth _{J_A}A\ge 2$, we have  $$\Hom
  _{A}(\bigwedge ^r MI,\bigwedge ^rMI)\cong H ^0_*(X\setminus
  V(J_A),\cO_{X})\cong A$$
  and
 $\Ext^{i}_{A}(\bigwedge ^rMI,\bigwedge ^rMI)\cong
 \Ext^{i}_{\cO_{X'}}(\widetilde{\bigwedge ^rMI},\widetilde{\bigwedge ^rMI})$ with
  $X':=\Spec(A)\setminus V(J_A)$.

   Since $\widetilde{\bigwedge ^rMI}$ is locally free of rank 1 on $X'$, we get
   $$\Ext^{i}_{\cO_{X'}}(\widetilde{\bigwedge ^rMI},\widetilde{\bigwedge
     ^rMI})\cong H ^i(X',\cH om(\widetilde{\bigwedge
     ^rMI},\widetilde{\bigwedge ^rMI}))\cong H ^{i}(X',\cO_{X'})=0 $$ (since
   $H^i(X',\cO_{X'})\cong H ^{i+1}_{J_A}(A)=0$ by $\depth
   _{J_A}A\ge j+1$) for $1\le i\le j-1$.
\end{proof}

We suppose $\depth _{J_A}A\ge 4$ and using Lemma \ref{42} we will compare
deformations of $X\subset \PP^n$ with deformations of $\bigwedge ^r MI$. Note
that
$\Hom _{R}(\bigwedge ^r MI,\bigwedge ^rMI)\cong \Hom _{A}(\bigwedge ^r
MI\otimes A,\bigwedge ^rMI)\cong A$
and recall that there is a spectral sequence
$$\Ext^{i}_{A}(\Tor _j^R(\bigwedge ^rMI,A), \bigwedge
^rMI)\Rightarrow \Ext^{\bullet}_{R}(\bigwedge ^rMI,\bigwedge ^rMI)$$
which induce an exact sequence
\begin{equation} \label{spectralseq}
0\longrightarrow \Ext^{1}_{A}(\bigwedge ^rMI,\bigwedge ^rMI)\longrightarrow \Ext^{1}_{R}(\bigwedge ^rMI,\bigwedge ^rMI)\longrightarrow \end{equation}
$$\Hom_{A}(\Tor _1^R(\bigwedge ^rMI,A),\bigwedge ^rMI)
\longrightarrow \Ext^{2}_{A}(\bigwedge ^rMI,\bigwedge ^rMI) \
$$ \\[-3mm]
where the $\Ext^{i}_{A}$-groups for $i=1, 2$ vanish.
Now the exact sequence
$0\longrightarrow I_{A}\longrightarrow R \longrightarrow A\longrightarrow 0$
leads to $$\Tor _1^R(\bigwedge ^rMI,A)\cong \bigwedge ^rMI\otimes I_A \ .$$
Therefore, we have
$$\begin{array}{rcl} \Hom_{A}(\Tor _1^R(\bigwedge ^rMI,A),\bigwedge ^rMI) &
\cong & \Hom_{A}(\bigwedge ^rMI\otimes I_A ,\bigwedge ^rMI) \\ & \cong & \Hom _R(I_A,A)\end{array}$$
because $\Hom _{A}(\bigwedge ^rMI,\bigwedge ^rMI)\cong A$. By
(\ref{spectralseq}) we get an isomorphism of the tangent spaces
$$ \Ext^{1}_{R}(\bigwedge ^rMI,\bigwedge ^rMI) \cong \Hom
_R(I_A,A)\cong \Hom
_{A}(I_A/I_A^2,A) \ $$
of the deformation functors we consider. Also note it is
possible to continue (\ref{spectralseq})  to see the injection
$$  \Ext^1 _{A}(I_A/I_A^2,A)\hookrightarrow \Ext^{2}_{R}(\bigwedge
^rMI,\bigwedge ^rMI)\ $$
of obstruction spaces of these functors. Thus we have the first isomorphism in
\begin{theorem} \label{defpower}  Let $c\ge 1$, $(c,r) \ne (1,1)$ and suppose
  $\depth _{J_A}A\ge 4$ for $c \ne 2$. Then
$$Def_{\bigwedge ^rMI /R}(-)\simeq Def_{A/R}(-) \simeq  \Hilb_{X}(-)$$
i.e. the deformation functors of $\bigwedge ^rMI$ as an $R$ module, the
deformation functor of the surjection $R\twoheadrightarrow A$ and the local
Hilbert functor are isomorphic.
\end{theorem}

\begin{remark} \rm For $(c,r) = (1,1)$, $\depth _{J_A}A\ge 4$ is impossible.
  Note that the functors is defined on the category $\underline{\ell}$
  described in section 5 and that some of the local functors in the theorem are
  considered in Lemma~\ref{lemma2}.
\end{remark}

\begin{proof} The first isomorphism follows from Lemma~\ref{42} and
  (\ref{spectralseq}) by the arguments above  while the second follows from \eqref{Grad}.
\end{proof}

\begin{remark}\rm We have now the vertical isomorphism in
\[
\begin{array}{cccccc}
Def_{MI /R}(-) &  \dashrightarrow & Def_{\bigwedge ^rMI /R}(-)
\\ \cong \downarrow & &
\cong \downarrow  &     \\ \Hilb_{X_1}(-) &
  & \Hilb_{X_r}(-)
\end{array}
\]
\vskip 2mm
\noindent provided $\depth _{J_{A_r}}A_r\ge 4$ and $\depth _{J_{A_1}}A_1\ge 3$
(Theorem~\ref{Amodulecor2}). To define the dotted arrow, i.e. a morphism
$ \alpha(-): Def_{MI /R}(-) \longrightarrow Def_{\bigwedge ^rMI /R}(-)$ of
functors on $\underline{\ell}$, it suffices to observe that if $MI_S$ is a
deformation of $MI$, hence $S$-flat, then
$MI_S^{\otimes r}:=\overbrace{MI_S\otimes \cdots \otimes MI_S}^r$ is obviously
$S$-flat, and thinking of $\bigwedge ^rMI_S$ as a direct summand of
$MI_S^{\otimes r}$, we get that $\bigwedge ^rMI_S$ is $S$-flat. Thus, we may
define on ${\underline{\ell}}$:
$$
\alpha(S): Def_{MI /R}(S) \longrightarrow Def_{\bigwedge ^rMI /R}(S)$$
$$ \hskip 5mm MI_S\mapsto \bigwedge ^rMI_S.
$$
and use it to compare all four deformation functors above.
\end{remark}
We will, however, treat this a little differently, using the theory we have
developed. Indeed in the diagram of the remark above, there is a direct
well-defined morphism
$$Def(\psi )(-):Def_{MI/R}(-) \longrightarrow
Def_{A/R}(-) \cong  \Hilb_{X_r}(-) $$
of local functors over $\underline{\ell}$ defined in Lemma~\ref{lemma2}, with
tangent map
$$\psi:\ _0\!\Ext^1_R(MI,MI) \longrightarrow \ _0\!\Hom(I_{t-r+1}(\varphi ^*),A)$$
which, thanks to Lemma~\ref{lemma2}, is surjective if, and only if, every
deformation of $A$ comes from deforming its matrix $\cA$. Thus redefining
$\alpha(-)$ to be $Def(\psi )(-)$ composed with the inverse of the isomorphism
$Def_{\bigwedge  ^rMI /R}(-)  \stackrel{\cong}{\longrightarrow} \Hilb_{X_r}(-)$
we  get

\begin{theorem} \label{defWed} Let $c\ge 1$, $1 \le r <t$, $(c,r) \ne (1,1)$,
  set $A:=A_r$ and suppose $\dim A \ge 2$ and if $c\ne 2$ that
  $\depth _{J_A}A\ge 4$ and $\depth _{J_{A_1}}A_{1}\ge 3$. Moreover, we suppose
  that every deformation of $A$ comes from deforming its matrix $\cA$ and that
$$ \  _0\!\hom(I_{t-r+1}(\varphi ^*),A_r) = \lambda _c+K_3+K_4+\cdots +K _c \ .
$$
If $c \ne 1$ then the following 4 functors are isomorphic and smooth:
\[
\begin{array}{cccccc}
Def_{MI /R}(-) &  \stackrel{\cong}{\longrightarrow} & Def_{\bigwedge ^rMI /R}(-)
\\ \cong \downarrow & &
\cong \downarrow  &     \\ \Hilb_{X_1}(-) &
  & \Hilb_{X_r}(-)
\end{array}
\]
In particular the dimension of their tangent spaces are all equal, and also equal to
\begin{equation}\label{dim}   \lambda _c+K_3+K_4+\cdots +K _c = \dim
  W(\underline{b};\underline{a};1) = \dim
  W(\underline{b};\underline{a};r)\ .
\end{equation}
Moreover if $c=1$ then the functors
$ Def_{MI /R}(-) \cong Def_{\bigwedge ^rMI /R}(-) \cong \Hilb_{X_r}(-) $ are
isomorphic and smooth, and the dimension of their tangent spaces are equal to
$\lambda _1 = \dim W(\underline{b};\underline{a};r)\ .$
\end{theorem}

\begin{proof} By Theorem~\ref{Amodulecor1}(iii),
  $\dim\, _0\!\Ext^1_R(MI,MI) =\lambda _c+K_3+K_4+\cdots +K _c$. So
  $Def(\psi )(-)$ is an isomorphism on tangent spaces and since it is smooth
  by Lemma~\ref{lemma2}, $Def(\psi )(-)$ is an isomorphism. So for $c \ne 1$,
  all 4 functors are isomorphic by Theorems~\ref{defpower}, cf.
  Remark~\ref{rema92}, and smooth by Theorem~\ref{Amodulecor1}(iii) or
  Corollary~\ref{unobstr}. Finally using Lemma~\ref{unobst}, we get
  $$\dim \overline{ W(\underline{b};\underline{a};r)} = \
  _0\!\hom(I_{t-r+1}(\varphi ^*),A)$$ which concludes the proof because for
  $c=1$ similar arguments apply.
\end{proof}

\begin{remark} \label{rema97} \rm (1) A consequence of Theorem~\ref{defWed} is
  that we have
  $$ \ _0\!\Ext^1_R(MI,MI) \stackrel{\cong}{\longrightarrow} \
  _0\!\Ext^1_R(\bigwedge ^rMI,\bigwedge ^rMI)
$$
under the assumptions of  Theorem~\ref{defWed}.

(2) As we have seen the assumption ``every deformation of $A$ comes from
deforming $\cA$'' often holds (Theorems~\ref{Wsmooth} and
\ref{corWsmoothcnew}) and is really the main reason that allows us to state
Conjecture~\ref{conjWsmooth}. The displayed assumption is then a main
conclusion in Theorems~\ref{dimW} and \ref{corWsmoothcnew} which, together
with Macaulay2 computations, inspired
Conjecture~\ref{conjdimW}. Note that in these results we always assume
$a_{t+c-1} < s_r -b_r+b_1$, i.e.
$a_{t+c-1} -b_1 < \sum _{i=1}^{t-r+1}(a_i-b_{r+i-1})$ which implies $K_i=0$ by
\eqref{Kis0} and \eqref{Ki0}. In particular assuming
$a_{t+c-1} < s_r -b_r+b_1$ the displayed assumption in Theorem~\ref{defWed} is
plainly $ \dim \overline{ W(\underline{b};\underline{a};r)} =\lambda _c$. In
fact we are not aware of cases where the displayed assumption above holds
unless $r=1$, or $r>1$ and $a_{t+c-1} < s_r -b_r+b_1$. But if the latter
assumption holds, both these assumptions of Theorem~\ref{defWed} are expected
if $\dim A$ is large enough. More precisely if in addition $\dim A \ge 4$ for $c \ne 2$ and $\dim A\ge 2$ for $c = 2$
(cf. Conjecture~\ref{conjWsmooth} and Remark~\ref{remconjWsmooth}) the
conclusions of Theorem~\ref{defWed} are expected.

(3) As a consequence of (2) and \eqref{Ki0}, if $ a_{t+c-1} < s_r -b_r+b_1$
holds then $ a_{t+c-1} < s_i -b_i+b_1$ hold for all $1 \le i \le r$. This
implies that if the Conjectures~\ref{conjdimW} and \ref{conjWsmooth}  hold for
all $A_i$, which is highly expected, and if say $\dim A_r \ge 4$ and
$c \ge 2$, then the three isomorphic deformation functors
$Def_{A_i/R}(-) \cong \Hilb_{X_i}(-) \cong Def_{\bigwedge ^iMI /R}(-)$ are for
{\it all} $i$, $1 \le i \le r$ further isomorphic to e.g. the local Hilbert
functor
$\Hilb_{X_1}(-)$ of deforming the scheme $X_1 \subset \Proj(R)$ defined by
maximal minors. In particular $ \overline{
  W(\underline{b};\underline{a};i)}$ is for all $i$ a generically smooth irreducible
component of dimension $\lambda _c$ of $\Hilb ^{p_{X_i}(t)} (\PP^{n})$!! Here
the dimension of $X_i$ are very different and so are their Hilbert polynomials.
\end{remark}

Let us finish this section by seeing how precisely the assumption
$a_{t+c-1} < s_r -b_r+b_1$ lead to the isomorphisms given in
Theorem~\ref{defWed}. To illustrate the theorem we include also Macaulay2
computations of $ _0\!\Ext^1_R(\bigwedge ^rMI,\bigwedge ^rMI)$ for various $r$.

\begin{example} \label{exex910} \rm In all of this example $\cA$ is a
  $4\times 4$ matrix with 16 variable as in the generic case, and (i) is the
  generic case where all five deformation functors which we consider should  be
  isomorphic due to Theorem~\ref{defWed}. We only need to check that the
  tangent spaces are isomorphic since all functors are smooth. Below we let
  $M_r = \bigwedge ^rMI$. Since $c=1$, Theorem~\ref{defWed} does not include
  $ \Hilb_{X_1}(-) $.

  (i) We let $\cA=(x_{i,j})$ be the generic $4\times 4$ linear matrix and
  using Macaulay2 we get
  $$\ _0\!\ext^1(MI,MI)=\ _0\!\ext^1(M_2,M_2)=\ _0\!\ext^1(M_3,M_3) =225 \, .$$
  Correspondingly for the local graded deformation functor,
  $ \ _0\!\hom(I_{2}(\varphi ^*),A_3)=\ _0\!\hom(I_{3}(\varphi ^*),A_2)=225$.
  Hence all five deformation functors are isomorphic.
  Note that $A_r$ and $I_{A_r}$ are defined by the $t-r+1=5-r$ minors, so e.g.
  for $r=2$, $A_2=R/I_{A_2}$ is defined by $3 \times 3$ minors (submaximal
  minors), so $I_{A_2}=I_{3}(\varphi ^*)$ according to our notation.

  (ii) In this case we consider $\cA=(x_{i,j}^{r_j})$ with $r_1=r_2=1$ and
  $r_3=r_4=2$, i.e. with the degree matrix
  $\left(\begin{smallmatrix}1 & 1 & 2 & 2\\ 1 & 1 & 2 & 2\\ 1 & 1 & 2 & 2\\ 1
      & 1 & 2 & 2\end{smallmatrix}\right)$.
  This means that we can take all $b_i=0$, $a_1=a_2=1$ and $a_3=a_4=2$. Then
  we get $s_2=4$ and $s_3=2$ by the definition of $s_r:=\sum_{i=1}^{5-r}a_i$,
  whence $a_4 < s_2$ is satisfied while $a_4 < s_3$ is not. So according to
  Theorem~\ref{defWed} and Remark~\ref{rema97} the deformation functors of
  $A_3$, $M_3$ and $MI$ should be isomorphic. Moreover by
  Theorem~\ref{defpower} the deformation functors of $A_2$ and $M_2$ are
  isomorphic. Now
  $$\ _0\!\ext^1(MI,MI)=\ _0\!\ext^1(M_2,M_2)= \ _0\!\hom(I_{3}(\varphi
  ^*),A_2)=1129$$
  by using Macaulay2 while $ _0\!\hom(I_{2}(\varphi ^*),A_3)=\ _0\!\ext^1(M_3,M_3)=1089$,
  confirming what we expected by Theorems~\ref{defWed} and \ref{defpower}.

(iii) In this case we consider $\cA=(x_{i,j}^{r_j})$ with $r_1=1$ and
$r_2=r_3=r_4=2$, i.e. the degree matrix of $\cA$ is
$\left(\begin{smallmatrix}1 & 2 & 2 & 2\\ 1 & 2 & 2 & 2\\ 1 & 2 & 2 & 2\\ 1 &
    2 & 2 & 2\end{smallmatrix}\right)$.
So we may take $b_i=0$, $a_1=1$ and $a_2=a_3=a_4=2$. Then we get $s_2=5$
and $s_3=2$ by definition, whence both $a_4 < s_2$ and
$a_4 < s_3$ are satisfied. So the deformation functors of $A_i$ for $i=2, 3$
and those of $M_2$, $M_3$ and $MI$ are again supposed to be isomorphic by
Theorem~\ref{defWed}.
Using Macaulay2 we get
$$\ _0\!\ext^1(MI,MI)=\ _0\!\ext^1(M_2,M_2)= \ _0\!\ext^1(M_3,M_3) = 1623$$ as
well as
$$\ _0\!\hom(I_{2}(\varphi ^*),A_3)=\ _0\!\hom(I_{3}(\varphi ^*),A_2)=1623\, ,$$
confirming what we expected by Theorem~\ref{defWed}. So
all five deformation functors are isomorphic, as in (i).
\end{example}

\section{Final comments and conjectures} \label{conjec}

This last section contains a collection of some natural open
questions/conjectures arising from our work.

First of all, we would like to say few words about the (linear) representation
of a hypersurface. The Hilbert scheme which parameterizes hypersurfaces
$X\subset \PP^n$ of fixed degree $d\ge 1$ is isomorphic to $\PP^{N}$,
$N:={n+d \choose n}-1$ and there is an open dense subset $U\subset \PP^{N}$
parameterizing smooth hypersurfaces. Inside $\PP^{N}$ we have the
determinantal locus $W(b_1,\cdots ,b_t;a_1,\cdots,a_t;1)$ which parameterizes
hypersurfaces defined by a form $F\in k[x_0,\cdots ,x_n]$ of degree $d$ which
can be written as the determinant of a $t\times t$ matrix with entries
homogeneous forms of degree $a_j-b_i$. We can ask, for instance, for
$\codim _{\PP^{N}} W(b_1,\cdots ,b_t;a_1,\cdots,a_t;1)$. In fact, if
$n \ge 3$, this codimension is $N-\lambda_1$ by Theorem \ref{Amodulethm5}.
This leads us to a more basic problem: Given a {\em general} form
$F\in k[x_0,\cdots ,x_n]$ of fixed degree $d\ge 2$, we would like to determine
the minimum integer $r=r(n,d;s)$ (resp. $p=p(n,d;s)$) such that $F^r$ (resp.
$F^p$) can be written as the determinant of a matrix (resp. the pfaffian of a
skew symmetric matrix) $\cA$ with entries homogeneous polynomials of degree
$s\ge 1$. As a special case of great interest we highlight the case $s=1$. In
this case we are dealing with matrices with linear entries and, for
simplicity, we set $r(n,d):=r(n,d;1)$.

It is a classical result that $r(2,d)=r(3,3)=1$. Indeed, in \cite{D}, Dickson
proved that a general form of degree $d$ in $k[x,y,z]$ is the determinant of a
$d\times d$ matrix of linear forms and, in \cite{G}, Grassmann proved that a
general cubic form in $k[x,y,z,t]$ can be written as the determinant of a
$3\times 3$ matrix with linear entries. More recently Beauville has proved
that $r(3,d)=2$ for $3<d<16$ and $r(3,d)>2$ for $d>15$ (see \cite{Beau}). To
our knowledge no other results are known concerning the (linear)
representation of a {\em general} form $F\in k[x_0,\cdots ,x_n]$ of degree
$d> 2$ and this is a natural problem that we should address before studying
the determinantal locus of hypersurfaces. So we propose the following to
problems:

\begin{pblm} Fix integers $n\ge 2$ and $d\ge 3$. To determine $r(n,d)$ or at
  least lower/upper bounds for $r(n,d)$.
\end{pblm}

Fix $n\ge 3$ and $d\ge 3$. It follows from Theorem \ref{Amodulethm5} that
$r(n,d)=1$ if and only if $(n,d)=(3,3)$. So we not only recover Grassmann's
result but we also prove that the converse holds.

Now, we come back to determinantal schemes with $c\ge 2$ and we collect the
conjectures scattered throughout the work and introduce a new one.

\begin{conj} \label{conj1} Fix integers $t\ge 2$,
  $a_1\le a_2\le \cdots \le a_{t+c-1}$ and
  $b_1\le b_2\le \cdots ..\le b_{t}$. Let $\cA$ be a homogeneous
  $t \times (t+c-1)$ matrix with entries forms of degree $a_j-b_i$, let
  $A=R/I_{t-r+1}(\cA)$ where $1 \le r \le t-1$, $2-r \le c$ and suppose that
  $\dim A\ge 2$ for $c \ne 1$ and $\dim A\ge 3$ for $c = 1$. Moreover suppose
  $\Proj(A) \in W(\underline{b};\underline{a};r)$, $a_1 >b_t$ and
  $a_{t+c-1} -b_1 < \sum _{i=1}^{t-r+1}(a_i-b_{r+i-1})$. Then
    $$\dim W(\underline{b};\underline{a};r) = \lambda_c \ .$$
  \end{conj}

  In particular, we would like to know if the above conjecture is at least
  true when the entries of $\cA$ are all of fixed degree $e\ge 1$. Note that
  the main assumption $a_{t+c-1} -b_1 < \sum _{i=1}^{t-r+1}(a_i-b_{r+i-1})$
  holds in this case. More precisely,

\begin{conj} \label{conj1bisbis} Fix integers $r \ge 1$, $c > 2-r$, $t\ge 2$
  and $d\ge 1$. Let $\cA$ be a homogeneous $t \times (t+c-1)$ matrix with
  entries forms of degree $e$. Let $A=R/I_{t-r+1}(\cA)$ and suppose
  $\Proj(A) \in W(\underline{0};\underline{d};r)$, $\dim A\ge 2$ for $c \ne 1$
  and $\dim A\ge 3$ for $c = 1$. Then,
   $$\dim W(\underline{b};\underline{a};r) = t(t+c-1){e+n\choose
     n}-t^2-(t+c-1)^2+1.$$
  \end{conj}

  It will be also interesting to drop the assumption
  $a_{t+c-1} -b_1 < \sum _{i=1}^{t-r+1}(a_i-b_{r+i-1})$ and find
  $\dim W(\underline{b};\underline{a};r)$ more generally. To give our best
  guess, we delete $c-j$ columns from the right-hand side of the matrix $\cA$
  of $\varphi ^*$ to define $\varphi ^*_{t+j-1}$, and we let
  $A_j= R/I_{t-r+1}(\varphi ^*_{t+j-1})$ and $N_j=\coker(\varphi ^*_{t+j-1})$
  for $2-r \le j \le c$.

  \begin{question} \label{questkap} Fix integers $t\ge 2$,
    $a_1\le a_2\le \cdots \le a_{t+c-1}$ and
    $b_1\le b_2\le \cdots ..\le b_{t}$. Let $\cA$ be a homogeneous
    $t \times (t+c-1)$ matrix with entries forms of positive degree $a_j-b_i$,
    let $A=R/I_{t-r+1}(\cA)$ where $1 \le r \le t-1$, $2-r \le c$. Suppose
    $\Proj(A) \in W(\underline{b};\underline{a};r)$ and
    $\dim A \ge 2$. If $c = 1$, we also suppose $\dim A\ge 3$. When is
    $$\dim W(\underline{b};\underline{a};r) = \lambda_c +K_3+K_4+\cdots +K_c
    -\kappa $$
    where $\kappa$ is the following non-negative integer:
    $\kappa= \sum_{j=3-r}^c( \dim_k(N_j)_{(a_{t+j-1})}- \dim_k(N_j \otimes
    A_j)_{(a_{t+j-1})})$?
\end{question}

There is at least one case where this guess is correct, namely for $r=1$.
Indeed, then $N_j$ is annihilated by $I_{t}(\varphi ^*_{t+j-1})$, so
$N_j \otimes A_j \cong N_j$ and $\kappa = 0$ and we get
$\dim W(\underline{b};\underline{a};r)$ as above by Theorem~\ref{Amodulethm5}.
There is one more case where we expect the answer to be true, namely in the
case where $a_{t+c-1} -b_1 < \sum _{i=1}^{t-r+1}(a_i-b_{r+i-1})$ because
Lemma~\ref{dimMI}(i) for $v=0$ and $a_{t+j-1} \le a_{t+c-1}$ show
$\dim (N_j \otimes A_j)_{(a_{t+j-1})} = \dim (N_j)_{(a_{t+j-1})}$, i.e. that
$\kappa=0$. Hence we get Conjecture~\ref{conj1}. Unfortunately in all other
cases we are aware of, $\kappa \ne 0$. So the correction term $\kappa$ in
Question~\ref{questkap} should not be skipped. Moreover due to
Theorem~\ref{corWsmoothnew2} and the fact that the assumptions of
Theorem~\ref{corWsmoothnew2} are designed to be true as far as our results and
computations by Macaulay2 indicate, except for the inequality below, we expect
\begin{conj} \label{conj1bis} With notations and assumptions as
  in Question~\ref{questkap}, if $b_r-b_1 < s_r -a_{t-r+1}$, then
$$\dim W(\underline{b};\underline{a};r) = \lambda_c +K_3+K_4+\cdots +K_c
-\kappa\ .$$
 \end{conj}

Unfortunately it seems difficult to compute $\kappa$, but
Theorems~\ref{dimW1} and \ref{corWsmoothnew2} give some answers
under restrictive assumptions.

  \begin{conj} \label{conj2}
  Fix integers  $r \ge 1$, $c > 2-r$,  $t\ge 2$,  $a_1\le a_2\le \cdots \le a_{t+c-1}$ and $b_1\le b_2\le \cdots ..\le b_{t}$. Let $\cA$ be a  homogeneous
  $t \times (t+c-1)$ matrix with entries forms of degree $a_j-b_i$.
  Let $A=R/I_{t-r+1}(\cA)$ and suppose
  $\Proj(A) \in W(\underline{b};\underline{a};r)$,
   $\dim A \ge 4$ for $c = 1$ and $\dim A\ge 3$ for $c \ne 0$.  If  $a_1 >
   b_t$ then
   $ \overline{W(\underline{b};\underline{a};r)}$ is a generically smooth
   irreducible component of $\Hilb ^{p_X}(\PP^n)$ and every deformation of $A$
   comes from deforming $\cA$.
  \end{conj}

  Related to the smoothness of the Hilbert scheme $\Hilb ^{p_X}(\PP^n)$ along
  $ W(\underline{b};\underline{a};r)$ and to whether any deformation of a
  determinantal scheme comes from deforming its associated homogeneous matrix,
  we have the following conjecture involving the codepth of the normal
  modules:

  \begin{conj} \label{conj3}
 Let $r \ge 1$, $c \ge 2-r$ and let $A=R/I_A$
    (resp. $B=R/I_B$ if $c > 2-r$) be defined by the vanishing of the
    $(t-r+1)\times (t-r+1)$ minors of a general $t\times (t+c-1)$ matrix $\cA$
    with $a_1 > b_t$ (resp. of $\cB$ obtained by deleting a column of
    $\cA$). Let $N_A:=\Hom_R(I_A,A)$ and suppose that $\dim A \ge 3$.
\begin{itemize}
\item[{\rm (i)}] If $c \notin \{0,1,2\}$ then \ $\codepth(N_A)= 1\,.$
 If $c=1$ (resp. $c =0, 2$), then
    $\codepth(N_A)= \min\{ 2,\dim A - 2\}$ (resp. $\codepth(N_A)= 0$).
    \item[{\rm (ii)}] Let $(r,c)\ne (1,2)$, $c\ge 2$ (resp. $3-r\le c\le 1$)
      then \
    $\codepth\,(\Hom_R(I_B,I_{A}/I_{B})) = r \ (resp.\ r+1).$
   \end{itemize}
    \end{conj}

    As we proved in Proposition \ref{genericdetring}, resp. in  Remark
    \ref{conjr=1}, this last conjecture essentially holds for generic
    determinantal algebras, resp.  for $r=1$.

    All results of this work deal with determinantal varieties of positive
    dimension and it will be very interesting generalize them to 0-dimensional
    schemes since so far only results for standard determinantal 0-dimensional
    schemes are known (see \cite{KM2011} and \cite{K2018}).

    \vskip 2mm We end this last chapter with three concrete problems that we
    have addressed in this work but unfortunately have not been able to
    completely solve. Their resolution will help to answer the above
    conjectures.

\begin{pblm} \rm
\begin{itemize} \hspace{0mm}
\item[(1)]  \ Show that:
$\dim\, _0\!\Hom_B(I_B/I^2_B, I_{A/B})=\sum _{j=1}^{t+c-2}
{a_j-a_{t+c-1}+n\choose n}$ if\  $\dim A \ge 3$\,.

\item[(2)] \ Show that: \ $_0\!\Ext^1_B(I_B/I^2_B, I_{A/B})= 0$
  if $ \dim A \ge 3$ (resp. $\dim A \ge 4$)  for $c \ne 1$ (resp. $c=1$)\,.

\item[(3)] \ Determine $ \dim (MI\otimes A)_{a_{t+c-1}}$ in the case:
  $a_{t+c-1} -b_1 \ge \sum _{i=1}^{t-r+1}(a_i-b_{r+i-1})$\,.
\end{itemize}
\end{pblm}


\newpage
\section{Appendix}

Many examples included in this work have been or can be computed using Macaulay2 \cite{Mac2}. We include the code of some of these examples  with the aim that the reader will be able to reproduce them.

\begin{example}\rm Set $R=k[x_0,\cdots ,x_{15}]$ and let
  $\PP ^{15}=\Proj (R)$. We consider a $4\times 4$ matrix $\cA=[\cB,v]$ with
  linear (resp. quadratic) entries in the first, second, third (resp. fourth)
  column. As we saw in Example \ref{ex1dimW},
  $$\dim \overline{ W(0^4;1^3,2;2)}=663=\dim _{(X)} \Hilb (\PP^{15})$$ where
  $X=\Proj(A)$ and $A=R/I_3(\cA)$. Let us check it with Macaulay2 .

\vskip 2mm
\begin{verbatim}
Macaulay2, version 1.9.2
with packages: ConwayPolynomials, Elimination, IntegralClosure, LLLBases,
PrimaryDecomposition, ReesAlgebra, TangentCone

i1 : kk=ZZ/3001;
i2 : R=kk[x0,x1,x2,x3,x4,x5,x6,x7,x8,x9,x10,x11,x12,x13,x14,x15]
i3 : IA3=minors(3,matrix{{x0,x1,x2,x3^2},{x4,x5,x6,x7^2},{x8,x9,x10,x11^2},
     {x12,x13,x14,x15^2}});
i4 : nA=Hom(IA3,R^1/IA3); apply(-3..5,i->hilbertFunction(i, nA))
o5 = (0, 4, 73, 663, 4087, 19418, 76153, 257315, 771393)
i6 : A=R/IA3;
i7 : IAIA=IA3*IA3; CoNA=IA3/IAIA;   %conormal module of A
i9 : Ext^1(CoNA**A,A)==0
o9 = true
\end{verbatim}

\vskip 2mm
\noindent
So the Hilbert scheme $\Hilb (\PP^{15})$ is smooth at $(X)$ by o9 of dimension
663 by o5. \vskip 2mm

\begin{verbatim}
i10 : use R
i11 : IB3=minors(3,matrix{{x0,x1,x2},{x4,x5,x6},{x8,x9,x10},{x12,x13,x14}});
i12 : nB=Hom(IB3,R^1/IB3);  apply(-3..5,i->hilbertFunction(i, nB))
o13 = (0, 0, 12, 168, 1260, 6720, 28560, 102816, 325584)
i14 : IAB=IA3/IB3;
i15 : Fibp2=Hom(IAB,R^1/IA3); apply(-3..5,i->hilbertFunction(i, Fibp2))
o16 = (0, 4, 61, 495, 2830, 12742, 47997, 157183, 459628)
i17 : Fibp1=Hom(IB3,IAB); apply(-3..5,i->hilbertFunction(i, Fibp1))
o18 = (0, 0, 0, 0, 3, 44, 404, 2684, 13819)
\end{verbatim}

\vskip 2mm
\noindent
So $\dim W(0^4;1^3,2;2) =  168$ (by o13) + 495 (by o16) - 0 (by o18) = 663\,.
\end{example}
\vskip 2mm
\noindent
\begin{example} \rm As in Example \ref{exgendet}(i) we consider a $3\times
  (c+2)$ generic matrix $\cA$, $1\le c\le  7$ and the ideal generated by all
  $2\times 2$ minors of $\cA$. Let us start with the case $c=1$.

\vskip 2mm

\begin{verbatim}
i55 : kk=ZZ/101;  %COMPUTED ALSO OVER : kk=ZZ/701; AND : kk=ZZ/3001;
i56 : R=kk[x0,x1,x2,x3,x4,x5,x6,x7,x8]
i59 : IB2=minors(2,matrix{{x0,x1},{x3,x4},{x6,x7}});
i60 : IA2=minors(2,matrix{{x0,x1,x2},{x3,x4,x5},{x6,x7,x8}});
i62 : codim IA2
o62 = 4
i63 : dim IA2
o63 = 5
i64 : nB=Hom(IB2,R^1/IB2); apply(-3..5,i->hilbertFunction(i, nB))
o65 = (0, 0, 6, 42, 168, 504, 1260, 2772, 5544)
i68 : nA=Hom(IA2,R^1/IA2); apply(-3..5,i->hilbertFunction(i, nA))
o69 = (0, 0, 9, 64, 225, 576, 1225, 2304, 3969)
i70 : IAB=IA2/IB2;
i71 : Fibp1=Hom(IB2,IAB); apply(-3..5,i->hilbertFunction(i, Fibp1))
o72 = (0, 0, 0, 2, 33, 168, 560, 1476, 3339)
\end{verbatim}
\vskip 1mm
\noindent
Only the value 2 in o72 is needed in Example \ref{exgendet} because then
(\ref{codgen})
holds. As we have seen in section 6, $\dim (nA)_0 = 8(c+2)^2-8$, which
is the number 64 in o69, so as a Macaulay2 check; $8(c+2)^2-8=64$ for c=1.
Similarly as  a Macaulay2 check; we also have $\dim (nB)_0 = 42$ which agrees
with the result obtained in the paper.

\vskip 2mm
For $c=2$, we have:

\begin{verbatim}
i1 : kk=ZZ/101;     % COMPUTED ALSO OVER : kk=ZZ/701; AND : kk=ZZ/3001;
i8 : R=kk[x0,x1,x2,x3,x4,x5,x6,x7,x8,x9,x10,x11]
i19 : IA2=minors(2,matrix{{x0,x1,x2,x9},{x3,x4,x5,x10},{x6,x7,x8,x11}});
i20 : IB2=minors(2,matrix{{x0,x1,x2},{x3,x4,x5},{x6,x7,x8}});
i27 : nA=Hom(IA2,R^1/IA2); apply(-3..5,i->hilbertFunction(i, nA))
o28 = (0, 0, 12, 120, 540, 1680, 4200, 9072, 17640)
i31 : IAB=IA2/IB2;
i32 : Fibp1=Hom(IB2,IAB); apply(-3..5,i->hilbertFunction(i, Fibp1))
o33 = (0, 0, 0, 3, 81, 525, 2103, 6450, 16614)
\end{verbatim}

\vskip 1mm
\noindent
Again only the value 3 in o33 is needed in Example \ref{exgendet} because then
(\ref{codgen}) holds. We use Macaulay2 to check:
$\dim \overline{W(0^3;1^4;2)}=8(c+2)^2-8=120$ (see value 120 in o28).
Analogously, for $3\le c\le 7$, we verify (not always over all three fields
above, and sometimes $kk=ZZ/13$ is used) with the code above that
(\ref{codgen}) holds and then check
$\dim \overline{W(0^3;1^4;2)}=8(c+2)^2-8=120$. All of Example \ref{exgendet}
is verified using (\ref{codgen}).

When the size of the matrices become large, it is much faster to verify
formula (\ref{exgenassump}) as we do in Example \ref{ex85}.
\end{example}
\vskip 2mm
\noindent

\begin{example} \label{ex623} \rm In Example \ref{dimW2} we consider a general
  $3\times 3$ matrix $\cA=[\cB,v]$ with degree matrix   $\left(\begin{smallmatrix}1 & 1  & 2 \\
      1 & 1 & 2 \\ 1 & 1 & 2 \end{smallmatrix}\right)$,
  and we let $A$ and $B$ be given by the ideal of all
  $2\times 2$ minors of $\cA$ and $\cB$.
\vskip 2mm
\noindent
\begin{verbatim}
GENERAL DETERMINANTAL,  dim A = 5  (A=R/IA2)
i1 : kk=ZZ/101;
i2 : R=kk[x0,x1,x2,x3,x4,x5,x6,x7,x8]
i3 : IA2=minors(2,matrix{{x0,x1,x2^2},{x3,x4,x5^2},{x6,x7,x8^2}});
i4 : IB2=minors(2,matrix{{x0,x1},{x3,x4},{x6,x7}});
i5 : nA=Hom(IA2,R^1/IA2); apply(-3..4,i->hilbertFunction(i, nA))
o6 = (0, 3, 31, 152, 502, 1286, 2776, 5312)
i7 : IAIA=IA2*IA2; CoNA=IA2/IAIA; A=R/IA2;
i10 : Ext^1(CoNA**A,A)
o10 = 0
o10 : A-module
\end{verbatim}
\vskip 2mm
\noindent
So the Hilbert scheme $\Hilb (\PP^{8})$ is smooth at $(X=\Proj(A))$ by o10 of dimension
152 by o6.
\vskip 2mm
\begin{verbatim}
i11 : use R
i12 : nB=Hom(IB2,R^1/IB2);  apply(-3..4,i->hilbertFunction(i, nB))
o13 = (0, 0, 6, 42, 168, 504, 1260, 2772)
i14 : IAB=IA2/IB2;
i15 : Fibp2=Hom(IAB,R^1/IA2); apply(-3..4,i->hilbertFunction(i, Fibp2))
o16 = (0, 3, 25, 110, 336, 815, 1693, 3150)
i5 : MI=coker matrix{{x0,x1,x2^2},{x3,x4,x5^2},{x6,x7,x8^2}};
i6 : Fibp2t=MI**R^1/IA2; apply(-3..5,i->hilbertFunction(i, Fibp2t))
o7 = (0, 0, 0, 3, 25, 110, 336, 815, 1693)
i17 : Fibp1=Hom(IB2,IAB); apply(-3..4,i->hilbertFunction(i, Fibp1))
o18 = (0, 0, 0, 0, 2, 33, 177, 610)
\end{verbatim}
\vskip 2mm
\noindent
By Proposition~\ref{propo8}, $\dim \overline{W(0^3;1^2,2;2)}$ = 42 (by o13) +
110 (by o16 or o7) - 0 (by o18) = 152, thus $\overline{W(0^3;1^2,2;2)}$ is a
generically smooth component of $\Hilb (\PP^{8})$ by o6. The case $\dim A = 4$
is computed in the same way.
\vskip 2mm
\noindent
Now $\dim A = 3:$
\begin{verbatim}
i18 : R=kk[x0,x1,x2,x3,x4,x5,x6]
i19 :  l0=random(1,R); l1=random(1,R);
i21 :  m0=random(2,R); m1=random(2,R);
i23 : IA2=minors(2,matrix{{x0,x1,x2^2},{x3,x4,x5^2},{x6,l0,m0}});
i24 : IB2=minors(2,matrix{{x0,x1},{x3,x4},{x6,l0}});
i25 : nA=Hom(IA2,R^1/IA2); apply(-3..4,i->hilbertFunction(i, nA))
o26 = (0, 3, 25, 94, 230, 434, 706, 1046)
i28 : nB=Hom(IB2,R^1/IB2);  apply(-3..4,i->hilbertFunction(i, nB))
o29 = (0, 0, 6, 30, 90, 210, 420, 756)
i30 : IAB=IA2/IB2;
i31 : Fibp2=Hom(IAB,R^1/IA2); apply(-3..4,i->hilbertFunction(i, Fibp2))
o32 = (0, 3, 19, 63, 141, 253, 399, 579)
i9 : MI=coker matrix{{x0,x1,x2^2},{x3,x4,x5^2},{x6,l0,m0}};
i10 : Fibp2t=MI**R^1/IA2; apply(-3..5,i->hilbertFunction(i, Fibp2t))
o11 = (0, 0, 0, 3, 19, 63, 141, 253, 399)
i33 : Fibp1=Hom(IB2,IAB); apply(-3..4,i->hilbertFunction(i, Fibp1))
o34 = (0, 0, 0, 0, 2, 29, 113, 289)
i35 : IAIA=IA2*IA2; CoNA=IA2/IAIA; A=R/IA2;
i38 : EA=Ext^5(CoNA,R);   % A is GORENSTEIN, so EA(-7)= Ext^1(CoNA**A,A)
i39 : apply(-11..-4,i->hilbertFunction(i, EA))
o39 = (0, 2, 2, 0, 0, 0, 0, 0)
\end{verbatim}
By Proposition~\ref{propo8} $\dim \overline{W(0^3;1^2,2;2)}$ = 30 (in o29) +
63 (in o32 or o11) = 93 while $\dim_{(X)} \Hilb (\PP^{6})$ = 94 by o26 and
o39, thus $\overline{W(0^3;1^2,2;2)}$ is of codimension 1 in
$\Hilb (\PP^{6})$. Using the isomorphism in o38 finding
$\Ext^1_A(I_A/I_A^2,A)$ allowed a much faster computation.

\noindent
The case $\dim A = 2$ is computed in the same way and shows
$\codim \overline{W(0^3;1^2,2;2)}= 3$ in $\Hilb (\PP^{5})$.
\end{example}

\vskip 2mm
\noindent
   \begin{example} \label{ex713} \rm In Example~\ref{examples712}(i)
    $\cA=[\cB,v]$ is a general $3\times 5$ matrix with linear entries, and let
    $A$ and $B$ be the quotients of $R$ defined by the $2 \times 2$ minors of
    $\cA$ and $\cB$, respectively. So $\dim A = 2$. Let us check that
    $\overline{W(0^3;1^5;2)}$ is not a generically smooth component of
    $\Hilb (\PP^{9})$ by using \eqref{thm61cond}. \vskip 1mm
\noindent
\begin{verbatim}

GENERAL DETERMINANTAL, A defined by 2 x 2 minors of a 3 x 5 matrix. dim A=2
i1 : kk=ZZ/13;
i3 : R=kk[x0,x1,x2,x3,x4,x5,x6,x7,x8,x9]
i4 : l0=random(1,R); l1=random(1,R); l2=random(1,R);
i7 : l3=random(1,R); l4=random(1,R);
i12 : IA2=minors(2,matrix{{x0,x1,x2,x3,l0},{x4,x5,x6,l4,l1},{x7,x8,x9,l3,l2}});
i13 : IB2=minors(2,matrix{{x0,x1,x2,x3},{x4,x5,x6,l4},{x7,x8,x9,l3}});
i14 : dim IA2
o14 = 2
i15 : dim IB2
o15 = 4
i16 : codim IA2
o16 = 8
i17 : nB=Hom(IB2,R^1/IB2); apply(-3..5,i->hilbertFunction(i, nB))
o18 = (0, 0, 12, 96, 312, 720, 1380, 2352, 3696)
i25 : IAB=IA2/IB2;
i18 : Fibp2=Hom(IAB,R^1/IA2); apply(-3..5,i->hilbertFunction(i, Fibp2))
o19 = (0, 0, 3, 25, 55, 85, 115, 145, 175)
i26 : Fibp1=Hom(IB2,IAB); apply(-3..5,i->hilbertFunction(i, Fibp1))
o27 = (0, 0, 0, 4, 127, 445, 1015, 1897, 3151)
i28 : nA=Hom(IA2,R^1/IA2); apply(-3..5,i->hilbertFunction(i, nA))
\end{verbatim}
\vskip 2mm
\noindent
Aborted, but we try again.
\vskip 2mm
\begin{verbatim}
i2 : kk=ZZ/3;
i3 : R=kk[x0,x1,x2,x3,x4,x5,x6,x7,x8,x9]
i4 : l0=random(1,R); l1=random(1,R); l2=random(1,R);
i7 : l3=random(1,R); l4=random(1,R);
i10 : IA2=minors(2,matrix{{x0,x1,x2,x3,l0},{x4,x5,x6,l4,l1},{x7,x8,x9,l3,l2}});
i11 : nA=Hom(IA2,R^1/IA2); apply(-3..5,i->hilbertFunction(i, nA))
o12 = (0, 0, 15, 120, 240, 360, 480, 600, 720)
\end{verbatim}
\vskip 2mm
\noindent
Hence $ \overline{W(0^3;1^5;2)}$ is not a generically smooth component because
120 (in o12) $\ne$ 96 (in o18) + 25 (in o19) - 4 (in o27), i.e.
\eqref{thm61cond} is not satisfied (because
$\dim (MI \otimes A)_1= \dim {\rm (Fibp2)}_0$). But to apply
Remark~\ref{remthm61}(2) we need to show that the conditions of
Theorem~\ref{Wsmooth} hold. Let us check that
$ _0\!\Ext^1_C(I_C/I_C^2,I_{B/C})=0$ because then every deformation of $B$
comes from deforming $\cB$.

\vskip 2mm
\begin{verbatim}
i10 : IB2=minors(2,matrix{{x0,x1,x2,x3},{x4,x5,x6,l4},{x7,x8,x9,l3}});
i13 : IC2=minors(2,matrix{{x0,x1,x2},{x4,x5,x6},{x7,x8,x9}});
i14 : IBC=IB2/IC2;
i15 : ICIC=IC2*IC2; CoNC=IC2/ICIC; C=R/IC2;
i19 : Ext^1(CoNC**C,IBC**C)==0
o19 = true
\end{verbatim}
\vskip 2mm
\noindent
The computations by Macaulay2 in Example~\ref{examples712} often check if
\eqref{thm61cond} holds, or even simpler for matrices in the {\it linear} case
where we only need to compute $\dim {\rm (Fibp1)}_0$ to see if
$\dim \overline{W(0^t;1^{t+c-1};r)}=\lambda_c$ holds. If so and if
$\lambda_c= \dim (nA)_0$, (or if \eqref{thm61cond} holds) and the other
conditions of Theorem~\ref{Wsmooth} hold, then $\overline{W(0^t;1^{t+c-1};r)}$
is a generically smooth component. And conversely for \eqref{thm61cond} by
Remark~\ref{remthm61}(2), as in the example above. Also in the linear case
verifying $\dim \overline{W(0^t;1^{t+c-1};r)}=\lambda_c < \dim (nA)_0$ and
$\Hilb (\PP^{n})$ smooth at $(X)$ imply a converse.
\end{example}
\vskip 2mm
\noindent
\begin{example} \rm In Example \ref{exbcn}  we consider $3\times (c+2)$
  linear and non-linear matrices $\cA=[\cB,v]$, $1\le c\le 4$ with
  number of variables of $R$ at least as in the generic case. The ideals of $A$
  and $B$ are generated by all $2\times 2$ minors of $\cA$ and $\cB$.
\vskip 2mm
\noindent
\begin{verbatim}
GENERIC 3 by 4 matrix \cA=[\cB,v] defining A
i1 : kk=ZZ/101;     %computed also over  kk=ZZ/701;
i2 : R=kk[x0,x1,x2,x3,x4,x5,x6,x7,x8,x9,x10,x11]
i3 : IA2=minors(2,matrix{{x0,x1,x2,x3},{x4,x5,x6,x7},{x8,x9,x10,x11}});
i4 : IB2=minors(2,matrix{{x0,x1,x2},{x4,x5,x6},{x8,x9,x10}});
i6 : IAB=IA2/IB2;
i7 : nB=Hom(IB2,R^1/IB2); pdim nB
o9 = 6
i10 : B=R^1/IB2; pdim B
o11 = 4
i12 : nA=Hom(IA2,R^1/IA2); pdim nA
o14 = 6
i15 : A=R^1/IA2; pdim A
o16 = 6
i62 : Fibp1=Hom(IB2,IA2/IB2); pdim Fibp1
o63 = 6
\end{verbatim}
\noindent

\begin{verbatim}
NON-LINEAR, thus NON-GENERIC A and B.
i3 : IA2=minors(2,matrix{{x0,x1,x2^2,x3^3},{x4,x5,x6^2,x7^3},{x8,x9,x10^2,x11^3}});
i4 : IB2=minors(2,matrix{{x0,x1,x2^2},{x4,x5,x6^2},{x8,x9,x10^2}});
i5 : IAB=IA2/IB2;
i6 : nB=Hom(IB2,R^1/IB2); pdim nB
o7 = 6
i8 : B=R^1/IB2; pdim B
o9 = 4
i10 : nA=Hom(IA2,R^1/IA2); pdim nA
o11 = 6
i12 : A=R^1/IA2; pdim A
o13 = 6
i14 : Fibp1=Hom(IB2,IA2/IB2); pdim Fibp1
o15 = 6
\end{verbatim}
\noindent

\begin{verbatim}
LINEAR, NON-GENERIC $A$ of $\dim A = 3$.
i1 : kk=ZZ/101;       % Computed also over   kk=ZZ/11;
i2 : R=kk[x0,x1,x2,x3,x4,x5,x6,x7,x8]
i3 :  l0=random(1,R); l1=random(1,R); l2=random(1,R);
i6 : IB2=minors(2,matrix{{x0,x1,x2},{x3,x4,x5},{x6,x7,x8}});
i7 : IA2=minors(2,matrix{{x0,x1,x2,l0},{x3,x4,x5,l1},{x6,x7,x8,l2}});
i8 : IAB=IA2/IB2;
i9 : nB=Hom(IB2,R^1/IB2); pdim nB
o10 = 6
i11 : B=R^1/IB2; pdim B
o12 = 4
i13 : nA=Hom(IA2,R^1/IA2); pdim nA
o14 = 6
i15 : A=R^1/IA2; pdim A
o16 = 6
i17 : Fibp1=Hom(IB2,IA2/IB2); pdim Fibp1
o18 = 6
\end{verbatim}

From the projective dimensions and Auslander-Buchsbaum's formula, we get that $\codepth nB = 2$ ($B$ is
Gorenstein) while  $\codepth nA = 0$ ($c=2$) and
$\codepth {\rm Fibp1} = 2$ ($r=2$) in the linear, as well as in the non-linear
case, confirming Conjecture~\ref{conjdepthNb} and
Proposition~\ref{genericdetring}. All examples in Example \ref{exbcn} are
computed in this way (or by using ``betti res'' in replacment of ``pdim'')
\end{example}
\vskip 2mm
\noindent
\begin{example} \label{ex85} \rm In Example \ref{exgendetcor83} $\cA=(x_{ij})$
  is a $3\times (c+2)$ generic matrix, $6\le c\le 18$ and $I_A=\ker(R \to A)$
  the ideal generated by all $2\times 2$ minors of $\cA$. For $c=6$ we get

\begin{verbatim}
i60 : kk=ZZ/101; % COMPUTED ALSO OVER : kk=ZZ/701; AND : kk=ZZ/3001;
i61 : R=kk[x0,x1,x2,x3,x4,x5,x6,x7,x8,x9,x10,x11,x12,x13,x14,x15,x16,x17,a,b,c,d,e,f]
i63 : IA2=minors(2,matrix{{x0,x1,x2,x3,x4,x5,a,d},{x6,x7,x8,x9,x10,x11,b,e},
      {x12,x13,x14,x15,x16,x17,c,f}});
i64 : codim IA2
o64 = 14
i66 : IB2=minors(2,matrix{{x0,x1,x2,x3,x4,x5,a},{x6,x7,x8,x9,x10,x11,b},
      {x12,x13,x14,x15,x16,x17,c}});
i67 : IAB=IA2/IB2;
i68 : IG2=minors(2,matrix{{x0,x1,x2},{x6,x7,x8},{x12,x13,x14}});
i69 : nGB=Hom(IG2,R^1/IB2); apply(-3..5,i->hilbertFunction(i, nGB))
o70 = (0, 0, 9, 187, 1689, 9891, 43903, 160173, 504231)
i73 : nGIAB=Hom(IG2,IAB); apply(-3..5,i->hilbertFunction(i, nGIAB))
o74 = (0, 0, 0, 3, 189, 2115, 13453, 61965, 229803)
i75 : nGA=Hom(IG2,R^1/IA2); apply(-3..5,i->hilbertFunction(i, nGA))
o76 = (0, 0, 9, 184, 1500, 7776, 30450, 98208, 274428)
\end{verbatim}

\vskip 2mm
\noindent
Only the value 3 in o74 is needed in Example \ref{exgendetcor83} to get
$\dim \overline{W(0^3;1^{c+2};2)}=\lambda_c$ because then (\ref{exgenassump})
holds. For a Macaulay2 check, to see if $\overline{W(0^3;1^{c+2};2)}$ is a
generically smooth component (which we know is true by
Proposition~\ref{genericdetring}), note that \eqref{thm61cond} is equivalent
to (\ref{a3r}). To check (\ref{a3r}) we subtract: 187 (in o70) - 184 (in o76)
= 3 (in o74). Analogously we compute the cases $7\le c\le 18$ and the
remainder of Example \ref{exgendetcor83} by verifying (not always over all
three fields above) that (\ref{exgenassump}) holds. Thus we get
$\dim \overline{W(0^3;1^{c+2};2)}=\lambda_c$. As a Macaulay2 check we also
verify (\ref{a3r}) which should hold by our results.
\end{example}

\vskip 2mm
\begin{example} \label{ex817} \rm In Example \ref{exgendetWit}(i) $\cA$ is a
  $3\times 6$ linear matrix, and we successively delete columns to get a flag
  of determinantal rings. In such a situation it is much faster to verify
  (\ref{a3r}) than using \eqref{thm61cond} even though they are equivalent.

\begin{verbatim}
i2 : kk=ZZ/101;
i3 : R=kk[x0,x1,x2,x3,x4,x5,x6,x7,x8,x9,x10,x11,x12]
i4 : l0=random(1,R); l1=random(1,R); l2=random(1,R); l3=random(1,R); l4=random(1,R);
i9 : IB2=minors(2,matrix{{x0,x1,x2,x3,x12},{x4,x5,x6,x7,l0},{x8,x9,x10,x11,l4}});
i10 : IG2=minors(2,matrix{{x0,x1,x2},{x4,x5,x6},{x8,x9,x10}});
i11 : nGB=Hom(IG2,R^1/IB2); apply(-3..2,i->hilbertFunction(i, nGB))
o12 = (0, 0, 9, 94, 385, 1072)
i13 : IC2=minors(2,matrix{{x0,x1,x2,x3},{x4,x5,x6,x7},{x8,x9,x10,x11}});
i14 : IBC=IB2/IC2;
i15 : nGC=Hom(IG2,R^1/IC2); apply(-3..2,i->hilbertFunction(i, nGC))
o16 = (0, 0, 9, 97, 487, 1687)
i17 : nGIBC=Hom(IG2,IBC); apply(-3..2,i->hilbertFunction(i, nGIBC))
o18 = (0, 0, 0, 3, 102, 615)
i19 : IA2=minors(2,matrix{{x0,x1,x2,x3,x12,l1},{x4,x5,x6,x7,l0,l2},
      {x8,x9,x10,x11,l4,l3}});
i20 : dim IA2
o20 = 3
i21 : IAB=IA2/IB2;
i22 : nGA=Hom(IG2,R^1/IA2); apply(-3..2,i->hilbertFunction(i, nGA))
o23 = (0, 0, 9, 91, 265, 523)
i24 : nGIAB=Hom(IG2,IAB); apply(-3..2,i->hilbertFunction(i, nGIAB))
o25 = (0, 0, 0, 3, 120, 549)
\end{verbatim}

\vskip 2mm
\noindent
To check (\ref{a3r}) for $C \to B$: 94 (in o12) + 3 (in o18) = 97 (in o16),
whence $\dim \overline{W(0^3;1^{5};2)}=\lambda_3$ and every deformation of $C$
comes from deforming its matrix by Theorem~\ref{corWsmoothcnew}. Similarly to
check (\ref{a3r}) for $B \to A$ we add: 91 (in o23) + 3 (in o25) = 94 (in
o12), whence $\dim \overline{W(0^3;1^{6};2)}=\lambda_4$ and
$\overline{W(0^3;1^{6};2)}$ is a generically smooth component of $\Hilb
  (\PP^{12})$ by Theorem~\ref{corWsmoothcnew}.
\end{example}
\vskip 2mm
\begin{example} \label{ex910} \rm In Example \ref{exex910}, $\cA$ is a
  $4\times 4$ matrix, and five smooth deformation functors are considered. By
  Theorem~\ref{defpower} the deformation functors of $M_r:=\bigwedge^rM$ and
  of $R/I_{t-r+1}(\varphi ^* )$ are isomorphic for fixed $r$. By
  Theorem~\ref{defWed} they should further be isomorphic to the deformation
  functor of $MI$ (and if $c > 1$ also to the deformation functor of
  $R/I_{t}(\varphi ^* )$, but in our example $c=1$). Theorem~\ref{defWed} are,
  however, true under the assumption: the compared functors have the same
  dimension of tangent spaces. Let us check this assumption by using
  Macaulay2. \vskip 2mm
\begin{verbatim}
GENERIC DETERMINANTAL, five isomorphic deformation functors
i1 : kk=ZZ/101;
i2 : R=kk[x0,x1,x2,x3,x4,x5,x6,x7,x8,x9,x10,x11,x12,x13,x14,x15]
i4 : MI=coker matrix{{x0,x1,x2,x3},{x4,x5,x6,x7},{x8,x9,x10,x11},
     {x12,x13,x14,x15}};
i5 : M2=exteriorPower(2,MI);
i6 : M3=exteriorPower(3,MI);
i11 : WED1=Ext^1(MI,MI); apply(-3..3,i->hilbertFunction(i, WED1))
o12 = (0, 0, 16, 225, 1696, 9096, 38896, 140963)
i13 : WED2=Ext^1(M2,M2); apply(-3..3,i->hilbertFunction(i, WED2))
o14 = (0, 0, 16, 225, 1696, 8996, 37600, 131882)
i15 : WED3=Ext^1(M3,M3); apply(-3..3,i->hilbertFunction(i, WED3))
o16 = (0, 0, 16, 225, 1296, 4900, 14400, 35721)
i18 : IA2=minors(2, matrix{{x0,x1,x2,x3},{x4,x5,x6,x7},{x8,x9,x10,x11},
      {x12,x13,x14,x15}});
i19 : IA3=minors(3, matrix{{x0,x1,x2,x3},{x4,x5,x6,x7},{x8,x9,x10,x11},
      {x12,x13,x14,x15}});
i20 : nA2=Hom(IA2,R^1/IA2); apply(-3..3,i->hilbertFunction(i, nA2))
o21 = (0, 0, 16, 225, 1296, 4900, 14400, 35721)
i22 : nA3=Hom(IA3,R^1/IA3); apply(-3..3,i->hilbertFunction(i, nA3))
o23 = (0, 0, 16, 225, 1696, 8996, 37600, 131882)
\end{verbatim}
\noindent
So the five deformation functors we consider have all tangent space of dimension
225 (see o12, o14, o16, o21, o23), so their corresponding deformation functors
are isomorphic.
\vskip 4mm
\noindent
In the non-linear case not all 5 deformation functors are necessarily isomorphic because
when we take all $b_i=0$, $a_1=a_2=1$ and $a_3=a_4=2$ we get
$s_2=\sum_{i=1}^{(5-2)}a_i =4$ and $s_3=\sum_{i=1}^{5-3}a_i=2$, whence
$a_4 < s_2$ is satisfied while $a_4 < s_3$ is not. So according to
Theorem~\ref{defWed} and Remark~\ref{rema97} the deformation functors of
$A_2=R/I_{3}(\varphi ^* )$ and $M_2$ which are isomorphic by
Theorem~\ref{defpower}, should further be isomorphic the deformation functor
of $MI$, while those of $A_3=R/I_{2}(\varphi ^* )$ and $M_3$ are isomorphic by
Theorem~\ref{defpower}. Let us check it \vskip 1mm
\noindent
\begin{verbatim}
NON-LINEAR : not all 5 deformation functors are isomorphic
i22 : MI=coker matrix{{x0,x1,x2^2,x3^2},{x4,x5,x6^2,x7^2},{x8,x9,x10^2,x11^2},
      {x12,x13,x14^2,x15^2}};
i23 : M2=exteriorPower(2,MI);
i24 : M3=exteriorPower(3,MI);
i25 : WED1=Ext^1(MI,MI); apply(-3..3,i->hilbertFunction(i, WED1))
o26 = (0, 8, 132, 1129, 6708, 31216, 121440)
i27 : WED2=Ext^1(M2,M2); apply(-3..3,i->hilbertFunction(i, WED2))
o28 = (0, 8, 132, 1129, 6708, 31186, 120944)
i29 : WED3=Ext^1(M3,M3); apply(-3..3,i->hilbertFunction(i, WED3))
o30 = (0, 8, 132, 1089, 5972, 24594, 81740)
i35 : IA2=minors(2, matrix{{x0,x1,x2^2,x3^2},{x4,x5,x6^2,x7^2},{x8,x9,x10^2,x11^2},
      {x12,x13,x14^2,x15^2}});
i36 : IA3=minors(3, matrix{{x0,x1,x2^2,x3^2},{x4,x5,x6^2,x7^2},{x8,x9,x10^2,x11^2},
      {x12,x13,x14^2,x15^2}});
i37 : nA2=Hom(IA2,R^1/IA2); apply(-3..3,i->hilbertFunction(i, nA2))
o38 = (0, 8, 132, 1089, 5972, 24594, 81740)
i39 : nA3=Hom(IA3,R^1/IA3); apply(-3..3,i->hilbertFunction(i, nA3))
o40 = (0, 8, 132, 1129, 6708, 31186, 120944)
\end{verbatim}
\vskip 2mm
\noindent
So three of the deformation functors we consider have the same tangential
dimension, namely equal to 1129 (see o26, o28, o40). Thus their corresponding
deformation functors are isomorphic. Moreover, the tangent space dimensions in
o30 and o38 coincide, so their deformation functors are isomorphic.

\vskip
4mm
\noindent
In the non-linear case, again 5 isomorphic deformation functors are expected
in the case
that all $b_i=0$, $a_1=1$ and $a_2=a_3=a_4=2$. We have $s_2=5$ and $s_3=3$,
whence both $a_4 < s_3$ and $a_4 < s_2$ are satisfied and Theorem~\ref{defWed}
predicts that all five deformation functors are isomorphic. Let us check it
\vskip 2mm
\begin{verbatim}
NON-LINEAR, but again all 5 isomorphic deformation functors are isomorphic
i45 : MI=coker matrix{{x0,x1^2,x2^2,x3^2},{x4,x5^2,x6^2,x7^2},{x8,x9^2,x10^2,x11^2},
      {x12,x13^2,x14^2,x15^2}};
i46 : M2=exteriorPower(2,MI);
i47 : M3=exteriorPower(3,MI);
i48 : WED1=Ext^1(MI,MI); apply(-3..3,i->hilbertFunction(i, WED1))
o49 = (0, 12, 193, 1623, 9529, 43956, 169852)
i50 : WED2=Ext^1(M2,M2); apply(-3..3,i->hilbertFunction(i, WED2))
o51 = (0, 12, 193, 1623, 9529, 43956, 169792)
i52 : WED3=Ext^1(M3,M3); apply(-3..3,i->hilbertFunction(i, WED3))
o53 = (0, 12, 193, 1623, 9369, 41344, 148212)
i55 : IA2=minors(2, matrix{{x0,x1^2,x2^2,x3^2},{x4,x5^2,x6^2,x7^2},
      {x8,x9^2,x10^2,x11^2},{x12,x13^2,x14^2,x15^2}});
i56 : IA3=minors(3, matrix{{x0,x1^2,x2^2,x3^2},{x4,x5^2,x6^2,x7^2},
      {x8,x9^2,x10^2,x11^2},{x12,x13^2,x14^2,x15^2}});
i57 : nA2=Hom(IA2,R^1/IA2); apply(-3..3,i->hilbertFunction(i, nA2))
o58 = (0, 12, 193, 1623, 9369, 41344, 148212)
i59 : nA3=Hom(IA3,R^1/IA3); apply(-3..3,i->hilbertFunction(i, nA3))
o60 = (0, 12, 193, 1623, 9529, 43956, 169792)

% 5 functors with the same tangential dimension = 1623 (o49, o51, o53, o58, o60)
\end{verbatim}
\end{example}

\end{document}